\documentclass[a4paper, 10pt]{report}
\usepackage{mathtools, microtype, mathrsfs, xfrac, amssymb, enumerate, xspace, stmaryrd, amsthm}
\usepackage[alphabetic]{amsrefs}
\usepackage{ascii}
\usepackage[latin1]{inputenc}
\usepackage{tikz-cd}
\usepackage[all,cmtip]{xy}
\usepackage{hyperref}
\hypersetup{
	colorlinks=true,
	linkcolor=blue,
	filecolor=magenta,      
	urlcolor=cyan,
}
\usepackage{faktor}
\usepackage{cleveref}
\usepackage{comment}
\usepackage{graphicx}
\usepackage{aliascnt}
\usepackage{float}

\numberwithin{equation}{section}

\numberwithin{subsection}{section}

\allowdisplaybreaks[1]

\newtheorem{theorem}{Theorem}[section]
\newtheorem{proposition}[theorem]{Proposition}
\newtheorem{proposition-definition}[theorem]
{Proposition-Definition}
\newtheorem{corollary}[theorem]{Corollary}
\newtheorem{lemma}[theorem]{Lemma}

\theoremstyle{definition}
\newtheorem{definition}[theorem]{Definition}

\newtheorem{example}[theorem]{Example}

\newtheorem{remark}[theorem]{Remark}

\theoremstyle{remark}

\newcommand\cA{\mathcal{A}} \newcommand\cB{\mathcal{B}}
\newcommand\cC{\mathcal{C}} \newcommand\cD{\mathcal{D}}
\newcommand\cE{\mathcal{E}} \newcommand\cF{\mathcal{F}}
 \newcommand\cH{\mathcal{H}}
 
 \newcommand\cL{\mathcal{L}}
\newcommand\cM{\mathcal{M}} 
\newcommand\cO{\mathcal{O}} \newcommand\cP{\mathcal{P}}
 
\newcommand\cS{\mathcal{S}} 
\newcommand\cU{\mathcal{U}} 
 \newcommand\cX{\mathcal{X}}
\newcommand\cY{\mathcal{Y}} \newcommand\cZ{\mathcal{Z}}

\renewcommand\AA{\mathbb{A}} 
\newcommand\CC{\mathbb{C}}  
\newcommand\GG{\mathbb{G}} \newcommand\HH{\mathbb{H}}

 \newcommand\NN{\mathbb{N}}
 \newcommand\PP{\mathbb{P}}
\newcommand\QQ{\mathbb{Q}} 
 
 \newcommand\VV{\mathbb{V}}
 
 \newcommand\ZZ{\mathbb{Z}}

\newcommand\rma{\mathrm{a}}

\newcommand\rmm{\mathrm{m}}

\newcommand\arr{\ifinner\to\else\longrightarrow\fi}

\newcommand\arrto{\ifinner\mapsto\else\longmapsto\fi}

\newcommand\larr{\longrightarrow}

\newcommand{\hooklongrightarrow}{\lhook\joinrel\longrightarrow}

\renewcommand\H{\operatorname{H}}

\newcommand\into{\hookrightarrow}

\newcommand\im[1]{\operatorname{im}(#1)}

\def\displaytimes_#1{\mathrel{\mathop{\times}\limits_{#1}}}

\def\displayotimes_#1{\mathrel{\mathop{\bigotimes}\limits_{#1}}}

\renewcommand\hom{\operatorname{Hom}}

\newcommand\ext{\operatorname{Ext}}

\newcommand\aut{\operatorname{Aut}}

\newcommand\tor{\operatorname{Tor}}

\newcommand\pic{\operatorname{Pic}}

\newcommand\spec{\operatorname{Spec}}

\newcommand\codim{\operatorname{codim}}

\newcommand{\proj}{\operatorname{Proj}}

\newcommand\id{\mathrm{id}}

\newcommand{\underhom}{\mathop{\underline{\mathrm{Hom}}}\nolimits}

\newcommand{\underaut}{\mathop{\underline{\mathrm{Aut}}}\nolimits}

\newlength{\ignora}

\renewcommand{\setminus}{\smallsetminus}

\newcommand{\mmu}{\boldsymbol{\mu}}
\newcommand{\gm}{\GG_{\rmm}}

\newcommand{\GL}{\mathrm{GL}}

\newcommand{\PGL}{\mathrm{PGL}}
\newcommand{\ga}{\GG_{\rma}}

\newcommand{\ds}[1]{[\mspace{-2mu}[#1]\mspace{-2mu}]}

\DeclareFontFamily{U}{mathx}{\hyphenchar\font45}
\DeclareFontShape{U}{mathx}{m}{n}{
	<5> <6> <7> <8> <9> <10>
	<10.95> <12> <14.4> <17.28> <20.74> <24.88>
	mathx10
}{}
\DeclareSymbolFont{mathx}{U}{mathx}{m}{n}
\DeclareFontSubstitution{U}{mathx}{m}{n}
\DeclareMathAccent{\widecheck}{0}{mathx}{"71}
\DeclareMathAccent{\wideparen}{0}{mathx}{"75}

\renewcommand{\epsilon}{\varepsilon}

\newcommand{\Mbar}{\overline{\cM}}

\newcommand{\Mtilde}{\widetilde{\mathcal M}}

\newcommand{\Ctilde}{\widetilde{\mathcal C}}

\newcommand{\Htilde}{\widetilde{\cH}}

\newcommand{\Hbar}{\overline{\cH}}

\newcommand{\ThTilde}{\widetilde{\Theta}}

\newcommand{\Detilde}{\widetilde{\Delta}}

\newcommand{\Debar}{\overline{\Delta}}

\newcommand{\ch}[1][*]{\operatorname{CH}^{#1}}

\newcommand{\mt}{\widetilde{\mathcal M}}

\newcommand{\htil}{\widetilde{\HH}}

\newcommand{\emptypage}{
	\newpage
	\thispagestyle{empty}
	\mbox{}
	\newpage
}

\title{$A_r$-stable curves \\ and the Chow ring of $\Mbar_{3}$}
\author{Michele Pernice}

\begin{document}

\thispagestyle{empty}

\begin{center}
\includegraphics[width=0.9\linewidth]{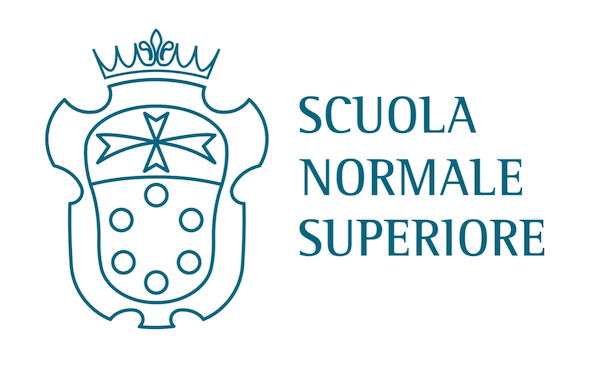}

\vspace{3cm}
Classe di SCIENZE\\
Corso di perfezionamento in MATEMATICA\\
XXXIV ciclo\\
Anno Accademico 2021-2022

\vfill 

{\Large $A_r$-stable curves and the Chow ring of $\Mbar_3$ \par}
\vfill
Tesi di perfezionamento in Matematica \\
MAT/03
\vfill
\textsc{Candidato}\\ \textit{Michele Pernice}

\vfill

\textsc{Relatore}\\
\textit{Prof. Angelo Vistoli} (Scuola Normale Superiore)

\end{center}   

\clearpage

\emptypage

\pagenumbering{Roman}
\tableofcontents

\emptypage

\chapter*{Introduction}
\addcontentsline{toc}{chapter}{Introduction}

The geometry of the moduli spaces of curves has always been the subject of intensive investigations, because of its manifold implications, for instance in the study of families of curves. One of the main aspects of this investigation is the intersection theory of these spaces, which has both enumerative and geometrical implication. In his groundbreaking paper \cite{Mum}, Mumford introduced the intersection theory with rational coefficients for the moduli spaces of stable curves. He also computed the Chow ring (with rational coefficients) of $\Mbar_2$, the moduli space of stable genus $2$ curves. While the rational Chow ring of $\cM_g$, the moduli space of smooth curves, is known for $2\leq g\leq 9$ (\cite{Mum}, \cite{Fab}, \cite{Iza}, \cite{PenVak}, \cite{CanLar}), the computations in the stable case are much harder. The complete description of the rational Chow ring has been obtained only for genus $2$ by Mumford and for genus $3$ by Faber in \cite{Fab}. In his PhD thesis, Faber also computed the rational Chow ring of $\Mbar_{2,1}$, the moduli space of $1$-pointed stable curves of genus $2$.

Edidin and Graham introduced in \cite{EdGra} the intersection theory of global quotient stacks with integer coefficients. It is a more refined invariant but as expected, the computations for the Chow ring with integral coefficients of the moduli stack of curves are much harder than the ones with rational coefficients. To date, the only complete description for the integral Chow ring of the moduli stack of stable curves is the case of $\Mbar_2$, obtained by Larson in \cite{Lar} and subsequently with a different strategy by Di Lorenzo and Vistoli in \cite{DiLorVis}. It is also worth mentioning the result of Di Lorenzo, Pernice and Vistoli regarding the integral Chow ring of $\Mbar_{2,1}$, see \cite{DiLorPerVis}.

The aim of this thesis is to describe the Chow ring with $\ZZ[1/6]$-coefficients of the moduli stack $\Mbar_3$ of stable genus $3$ curves. This provides a refinement of the result of Faber with a completely indipendent method. The approach is a generalization of the one used in \cite{DiLorPerVis}: we introduce an Artin stack, which is called the stack of $A_r$-stable curves, where we allow curves with $A_r$-singularities to appear. The idea is to compute the Chow ring of this newly introduced stack in the genus $3$ case and then, using localization sequence, find a description for the Chow ring of $\Mbar_3$. The stack $\Mtilde_{g,n}$ introduced in \cite{DiLorPerVis} is cointained as an open substack inside our stack. We state the main theorem. 

\begin{theorem}
	The Chow ring of $\Mbar_3$ with $\ZZ[1/6]$-coefficients is the quotient of the graded polynomial algebra 
	$$\ZZ[1/6,\lambda_1,\lambda_2,\lambda_3,\delta_{1},\delta_{1,1},\delta_{1,1,1},H]$$
	where 
	\begin{itemize}
		\item[] $\lambda_1,\delta_1,H$ have degree $1$, \item[]$\lambda_2,\delta_{1,1}$ have degree $2$, \item[]$\lambda_3,\delta_{1,1,1}$ have degree $3$.
	\end{itemize}
	The quotient ideal is generated by 15 homogeneous relations, where
	\begin{itemize}
		\item  $1$ of them is in codimension $2$,
		\item  $5$ of them are in codimension $3$,
		\item  $8$ of them are in codimension $4$,
		\item  $1$ of them is in codimension $5$.
	\end{itemize}  .
\end{theorem}

At the end of the thesis, we explain how to compare the result of Faber with ours and comment about the information we lose in the process of tensoring with rational coefficients.

\subsection*{Stable $A_r$-curves and the strategy of the proof}

The strategy for the computation is the same used in \cite{DiLorPerVis} for the integral Chow ring of $\Mbar_{2,1}$. Suppose we have a closed immersion of smooth stacks $\cZ \into \cX$ and we know how to compute the Chow rings of $\cZ$ and $\cX \setminus \cZ$. We would like to use the well-known localization sequence
$$ 
 \ch(\cZ) \rightarrow \ch(\cX) \rightarrow \ch(\cX \setminus \cZ) \rightarrow 0
$$
to get the complete description of the Chow ring of $\cX$. To this end, we  make use of a patching technique which is at the heart of the Borel-Atiyah-Seigel-Quillen localization theorem, which has been used by many authors in the study of equivariant cohomology, equivariant Chow ring and equivariant K-theory. See the introduction of \cite{DiLorVis} for a more detailed discussion. 

However, without information regarding the kernel of the pushforward of the closed immersion $\cZ \into \cX$, there is no hope to get the complete description of the Chow ring of $\cX$. If the top Chern class of the normal bundle $N_{\cZ|\cX}$ is a non-zero divisor inside the Chow ring of $\cZ$, we can recover $\ch(\cX)$ from $\ch(\cX \setminus \cZ)$, $\ch(\cZ)$ and some patching information. We will refer to the condition on the top Chern class of the normal bundle as the \emph{gluing condition}. The gluing condition implies that the troublesome kernel is trivial. See \Cref{lem:gluing} for a more detailed statement. 

Unfortunately, there is no hope that this condition is verified if $\cZ$ is a Deligne-Mumford separated stack, because in this hypothesis the integral Chow ring is torsion above the dimension of the stack. This follows from Theorem 3.2 of \cite{EdGra}. This is exactly the reason that motivated the authors of  \cite{DiLorPerVis} to introduce the stack of cuspidal stable curves, which is not a Deligne-Mumford separated stack because it has some positive-dimensional affine stabilizers. However, introducing cusps is not enough in the case of $\Mbar_3$ to have the gluing condition verified (for the stratification we choose). This motivated us to introduce a generalization of the moduli stack of cuspidal stable curves, allowing curves with $A_r$-singular points to appear in our stack. They are a natural generalization of nodes and cusps, and \'etale locally are plane singularities described by the equation $y^2=x^{n+1}$.

\subsection*{Future Prospects}

As pointed out in the introduction of \cite{DiLorPerVis},  the limitations of this strategy are not clear. It seems that the more singularities we add, the more it is likely that the gluing condition is verified. However, adding more singularities implies that we  have to eventually compute the relations coming from such loci, which can be hard. Moreover, we are left with a difficult problem, namely to find the right stratification for these newly introduced stacks. We hope that this strategy will be useful to study the intersection theory of $\Mbar_{3,1}$ or $\Mbar_{4}$. Moreover, we believe that our approach can be used to obtain a complete description for the integral Chow ring of $\Mbar_3$. We have not verified the gluing condition with integer coefficients because we do not know the integral Chow ring of some of the strata. However, one can try to prove alternative descriptions for these strata, for instance using weighted blowups, and compute their integral Chow ring using these descriptions. See \cite{Ink} for an example.

\subsection*{Outline of the thesis}

\Cref{chap:1} focuses on introducing the moduli stack of $A_r$-stable curves and proving results that are useful for the computations. Specifically, \Cref{sec:1-1} is dedicated to studying the possible involutions (and relative quotients) of the complete local ring of an  $A_r$-singularity. In \Cref{sec:1-2}, we define the moduli stack $\Mtilde_{g,n}^r$ of $n$-pointed $A_r$-stable curves of genus $g$ and prove that it is a smooth global quotient stack. We also prove that it shares some features with the classic stable case, for example the existence of the Hodge bundle and of the contraction morphism (as defined in \cite{Knu}), which is just an open immersion (instead of an isomorphism as in the stable case). \Cref{sec:1-3} focuses on the second important actor of this computation: the hyperelliptic locus inside $\Mtilde_g^r$. We  introduce the moduli stack $\Htilde_g^r$ of hyperelliptic $A_r$-stable curves of genus $g$ generalizing the definition for the stable case, prove that it is a smooth stack and that it contains the stack of smooth hyperelliptic curves of genus $g$ as a dense open substack. To do this, we give an alternative description of $\Htilde_g^r$ as the moduli stack of cyclic covers of degree $2$ over twisted curves of genus $0$. This description is one of the main reasons why we choose $A_r$-singularities as they appear naturally in the case of ramified branching divisor. We also give a natural morphism $$\eta:\Htilde_g^r \rightarrow \Mtilde_g^r.$$
 Finally, \Cref{sec:1-4} is entirely dedicated to proving that the morphism $\eta$ is a closed immersion of algebraic stacks as we expect from the stable case. The proof is long and uses combinatorics of hyperelliptic curves over algebraically closed fields for the injectivity on geometric points, deformation theory for the unramifiedness, and degenerations of families of $A_r$-curves for properness. 

In \Cref{chap:2}, we compute the Chow ring of $\Mtilde_3^7$. We start by introducing the stratification we use for the computations. This includes two closed substacks of codimension $1$, namely the hyperelliptic locus $\Htilde_3$ and $\Detilde_1$, which parametrizes curves with at least one separating node. We have the codimension $2$ substack $\Detilde_{1,1}$, which parametrizes curves with at least two separating nodes. Finally, we have the codimension $3$ substack $\Detilde_{1,1,1}$ which parametrizes curves with three separating nodes. For this sequence of closed immersions 
$$\Detilde_{1,1,1} \subset \Detilde_{1,1} \subset \Detilde_1$$ 
 the gluing condition is verified even if we add only cusps, i.e. if we consider $A_2$-stable curves. The hyperelliptic locus is the reason why we need to add all the other singularities.
 
 \Cref{sec:H3tilde} focuses on the description of $\Htilde_3 \setminus \Detilde_1$ as a global quotient stack and on the computation of its Chow ring. In \Cref{sec:m3tilde-open}, we describe the open complement of the two closed codimension $1$ strata and prove that the description given in \cite{DiLor2} of $\cM_3 \setminus \cH_3$ as an open of the space of quartics in $\PP^2$ with a $\GL_3$-action extends to $\Mtilde_3^7 \setminus (\Htilde_3^7 \cup \Detilde_1)$. Then we  compute its Chow ring. Similarly, \Cref{sec:detilde-1}, \Cref{sec:detilde-1-1} and \Cref{sec:detilde-1-1-1}  are dedicated to describing respectively $\Detilde_1\setminus \Detilde_{1,1}$, $\Detilde_{1,1}\setminus \Detilde_{1,1,1}$ and $\Detilde_{1,1,1}$ and compute their Chow rings. Finally, in \Cref{sec:chow-m3tilde} we  describe how to glue the informations we collected in the previous sections to get the Chow ring of $\Mtilde_3^7$ and to write an explicit presentation of this ring. 
 
 Finally, \Cref{chap:3} revolves around the computations of the Chow ring of $\Mbar_3$. In \Cref{sec:3-1}, we study in detail the moduli stack $\cA_n$ for $n\leq r$, which parametrizes pairs $(C,p)$ where $C$ is an $A_r$-stable curve of some fixed genus and $p$ is an $A_n$-singularity. \Cref{sec:3-2} focuses on finding the generators for the ideal of relations coming from $\Mtilde_3^7 \setminus \Mbar_3$ while in \Cref{sec:strategy} we compute the explicit description of these generators in the Chow ring of $\Mtilde_3^7$. Finally, in \Cref{sec:3-4} we put all our relations together and describe the Chow ring of $\Mbar_3$ and we compare our result with Faber's.
 
 We also have added three appendices, where we prove (or cite) some results needed for the computations or some technical lemmas that are probably well-know to experts. In Appendix A we describe a variant of the stack of finite flat algebras, namely the stack of finite flat extensions of degree $d$ of finite flat (curvilinear) algebras. In Appendix B we recall some results regarding blowups and pushouts in families and some conditions to have functoriality results. In Appendix C we generalize Proposition 4.2 of \cite{EdFul} and as a byproduct we also obtain a general formula for the stratum of $A_n$-singularities restricted to the open of the hyperelliptic locus parametrizing cyclic covers of the projective line. This formula is used several times in \Cref{sec:strategy}.

\pagenumbering{arabic}

\chapter{$A_r$-stable curves and the hyperelliptic locus}\label{chap:1}
 
This chapter is dedicated to the study of $A_r$-singularities and of $A_r$-stable curves. In the first section we study the possible involutions acting on the complete local ring of an $A_r$-singularity. The main result is \Cref{lem:quotient}, where we describe explicitly the possible quotients. In the second section we introduce the moduli stack $\Mtilde_g^r$ of $A_r$-stable curves of genus $g$ and prove some standard results known for the (classic) stable case. The third section focuses on the moduli stack $\Htilde_g^r$ of hyperelliptic curves and the main theorem proves an alternative description of $\Htilde_g^r$ as the moduli stack of cyclic covers over a genus $0$ twisted curve. Finally, the last section is dedicated to proving that the moduli stack of hyperelliptic $A_r$-stable curves embeds as a closed substack inside the moduli stack of $A_r$-stable curves.  
	
\section{$A_r$-singularities and involutions}\label{sec:1-1}

In this section we study the possible involutions acting on a singularity of type $A_r$.

Let $r$ be a non-negative integer and $k$ be an algebraically closed field with characteristic coprime with $r+1$ and $2$. The complete local $k$-algebra  $$A_r:=k[[x,y]]/(y^2-x^{r+1})$$
is called an $A_r$-singularity. By definition, $A_0$ is a regular ring.		

Suppose we have an involution $\sigma$ of $A_r$. Because the normalization of a noetherian ring is universal among dominant morphism from normal rings, we know the involution lifts to the normalization of $A_r$.

\begin{remark}\label{rem:norm}
	We recall the description of the normalization of $A_r$: 
\begin{itemize}
	\item if $r$ is even, then the normalization is the morphism
	$$ \iota: A_r \arr k[[t]] $$ 
	defined by the associations $x\mapsto t^2$ and $y\mapsto t^{r+1}$;
	 
	\item if $r$ is odd, then the normalization is the morphism 
	$$ \iota: A_r \arr k[[t]]\oplus k[[t]]$$ 
	defined by the associations $x\mapsto (t,t)$ and $y \mapsto (t^{\frac{r+1}{2}}, -t^{\frac{r+1}{2}})$.
\end{itemize}

In the even case, we are identifying $A_r$ to the subalgebra of $k[[t]]$ of power series with only even degrees up to degree $r+1$. In the odd case, we are identifying $A_r$ to the subalgebra of $k[[t]]\oplus k[[t]]$ of pairs of power series which coincide up to degree $(r+1)/2$.
\end{remark}

If $\sigma$ is an involution of $k[[t]]$, we know that the differential $d\sigma$ is an involution of $k$ as a vector space over itself, therefore there exists $\xi_{\sigma} \in k$ such that $d\sigma=\xi_{\sigma} \id$ with $\xi_{\sigma}^2=1$.
We define an endomorphism $\phi_{\sigma}$ of $k[[t]]$ by the association $t\mapsto (t+\xi_{\sigma}\sigma(t))/2$. 

\begin{lemma}
	In the setting above, we have that $\phi_{\sigma}$ is an automorphism and $\sigma':=\phi_{\sigma}^{-1}\sigma\phi_{\sigma}$ is the involution of $k[[t]]$ defined by the association $t\mapsto \xi_{\sigma} t$.
\end{lemma}
\begin{proof}
	The fact that $d\phi_{\sigma}=\id$ implies $\phi_{\sigma}$ is an automorphism. The second statement follows from a straightforward computation.
\end{proof}

The idea is to prove the above lemma also for the algebra $A_r$ using the morphism $\phi_{\sigma}$. In fact we prove that in the even case the automorphism $\phi_{\sigma}$ restricts to an automorphism of $A_r$. Similarly, in the odd case we can construct an automorphism of $k[[t]]\oplus k[[t]]$, prove that it restricts to an automorphism of the subalgebra $A_r$ and describe explicitly the conjugation of the involution by this automorphism.

 \begin{proposition}\label{prop:descr-inv}
 	Every non-trivial involution of $A_r$ is one of the following: 
 	\begin{itemize}
 		\item[$(a)$] if $r$ is even, $\sigma:k[[x,y]]/(y^2-x^{r+1}) \arr k[[x,y]]/(y^2-x^{r+1})$ is defined by the associations $x\mapsto x$ and $y\mapsto -y$;	
 		\item[$(b)$] if $r$ is odd and $r\geq 3$, we get that $\sigma:k[[x,y]]/(y^2-x^{r+1}) \arr k[[x,y]]/(y^2-x^{r+1})$ is defined by one of the following associations:
 		\begin{itemize}
 			\item[$(b_1)$] $x\mapsto x$ and $ y \mapsto -y$,
 			\item[$(b_2)$] $x\mapsto -x$ and $y \mapsto -y$,
 			\item[$(b_3)$] $x\mapsto -x$ and $y \mapsto -y$;
 		\end{itemize}
 		\item[$(c)$] if $r=1$, we get that $\sigma:k[[x,y]]/(y^2-x^2) \arr k[[x,y]]/(y^2-x^2)$ is defined by one of the following associations: 
 		\begin{itemize}
 			\item[$(c_1)$] $x\mapsto x$ and $y\mapsto -y$,
 			\item[$(c_2)$] $x\mapsto -x$ and $y\mapsto -y$,
 			\item[$(c_3)$] $x\mapsto y$ and $y\mapsto x$;
 		\end{itemize}
 	\end{itemize}
 up to conjugation by an automorphism of $A_r$. 
\end{proposition}  

\begin{proof}

We start with the even case. We identify $\sigma$ with its lifting to the normalization of $A_r$ by abuse of notation. First of all, we know that $\sigma(t)=\xi_{\sigma}tp(t)$ where $p(t)\in k[[t]]$ with $p(0)=1$. Because $\sigma$ is induced by an involution on $A_r$, we have that $\sigma(t)^2=\sigma(t^2) \in A_r$. An easy computation shows that this implies $t^2p(t) \in A_r$.
We see that the images of the two generators of $A_r$ through the morphism $\phi_{\sigma}$  are inside $A_r$:
$$\phi_{\sigma}(t^2)=\frac{t^2+\sigma(t)^2 + 2\xi_{\sigma}t\sigma(t)}{4}= \frac{t^2+\sigma(t^2)+2t^2p(t)}{4} \in A_r$$
and  
$$ \phi_{\sigma}(t^{r+1})=t^{r+1}\frac{(1+p(t))^{r+1}}{2^{r+1}}\in t^{r+1}k[[t]] \subset A_r;$$
notice that if we compute the differential of the restriction $\phi_{\sigma}\vert_{A_r}:A_r \arr A_r$ we get an endomorphism of the tangent space of $A_r$ of the form 
$$
\begin{pmatrix}
  1  &  \star \\
  0  &    1
\end{pmatrix} 
$$ 
therefore  $\phi_{\sigma}\vert_{A_r}$ is an injective morphism with surjective differential between complete noetherian rings, thus it is an automorphism. Finally, if we define $\sigma':=\phi_{\sigma}^{-1}\sigma\phi_{\sigma}$, then $\sigma'\vert_{A_r}=(\phi_{\sigma}^{-1}\vert_{A_r})(\sigma\vert_{A_r})(\phi_{\sigma}\vert_{A_r})$ and we can describe its action on the generators:
\begin{itemize}
	\item $\sigma'(x)=\sigma'(t^2)=\xi_{\sigma}^2 t^2 = x$,
	\item $\sigma'(y)=\sigma'(t^{r+1})=\xi_{\sigma}^{r+1} t^{r+1} = \xi_{\sigma} y$
\end{itemize} 
 where we know that $\xi_{\sigma}^2=1$. We can have both $\xi_{\sigma}=1$ and $\xi_{\sigma}=-1$, although the first one is just the identity.

They same idea works for the odd case. We describe the lifting $\Sigma$ of the involution $\sigma$ of $A_r$ to the normalization $k[[t]]\oplus k[[t]]$. We have two possibilities: the involution $\sigma$ exchanges the two branches or not. This translates in the condition that $\Sigma$ exchanges the two connected component of the normalization (or not). Firstly, we consider the case where $\Sigma$ fixes the two connected components of the normalization, therefore we can describe $\Sigma:k[[t]]^{\oplus 2}\arr k[[t]]^{\oplus 2}$ as a matrix of the form
$$
\Sigma=
\begin{pmatrix}
      \sigma_1 & 0 \\ 0 & \sigma_2
\end{pmatrix}
$$
where $\sigma_1,\sigma_2$ are involutions of $k[[t]]$. Because $\Sigma$ is induced by an involution of $A_r$, we have that $(\sigma_1(t),\sigma_2(t))\in A_r$, i.e. $\sigma_1(t)=\sigma_2(t) \mod t^{\frac{r+1}{2}}$.

 We consider the automorphism $\Phi_{\Sigma}$ of $k[[t]]^{\oplus 2}$ which can described as a matrix of the form
$$
\Phi_{\Sigma}:=
\begin{pmatrix}
\phi_{\sigma_1} & 0 \\ 0 & \phi_{\sigma_2}
\end{pmatrix};
$$
 we have the following equalities:
$$
\Phi_{\Sigma}(t,t)=(\phi_{\sigma_1}(t),\phi_{\sigma_2}(t))=1/2(t+\xi_{\sigma_1}\sigma_1(t),t+\xi_{\sigma_2}\sigma_2(t)) \in A_r
$$ 
and 
$$
\Phi_{\Sigma}(t^{\frac{r+1}{2}},-t^{\frac{r+1}{2}})=
(\phi_{\sigma_1}(t)^{\frac{r+1}{2}},-\phi_{\sigma_2}(t)^{\frac{r+1}{2}}) \in A_r
$$
 and again we have that the differential of $\Phi_{\Sigma}\vert_{A_r}$ is of the form
$$
\begin{pmatrix}
1  &  \star \\
0  &    1
\end{pmatrix} 
$$ 
therefore $\Phi_{\sigma}\vert_{A_r}$ is an automorphism of $A_r$. Notice that if $r\geq 3$ we have that $\xi_{\sigma_1}$ and $\xi_{\sigma_2}$ are the same, but if $r=1$ we don't; nevertheless it is still true that $\Phi_{\Sigma}(t,t) \in A_1$. 
Again, we have proved that if the involution of $A_r$ fixes the two branches, then we have only a finite number of involutions up to conjugation. If $r\geq 3$ then we have only the involution described on generators by $x\mapsto \xi x$ and $y\mapsto \xi^{\frac{r+1}{2}} y$ where $\xi^2=1$. Conversely, if $r=1$ we get more possible involutions, as $\xi_{\sigma_1}$ and $\xi_{\sigma_2}$ can be different. Specifically, we get four of them; if we consider their action on the pair of generators $(x,y)$ of $A_r$, we get the following matrices describing the four involutions:
$$
\begin{pmatrix}
1  &  0 \\
0  &    1
\end{pmatrix} ,
\begin{pmatrix}
	-1  &  0 \\
	0  &   -1
\end{pmatrix},
\begin{pmatrix}
0  &  1 \\
1  &  0
\end{pmatrix},
\begin{pmatrix}
	0  &  -1 \\
	-1  &  0
\end{pmatrix};
$$
notice that the third matrix is the conjugate of the fourth one by the automorphism of $A_1$ defined on the pair of generators $(x,y)$ by the following matrix:
$$
\begin{pmatrix}
-1  &  0 \\
 0 &  1
\end{pmatrix}.
$$
Lastly, we consider the case when $\Sigma$ exchanges the connected components. Because it is an involution, $\Sigma$ is an automorphism of $k[[t]]^{\oplus 2}$ of the form 
$$
\begin{pmatrix}
0  &  \tilde{\sigma} \\
\tilde{\sigma}^{-1}  &    0
\end{pmatrix} 
$$
and because $\Sigma$ is induced by the involution $\sigma$ of $A_r$, we get that $\tilde{\sigma}$ is an involution of $k[[t]]/(t^{\frac{r+1}{2}})$, i.e. $\tilde{\sigma}^2(t)=t+t^{\frac{r+1}{2}}p(t)$ with $p(t)\in k[[t]]$. Notice that if $r\geq 3$ we get that $\tilde{\sigma}(t)=\xi_{\tilde{\sigma}}tp(t)$ where $p(t)\in k[[t]]$ with $p(0)=1$ and $\xi_{\tilde{\sigma}}^2=1$. On the contrary, if $r=1$ the previous condition is empty, i.e. $\tilde{\sigma}$ is any automorphism,. Let us first consider the case $r\geq 3$. Consider the endomorphism $\Phi_{\Sigma}$ of $k[[t]]^{\oplus 2}$ of the form 
$$
\Phi_{\Sigma}:=
\begin{pmatrix}
\phi_1  &  0 \\
0  &    \phi_2
\end{pmatrix} 
$$ 
where we define $\phi_1(t)=1/2(t+\xi_{\tilde{\sigma}}\tilde{\sigma}(t))$ and $\phi_2(t)=1/2(t+\xi_{\tilde{\sigma}}\tilde{\sigma}(t)+t^{\frac{r+1}{2}}p(t))$. Again, an easy computation shows that the automorphism $\Phi_{\Sigma}$ restricts to the algebra $A_r$ and it is in fact an automorphism. For the case $r=1$, we can simply consider $\Phi_{\Sigma}$ of the form 
$$
\Phi_{\Sigma}:=
\begin{pmatrix}
 \tilde{\sigma} &  0 \\
0  &    \id
\end{pmatrix} 
$$
which restricts to an automorphism of $A_r$ as well. As before, if $r\geq 3$ we have that the involution is of the form $x\mapsto \xi x$ and $y \mapsto -\xi^{\frac{r+1}{2}}y$. Instead, if $r=1$ we have the involution described by the association $ x \mapsto x$ and $y \mapsto -y$. 
	
\end{proof}

The previous corollary finally implies the description of the invariant we were looking for. Let us focus on the description of this invariant subalgebras (in the case of a non trivial involution). We prove the following statement.
 
\begin{corollary}
	Let $\sigma$ be a non-trivial involution of the algebra $A_r$ and let us denote by $A_r^{\sigma}$ the invariant subalgebra and by $i:A_r^{\sigma}\into A_r$ the inclusion. If we refer to the classification proved in \Cref{prop:descr-inv}, we have that
	\begin{itemize}
		\item[$(a)$] if $r$ is even, we have that $A_r^{\sigma} \simeq A_0$, the inclusion $i$ is faithfully flat and the fixed locus of $\sigma$ has length $r+1$ and it is a Cartier divisor;
		\item[$(b)$] if $r:=2k-1$ is odd and $k\geq 2$ we have that 
		\begin{itemize}
			\item[$(b_1)$] $A_r^{\sigma} \simeq A_0$, the inclusion $i$ is faithfully flat, the fixed locus of $\sigma$ has length $r+1$ and it is the support of a Cartier divisor;
			\item[$(b_2)$] $A_r^{\sigma} \simeq A_{k-1}$, the inclusion $i$ is faithfully flat, the fixed locus of $\sigma$ has length $2$ and it is the support of a Cartier divisor;
			\item[$(b_3)$] $A_r^{\sigma} \simeq A_k$, the inclusion $i$ is not flat, the fixed locus is of length $1$ and it is not the support of a Cartier divisor;
		\end{itemize}
		\item[$(c)$] if $r=1$, we have that
		\begin{itemize}
			\item[$(c_1)$] $A_1^{\sigma} \simeq A_0$, the inclusion $i$ is faithfully flat, the fixed locus of $\sigma$ has length $2$ and it is the support of a Cartier divisor;
			\item[$(c_2)$] $A_1^{\sigma} \simeq A_1$, the inclusion $i$ is not faithfully flat, the fixed locus of $\sigma$ has length $1$ and it is not the support of a Cartier divisor;
			\item[$(c_3)$] $A_1^{\sigma} \simeq A_1$, the inclusion $i$ is faithfully flat and the fixed locus coincides with one of two irreducible components.
		\end{itemize}
	\end{itemize}
\end{corollary}

\begin{proof}
Let us start with the even case. Thanks to \Cref{prop:descr-inv}, we have that the involution $\sigma$ of $A_r$ is defined by the associations $x\mapsto x$ and $y\mapsto -y$ (up to conjugation by an isomorphism of $A_r$). Therefore it is clear that $A_r^{\sigma} \simeq  k[[x]]$ and the quotient morphism is induced by the inclusion $$A_0 \simeq k[[x]] \subset \frac{k[[x,y]]}{(y^2-x^{r+1})} \simeq A_r;$$
 the same is true for the odd case when the involution $\sigma$ acts in the same way. In this case the algebras extension (corrisponding to the quotient morphism) is faithfully flat, the fixed locus is the support of a Cartier divisor defined by the ideal $(y)$ in $A_r$ and it has length $r+1$.

Now we consider the case $r=2k-1$ with $\sigma$ acting as follows: $\sigma(x)=-x$ and $\sigma(y)=y$. A straightforward computation shows that the invariant algebra $A_r^{\sigma}$ is of the type $A_{k-1}$. To be precise, the inclusion of the invariant subalgebra can be described by the morphism
$$ i: A_{k-1}\simeq\frac{k[[x,y]]}{(y^2-x^{k})} \hooklongrightarrow \frac{k[[x,y]]}{(y^2-x^{2k})}\simeq A_r$$ 
where $i(x)=x^2$ and $i(y)=y$. In this case we get the $A_r$ is a faithfully flat $A_{k-1}$-algebra and the fixed locus is the support of a Cartier divisor defined by the ideal $(x)$ in $A_r$ and it has length $2$.

If $\sigma$ is defined by the associations $x\mapsto -x$ and $y \mapsto -y$ and $r=2k-1$, then the invariant subalgebra $A_r^{\sigma}$ is of type $A_k$ and the quotient morphism is defined by the inclusion

$$ i: A_k\simeq \frac{k[[x,y]]}{(y^2-x^{k+1})} \hooklongrightarrow \frac{k[[x,y]]}{(y^2-x^{r+1})}\simeq A_r $$

where $i(x)=x^2$ and $i(y)=xy$. In contrast with the two previous cases, $A_r$ is not a flat $A_k$-algebra and the fixed locus is not (the support of) a Cartier divisor, it is in fact defined by the ideal $(x,y)$ in $A_r$ and it has length $1$. 

Finally, we consider the case where $r=2$ and the action of $\sigma$ is described by $x\mapsto y$ and $y \mapsto x$. If you consider the isomorphism 
$$ \frac{k[[u,v]]}{(uv)} \rightarrow  \frac{k[[x,y]]}{(y^2-x^2)}$$
defined by the associations $u\mapsto y+x$ and $v\mapsto y-x$  we get that the invariant subalgebra is defined by the inclusion 
$$ i: A_1\simeq \frac{k[[u,v]]}{(uv)} \hooklongrightarrow \frac{k[[u,v]]}{(uv)}$$ 
where $i(u)=u$ and $i(v)=v^2$. Notice that in this situation the algebras extension is flat but the fixed locus is not a Cartier divisor, as it is defined by $(v)$, which is a zero divisor in $A_1$.
\end{proof}

\begin{remark}
	If $r$ is odd and $r\geq 3$ (case $(b)$), we have that every involution gives a different quotient. The same is not true for the case $r=1$ as we can obtain the nodal singularity in two ways.
\end{remark}

\begin{remark}\label{rem:fix-locus}
 Notice that $(c_3)$ is the only case when the fixed locus is an irreducible component. This situation do not appear in the stack of hyperelliptic curves as we consider only involutions with finite fixed locus. 
\end{remark}
 
 We end this section with a technical lemma which will be useful afterwards. 
 \begin{lemma}\label{lem:local-node-involution}
 	Let $(R,m)\hookrightarrow (S,n)$ be a flat extension of noetherian complete local rings over $k$ such that  $$S\otimes_R R/m \simeq A_1.$$
 	Suppose we have an $R$-involution $\sigma$ of $S$ such that $\sigma \otimes R/m$ (seen as an involution of $A_1$) does not exchange the two irreducible components. Hence there exists a $R$-isomorphism 
 	$$ S \simeq R[[x,y]]/(xy-t) $$
 	where $t \in R$ such that $\sigma$ (seen as an involution of the right-hand side of the isomorphism) acts as follows: $\sigma(x)=\xi_2x$ and $\sigma(y)=\xi_1y$ for some $\xi_i \in k$ such that $\xi_i^2=1$ for $i=1,2$. Furthermore, if $\xi_1=-\xi_2$ we have $t=0$.

 \end{lemma}

 \begin{proof}
 For the sake of notation, we still denote by $\sigma$ the involution $\sigma \otimes R/m^{n+1}$. We inductively construct elements $x_n,y_n$ in $S_n:=S\otimes R/m^{n+1}$ and $t_n \in R/m^{n+1}$ such that 
 	\begin{itemize}
 		\item[1)] $\sigma(x_n)=\xi_1x_n$,
 		\item[2)] $\sigma(y_n)=\xi_2y_n$,
 		\item[3)] $x_ny_n=t_n$ in $S_n$;
 	\end{itemize}
 for some $\xi_i \in k$ indipendent of $n$ such that $\xi_i^2=1$ for $i=1,2$. The case $n=0$ follows from \Cref{prop:descr-inv} and it gives us $\xi_i$ for $i=1,2$. Suppose we have constructed $(x_n,y_n,t_n)$ with the properties listed above. 

 Consider two general liftings $x'_{n+1},y'_{n+1}$ in $S_{n+1}$. 
 
 We define $x''_{n+1}:=(x'_{n+1}+\xi_1\sigma(x'_{n+1}))/2$ and $y''_{n+1}:=(y'_{n+1}+\xi_2\sigma(y'_{n+1}))/2$. The pair $(x''_{n+1},y''_{n+1})$ clearly verify the properties 1) and 2). A priori, $x''_{n+1}y''_{n+1}=t_n+h$ with $h$ an element in $S_{n+1}$ such that its restriction in $S_n$ is zero. The flatness of the extension implies that 
 $$ \ker(S_{n+1}\rightarrow S_n)\simeq A \otimes_R (m^{n+1}/m^{n+2})\simeq S_0 \otimes_k (m^{n+1}/m^{n+2})$$ and therefore $$ h = h_0 + x''_{n+1}p(x''_{n+1}) + y''_{n+1}q(y''_{n+1})$$ 
 where all the coefficients of the polynomial $p$ and $q$ (and clearly $h_0$) are in $m^{n+1}/m^{n+2}$. if we define 
 \begin{itemize}
 	\item $t_{n+1}:=t_n+h_0$,
 	\item $x_{n+1}:=x''_{n+1}+q(y''_{n+1})$,
 	\item $y_{n+1}:=y''_{n+1}+p(x''_{n+1})$,
 \end{itemize}  
 the third condition above is verified but we have to prove that the first two are still verified for $x_{n+1}$ and $y_{n+1}$. 

 Using the fact that $\sigma(x''_{n+1}y''_{n+1})=\xi_1\xi_2x''_{n+1}y''_{n+1}$, we reduce to analyze three cases.
 
 If $\xi_1=\xi_2=1$, there is nothing to prove.
 If $\xi_1=-\xi_2$, a computation shows that   
 \begin{itemize}
 	\item $h_0=0$,
 	\item $ p\equiv 0$,
 	\item $q(y''_{n+1})=\tilde{q}(y''^2_{n+1})$
 \end{itemize}
for a suitable polynomial $\tilde{q}$ with coefficients in $m^{n+1}/m^{n+2}$. 
 If $\xi_1=\xi_2=-1$, a computation shows that   
\begin{itemize}
	\item $ p(x''_{n+1})=\tilde{p}(x''^2_{n+1})$,
	\item $q(y''_{n+1})=\tilde{q}(y''^2_{n+1})$
\end{itemize}
for a suitable $\tilde{p},\tilde{q}$ polynomials with coefficients in $m^{n+1}/m^{n+2}$. 

Hence $(x_{n+1},y_{n+1},t_{n+1})$ satisfies the conditions 1), 2) and 3),
therefore we have a morphism of flat $R/m^{n+1}$-algebras
$$ (R/m^{n+1})[[x_{n+1},y_{n+1}]]/(x_{n+1}y_{n+1}-t_{n+1}) \longrightarrow S_{n+1}$$ 
which is an isomorphism modulo $m$, therefore it is an isomorphism. If we pass to the limit we get the result.
 \end{proof}

\section{$A_r$-stable curves and moduli stack}\label{sec:1-2}

Fix a nonnegative integer $r$.
\begin{definition}	Let $k$ be an algebraically closed field and $C/k$ be a proper reduced connected one-dimensional scheme over $k$. We say the $C$ is an \emph{$A_r$-prestable curve} if it has at most $A_r$-singularity, i.e. for every $p\in C(k)$, we have an isomorphism
		$$ \widehat{\cO}_{C,p} \simeq k[[x,y]]/(y^2-x^{h+1}) $$ 
	with $ 0\leq h\leq r$. Furthermore, we say that $C$ is $A_r$-stable if it is $A_r$-prestable and the dualizing sheaf $\omega_C$ is ample. A $n$-pointed $A_r$-stable curve over $k$ is $A_r$-prestable curve together with $n$ smooth distinct closed points $p_1,\dots,p_n$ such that $\omega_C(p_1+\dots+p_n)$ is ample.
\end{definition}
\begin{remark}
	Notice that a $A_r$-prestable curve is l.c.i by definition, therefore the dualizing complex is in fact a line bundle. 
\end{remark}

For the rest of the chapter, we fix a base field $\kappa$ where all the primes smaller than $r+1$ are invertible. Every time we talk about genus, we intend arithmetic genus, unless specified otherwise.

\begin{remark}\label{rem:genus-count}
	Let $C$ be a connected, reduced, one-dimensional, proper scheme over an algebraically closed field. Let $p$ be a rational point which is a singularity of $A_r$-type. We denote by $b:\widetilde{C}\arr C$ the partial normalization at the point $p$ and by $J_b$ the conductor ideal of $b$. Then a straightforward computation shows that \begin{enumerate}
		\item if $r=2h$, then $g(C)=g(\widetilde{C})+h$;
		\item if $r=2h+1$ and $\widetilde{C}$ is connected, then $g(C)=g(\widetilde{C})+h+1$,
		\item if $r=2h+1$ and $\widetilde{C}$ is not connected, then $g(C)=g(\widetilde{C})+h$.
	\end{enumerate}
	If $\widetilde{C}$ is not connected, we say that $p$ is a separating point. Furthermore, Noether formula gives us that $b^*\omega_C \simeq \omega_{\widetilde{C}}(J_b^{\vee})$.
\end{remark}
Let us define the moduli stack we are interested in. Let $g$ be an integer with $g\geq 2$.
We denote by $\Mtilde_{g}^r$ the category defined in the following way: an object is a proper flat finitely presented morphism $C\arr S$ over $\kappa$ such that every geometric fiber over $S$ is a $A_r$-stable curve of genus $g$. These families are called \emph{$A_r$-stable curves} over $S$. Morphisms are defined in the usual way. This is clearly a fibered category over the category of schemes over $\kappa$.

Fix a positive integer $n$. In the same way, we can define $\Mtilde_{g,n}^r$ the fibered category whose objects are the datum of $A_r$-stable curves over $S$ with $n$ distinct sections $p_1,\dots,p_n$ such that every geometric fiber over $S$ is a  $n$-pointed $A_r$-stable curve. These families are called \emph{$n$-pointed $A_r$-stable curves} over $S$. Morphisms are just morphisms of $n$-pointed curves.

The main result of this section is the description of $\Mtilde_{g,n}^r$ as a quotient stack. Firstly, we need to prove two results which are classical in the case of ($A_1$)-stable curves.
 
 \begin{proposition}\label{prop:openness}
 	Let $C \arr S$ a proper flat finitely presented morphism with $n$-sections $s_i:S \arr C$ for $i=1,\dots,n$. There exists an open subscheme $S' \subseteq S$ with the property that a morphism $T \arr S$ factors through $S'$ if and only if the projection $T\times_{S} C \arr T$, with the sections induced by the $s_{i}$, is a $n$-pointed $A_r$-stable curve.
 \end{proposition}
 
 \begin{proof}
 	It is well known that a small deformation of a curve with $A_h$-singularities for $h\leq r$ still has $A_h$-singularities for $h\leq r$. Hence, after restricting to an open subscheme of $S$ we can assume that $C \arr S$ is an $A_{r}$-prestable curve. By further restricting $S$ we can assume that the sections land in the smooth locus of $C \arr S$, and are disjoint. Then the result follows from openness of ampleness for invertible sheaves. 
 \end{proof}

The following result is already known for canonically positive Gorestein curves (see \cite[Theorem B and Theorem C]{Cat}). We extend the proof to the case of $n$-pointed $A_r$-stable curves. As a matter of fact, we also prove Theorem 1.8 of \cite{Knu} for $n$-pointed $A_r$-stable curves of genus $g$.

\begin{proposition}\label{prop:boundedness}
	Let $(C,p_1,\dots,p_n)$ be a $n$-pointed $A_r$-stable curve over an algebraically closed field $k/\kappa$. Then 
	\begin{enumerate}
		\item[i)] $\H^1(C,\omega_C(p_1+\dots+p_n)^{\otimes m})=0$ for every $m\geq 2$,
		\item[ii)] $\omega_C(p_1+\dots+p_n)^{\otimes m}$ is very ample for every $m\geq 3$.
		\item[iii)] $\omega_C(p_1+\dots+p_n)^{\otimes m}$ is normally generated for $m\geq 6$.
	\end{enumerate}
\end{proposition}

\begin{proof}
	We denote by $\Sigma$ the Cartier divisor $p_1+\dots+p_n$ of $C$. Using duality, i) is equivalent to 
	$$ \H^0(C,\omega_C^{\otimes (1-m)}(-m\Sigma))=\H^0(C,\omega_C(\Sigma)^{\otimes (1-m)}(-\Sigma))=0$$
	for every $m\geq 2$. If $\Gamma$ is an irreducible component of $C$, then we denote by $a_{\Gamma}$ the degree of $\omega_C(\Sigma)\vert_{\Gamma}$, which is a positive integer by the stability condition. Thus, 
	$$\deg \big(\omega_C(\Sigma)^{\otimes 1-m}(-\Sigma)\big) \leq (1-m)a_{\Gamma}<0$$
	for every $m\geq2$. Thus the sections of the line bundle are zero restricted on every irreducible component, which implies the line bundle does not have non trivial sections.
	
	Regarding ii), we need to prove that for every pair of closed point $q_1,q_2 \in C$ (possibly $q_1=q_2$) the restriction morphism 
	$$ \H^0(C,\omega_C(\Sigma)^{\otimes m})\longrightarrow \omega_C(\Sigma)^{\otimes m}\otimes D $$
	is surjective, where $D$ is the closed subscheme of $C$ defined by the ideal $I_D:=m_{q_1}m_{q_2}$ ($m_{q_i}$ is the ideal defining $q_i$ for $i=1,2$). This is equivalent to prove that 
	$$ \H^1(C,I_D\omega_C(\Sigma)^{\otimes m})=0$$
	or by duality 
	$$ \hom_{\cO_C}(I_D, \omega_C(\Sigma)^{\otimes (1-m)}(-\Sigma))=0.$$
	
	The proof is divided into $4$ cases, starting with the situation when both $q_1$ and $q_2$ are smooth.

	First of all, notice that, if $\Gamma \subset C$ is an irreducible component and $q$ is a smooth point on $\Gamma$, then 
	$$\deg\big(\omega_C(\Sigma)^{\otimes (1-m)}(-\Sigma+q)\vert_{\Gamma}\big)<0$$
	for $m\geq 3$. Therefore, we know the vanishing result for every pair of smooth point that do not lie in the same irreducible component. However, if $q_1,q_2 \in \Gamma$ with $\Gamma$ irreducible component, this implies that $$\deg\big(\omega_C(\Sigma)^{\otimes (1-m)}(-\Sigma+q_1+q_2)\vert_{\tilde{\Gamma}}\big)\leq0$$
	where equality is possible only for $\tilde{\Gamma}=\Gamma$. Therefore, if $C$ is not irreducible, we get the vanishing result again. Finally, if $C=\Gamma$ is integral, one can prove that 
	$$\deg\big(\omega_C(\Sigma)^{\otimes (1-m)}(-\Sigma+q_1+q_2)\vert_{\Gamma}\big)<0$$
	by simply using the stability condition. Thus we have proved that the morphism associated to the complete linear system of $\omega_C(\Sigma)^{\otimes m}$ separates smooth points.
	
	Suppose now that $q_1$ is singular and $q_2$ is smooth. Let us call $\pi:\tilde{C} \rightarrow C$ the partial normalization in $q_1$ of $C$. We have that (see Lemma 2.1 of \cite{Cat})
	$$ \underhom(m_{q_1},\cO_C) \subset \pi_{*}\cO_{\tilde{C}}$$
	therefore  
	$$\hom_{\cO_C}(I_D, \omega_C(\Sigma)^{\otimes (1-m)}(-\Sigma))\subset \H^0(\tilde{C}, \pi^*\omega_C(\Sigma)^{\otimes (1-m)}(-\Sigma+q_2)).$$
	If we denote by $\cF$ the line bundle $\omega_C(\Sigma)^{\otimes (1-m)}(-\Sigma+q_2)$, then clearly the restriction of $\pi^*\cF$ to every irreducible component $\tilde{\Gamma}$ of $\tilde{C}$ has the same degree of the restriction of $\cF$ to the irreducible component $\Gamma=\pi(\tilde{\Gamma})$ because $\pi$ is finite and birational. Again, the restriction $\cF\vert_{\Gamma}$ has negative degree, therefore the vanishing result follows.
	
	Suppose that both $q_1$ and $q_2$ are singular but distinct. Let us call $\pi:\tilde{C} \rightarrow C$ the partial normalization in $q_1$ and $q_2$ of $C$. Then we have that (see Lemma 2.1 of \cite{Cat})
	$$ \underhom(m_{q_1}m_{q_2},\cO_C) \subset \pi_{*}\cO_{\tilde{C}}$$
	therefore
	$$\hom_{\cO_C}(I_D, \omega_C(\Sigma)^{\otimes (1-m)}(-\Sigma))\subset \H^0(\tilde{C}, \pi^*\omega_C(\Sigma)^{\otimes (1-m)}(-\Sigma)).$$
	By the same argument again, we get the vanishing result.
	
	Last case we need to consider is when $q:=q_1=q_2$ and it is singular. Let $\pi:\tilde{C} \rightarrow C$ be the partial normalization at $q$ and $M$ be the ideal generated by $\pi^{-1}m_q$ in $\tilde{C}$.  Then we have that (see Lemma 2.2 of \cite{Cat})
	$$ \underhom(m_q^2,\cO_C) \subset \pi_{*}M^{\vee}$$
	which implies 
	$$\hom_{\cO_C}(I_D, \omega_C(\Sigma)^{\otimes (1-m)}(-\Sigma))\subset \H^0(\tilde{C}, M^{\vee}\otimes\pi^*\omega_C(\Sigma)^{\otimes (1-m)}(-\Sigma)).$$
	A local computation shows that $M$ is the (invertible) ideal of definition of a subscheme of length $2$ in $\tilde{C}$. Therefore the same argument used for the case when $q_1$ and $q_2$ are smooth can be applied here to get the vanishing result. 
	
	Finally, we prove that if $L:=\omega(\Sigma)$ then $L^{\otimes m}$ is normally generated for $m \geq 6$. This is a simplified version of the proof in \cite{Knu} as the author proved the same statement for $m\geq 3$ which requires more work. For simplicity, we write $\H^0(-)$ and $\H^1(-)$ instead of $\H^0(C,-)$ and $\H^1(C,-)$.
	
	Consider the following commutative diagram for $k\geq 1$
	$$
	\begin{tikzcd}
		\H^0(L^{\otimes 3}) \otimes \H^{0}(L^{\otimes m-3}) \otimes H^0(L^{\otimes km}) \arrow[rr] \arrow[d, "\varphi_1"] &  & \H^0(L^{\otimes m})\otimes \H^0(L^{\otimes km}) \arrow[d, "\phi"] \\
		\H^0(L^{\otimes m-3}) \otimes \H^0(L^{\otimes km+3}) \arrow[rr, "\varphi_2"]                                      &  & \H^0(L^{\otimes (k+1)m});                                        
	\end{tikzcd}
	$$
	we need to prove that $\phi$ is surjective, therefore it is enough to prove that $\varphi_1$ and $\varphi_2$ are surjective. As both $L^{\otimes 3}$ and $L^{\otimes m-3}$ are globally generated because $m\geq 6$, we can use the Generalized Lemma of Castelnuovo (see pag 170 of \cite{Knu}) and reduce to prove that $\H^1(L^{\otimes (km-3)})$ and $\H^1(L^{\otimes (km+6-m)})$ are zero for $k\geq 1$. We have that by Grothendieck duality 
	$$\H^1(L^{\otimes km-3})=\H^0(\omega_C \otimes L^{\otimes 3-km}).$$ 
	If we focus on a single irreducible component $\Gamma$ and denote by $a_{\Gamma}$ the quantity $deg \omega_C(\Sigma)\vert_{\Gamma}$, we have that 
	$$ \deg (\omega_{C} \otimes L^{3-km}) \leq 4-km <0$$ 
	for $k\geq 1$. The same reasoning prove the vanishing of the other cohomology group.
	
\end{proof}

\begin{remark} 
	One can also prove that $\omega_C(\Sigma)^{\otimes m}$ is globally generated for $m\geq 2$. For our purpose, we just need to prove that there exists an integer $m$, indipendent of the field $k$, such that  $\omega_C(\Sigma)^{\otimes m}$ is very ample.
\end{remark}

Now we are ready to prove the theorem. To be precise, both the proofs of the following theorem and of \Cref{prop:openness} are an adaptation of Theorem 1.3 and of Proposition 1.2 of \cite{DiLorPerVis} to the more general case of $A_r$-stable curve. 

\begin{theorem}\label{theo:descr-quot}
	$\Mtilde_{g,n}^r$ is a smooth algebraic stack of finite type over $\kappa$. Furthermore, it is a quotient stack: that is, there exists a smooth quasi-projective scheme X with an action of $\GL_N$ for some positive $N$, such that 
	$ \Mtilde_{g,n}^r \simeq [X/\GL_N]$. If $k$ is a field, then it is connected.
\end{theorem}

\begin{proof}
	It follows from \Cref{prop:boundedness} that if $(\pi\colon C \arr S,\Sigma_1,\dots,\Sigma_n)$ is a $n$-pointed $A_{r}$-stable curve of genus $g$, then $\pi_{*}\omega_{C/S}(\Sigma_{1}+ \dots + \Sigma_{n})^{\otimes 3}$ is locally free sheaf of rank $N:=5g-5+3n$, and its formation commutes with base change, because of Grothendieck's base change theorem.
	
	Call $X$ the stack over $k$, whose sections over a scheme $S$ consist of a $A_{r}$-stable $n$-pointed curve as above, and an isomorphism $\cO_{S}^{N} \simeq \pi_{*}\omega_{C/S}(\Sigma_{1}+ \dots + \Sigma_{n})^{\otimes m}$ of sheaves of $\cO_{S}$-modules. Since $\pi_{*}\omega_{C/S}(\Sigma_{1}+ \dots + \Sigma_{n})^{\otimes m}$ is very ample, the automorphism group of an object of $X$ is trivial, and $X$ is equivalent to its functor of isomorphism classes.
	
	Call $H$ the Hilbert scheme of subschemes of $\PP^{N-1}_{\kappa}$ with Hilbert polynomial $P(t)$, and $D \arr H$ the universal family. Call $F$ the fiber product of $n$ copies of $D$ over $S$, and $C \arr F$ the pullback of $D \arr H$ to $F$; there are $n$ tautological sections $s_{1}$, \dots,~$s_{n}\colon F \arr C$. Consider the largest open subscheme $F'$ of $F$ such that the restriction $C'$ of $C$, with the restrictions of the $n$ tautological sections, is a $n$-pointed $A_{r}$-stable curve, as in Proposition~\ref{prop:openness}. Call $Y \subseteq F'$ the open subscheme whose points are those for which the corresponding curve is nondegenerate, $E \arr Y$ the restriction of the universal family, $\Sigma_{1}$, \dots,~$\Sigma_{n} \subseteq E$ the tautological sections. Call $\cO_{E}(1)$ the restriction of $\cO_{\PP^{N-1}_{Y}}(1)$ via the tautological embedding $E \subseteq \PP^{N-1}_{Y}$; there are two section of the projection $\pic_{E/Y}^{3(2g-2 + n)}\arr Y$ from the Picard scheme parametrizing invertible sheaves of degree $3(2g-2 + n)$, one defined by $\cO_{E}(1)$, the other by $\omega_{E/Y}(\Sigma_{1} + \dots + \Sigma_{n})^{\otimes 3}$; let $Z \subseteq Y$ the equalizer of these two sections, which is a locally closed subscheme of $Y$.
	
	Then $Z$ is a quasi-projective scheme over $\kappa$ representing the functor sending a scheme $S$ into the isomorphism class of tuples consisting of a $n$-pointed  $A_{r}$-stable curve $\pi\colon C \arr S$, together with an isomorphism of $S$-schemes
	\[
	\PP^{N-1}_{S} \simeq \PP(\pi_{*}\omega_{C/S}(\Sigma_{1} + \dots + \Sigma_{n})^{\otimes 3})\,.
	\]
	There is an obvious functor $X \arr Z$, associating with an isomorphism $\cO_{S}^{N} \simeq \pi_{*}\omega_{C/S}(\Sigma_{1}+ \dots + \Sigma_{n})^{\otimes 3}$ its projectivization. It is immediate to check that $X \arr Z$ is a $\gm$-torsor; hence it is representable and affine, and $X$ is a quasi-projective scheme over $\spec \kappa$.
	
	On the other hand, there is an obvious morphism $X \arr \mt_{g,n}^r$ which forgets the isomorphism $\cO_{S}^{N} \simeq \pi_{*}\omega_{C/S}(\Sigma_{1}+ \dots + \Sigma_{n})^{\otimes 3}$; this is immediately seen to be a $\GL_{N}$ torsor. We deduce that $\mt_{g,n}^r$ is isomorphic to $[X/\GL_{N}]$. This shows that is a quotient stack, as in the last statement; this implies that $\mt_{g,n}^r$ is an algebraic stack of finite type over $\kappa$.
	
	The fact that $\mt_{g,n}^r$ is smooth follows from the fact that $A_{r}$-prestable curves are unobstructed.
	
	Finally, to check that $\mt_{g,n}^r$ is connected it is enough to check that the open embedding $\cM_{g,n} \subseteq \mt_{g,n}^r$ has a dense image, since $\cM_{g,n}$ is well known to be connected. This is equivalent to saying that every $n$-pointed $A_r$-stable curve over an algebraically closed extension $\Omega$ of $\kappa$ has a small deformation that is smooth. Let $(C, p_{1}, \dots, p_{n})$ be a $n$-pointed $A_r$-stable curve; the singularities of $C$ are unobstructed, so we can choose a lifting $\overline{C}\arr \spec \Omega\ds{t}$, with smooth generic fiber. The points $p_{i}$ lift to sections $\spec\Omega\ds{t} \arr \overline{C}$, and then the result follows from  Proposition~\ref{prop:openness}.

\end{proof}
\begin{remark}\label{rem: max-sing}
Clearly, we have an open embedding $\Mtilde_{g,n}^r  \subset \Mtilde_{g,n}^s$ for every $r\leq s$. Notice that $\Mtilde_{g,n}^r=\Mtilde_{g,n}^{2g+1}$ for every $r\geq 2g+1$. 
\end{remark}

We prove that the usual definition of Hodge bundle extends to our setting.  As a consequence we obtain a locally free sheaf $\htil_{g}$ of rank~$g$ on $\mt_{g, n}^r$, which is called \emph{Hodge bundle}.

 	\begin{proposition}\label{prop:hodge-bundle}
 	Let $\pi\colon C \arr S$ be an $A_r$-stable of genus $g$. Then $\pi_{*}{\omega}_{C/S}$ is a locally free sheaf of rank $g$ on $S$, and its formation commutes with base change.
 \end{proposition}
 \begin{proof}
 	If $C$ is an $A_r$-stable curve of genus $g$ over a field $k$, the dimension of $\H^{0}(C, \omega_{C/k})$ is $g$; so the result follows from Grauert's theorem when $S$ is reduced. But the versal deformation space of an $A_{r}$-stable curve over a field is smooth, so every $A_{r}$-stable curve comes, \'{e}tale-locally, from an $A_{r}$-stable curve over a reduced scheme, and this proves the result.
 \end{proof}

 Finally, we define the contraction morphism which will be useful later. First of all, we study the possible contractions over an algebraically closed field. We refer to Definition 1.3 of \cite{Knu} for the definition of contraction.
 
 \begin{remark}\label{rem:contrac}
 	Let $k$ be an algebraically closed field and $(C,p_1,\dots,p_{n+1})$ a $n+1$-pointed $A_r$-stable curve of genus $g\geq 2$. Suppose $(C,p_1,\dots,p_n)$ is not stable. If $r \leq 2$, this implies that $p_{n+1}$ lies in either a rational bridge or a rational tail, see case $(1)$ and case $(2)$ in \Cref{fig:contrac}. If $r \geq 3$, we have another possible situation: the point $p_{n+1}$ lies in an irreducible component of genus $0$ which intersect the rest of the curve in a tacnode, see case $(3)$ in \Cref{fig:contrac}. This is a conseguence of \Cref{rem:genus-count}. We call this irreducible component a rational almost-bridge; it is a limit of  rational bridges. One can construct a contraction $c$ of an almost-bridge which satisfy all the properties listed in Definition 1.3 of \cite{Knu}. The image of the almost-bridge is a $A_2$-singularity, i.e. a cusp. In \Cref{fig:contrac}, we describe the three possible situations, where $p$ is a smooth point and $q$ is the image of $p$ through $c$.
 	\begin{figure}[H]
 		\caption{Contraction morphisms}
 		\centering
 		\includegraphics[width=1\textwidth]{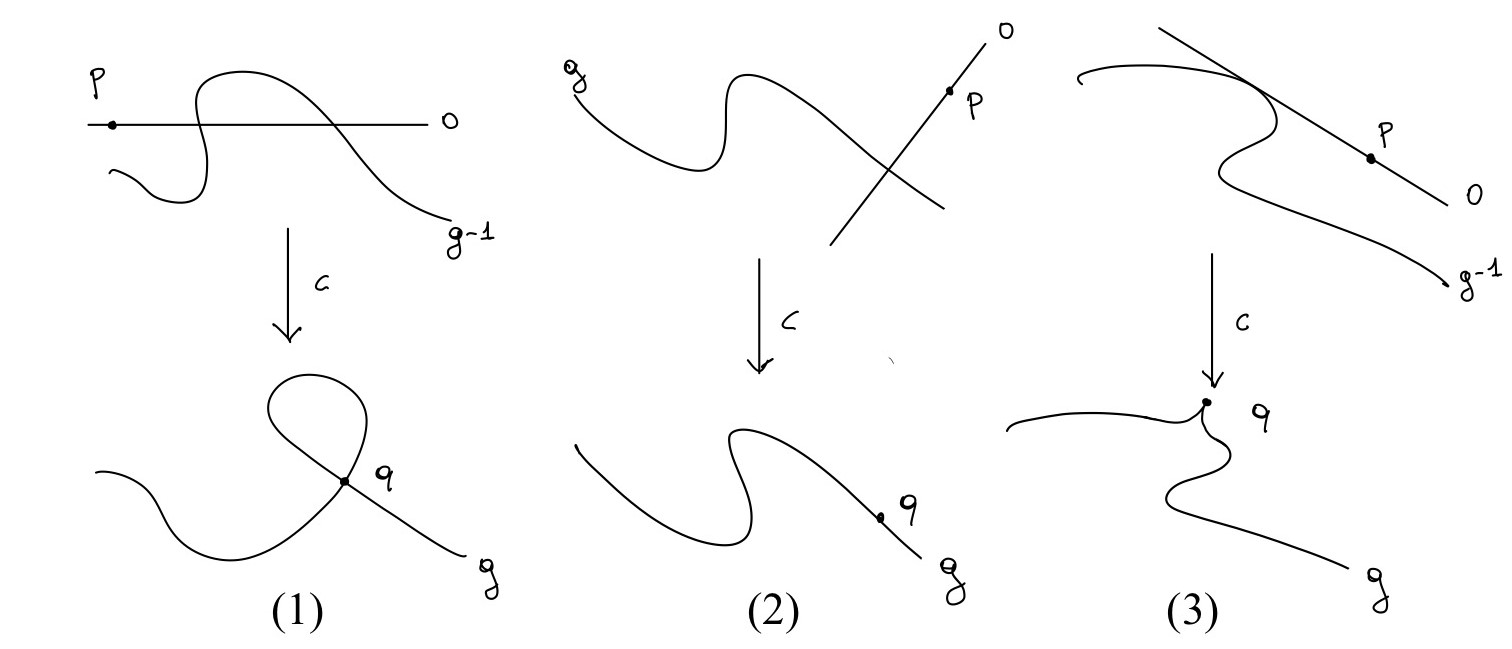}
 		\label{fig:contrac}
 	\end{figure}
 \end{remark} 
 
 \begin{lemma}\label{lem:knutsen}
 	Let $(C,p_1,\dots,p_n,p_{n+1})$ be a $n+1$-pointed $A_r$-stable curve of genus $g$ over an algebraically closed field $k/\kappa$ such that $2g-2+n>0$. Then the line bundle $\omega_C(p_1,\dots,p_n)$ satisfies the same statement as in \Cref{prop:boundedness}.
 \end{lemma}
 
 \begin{proof}
 	We do not give all the detail as again we rely on the results in \cite{Knu}. The idea is to use Lemma 1.6 of \cite{Knu}. If $(C,p_1,\dots,p_n)$ is a $n$-pointed $A_r$-stable curve, then the same proof of \Cref{prop:boundedness} works. Otherwise, we can contract the irreducible component of genus $0$ which is not stable without the point $p_{n+1}$. Therefore Lemma 1.6 of \cite{Knu} generalizes in our situation and we can conclude.  
 \end{proof}

 \begin{proposition}\label{prop:contrac}
 	We have a morphism of algebraic stacks 
 	$$ \gamma:\Mtilde_{g,n+1}^r \longrightarrow \Ctilde_{g,n}^r$$
 	where $\Ctilde_{g,n}^r$ is the universal curve of $\Mtilde_{g,n}^r$. Furthermore, it is an open immersion and its image is the open locus in $\Ctilde_{g,n}^r$ parametrizing $n$-pointed $A_r$-stable genus $g$ curves $(C,p_1,\dots,p_n)$ and a (non-necessarily smooth) section $q$ such that $q$ is an $A_h$-singularity for $h\leq 2$.
 \end{proposition}

\begin{proof}
	We again follow the proof of Proposition 2.1 in \cite{Knu}. Let $$(C/S,p_1,\dots,p_n,p_{n+1})$$ be an object in $\Mtilde_{g,n+1}^r$ and consider the morphism induced by $\pi_*\omega_{C/S}(p_1+\dots+p_n)^{\otimes 6}$. We know that the line bundle satisfy base change because of \Cref{lem:knutsen}. Furthermore, because it is normally generated we have that the image of the morphism can be described as the relative $\proj$ of the graded algebra $$\cA:=\bigoplus_{n \in \NN} \pi_*\omega_{C/S}(p_1+\dots+p_n)^{\otimes 6i}$$ which is flat over $S$ because of the base change property. Therefore we have that $\widetilde{C}:=\proj_S(\cA) \rightarrow S$ is a flat proper finitely presented morphism and, if we define $q_i$ as the image of $p_i$ through the morphism for $i=1,\dots,n+1$, we have that $(\widetilde{C},q_1,\dots,q_n)$ is a $n$-pointed $A_r$-stable genus $g$ curve. Over an algebraically closed field, a computation shows that the morphism is either the identity if the $n$-pointed curve $(C,p_1,\dots,p_n)$ is stable or one of the three contractions described in \Cref{rem:contrac}. Therefore the uniquess follows from Proposition 2.1 of \cite{Knu}. This implies we have that $\gamma$ is a monomorphism and in particular representable. It is enough to prove that $\gamma$ is \'etale to have that it is an open immersion.
		
    Notice that if we consider a geometric point $(C,p_1,\dots,p_{n+1})$ where there are no almost-bridges, then the morphism $\gamma$ is an isomorphism in a neighborhood of that point. This follows from \Cref{lem:sep-sing} and the result for stable curves. The only non trivial case is when we have an almost-bridge, which does not depend on the existence of the other sections $p_1,\dots,p_n$. For the sake of clarity, we do the case $n=0$. The general case is analogous.
    
    Let us consider the case when the geometric object $(C,p)$ has an almost-bridge, that is going to be contracted to  a cusp through the morphism $\gamma$. Let us denote by $(\widetilde{C},q)$ the image of $(C,p)$ through $\gamma$. We know that $q$ is the image of the almost-bridge and in particular a cusp. As usual, deformation theory tells us that $\Ctilde_g^r$ is smooth, therefore we can consider a smooth neighborhood $(\spec A,m_q)$ of $(\widetilde{C},q)$ in $\Ctilde_g^r$ with $A$ a smooth algebra. This implies that we have an $A_r$-stable curve
    $$
    \begin{tikzcd}
    	\widetilde{C}_A \arrow[r] & \spec A \arrow[l, "q_A"', bend right]
    \end{tikzcd}
	$$
	with $q_A$ a global section such that $q_A \otimes_A A/m_q = q$. By deformation theory of cusps (see Example 6.2.12 of \cite{TalVis}), we know that the completion of the local ring of $q$ in $\widetilde{C}_A$ is of the form
	$$ A[[x,y]]/(y^2-x^3-r_1x-r_2)$$
	where $r_1,r_2$ are part of a system of parameters for $A$ because $(\spec A, m_q)$ is versal. Consider know the element $f:=4r_1^3+27r_2^2 \in A$, which parametrizes the locus when the section $q_A$ ends in a singular point, which is either a node or a cusp. Let $q_1:\spec A/f \into \widetilde{C}_A$ be the codimension two closed immersion and denote by $C'_A:={\rm Bl}_{q_1}\widetilde{C}$. We have $C'_A$ is clearly proper and finitely presented over $A$, but it is also flat because $\tor_2^A(A/f^n,N)=0$ for any $A$-module $N$, which implies $\tor_1^A(I_{q_1}^n,N)=0$ for any $A$-module $N$ (we denote by $I_{q_1}$ the defining ideal of $q_1$ in $C_A$). We want to prove that the geometric fiber of $C_A'$ over $m_q$ is an almost-bridge.
 	
	 Notice that the formation of the blowup does not commute with arbitrary base change. However, consider the system of parameters $(r_1,r_2,\dots,r_a)$ with $a:=\dim A=3g-2$. Then if we denote by $J:=(r_2,\dots,r_a)$ we have that 
	$$ \tor_1^A(A/f^n,A/J)=0$$
	for every $n$. This implies that the blowup commutes with the base change for $A/J$ and therefore we can reduce to the case when $A$ is a DVR, $r:=r_1$ is the uniformizer and the completion of the local ring of $q$ in $\widetilde{C}_A$ is of the form 
	$$ A[[x,y]]/(y^2-x^3-rx).$$ A computation proves that over the special fiber we have
	\begin{itemize}
		\item $C_A'\otimes_A A/m_q$ is an $A_r$-prestable curve and if we denote by $p'_A$ the proper transform of $q_A$, we have that $(C_A' \otimes_A A/m_q, p'_A \otimes_A A/m_q)$ is an $1$-pointed $A_r$-stable curve of genus $g$,  
		\item $C_A' \rightarrow \widetilde{C}_A \otimes_A A/m_q$ is the contraction of an almost-bridge.
		\item  we have an isomorphism $(C_A'\otimes_A A/m_q,p_A'\otimes_A A/m_q) \rightarrow (C,p)$ of schemes over $(\widetilde{C},q)$.
	\end{itemize}  

Because the nodal case is already being studied in \cite{Knu} and Proposition 2.1 of \cite{Knu} is still true in our setting, we have that the object $(C_A',p_A')$ gives rise to a lifting
$$
\begin{tikzcd}
	& \spec A \arrow[d, "{(\widetilde{C}_A,q_A)}"] \arrow[ld, "{(C_A',p_A')}"', dashed] \\
	{\Mtilde_{g,1}^r} \arrow[r, "\gamma"'] & \Ctilde_{g}^r.                                                            
\end{tikzcd}
$$  
which implies that $\gamma$ is \'etale. 
\end{proof}

\section{Moduli stack of hyperelliptic curves}\label{sec:1-3}

Another important actor is the moduli stack of hyperelliptic $A_r$-stable curves. We start by proving a technical lemma.

\begin{lemma}\label{lem:quotient}
	Let $n$ be a positive integer coprime with the characteristic of $\kappa$.
	Let $X\arr S$ be a finitely presented, proper, flat morphism of schemes over $\kappa$ and let $\mmu_{n,S}\arr S$ the group of $n$-th roots acting on $X$ over $S$. Then there exists a geometric categorical quotient $X/\mmu_n\arr S$ which it is still  flat, proper and finitely presented.  
\end{lemma}
\begin{proof}
First of all, we know the existence and also the separatedness (see Corollary 5.4 of\cite{Rydh}) because $\mmu_{n,S}$ is a finite locally free group scheme over $S$. Since $\mmu_{n}$ is diagonalizable, we know that the formation of the quotient commutes with base change and that flatness is preserved. We also know that if $S$ is locally noetherian, we get that $X/\mmu_{n,S} \arr S$ is locally of finite presentation and proper (see Proposition 4.7 of \cite{Rydh}). Because proper and locally of finite presentation are two local conditions on the target, we can reduce to the case $S=\spec R$ is affine. Using now that the morphism $X\arr S$ is of finite presentation we get that there exists $R_0 \subset R$ subalgebra of $R$ which is of finite type over $\kappa$ (therefore noetherian) such that there exist a morphism $X_0\arr S_0:=\spec R_0$ proper and flat, and an action of $\mmu_{S_0,n}$ over $X_0/S_0$ such that the pullback to $S$ is our initial data. Because formation of the quotient commutes with base change, we get that we can assume $S$ noetherian and we are done.
\end{proof}
\begin{definition}
	Let $C$ be an $A_r$-stable curve of genus $g$ over an algebraically closed field. We say that $C$ is hyperelliptic if there exists an involution $\sigma$ of $C$ such that the fixed locus of $\sigma$ is finite and the geometric categorical quotient, which is denoted by $Z$, is a reduced connected nodal curve of genus $0$. We  call the pair $(C,\sigma)$ a \emph{hyperelliptic $A_r$-stable curve} and such $\sigma$  is called a \emph{hyperelliptic involution}.
\end{definition}

Finally we can define $\Htilde_g^r$ as the following fibered category: its objects are the data of a pair $(C/S,\sigma)$ where $C/S$ is an $A_r$-stable curve over $S$ and $\sigma$ is an involution of $C$ over $S$ such that $(C_s,\sigma_s)$ is a $A_r$-stable hyperelliptic curve of genus $g$ for every geometric point $s \in S$. These are called \emph{hyperelliptic $A_r$-stable curves over $S$}. A morphism is a morphism of $A_r$-stable curves that commutes with the involutions. We clearly have a morphism of fibered categories 
$$\eta:\Htilde_g^r  \larr \Mtilde_g^r$$
 over $\kappa$ defined by forgetting the involution. This morphism is known to be a closed embedding for $r\leq 1$, where both source and target are smooth algebraic stacks of finite type over $\kappa$.
 In the case of $r = 0$ the theory is well-known, so from now on, we can suppose $r\geq 1$. We assume ${\rm char} \, \kappa > r$ and in particular $2$ is invertible.
 \begin{remark}
 	 We have a morphism $i:\Htilde_g^r \arr \mathit{I}_{\Mtilde_g^r}$ over $\Mtilde_g^r$ induced by $\eta$ where $\mathit{I}_{\Mtilde_g^r}$ is the inertia stack of $\Mtilde_g^r$. This factors through $\mathit{I}_{\Mtilde_g^r}[2]$, the two torsion elements of the inertia, which is closed inside the inertia stack. It is fully faithful by definition. This implies $\Htilde_g^r$ is a prestack in groupoid. It is easy to see that it is in fact a stack in groupoid.
 \end{remark}

We want to describe $\Htilde_g^r$ as a connected component of the $2$-torsion of the inertia. To do so, we need to following lemma.  

\begin{lemma}\label{lem:conn-comp}
	Let $C\arr S$ be a $A_r$-stable curve over $S$, and $\sigma$ be an automorphism of the curve over $S$ such that $\sigma^2=\id$. Then there exists an open and closed subscheme $S'\subset S$ such that the following holds: a morphism $f:T\arr S$  factors through $S'$ if and only if $(C\times_S T, \sigma \times_S T)$ is a hyperelliptic $A_r$-stable curve over $T$.
\end{lemma}
  
\begin{proof}
Consider $$S':=\{s\in S| (C_s,\sigma_s) \in \Htilde_g^r(s)  \}\subset S$$
i.e. the subset where the geometric fibers over $S$ are hyperelliptic $A_r$-stable curves of genus $g$. If we prove $S'$ is open and closed, we are done.
Let $Z\rightarrow S$ be the geometric quotient of $C$ by the involution. Because of \Cref{lem:quotient}, $Z\rightarrow S$ is flat, proper and finitely presented and thus both the dimension and the genus of the geometric fibers over $S$ are locally costant function on $S$. It remain to prove that the finiteness of the fixed locus is an open and closed condition. Notice that in general the fixed locus is not flat over $S$, therefore this is not trivial.

The openess follows from the fact the fixed locus of the involution is proper over $S$, and therefore we can use the semicontinuity of the dimension of the fibers. 
 
Regarding the closedness, the fixed locus of a geometric fiber over a point $s\in S$ is positive dimensional only when we have a projective line on the fiber $C_s$ such that intersects the rest of the curve only in separating nodes (because of \Cref{prop:descr-inv}). Therefore the result is a direct conseguence of \Cref{lem:local-node-involution}.
\end{proof}
 	
\begin{proposition}\label{prop:open-closed-imm}
	The morphism $i:\Htilde_g^r \arr \mathit{I}_{\Mtilde_g^r}[2]$ is an open and closed immersion of algebraic stacks. In particular, $\Htilde_g^r$ is a closed substack of $\mathit{I}_{\Mtilde_g^r}$ and it is locally of finite type (over $\kappa$). 
\end{proposition}

\begin{proof}
We first prove that $\Htilde_g^r$ is an algebraic stack, and an open and closed substack of $\mathit{I}_{\Mtilde_g^r}[2]$.
First of all, we need to prove that the diagonal of $\Htilde_g^r$ is representable by algebraic spaces. It follows from the following fact: given a morphism of fibered categories $X\arr Y$, we can consider the $2$-commutative diagram
$$
\begin{tikzcd}
X \arrow[d, "\Delta_X"] \arrow[r, "f"] & Y \arrow[d, "\Delta_Y"] \\
X\times X \arrow[r, "{(f,f)}"]       & Y\times Y;            
\end{tikzcd}
$$
 if $f$ is fully faithful, then the diagram is also cartesian.
Secondly, we need to prove that the morphism $i$ is representable by algebraic spaces and it is an open and closed immersion. Suppose we have a morphism $ V\arr \mathit{I}_{\Mtilde_g^r}$ from a $\kappa$-scheme $S$. Thus we have a cartesian diagram:
 $$
 \begin{tikzcd}
 F \arrow[d] \arrow[r, "i_S"] & S \arrow[d]              \\
 \Htilde_g^r \arrow[r, "i"]   & \mathit{I}_{\Mtilde_g^r}
 \end{tikzcd}
 $$
 where $F$ is equivalent to a category fibered in sets as $i$ is fully faithful. We can describe $F$ in the following way: for every $T$ a $\kappa$-scheme, we have $$F(T)=\{ f:T\arr S |(C_S\times_S T,\sigma_S \times_S T) \in \Htilde_g^r(T) \}$$
 where $(C_S,\sigma_S) \in \mathit{I}_{\Mtilde_g^r}(S)$ is the object associated to the morphism $S\arr   \mathit{I}_{\Mtilde_g^r}$. By \Cref{lem:conn-comp}, we deduce $F=S'$ and the morphism $i_S$ is an open and closed immersion.
\end{proof}

In the last part of this section we introduce another description of $\Htilde_g^r$, useful for understanding the link with the smooth case. We refer to \cite{AbOlVis} for the theory of twisted nodal curves, although we consider only twisted curves with $\mu_2$ as stabilizers as with no markings. 

The first description is a way of getting cyclic covers from $A_r$-stable hyperelliptic curves. 

\begin{definition}
	Let $\cC_g^r$ be the category fibered in groupois whose object are morphisms $f:C\rightarrow \cZ$ over some base scheme $S$ such that $C\rightarrow S$ is a family of $A_r$-stable genus $g$ curve, $\cZ\rightarrow S$ is a family of twisted curves of genus $0$ and $f$ is a finite flat morphism of degree $2$ which is generically \'etale. Morphisms are commutative diagrams of the form
	$$
	\begin{tikzcd}
	C \arrow[d, "\phi_C"] \arrow[r, "f"] & \cZ \arrow[d, "\phi_Z"] \arrow[r] & S \arrow[d] \\
	C' \arrow[r, "f'"]                   & \cZ' \arrow[r]                    & S'.         
	\end{tikzcd}
	$$
\end{definition}

\begin{remark}
	The definition implies that the morphism $f$ is \'etale over the stacky locus of $\cZ$ . First of all, we can reduce to the case of $S$ being a spectrum of an algebraically closed field. Let $\xi: B\mu_2\hookrightarrow \cZ$ be a stacky node of $\cZ$, thus $f^{-1}(\xi)\rightarrow B\mu_2$ is finite flat of degree two. It is clear that  $f^{-1}(\xi)\subset C$ implies $f^{-1}(\xi)\rightarrow B\mu_2$ \'etale. It is easy to prove that over a stacky node of $\cZ$ there can only be a separating node of $C$. 
\end{remark}

The theory of cyclic covers (see for instance \cite{ArVis}) guarantees the existence of the functor of fibered categories 
$$\gamma:\cC_g^r \longrightarrow \Htilde_g^r$$
defined in the following way on objects: if $f:C\rightarrow Z$ is finite flat of degree $2$, we can give a $\ZZ/(2)$-grading on $f_*\cO_C$, because it splits as the sum of $\cO_{\cZ}$ and some line bundle $\cL$ on $\cZ$ with a section $\cL^{\otimes 2}\hookrightarrow \cO_{\cZ}$. This grading defines an action of $\mu_2$ over $C$. Everything is relative to a base scheme $S$. The geometric quotient by this action is the coarse moduli space of $\cZ$, which is a genus $0$ curve. The fact that $f$ is generically \'etale, implies that the fixed locus is finite. 

We prove a general lemma which gives us the uniqueness of an involution once we fix the geometric quotient. 

\begin{lemma}\label{lem:unique-inv-quotient}
	Suppose $S$ is a $\kappa$-scheme.  Let $X/S$ be a finitely presented, proper, flat $S$-scheme and $\sigma_1,\sigma_2$ be two involutions of $X/S$. Consider a geometric quotient $\pi_i:X\rightarrow Y_i$ of the involution $\sigma_i$ for $i=1,2$. If there exists an isomorphism  $\psi:Y_1 \rightarrow Y_2$ of $S$-schemes which commutes with the quotient maps, then $\sigma_1=\sigma_2$. 
\end{lemma}
\begin{proof}
	Let $T$ be a $S$-scheme and $t:T\rightarrow X$ be a $T$-point of $X$. We want to prove that $\sigma_1(t)=\sigma_2(t)$.
	Fix $i\in \{1,2\}$ and consider the cartesian diagram
	$$
	\begin{tikzcd}
	X_{\pi_i(t)} \arrow[d] \arrow[r] & X \arrow[d, "\pi_i"] \\
	T \arrow[r, "\pi_i(t)"]          & Y_i,                
	\end{tikzcd}
	$$
	thus we have that there are at most two sections of the morphism $X_{\pi_i(t)}\rightarrow T$, namely $t$ and $\sigma_i(t)$. Using the fact that $\pi_2=\psi\circ \pi_1$, it follows easily that $\sigma_1(t)=\sigma_2(t)$.
\end{proof}	

\begin{proposition}\label{prop:cyclic-covers}
	The functor $\gamma:\cC_g^r \rightarrow \Htilde_g^r$ is an equivalence of fibered categories.
\end{proposition}

\begin{proof}
	We explicitly construct an inverse. Let $(C,\sigma)$ be a hyperelliptic $A_r$-stable genus $g$ curve over $S$. Consider $F_{\sigma}\subset C$ the fixed locus of $\sigma$, which is finite over $S$. \Cref{prop:descr-inv} implies that the defining ideal of $F_{\sigma}$ in $C$ is locally generated by at most $2$ elements and the quotient morphism is not flat exactly in the locus where it is generated by $2$ element. By the theory of Fitting ideals, we can describe the non-flat locus as the vanishing locus of the $1$-st Fitting ideal of $F_{\sigma}$. Let $N_{\sigma}$ be such locus in $C$. \Cref{lem:local-node-involution} implies that $N_{\sigma}$ is also open inside $F_{\sigma}$, therefore $F_{\sigma}\setminus N_{\sigma}$ is closed inside $C$. If we look at the stacky structure on the image of $F_{\sigma}\setminus N_{\sigma}$ through the stacky quotient morphism $C\rightarrow [C/\sigma]$, we know that the stabilizers are isomorphic to $\mu_2$ as it is contained in the fixed locus. We denote by $[C/\sigma] \rightarrow \cZ$ the rigidification along $F_{\sigma}\setminus N_{\sigma}$ of $[C/\sigma]$. Because we are dealing with linearly reductive stabilizers, we know that the rigidification is functorial and $\cZ\rightarrow S$ is still flat (proper and finitely presented).
	
	We claim that the composition $$C\longrightarrow [C/\sigma] \longrightarrow \cZ$$ is an object of $\cC_g^r$. As the quotient morphism is \'etale, the only points where we need to prove flatness are the ones in $F_{\sigma}\setminus N_{\sigma}$, which are fixed points. Because the morphism $C\rightarrow \cZ$ locally at a point $p\in F_{\sigma}\setminus N_{\sigma}$ is the same as the geometric quotient, it follows from \Cref{prop:descr-inv} that it is flat. Thus $C\rightarrow \cZ$ is a finite flat morphism of degree $2$, generically \'etale as $F_{\sigma}$ is finite. 
	
	Hence, this construction defines a functor 
	$$\tau: \Htilde_g^r \longrightarrow \cC_g^r$$
	where the association on morphisms is defined in the natural way. A direct inspection using \Cref{lem:unique-inv-quotient} shows that $\tau$ and $\gamma$ are one the quasi-inverse of the other. 
	
\end{proof}

Using the theory developed in \cite{ArVis}, we know that $\cC_g^r$ is isomorphic to a stack of cyclic covers, namely the datum $C\rightarrow \cZ$ is equivalent to the triplet $(\cZ,\cL,i)$ where $\cZ$ is a twisted nodal curve, $\cL$ is a line bundle over $\cZ$ and $i:\cL^{\otimes 2}\rightarrow \cO_{\cZ}$ is a morphism of $\cO_{\cZ}$-modules. The vanishing locus of $i^{\vee}$ determines a subscheme of $\cZ$ which consists of the branching points of the cyclic cover. 

To recover $C$, we consider the sheaf of $\cO_{\cZ}$-algebras $\cA:=\cO_{\cZ}\oplus \cL$ where the algebra structure is defined by the section $i:\cL^{\otimes 2}\hookrightarrow \cO_{\cZ}$. Thus we define $C:=\spec_{\cZ}(\cA)$. Clearly not every triplet as above recovers a $A_r$-stable genus $g$ curve $C$. We need to understand what are the conditions for $\cL$ and $i$ such that $C$ is a $A_r$-stable genus $g$ curve. Because $C\rightarrow S$ is clearly flat, proper and of finite presentation,  it is a family of $A_r$-stable genus $g$ curve if and only if the geometric fiber $C_{s}$ is a $A_r$-stable genus $g$ curve for every point $s \in S$. Therefore we can reduce to understand the conditions for $\cL$ and $i$ over an algebraically closed field $k$. 

\begin{lemma}
	In the situation above, we have that $C$ is a proper one-dimensional Deligne-Mumford stack over $k$. Furthermore:
	\begin{itemize}
		\item[(i)] $C$ is a scheme if and only if the morphism $\cZ \rightarrow B\GG_m$ induced by $\cL$ is representable and $i$ does not vanish on the stacky nodes, 
		\item[(ii)] $C$ is reduced if and only if $i$ is injective,
		\item[(iii)] $C$ has arithmetic genus $g$ if and only if $\chi(\cL)=-g$.
	\end{itemize}
\end{lemma}

\begin{proof}
	 Because $C\rightarrow \cZ$ is finite, then $C$ is a proper one-dimensional Deligne-Mumford stack. To prove $(i)$, it is enough to prove that $C$ is an algebraic space, or equivalently it has trivial stabilizers. Because the map $f: C\rightarrow \cZ$ is affine, it is enough to check the fiber of the stacky points of $\cZ$. By the local description of the morphism, it is clear that if $p:B\mu_2 \hookrightarrow \cZ$ is a stacky point of $\cZ$, we have that the fiber $f^{-1}(p)$ is the quotient of the artinian algebra $k[x]/(x^2-h)$ by $\mu_2$, where $h \in k$. By the usual description of the cyclic cover, we know that $x$ is the local generator of $\cL$ and $h$ is the value of $i^{\vee}$ at the point $p$. The representability of the fiber then is equivalent to $h\neq 0$ and to the $\mu_2$-action being not trivial on $x$. Thus (i) follows.
	 	
	 Notice that $\cZ$ is clearly a Cohen-Macaulay stack and $f$ is finite flat and representable, therefore $C$ is also CM. This implies that $C$ is reduced if and only if it is generically reduced, i.e. the local ring of every irreducible component is a field. It is easy to see that this is equivalent to the fact that $i$ does not vanish in the generic point of any component.
	 	
	 As far as (iii) is concerned, firstly we prove that $\chi(\cL)\leq 0$ implies $C$ is connected. Suppose $C$ is the disjoint union of $C_1$ and $C_2$, then the involution has to send at least one irreducible component of $C_1$ to one irreducible component of $C_2$, otherwise the quotient $\cZ$ is not connected. But this implies that the restriction $f\vert_{C_1}$ is a finite flat representable morphism of degree $1$, therefore an isomorphism. In particular $f$ is a trivial \'etale cover of $\cZ$ and therefore $\chi(\cO_C)=2$, or equivalently $\chi(\cL)=1$. A straightforward computation now shows the equivalence.
	 	
\end{proof}
 
 \begin{remark}
 	Notice that $C$ is disconnected if and only if $\chi(\cO_C)=2$. We say that in this case $C$ has genus $-1$ and it is equivalent to $\chi(\cL)=1$.
 \end{remark}

Secondly, we have to take care of the $A_r$-prestable condition which makes sense only in the case when $C$ is a scheme. Therefore, we suppose $i^{\vee}$ does not vanish on the stacky points.

The $A_r$-prestable condition is encoded in the section $i^{\vee}\in \H^0(\cZ,\cL^{\otimes -2})$. The local description of the cyclic covers of degree $2$ (see \Cref{prop:descr-inv}) implies the following lemma.

\begin{lemma}
	In the situation above, $C$ is an $A_r$-stable curves if and only if $i^{\vee}$ has the following properties: 
	\begin{itemize}
		\item if it vanishes at a non-stacky node, then $r\geq 3$, $i^{\vee}$ vanishes with order $2$ and, locally at the node, it is not a zero divisor (there is a tacnode over the regular node),
		\item if it vanishes at a smooth point, than it vanishes with order at most $r$.
	\end{itemize}
\end{lemma}

Lastly, we want to understand how to describe the stability condition. For this, we need a lemma.

\begin{lemma}
	Let $(C,\sigma)$ be an $A_r$-prestable hyperelliptic curve of genus $g\geq 2$ over an algebraically closed field such that the geometric quotient $Z:=C/\sigma$ is integral of genus $0$, i.e. a projective line. Then $(C,\sigma)$ is $A_r$-stable. 
	
	Furthermore, suppose instead that $g=1$ and let $p_1,p_2$ be two smooth points such that $\sigma(p_1)=p_2$. Then $(C,p_1,p_2)$ is stable.
\end{lemma}
\begin{proof}
	If the curve $C$ is integral, there is nothing to prove. Suppose $C$ is not, then it has two irreducible components $C_1$ and $C_2$ of genus $0$ that has to be exchanged by the involution and their intersection $C_1 \cap C_2$ is a disjoint union of $A_{2k+1}$-singularities for $2k+1\leq r$. Let $p_1,\dots,p_h$ be the support of the intersection and $k_1,\dots,k_h$ integers such that $p_i$ is a $A_{2k_i+1}$-singularity for $i=1,\dots,h$. By \Cref{rem:genus-count}, we have that 
	$$  k_1+\dots+k_h-1=g\geq 2$$ 
	but at the same time \cite[Lemma 1.12]{Cat} implies
	$$\deg\omega_C\vert_{C_j}=-2+k_1+\dots+k_h=g-1>0$$
    for $j=1,2$. 
    
    Suppose now $g=1$. If $C$ is integral, then there is nothing to prove. If $C$ is not integral, again it has two irreducible components of genus $0$ such that their intersection is either two nodes or a tacnode, because of \Cref{rem:genus-count} (see the proof of \Cref{lem:genus1} for a more detailed discussion). Then again the statement follows.
\end{proof}

\begin{remark}
	In the previous lemma, we can take $p:=p_1=p_2$ to be a fixed smooth point of the hyperelliptic involution. In this case, $C$ has to be integral and therefore $(C,p)$ is $A_r$-stable. 
\end{remark}

The stability condition makes sense when $C$ is Gorestein, and in particular when it is $A_r$-prestable. Therefore suppose we are in the situation when $C$ is a $A_r$-prestable curve. We translate the stability condition on $C$ to a condition on the restrictions of $i$ to the irreducible components of $\cZ$.

Given an irreducible component $\Gamma$ of $\cZ$, we can define a quantity $$g_{\Gamma}:=\frac{n_{\Gamma}}{2}-\deg \cL\vert_{\Gamma}-1$$ where $n_{\Gamma}$ is the number of stacky points of $\Gamma$. It is easy to see that  $g_{\Gamma}$ coincides with the arithmetic genus of the preimage $C_{\Gamma}:=f^{-1}(\Gamma)$, when it is connected. The previous lemma implies that $\omega_{C}\vert_{C_{\Gamma}}$ is ample for all the components $\Gamma$ such that $g_{\Gamma}\geq 1$ . 

Let us try to understand the stability condition for $g_{\Gamma}=0$, i.e. $\deg\cL\vert_{\Gamma}=n_{\Gamma}/2-1$. Let $m_{\Gamma}$ be the number of points of the intersection of $\Gamma$ with the rest of the curve $\cZ$, or equivalently the number of nodes (stacky or not) lying on the component. Then the stability condition on $C$ implies that $2m_{\Gamma}-n_{\Gamma}\geq 3$, because the fiber of the morphism $C\rightarrow \cZ$ of every non-stacky node of $\cZ$ is of length $2$ (either two disjoint nodes or a tacnode) in $C$ while the fiber of a stacky node ($B\mu_2 \hookrightarrow \cZ$) has length $1$.

Suppose now that the curve $C_{\Gamma}$ is disconnected. Thus $C_{\Gamma}$ is the disjoint union of two projective line with an involution that exchange them. In this case $\cL\vert_{\Gamma}$ is trivial but also $n_{\Gamma}=0$. Stability condition on $C$ is equivalent to $m_{\Gamma}\geq 3$.  Notice that $g_{\Gamma}=-1$ if and only if $C_{\Gamma}$ is disconnected, or equivalently it is the \'etale trivial cover of projective line. 

This motivates the following definition.

\begin{definition}\label{def:hyp-A_r}
	Let $\cZ$ be a twisted nodal curve over an algebraically closed field. We denote by $n_{\Gamma}$ the number of stacky points of $\Gamma$ and by $m_{\Gamma}$ the number of intersections of $\Gamma$ with the rest of the curve for every $\Gamma$ irreducible component of $\cZ$. Let $\cL$ be a line bundle on $\cZ$ and $i:\cL^{\otimes 2} \rightarrow \cO_{\cZ}$ be a morphism of $\cO_{\cZ}$-modules.  We denote by $g_{\Gamma}$ the quantity $n_{\Gamma}/2-1-\deg\cL\vert_{\Gamma}$. 
	\begin{itemize}
	\item[(a)] We say that $(\cL,i)$ is hyperelliptic if the following are true:
			\begin{itemize}
				\item[(a1)] the morphism $\cZ \rightarrow B\GG_m$ induced by $\cL$ is representable,
				\item[(a2)] $i^{\vee}$ does not vanish restricted to any stacky point.
			\end{itemize}
	\item[(b)] We say that $(\cL,i)$ is $A_r$-prestable and hyperelliptic of genus $g$ if $(\cL,i)$ is hyperelliptic, $\chi(\cL)=-g$ and the following are true:
		\begin{itemize}
			\item[(b1)] $i^{\vee}$ does not vanish restricted to any irreducible component of $\cZ$ or equivalently the morphism $i:\cL^{\otimes 2 }\rightarrow \cO_{\cZ}$ is injective,
			\item[(b2)] if $p$ is a non-stacky node and $i^{\vee}$ vanishes at $p$, then $r\geq 3$ and the vanishing locus $\VV(i^{\vee})_p$ of $i^{\vee}$ localized at $p$ is a Cartier divisor of length $2$;
			\item[(b3)] if $p$ is a smooth point and $i^{\vee}$ vanishes at $p$, then the vanishing locus $\VV(i^{\vee})_p$ of $i^{\vee}$ localized at $p$ has length at most $r+1$.
		\end{itemize}	
	\item[(c)] We say that $(\cL,i)$ is $A_r$-stable and hyperelliptic of genus $g$ if it is $A_r$-prestable and hyperelliptic of genus $g$ and the following are true for every irreducible component $\Gamma$ in $\cZ$:
			\begin{itemize}
				\item[(c1)] if $g_{\Gamma}=0$ then we have $2m_{\Gamma}-n_{\Gamma}\geq 3$,
				\item[(c2)] if $g_{\Gamma}=-1$ then we have
				$m_{\Gamma}\geq 3$ ($n_{\Gamma}=0$).
			\end{itemize}
\end{itemize}
\end{definition}

Let us define now the stack classifying these data. We denote by $\widetilde{\cC}_g^r$ the fibered category defined in the following way: the objects are triplet $(\cZ\rightarrow S,\cL,i)$ where $\cZ \rightarrow S$ is a family of twisted curves of genus $0$, $\cL$ is a line bundle of $\cZ$ and $i:\cL^{\otimes 2}\rightarrow \cO_{\cZ}$ is a morphism of $\cO_{\cZ}$-modules such that the restriction $(\cL_s,i_s)$ to the geometric fiber is $A_r$-stable and hyperelliptic of genus $g$ for every point $s \in S$ .  Morphisms are defined as in \cite{ArVis}. We have proven the following equivalence.

\begin{proposition}\label{prop:descr-hyper}
	The fibered category $\widetilde{\cC}_g^r$ is isomorphic to $\cC_g^r$.
\end{proposition}

We can use this description to get the smoothness of $\Htilde_g^r$. Firstly, we need to understand what kind of line bundles $\cL$ on $\cZ$ can appear in $\Ctilde_g^r$.

\begin{lemma}\label{lem:hyp-line-bun}
	Let $Z$ be a nodal curve of genus $0$ over an algebraically closed field $k/\kappa$, $\cL$ a line bundle on $Z$ and $s \in \H^0(Z,\cL)$. We consider the following assertions:
	\begin{itemize}
		\item[$(i)$] $s$ does not vanish identically on any irreducible component of $Z$,
		\item[$(ii)$] $\cL$ is globally generated,
		\item[$(ii')$] $\deg \cL\vert_{\Gamma}\geq 0$ for every $\Gamma$ irreducible component of $Z$,
		\item[$(iii)$] $\H^1(Z,\cL)=0$;
	\end{itemize}
	then we have $(i)\implies (ii)\iff (ii')\implies (iii)$.
\end{lemma}

\begin{proof}
	It is easy to prove that $(i) \implies(ii')$ and $(ii)\implies (ii')$. To prove that $(ii')$ implies $(ii)$ we proceed by induction on the number of components of $Z$. Let $p$ a smooth point on $Z$ and $\Gamma_p$ the irreducible component which cointains $p$. Then the morphism $$\H^0(\Gamma_p,\cL\vert_{\Gamma_p})\longrightarrow k(p)$$ is clearly surjective because $\deg \cL\vert_{\Gamma_p}\geq 0$ and $\Gamma \simeq \PP^1$. Therefore it is enough to extend a section from $\Gamma_p$ to a section to the whole $Z$. 
	
	Because the dual graph of $Z$ is a tree, we know that $Z$ can be obtained by gluing a finite number of genus $0$ curve $Z_i$ to $\Gamma_p$ such that these curves are disjoint. Hence it is enough to find a section for every $Z_i$ such that it glues in the point of intersection with $\Gamma_p$. Because everyone of the $Z_i$'s has fewer irreducible components than $Z$, we are done.
	
	Finally, we need to prove that $(ii')\implies (iii)$. We know there exists a decomposition $Z=Z_1\cup Z_2$ where $Z_1$ and $Z_2$ are nodal genus $0$ curves such that $Z_1\cap Z_2$ has length $1$. Let $i_h:Z_h \hookrightarrow Z$ the closed embedding for $h=1,2$. Thus if we can consider the exact sequence of vector spaces 
	$$ 0\rightarrow \H^0(\cL)\rightarrow \H^0(i_1^*\cL)\oplus\H^0(i_2^*\cL) \rightarrow k \rightarrow \H^1(\cL) \rightarrow \H^1(i_1^*\cL)\oplus \H^1(i_2^*\cL) \rightarrow 0;$$  
	it is clear that the morphism $\H^0(i_1^*\cL)\oplus\H^0(i_2^*\cL) \rightarrow k$ is surjective (both $i_1^*\cL$ and $i_2^*\cL$ are globally generated) and therefore $\H^1(Z,\cL)=\H^1(Z_1,i_1^*\cL)\oplus \H^1(Z_2,i_2^*\cL)$. We get that 
	$$\H^1(Z,\cL)=\bigoplus_{\Gamma} \H^1(\Gamma, \cL\vert_{\Gamma})$$
	where the sum is indexed over the irreducible components $\Gamma$ of $Z$. Thus it is enough to prove that $\H^1(\Gamma,\cL\vert_{\Gamma})=0$ for every $\Gamma$ irreducible component of $Z$, which follows from $(iii))$.
\end{proof} 

\begin{proposition}\label{prop:smooth-hyp}
	The moduli stack $\Htilde_g^r$ of $A_r$-stable hyperelliptic curves of genus $g$ is smooth and the open $\cH_g$ parametrizing smooth hyperelliptic curves is dense in $\Htilde_g^r$. In particular $\Htilde_g^r$ is connected.
\end{proposition}

\begin{proof}
We use the description of $\Htilde_g^r$ as $\Ctilde_g^r$. First of all, let us denote by $\cP$ the moduli stack parametrizing the pair $(\cZ,\cL)$ where $\cZ$ is a twisted curve of genus $0$ and $\cL$ is a line bundle on $\cZ$. Consider the natural morphism $\cP \rightarrow \cM_0^{\rm tw}$ defined by the association $(\cZ,\cL)\mapsto \cZ$, where $\cM_0^{\rm tw}$ is the moduli stack of twisted curve of genus $0$ (see Section 4.1 of \cite{AbGrVis}).Proposition 2.7 \cite{AbOlVis} implies the  formal smoothness of this morphism. We have the smoothness of $\cM_0^{\rm tw}$ thanks to Theorem A.6 of \cite{AbOlVis}. Therefore $\cP$ is formally smooth over the base field $\kappa$.

 Let us consider now $\cP_g^0$, the substack of $\cP_g$ whose geometric objects are pairs $(\cZ,\cL)$ such that $\chi(\cL)=-g$ and $\H^1(\cZ,\cL^{\otimes -2})=0$. To be precise, its objects consist of families of twisted curves $\cZ\rightarrow S$ and a line bundle $\cL$ on $\cZ$ such that $\H^1(\cZ_s,\cL_s^{\otimes -2})=0$ for every $s \in S$. Because the Euler characteristic $\chi$ is locally constant for families of line bundles, the semicontinuity of the $\H^1$  implies that $\cP_g^0$ is an open inside $\cP$, therefore it is formally smooth over $\kappa$.

 We have a morphism 
 $$ \Htilde_g^r\simeq \Ctilde_g^r \rightarrow \cP$$ 
 defined by the association $(\cZ,\cL,i) \mapsto (\cZ,\cL)$. This factors through $\cP_g^0$ because of \Cref{lem:hyp-line-bun}. Consider the universal object $(\pi:\cZ_{\cP}\rightarrow \cP_g^0,\cL_{\cP})$ over $\cP_g^0$. Then $\cL_{\cP}^{\otimes -2}$ satisfies base change by construction and therefore we have that $\Ctilde_g^r$ is a substack of $\VV(\pi_*\cL_{\cP}^{\otimes -2})$, the geometric vector bundle over $\cP_g^0$ associated to $\pi_*\cL_{\cP}^{\otimes -2}$. The inclusion $\Ctilde_g^r\subset \VV(\pi_*\cL_{\cP}^{\otimes -2})$ is an open immersion because of \Cref{prop:openness}, which implies the smoothness of $\Ctilde_g^r$. 
 
 Given a twisted curve $\cZ$ over an algebraically closed field $k$, we can construct a family $\cZ_R\rightarrow \spec R$ of twisted curves where $R$ is a DVR such that the special fiber is $\cZ$ and the generic fiber is smooth. The smoothness of the morphism $\VV(\pi_{*}\cL^{\otimes -2})\rightarrow \Mtilde_0^{\rm tw}$ implies that given the datum $(\cZ,\cL,i) \in \Ctilde_g^r(k)$ with $k$ an algebraically closed field, we can lift it to $(\cZ_R,\cL_R,i_R) \in \Ctilde_g^r(R)$ where $R$ is a DVR such that it restricts to $(\cZ,\cL,i)$ in the special fiber and such that the generic fiber is isomorphic $(\PP^1, \cO(-g-1), i)$ with $i^{\vee} \in \H^0(\PP^1,\cO(2g+2))$. Finally, the open substack of $\VV(\pi_*\cL_{\cP}^{\otimes -2})$ parametrizing sections $i^{\vee}$ without multiple roots is dense, therefore we get that we can deform every datum $(\cZ,\cL,i)$ to the datum of a smooth hyperelliptic curve.
\end{proof}

\begin{remark}
	In the proof of \Cref{prop:smooth-hyp}, we have used the implication $(i) \implies (iii)$ as in  \Cref{lem:hyp-line-bun} for a twisted curve $\cZ$ of genus $0$. In fact, we can apply \Cref{lem:hyp-line-bun} to the coarse moduli space $Z$ of $\cZ$ and get the implication for $\cZ$ because the line bundle  $\cL^{\otimes -2}$ descends to $Z$ and the morphism $\cZ\rightarrow Z$ is cohomologically affine.
\end{remark}

\section{$\eta$ is a closed immersion}\label{sec:1-4}

We have a morphism $\eta:\Htilde_g^r \arr \Mtilde_g^r$ between smooth algebraic stacks. If we prove $\eta$ is representable, formally unramified, injective on geometric points and universally closed, then we get that $\eta$ is a closed immersion.
\begin{remark}
The morphism $\eta$ is faithful, as the automorphisms of $(C,\sigma)$ are by definition a subset of the ones of $C$ over any $\kappa$-scheme $S$. This implies that $\eta$ is representable.
\end{remark}
Firstly, we discuss why $\eta$ is injective on geometric points. We just need to prove that for every geometric point the morphism $\eta$ is full, i.e. every automorphism of a $A_r$-stable curve which is also hyperelliptic have to commute with the involution. This follows directly from the unicity of the hyperelliptic involution. Therefore our next goal is to prove that the hyperelliptic involution is unique over an algebraically closed field.

\subsection*{Injectivity on geometric points}

First of all, we use the results of the first section to describe the possible quotients.

\begin{proposition}\label{prop:description-quotient}
	Let $k/\kappa$ be an algebraically closed field and $(C,\sigma)$ be a hyperelliptic $A_r$-stable curve of genus $g$. Denote by $Z$ the geometric quotient by the involution and suppose $Z$ is a reduced nodal curve of genus $0$ (with only separating nodes). Furthermore, let $c\in C$ be a closed point, $z \in Z$ be the image of $c$ in the quotient and $C_z$ be the schematic fiber of $z$, which is the spectrum of an artinian $k$-algebra. Then
	
	\begin{itemize}
		\item if $z$ is a smooth point, either 
		\begin{itemize}
			\item[(s1)] $C_z$ is disconnected and supported on two smooth points of $C$ (i.e. the quotient morphism is étale at $c$),
			\item[(s2)] or $C_z$ is connected and supported on a possibly singular point of $C$ (i.e. the quotient morphism is flat and ramified at $c$); 
		\end{itemize}
		\item if $z$ is a separating node, either
		\begin{itemize}
			\item[(n1)] $C_z$ is disconnected and supported on two nodes (i.e. the quotient morphism is étale at $c$),
			\item[(n2)] or $C_z$ is connected and supported on a tacnode (i.e. the quotient morphism is flat and ramified at $c$),
			\item[(n3)] or $C_z$ is connected and supported on a node (i.e. the quotient morphism is ramified at $c$ but not flat). 
		\end{itemize}
	\end{itemize} 

Finally, the quotient morphism is finite of generic degree $2$ and ${\rm (n3)}$ is the only case when the length of $C_z$ is not $2$ but $3$. 
\end{proposition}

\begin{proof}
	As we are dealing with the quotient by an involution, the quotient morphism is either étale at closed point $c\in C$ or the point $c\in C$ is fixed by the involution. If the point is in the fixed locus, we can pass to the completion and apply  \Cref{prop:descr-inv}. 
\end{proof}
\begin{remark}
	We claim that if $C$ is an $A_r$-prestable curve of genus $g$ and $\sigma$ is any involution, then the geometric quotient $Z$ is automatically an $A_r$-prestable curve. The quotient is still reduced because $C$ is reduced. We have connectedness because of the surjectivity of the quotient morphism. The description of the quotient singularities in the first section implies the claim. Therefore if the quotient $Z$ has genus $0$, then it is a nodal curve 
	(with only separating nodes).
\end{remark}
\begin{remark}
      If $z$ is a node, we can say more about $C_z$. In fact, it is easy to prove that $C_z$ is a separating closed subset, i.e. the partial normalization of $C$ along the support of $C_z$ is not connected. This follows from the fact that every node in $Z$ is separating.
\end{remark}
We start by proving the uniqueness of the hyperelliptic involution for the case of an integral curve. For the remaining part of the section, $k$ is an algebraically closed field over $\kappa$. 

\begin{proposition}\label{prop:integral}
	Let $C$ be a $A_r$-stable integral curve of genus $g\geq 2$ over $k$ and suppose are given two hyperelliptic involution $\sigma_1,\sigma_2$ of $C$. Then $\sigma_1=\sigma_2$. 
\end{proposition}

\begin{proof}
First of all, notice that the quotient $Z:=C/\sigma$ is an integral curve with arithmetic genus $0$ over $k$, therefore $Z\simeq \PP^1$. Consider now the following morphism:
$$ 
\begin{tikzcd}
\phi: C \arrow[r, "f"] & \PP^1 \arrow[r, "i_{g-1}", hook] & \PP^{g-1}
\end{tikzcd}
$$ 
where $f$ is the quotient morphism and $i_{g-1}$ is the $(g-1)$-embedding. It is enough to prove that $\phi^*(\cO_{\PP^{g-1}}(1)) \simeq \omega_C$ as it implies that every hyperelliptic involution comes from the canonical morphism (therefore it is unique).

We denote by $\cL$ the line bundle $\phi^*\cO_{\PP^{g-1}}(1)$. Using Riemann-Roch for integral curves, we get that 
$$ h^0(C,\omega_C \otimes \cL^{-1})= h^0(C,\cL) + 1 -g $$ 
therefore if we prove that $h^0(C,\cL)= g$, we get that $\deg (\omega_C \otimes \cL^{-1})=0$ and that $h^0(\omega_C \otimes \cL^{-1})=1$ which implies $\cL \simeq \omega_C$ (as $C$ is integral). Because $f$ is finite, we have that $$\H^0(C,\cL)=\H^0(\PP^1,i_{g-1}^*\cO_{\PP^{g-1}}(1) \otimes f_*\cO_C) $$
thus using that $f_*\cO_C= \cO_{\PP^1}\oplus \cO_{\PP^1}(-g-1)$ as $f$ is a cyclic cover of degree $2$, we get that 
$$ \H^0(C,\cL)= \H^0(\PP^1, \cO_{\PP^1}(g-1))\oplus \H^0(\PP^1, \cO_{\PP^1}(-2))$$
which implies $h^0(C,\cL)=g$. 
\end{proof}

Now we deal with the genus $1$ case.
\begin{lemma}\label{lem:genus1}
	Let $(C,p_1,p_2)$ be a $2$-pointed $A_r$-stable curve of genus $1$. Then there exists a unique hyperelliptic involution which sends one point into the other. 
\end{lemma}

\begin{proof}
	The proof of this lemma consists in describing all the possible cases. First of all, the condition on the genus implies that the curve is $A_3$-stable, as the arithmetic genus would be too high with more complicated $A_r$-singularities. Clearly, 
	$$\sum_{\Gamma \subset C} g(\Gamma) \leq g(C)=1$$ 
	where $\Gamma$ varies in the set of irreducible components of $C$. Consider first the case where there exists $\Gamma$ such that $g(\Gamma)=1$, therefore all the other irreducible components have genus $0$ and all the separating points are nodes. Thanks to the stability condition, it is clear that $C$ is either integral (case (a) in \Cref{fig:Genus1} ) with two smooth points or it has two irreducible components $\Gamma_1$ and $\Gamma_0$, where $g(\Gamma_1)=1$, $g(\Gamma_0)=0$, they intersect in a separating node and the two smooth points lie on $\Gamma_0$ (case (b) in \Cref{fig:Genus1}).
	Suppose then that for every irreducible component $\Gamma$ of $C$ we have $g(\Gamma)=0$. Therefore, if the curve is not $A_1$-prestable, we have that the only possibility is that there exists a separating tacnode between two genus $0$ curves. Stability condition implies that $C$ can be described as two integral curves of genus $0$ intersecting in a tacnode, and the two points lie in different components (case (d) in \Cref{fig:Genus1}). Finally, if $C$ is $1$-prestable we get that $C$ has two irreducible components of genus $0$ intersecting in two points and the smooth sections lie in different components (case (c) in \Cref{fig:Genus1}). 
	\begin{figure}[H]
		\caption{$2$-pointed genus $1$ curves}
		\centering
		\includegraphics[width=1\textwidth]{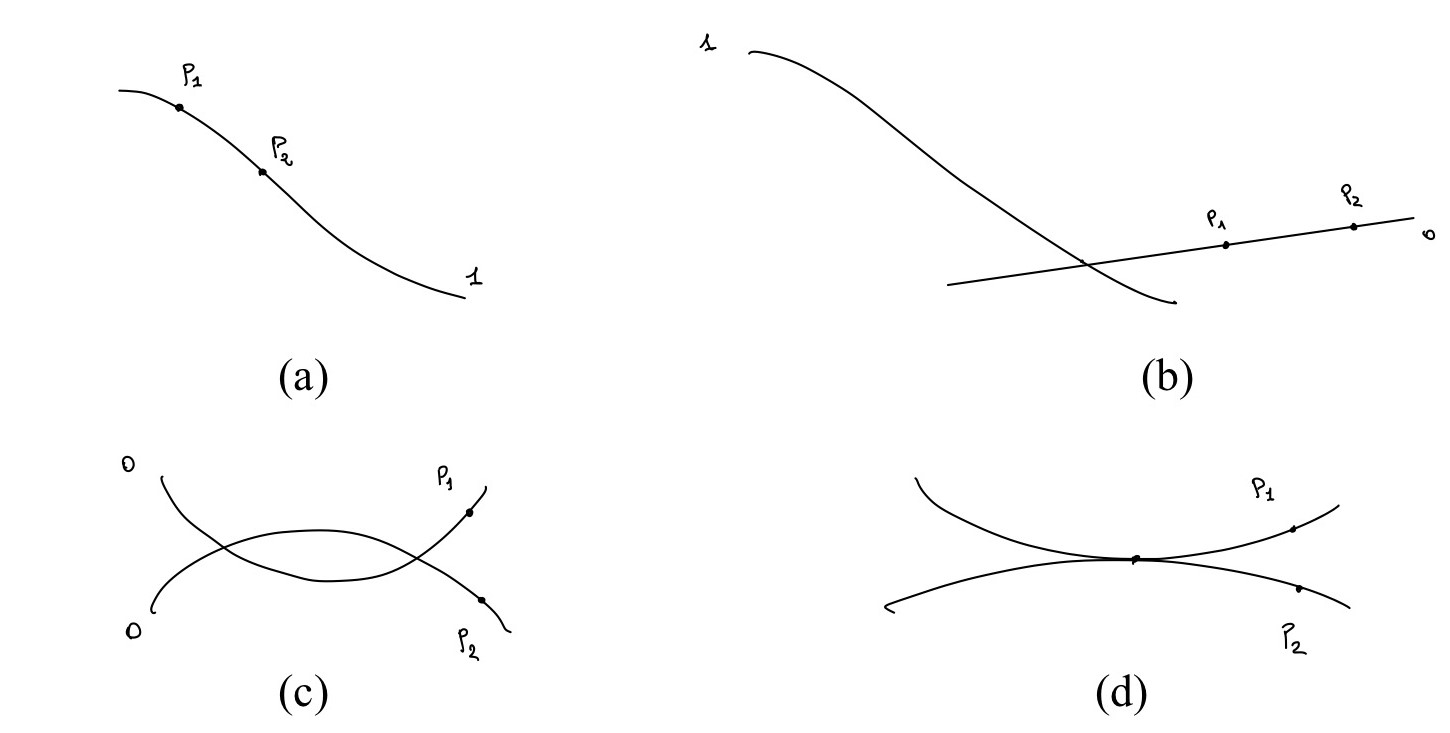}
		\label{fig:Genus1}
	\end{figure}
	In all of these four situations, it is easy to prove existence and uniqueness of the hyperelliptic involution. 
\end{proof}

\begin{remark}
In the previous lemma, we can also consider the case when  $p:=p_1=p_2$ and $(C,p)$ is an $1$-pointed $A_2$-stable curve. The same result is true.
\end{remark}
We want to treat the case of reducible curves.

\begin{definition}\label{def:subcurve}
	 Let $(C,\sigma)$ a hyperelliptic $A_r$-stable curve over an algebraically closed field. A one-equidimensional reduced closed subscheme $\Gamma \subset C$ is called a \emph{subcurve} of $C$. We denote by $\iota_{\Gamma}:\Gamma \rightarrow C$ the closed immersion. If $\Gamma$ is a subcurve of $C$, we denote by $C-C'$ the complementary subcurve of $C'$, i.e. the closure of $C\setminus C'$ in $C$. 
\end{definition}

\begin{remark}
	A subcurve is just a union of irreducible components of $C$ with the reduced scheme structure. Furthermore, given an one-equidimensional closed subset of $C$, we can always consider the associated subcurve with the reduced scheme structure.
\end{remark}

\begin{lemma}\label{lem:subcurve}
	 Let $(C,\sigma)$ be a hyperelliptic $A_r$-stable curve of genus $g\geq 2$ over an algebraically closed field and $\Gamma \subset C$ be a subcurve such that $g(\Gamma)\geq 1$. Then $\dim(\Gamma \cap \sigma(\Gamma))=1$.  
\end{lemma}

\begin{proof}
Suppose $\dim(\Gamma \cap \sigma(\Gamma))=0$ and consider the quotient morphism $f: C \arr Z:=C/\sigma$. Consider now the schematic image $f(\Gamma) \subset Z$, which is a subcurve of $Z$ as it is reduced and $f$ is finite. If we restrict $f$ to $f(\Gamma)$, we get the following commutative diagram:
$$
\begin{tikzcd}
\Gamma \arrow[r, hook] \arrow[rd] & f^{-1}f(\Gamma) \arrow[r, hook] \arrow[d, "f_{\Gamma}"] & C \arrow[d, "f"] \\
& f(\Gamma) \arrow[r, hook]                               & Z               
\end{tikzcd}
$$
where the square diagram is cartesian. If we consider $U:=f(\Gamma) \setminus f(\Gamma \cap \sigma(\Gamma))$, we get that $f^{-1}(U)$ is a disjoint union of two open subsets and $\sigma$ maps one into the other,  therefore the action on $f^{-1}(U)$ is free, giving that the degree of $f_{\Gamma}$ restricted to $f^{-1}(U)$ is $2$. The condition $\dim(\Gamma \cap \sigma(\Gamma))=0$ assures us that $f^{-1}(U)=(\Gamma \cup \sigma(\Gamma))\setminus (\Gamma \cap \sigma(\Gamma))$ is in fact dense in $f^{-1}f(\Gamma)$. Thus the morphism $f\vert_{\Gamma}:\Gamma \arr \pi(\Gamma)$ is a finite morphism which is in fact an isomorphism over $\Gamma\setminus \sigma(\Gamma)$, which is a dense open, therefore it is birational. This implies the following inequality
$$g(Z)\geq g(f(\Gamma))\geq g(\Gamma)\geq 1$$
which is absurd because $g(Z)=0$. 
\end{proof}

\begin{remark}
	Notice that this lemma implies that the only irreducible components that are not fixed by the hyperelliptic involution have (arithmetic) genus equal to $0$. Therefore, if we have a hyperelliptic involution $\sigma$ of $C$ and $\Gamma$ is an irreducible component of positive genus, we get that $\sigma(\Gamma)=\Gamma$ and $\sigma\vert_{\Gamma}$ is a hyperelliptic involution of $\Gamma$. This is true because we are quotienting by a linearly reductive group, thus the  morphism 
	$$\Gamma/\sigma \rightarrow C/\sigma$$
	induced by the closed immersion $\Gamma \subset C$ is still a closed immersion.
\end{remark}	

Let $C$ be an $A_r$-stable curve of genus $g\geq 2$. We denote by ${\rm Irr}(C)$ the set whose elements are the irreducible components of $C$. Then every automorphism $\phi$ of $C$ induces a permutation $\tau_{\phi}$ of the set ${\rm Irr}(C)$. First of all, we need to prove that the action on ${\rm Irr}(C)$ is the same for every hyperelliptic involution. Then we prove that the action on the $0$-dimensional locus described by all the intersections between the irreducible components is the same for all hyperelliptic involutions. Finally we see how these two facts imply the uniqueness of the hyperelliptic involution. We denote by $l(Q)$ the length of a $0$-dimensional subscheme $Q\subset C$.

\begin{lemma}\label{lem:exist-decomposition}
	Let $(C,\sigma)$ be a hyperelliptic $A_r$-stable curve of genus $g$ over $k$ and suppose there exists two irreducible components $\Gamma_1$ and $\Gamma_2$  of genus $0$ of $C$ such that $\sigma$ send $\Gamma_1$ in $\Gamma_2$. Let $n$ be the length of $\Gamma_1\cap \Gamma_2$. Then there exists a nonnegative integer $m$ and $m$ disjoint subcurves $D_i \subset C$ such that the following properties holds:
	\begin{itemize}
		\item[a)] $m+n\geq 3$,
		\item[b)] $\sigma(D_i)=D_i$ for every $i=1,\dots,m$,
		\item[c)] the length of the subscheme $D_i\cap \Gamma_j$ is $1$ for every $i=1,\dots,m$ and $j=1,2$ and $D_i \cap \Gamma_1 \cap \Gamma_2=\emptyset$,
		\item[d)] if we denote by $P_j^i$ the intersection $D_i \cap \Gamma_j$, we have that $(D_i,P_1^i,P_2^i)$ is a $2$-pointed $A_r$-stable curve of genus $g_i>0$ such that $\sigma\vert_{D_i}$ is a hyperelliptic involution of $D_i$ which maps $P_1^i$ to $P_2^i$,
		\item[e)] $C=\Gamma_1\cup \Gamma_2 \cup \bigcup_{i=1}^m D_i$.
	\end{itemize}
Furthermore, the following equality holds $$g=m+n-1+\sum_{i=0}^m g_i. $$
\end{lemma}

Before proving the lemma, let us explain it more concretely. First of all, the intersections in the lemma are scheme-theoretic, therefore we can have non-reduced ones. More precisely, the intersections have to be supported in singular points of type $A_{2h-1}$, as the local ring is not integral. It is easy to prove that the scheme-theoretic intersection in such a point is in fact of length $h$. Notice that both $m$ or $n$ can be zero.

\begin{figure}[H]
	\caption{Decomposition as in \Cref{lem:exist-decomposition}}
	\centering
	\includegraphics[width=1\textwidth]{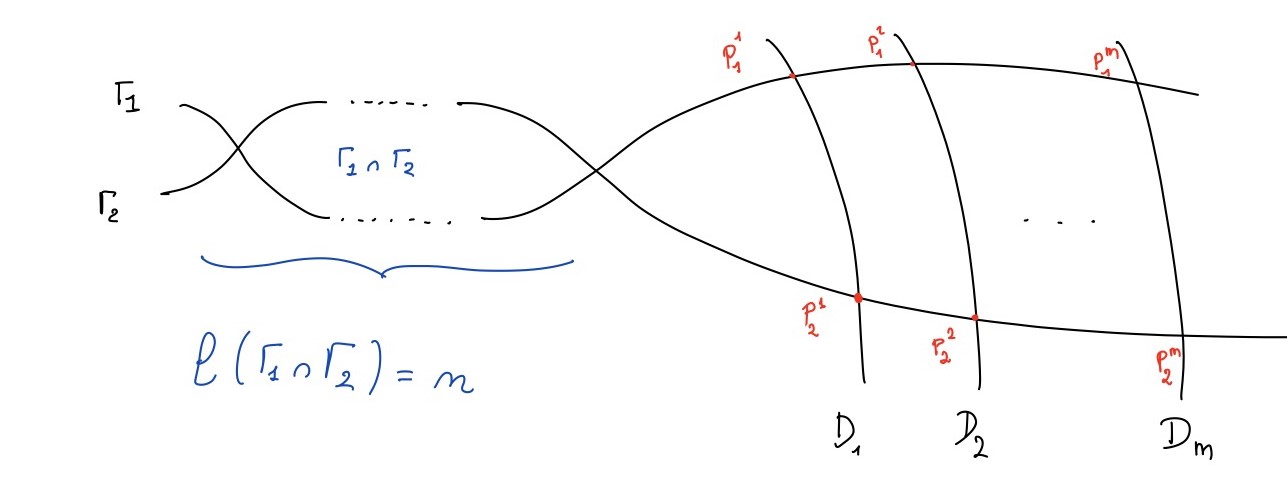}
	\label{fig:Decomp}
\end{figure}
	
\begin{proof}
	  Let $f:C\rightarrow Z$ be the quotient morphism and $\Gamma$ the image of $\Gamma_1\cup\Gamma_2$ through $\pi$. Because every node in $Z$ is a separating node, we know that $Z-\Gamma = \bigsqcup_{i=1}^m E_i$, i.e. it is a disjoint union (possibly empty) of $m$ subcurves of $Z$, which are still reduced, connected and of genus $0$.  Let $D_i$ be the subcurve of $C$ associated to the closed subset $\pi^{-1}(E_i)$ of $C$. We prove that $D_1,\dots,D_m$ verify the properties listed in the proposition.
	
	 Clearly, $D_i\cap D_j=\emptyset$ for every $i\neq j$ by construction. Properties b) and e) are verified by construction as well. Notice that $\Gamma_1\cap\Gamma_2\cap D_i=\emptyset$ for every $i=1,\dots,m$, otherwise if $p\in \Gamma_1\cap\Gamma_2\cap D_i$ is a closed point, then the local ring $\cO_{C,p}$ would have $3$ minimal primes. This cannot occur as the only singularities allowed are of type $A_r$, which have at most $2$ minimal primes. Therefore if we define $Q_i:=E_i\cap \Gamma$, we have that $\pi^{-1}(Q_i)$ is disconnected (as it does not belong to $\pi(\Gamma_1\cap\Gamma_2)$) and thus we are in the situation (n1) of \Cref{prop:description-quotient}. Property c) follows easily from this. Because $C$ is $A_r$-stable, we also get property a). 
	 
	 Regarding property d), we know that $D_i$ is reduced and that $D_i\cap (C-D_i)$ has length $2$, hence it is enough to prove $D_i$ is connected and then the statement follows. However, suppose $D_i$ is not connected, namely  $D_i=D_i^1 \bigsqcup D_i^2$ with $D_i^j$ two subcurves of $C$ for $j=1,2$ such that $\sigma(D_i^1)=D_i^2$. Using \Cref{lem:subcurve} again, we get $g(D_i^j)=0$ for $j=1,2$. This is not possible because of the stability condition on $C$. The genus formula follows from a straightforward computation using \Cref{rem:genus-count}.
\end{proof}	

Now that we have described the geometric structure of $C$, we can use it to prove that any other hyperelliptic involution has to act in the same way over the set ${\rm Irr}(C)$ of the irreducible components. 
\begin{lemma}
	Let $C/k$ be a $A_r$-stable curve of genus $g$ and suppose we have a decomposition as in \Cref{lem:exist-decomposition}. Thus every hyperelliptic involution $\sigma$ of $C$ commutes with the decomposition, i.e. we have $\sigma(\Gamma_1)=\Gamma_2$.
\end{lemma}

\begin{proof}
	Let $\pi:C\rightarrow Z$ be the quotient morphism. We start with the case $m=0$ and $n:=l(\Gamma_1\cap \Gamma_2)\geq 3$. Suppose $\sigma(\Gamma_j)=\Gamma_j$ for $j=1,2$. Recall that the only singularities that can appear in the intersection $\Gamma_1\cap \Gamma_2$ are of the form $A_{2h-1}$. As the involution does not exchange the two irreducible components, \Cref{prop:descr-inv} implies that the only possible singularities in the intersection are nodes or tacnodes with nodes as quotients. Because $n\geq 3$ we have that the quotient has two irreducible components intersecting in at least two nodes, but this does not have genus $0$. Therefore $\sigma(\Gamma_1)=\Gamma_2$.
	
	Suppose now $m\geq 1$ and $n\geq 2$. Because $n\geq 2$ then $\Gamma_1 \cup \Gamma_2$ has positive genus therefore $\sigma(\Gamma_1\cup\Gamma_2)$ and $\Gamma_1\cup\Gamma_2$ have a common component. Suppose $\sigma(\Gamma_1)\neq \Gamma_2$ and thus without loss of generality $\sigma(\Gamma_2)=\Gamma_2$. If $\sigma(\Gamma_1)=\Gamma_1$, then $\pi(\Gamma_1)\cap\pi(\Gamma_2)$ contains at least a node because $n\geq 2$ and the subcurve $\pi(\Gamma_1\cup\Gamma_2\cup D_1)$ of $Z$ does not have genus $0$ as $D_1$ is connected. Therefore we can suppose $\sigma(\Gamma_1)\subset D_1$. Then $\sigma(\Gamma_1\cap\Gamma_2)\subset D_1\cap\Gamma_2$, but this implies $l(\Gamma_1\cap\Gamma_2)\leq 1$ which is in contradiction with $n\geq 2$.
	
	Now we consider the case $n=1$ and $m\geq 2$. Again it is easy to prove that we cannot have $\sigma(\Gamma_i)=\Gamma_i$ for $i=1,2$. Without loss of generality we can suppose $\sigma(\Gamma_1)\subset D_1$. However $\sigma(D_2)\neq D_2$ as $D_2\cap \Gamma_1\neq \emptyset$ and thus $\sigma(D_2)\cap D_2$ has to share an irreducible component because of \Cref{lem:subcurve}. This implies that $\Gamma_2 \subset \sigma(D_2)$ because $D_2$ is connected. Thus $\sigma(\Gamma_1)\cap\sigma(\Gamma_2)\subset D_1\cap D_2 = \emptyset$, which is absurd. The only possibility is $\sigma(\Gamma_1)=\Gamma_2$.
	
	Finally, we consider the case $n=0$ and $m\geq 3$. As above, we cannot have $\sigma(\Gamma_j)=\Gamma_j$ for $j=1,2$, otherwise the quotient does not have genus $0$. If $\sigma(\Gamma_1)\subset D_1$, then $\sigma(\Gamma_2)$ is contained in only one of the subcurves $\{D_i\}_{i=1\dots,m}$. Thus there exists at least one between the $D_i$'s, say $D_2$, which is stable under the action of the involution (because $m\geq 3$, $D_i$ is connected for every $i=1,\dots,m$ and $\dim D_i \cap \sigma(D_i) = 1$). This is absurd because $\sigma(\Gamma_1)\cap \sigma(D_2) \subset D_1\cap D_2=\emptyset$. This implies $\sigma(\Gamma_1)=\Gamma_2$ and we are done.
\end{proof}

\begin{corollary}\label{cor:action-irrcomp}
	 If $C/k$ is an $A_r$-stable curve of genus $g\geq 2$ and $\sigma_1$ and $\sigma_2$ are two hyperelliptic involution, then $\tau_{\sigma_1}=\tau_{\sigma_2}$.
\end{corollary}

Let us study now the action on the intersections between irreducible components.

\begin{lemma}\label{lem:action-intersection}
	Let $\Gamma_1$ and $\Gamma_2$ two irreducible components of $C$ and $p\in \Gamma_1\cap \Gamma_2$ be a closed point. If $\sigma_1$ and $\sigma_2$ are two hyperelliptic involution of $C$, then $\sigma_1(p)=\sigma_2(p)$. 
\end{lemma}
\begin{proof}
 Suppose $\sigma$ is a hyperelliptic involution such that $\sigma(\Gamma_1)=\Gamma_2$. Because the quotient of $\Gamma_1\cup\Gamma_2$ by $\sigma$ is irreducible of genus $0$, we have that $\sigma(p)=p$ for every $p \in \Gamma_1\cap\Gamma_2$. Suppose now $\sigma(\Gamma_1)\neq \Gamma_2$ and denote by $\pi:C\rightarrow Z$ the quotient morphism. Clearly $\pi(\Gamma_1\cup\Gamma_2)$ is not irreducible and $\pi(\Gamma_1\cap\Gamma_2)$ is a separating node, therefore $\Gamma_1\cap\Gamma_2$ is either supported on a node, a tacnode or two disjoint nodes exchanged by the involution. In the first two cases, the intersection is supported on one point therefore there is nothing to prove. If the intersection is supported on two nodes, then every involution has to exchange them because otherwise the quotient would not have genus $0$.
\end{proof}

\begin{remark}\label{rem:part-case}
	We can say more about the case $\sigma(\Gamma_1)=\Gamma_2$. In this situation, we claim every hyperelliptic involution acts trivially on $\Gamma_1\cap \Gamma_2$ scheme-theoretically, not only set-theoretically. In fact, suppose we have a fix point $p\in \Gamma_1\cap\Gamma_2$.  We know that $\Gamma_1$ and $\Gamma_2$ are irreducible components of genus $0$ and the image of $p$ in the quotient by $\sigma$ is a smooth point. Therefore we can consider the local ring $A:=\cO_{C,p}$ and we know that the invariant subalgebra is a DVR which is denoted by $R$. By flatness, we know that $A=R[y]/(y^2-h)$ where $h$ is an element of $R$ and $\sigma$ is defined by the association $y\mapsto -y$. By hypothesis, $A$ has two minimal primes, namely $p$ and $q$, such that $p\cap q=0$. We can consider the morphism 
	$$A \rightarrow A/p\oplus A/q$$
	which is injective after completion, therefore injective. Because $R$ is a DVR, an easy computation shows that $h$ is a square, thus $A$ is of the form $R[y]/(y^2-r^2)$ with $r\in R$ and the involution is defined by the formula $\sigma(a+by)=a-by$. Clearly we have that the two minimal primes are $y-r$ and $y+r$ and the intersection $\Gamma_1\cap\Gamma_2$ is contained in the fixed locus of the action because is defined by the ideal $(y,r)$ (the fixed locus is defined by the ideal $(y)$).
 \end{remark}
Finally we can prove the theorem.
\begin{theorem}
	Let $C/k$ be a $A_r$-stable curves of genus $g\geq 2$. If $\sigma_1$ and $\sigma_2$ are two hyperelliptic involution of $C$, then $\sigma_1=\sigma_2$.
\end{theorem}

\begin{proof}
	 We prove that two hyperelliptic involutions coincide restricted to a decomposition of $C$ in subcurves.
	The integral case is done in \Cref{prop:integral}. Because of \Cref{cor:action-irrcomp} and \Cref{lem:action-intersection}, we know that every hyperelliptic involution $\sigma$ acts in the same way on the set of irreducible components and on the (set-theoretic) intersection on every pair of irreducible components. 
	
	 If we consider an irreducible component of genus greater than $2$, then $\sigma$ restricts to the irreducible component, it still is a hyperelliptic involution  and we can use \Cref{prop:integral} to get the uniqueness restricted to the component. If we restrict $\sigma$ to a component $\Gamma$ of genus $1$, we still have that it is a hyperelliptic involution. If we consider a point $p \in \Gamma\cap C-\Gamma$, we get that every other hyperelliptic involution $\sigma'$ has to verify $\sigma'(p)=\sigma(p)$ thanks to \Cref{lem:action-intersection}. Then we can apply \Cref{lem:genus1} to get the uniqueness restricted to the component $\Gamma$. 
	  
	  Finally, if $\Gamma$ has genus $0$, one can have only two possibilities, namely $\sigma(\Gamma)=\Gamma$ or $\sigma(\Gamma)\neq \Gamma$. If $\sigma(\Gamma)=\Gamma$, then two different involutions have to coincide when restricted to $\Gamma$ in at least two points by stability (we have tacnodes or nodes as intersections). This easily implies they coincide on the whole component. If $\sigma(\Gamma)\neq \Gamma$, it is easy to see that every hyperelliptic involution restricted to $\sigma(\Gamma) \cup \Gamma$ are determined by an involution of $\PP^1$ (one can just choose two identification $\Gamma\simeq \PP^1$ and $\sigma(\Gamma)\simeq \PP^1$). Because $\sigma$ does not fix $\Gamma$, we can use \Cref{lem:exist-decomposition} to get a decomposition $C$ in subcurves $D_1,\dots,D_m$ with the properties listed in the lemma. The stability condition, i.e. $m+l(\Gamma\cap \sigma(\Gamma))\geq 3$, implies that two hyperelliptic involution restricted to $\Gamma\cup \sigma(\Gamma)$ are determined by two involution of $\PP^1$ which coincide restricted to a subscheme of length $m+l(\Gamma\cap\sigma(\Gamma))$, therefore they coincide.
\end{proof}

\begin{remark}
	 Notice that for the case $m=0$ and $\Gamma\cap\sigma(\Gamma)$ supported on a single point we really need \Cref{rem:part-case}. In fact for the case of two $\PP^1$ glued on a subscheme of length greater or equal than $3$ concentrated on a single point, the condition that two hyperelliptic involution have to coincide on the set-theoretic intersection is not enough to conclude they are the same. 
\end{remark}
\subsection*{Unramifiedness}

Next we focus on the unramifiedness of the map $\eta$. We prove it using deformation theory of curves and of morphisms of curves.

\begin{remark}\label{rem:deformation}
Let $\cX,\cY$ be two algebraic stacks and let $f:\cY\arr \cX$ by a morphism. Suppose it is given a 2-commutative diagram
$$
\begin{tikzcd}
\spec k\arrow[r, "y"] \arrow[d] & \cY \arrow[d] \\
{\spec k[\epsilon]} \arrow[r, "x_{\epsilon}"]          & \cX                       
\end{tikzcd}
$$ 
with $k$ a field and $k[\epsilon]$ the ring of dual numbers over $k$. We define $x:=f(y)$ and $\cY_x$ the fiber product of $f$ with the morphism induced by $x$; clearly we have a lifting of $y$ from $\cY$ to $\cY_x$ which is denoted by $y$ by abuse of notation. By standard argument in deformation theory, we get the following exact sequence of vector spaces over $k$:
$$
\begin{tikzcd}
0 \arrow[r] & T_{\id}\aut_{\cY_x}(y) \arrow[r] & T_{\id}\aut_{\cY}(y) \arrow[r] & T_{\id}\aut_{\cX}(x) \arrow[lld, "\alpha_f(y)" description ] \\
& \pi_0(T_{y}\cY_x) \arrow[r]        & \pi_0(T_y\cY) \arrow[r]          & \pi_0(T_x\cX)                     
\end{tikzcd}
$$ 
where $T_x \cX$ is the groupoid of morphisms $x_{\epsilon}:\spec k[\epsilon] \arr \cX$ such that the composition with $\spec k \hookrightarrow \spec k[\epsilon]$ is exactly $x:\spec k \arr \cX$. By standard notation, we call $T_x\cX$ the tangent space of $\cX$ at $x$.
\end{remark} 

If we prove that $\pi_0(T_y\cY_x)=0$, then the morphism $f$ is fully faithful at the level of tangent spaces, and therefore unramified. We prove that this is true for the morphism $\eta$. 

First of all, we need to describe the fiber $\Htilde_C$ of the morphism $\eta: \Htilde_g^r \rightarrow \Mtilde_g^r$ in a point $C \in \Mtilde_g^r(k)$, where $k/\kappa$ is an extension of fields with $k$ algebraically closed. As the map $\eta$ is faithful, we know that $\Htilde_C$ is equivalent to a set. Given a point $(C,\sigma)\in \Htilde_C(k)$, we have that an element in $T_{(C,\sigma)}\Htilde_C$ is a pair $(C[\epsilon],\sigma_{\epsilon})$ where $C[\epsilon]$ is the trivial deformation of $C$ and $\sigma_{\epsilon}:C[\epsilon]\rightarrow C[\epsilon]$ is a deformation of $\sigma$. Therefore, we need to prove the uniqueness of the hyperelliptic involution for deformations. To do so, we study the deformations of the quotient map $\pi: C \rightarrow Z:=C/\sigma$. 

Let ${\rm Def}^{\rm fix}_{C/Z}$ be the deformation functor associated to the problem of deforming the morphism $\pi:C \rightarrow Z$ with both source and target fixed and ${\rm InfAut}(Z)$ be the deformation functor of infinitesimal automorphisms of $Z$. There is a natural morphism of deformation functors
$$ \alpha:{\rm InfAut}(Z) \longrightarrow {\rm Def}^{\rm fix}_{C/Z}$$
whose restriction to the tangent spaces $d\alpha$ induces a morphism of $k$-vector spaces. Furthermore, we have a map $$\gamma:T_{(C,\sigma)}\Htilde_C \longrightarrow {\rm Def}^{\rm fix}_{C/Z} $$ defined by the association $(C[\epsilon],\sigma_{\epsilon})\mapsto \pi_{\epsilon}:C[\epsilon]\rightarrow Z[\epsilon]\simeq C[\epsilon]/\sigma_{\epsilon}$.

\begin{remark}
It is not completely trivial that the quotient $C[\epsilon]/\sigma_{\epsilon}$ is isomorphic to the trivial deformation of $Z$. One can prove it using the fact that the morphism $\pi$ is finite reducing to the affine case.
\end{remark}

Notice that $\gamma(\sigma_{\epsilon}) \in \im{d\alpha}$ implies $\sigma_{\epsilon}=0$ because of \Cref{lem:unique-inv-quotient}. Therefore, it is enough to prove $\im{\gamma}\subset \im{d\alpha}$.

Let us focus on the morphism $\alpha$. The morphism $d\alpha$ can be identified with the map 
$$ \hom_{\cO_Z}(\Omega_Z,\cO_Z) \longrightarrow \hom_{\cO_Z}(\Omega_Z,\pi_*\cO_C)$$ 
induced by applying $\hom_{\cO_Z}(\Omega_Z,-)$ to the natural exact sequence
$$ 0 \rightarrow \cO_Z \rightarrow \pi_*\cO_C \rightarrow L\rightarrow 0 .$$
Clearly, if we have $\hom_{\cO_Z}(\Omega_Z, L)=0$, we get that $d\alpha$ is surjective and we have done. This is not true in general and we will see why in \Cref{ex:not-involution}. Therefore we need to treat the problem with care, studying the deformation spaces involved. 

Let us start with the case of $(C,\sigma)$, where $C$ does not have separating node. Thanks to the description in \Cref{prop:description-quotient}, we have that the quotient map $\pi:C \rightarrow Z$ is finite flat of degree $2$. The theory of cyclic covers implies that every deformation of $\pi$ still induces a hyperelliptic involution, meaning that in this case the composition 
$$ \bar{\gamma}: T_{(C,\sigma)}\Htilde_C \longrightarrow \frac{{\rm Def}^{\rm fix}_{C/Z}}{\im{d\alpha}} $$
is an isomorphism of vector spaces. 

However, if $C$ has at least one separating node, the deformations of $\pi$ do not always give a deformation of the involution of $\sigma$. 

\begin{example}\label{ex:not-involution}
	Let $C$ be an ($A_1$)-stable curve of genus $2$ over $k$ with a separating node $p \in C(k)$. We have a hyperelliptic involution $\sigma$ which fixes the two genus $1$ components such that the two points $p_1,p_2$ over the node $p$ in the normalization of $C$ are fixed by the involution as well. Therefore we have a quotient morphism $C \rightarrow Z$, where $Z$ is a genus $0$ curve with two components meeting in on separating node $q$, image of $p$. The morphism $\pi$ is finite and it is flat of degree $2$ restricted to the open $C \setminus p$. On the contrary, locally around the point $p$ is finite completely ramified of degree $3$, i.e. $\pi^{-1}(q)$ is the spectrum of a length $3$ local artinian $k$-algebra. Clearly, we can deform $\pi$ without deforming source and target but making it unramified over $q$. This deformation cannot correspond to a hyperelliptic involution. 
	
	In another way, we are deforming the two involutions on the two components of $C$ such that the fixed locus moves away from $p_1$ and $p_2$. This implies that we cannot patch them together to get an involution of $C$. To sum up, we need to consider only deformations where the fixed locus of the two involutions does not move. 
	
	Now we look at the tangent space. Let $C_1$ and $C_2$ be the two genus $1$ component of $C$ and $Z_1$ and $Z_2$ are the two components of $Z$ (reduced schematic images of $C_1$ and $C_2$ respectively). Let $L_i$ be the quotient line bundle of $\cO_{Z_i} \hookrightarrow \pi_*\cO_{C_i}$ for $i=1,2$. An easy computation (see \Cref{lem:decomp-line-bundle}) shows that $$\hom_{\cO_Z}(\Omega_Z,L)=\hom_{\cO_{Z_1}}(\Omega_{Z_1},L_1)\oplus \hom_{\cO_{Z_2}}(\Omega_{Z_2},L_2)$$
	and in particular $\hom_{\cO_Z}(\Omega_Z,L)$ is not zero. Nevertheless, if we ask that our deformation of $\pi$ is totally ramified over the separating node, we get exactly the space $\hom_{\cO_{Z_1}}(\Omega_{Z_1},L_1(-p_1))\oplus \hom_{\cO_{Z_2}}(\Omega_{Z_2},L_2(-p_2))$ which is zero. We are going to generalize this computation to our setting.
\end{example}

Suppose $\Gamma$ is an  subcurve of $Z$, then we define the subcurve $C_{\Gamma}:=\pi^{-1}(\Gamma)_{\rm red}$ of $C$. Let $\pi_{\Gamma}:C_{\Gamma}\rightarrow \Gamma$ be the restriction of $\pi$ to $C_{\Gamma}$ and $L_{\Gamma}$ be the quotient bundle of the natural map $\cO_{\Gamma}\hookrightarrow \pi_{\Gamma,*}\cO_{C_{\Gamma}}$.
The main statement in this section is the following theorem.

\begin{theorem}
	In the situation above, the morphism $\bar{\gamma}$ factors through the following inclusion of vector spaces
	$$ \bigoplus_{\Gamma \in {\rm Irr}(Z)}\hom_{\cO_{\Gamma}}(\Omega_{\Gamma},L_{\Gamma}(-D_{\Gamma}))\subset \hom_{\cO_Z}(\Omega_{Z},L)$$
where $D_{\Gamma}$ is the Cartier divisor on $\Gamma$ defined as $\sum^{\Gamma'\neq \Gamma}_{\Gamma' \in {\rm Irr(Z)}}\Gamma \cap \Gamma'$, or equivalently $\Gamma \cap (Z - \Gamma)$.
\end{theorem}
The theorem above implies the result we need to conclude the study of the unramifiedness of $\eta$.
\begin{corollary}
	For every $(C,\sigma) \in \Htilde_C$, the morphism $\bar{\gamma}\equiv 0$, which implies that $\im \gamma \subset \im{d\alpha}$. 
\end{corollary}

\begin{proof}
	It is enough to prove that for every $\Gamma$ irreducible component of $Z$, we have $\hom_{\cO_{\Gamma}}(\Omega_{\Gamma},L_{\Gamma}(-D_{\Gamma}))=0$. Let us explain why this follows from the stability condition on $C$. Clearly $C_{\Gamma}$ is a $A_r$-prestable hyperelliptic curve, with a $2:1$-morphism over $\Gamma\simeq \PP^1$. Let $n_{\Gamma}$ be the number of nodal point on the component $\Gamma$, or equivalently the degree of $D_{\Gamma}$. Thus $\deg(\Omega_{\PP^1}^{\vee} \otimes L_{\Gamma}(-D_{\Gamma}))=+2+(-h_{\Gamma}-1-n_{\Gamma})=1-h_{\Gamma}-n_{\Gamma}$, where $h_{\Gamma}$ is the  (arithmetic) genus of $C_{\Gamma}$. The stability condition on $C$ implies that $2h_{\Gamma}-2+2n_{\Gamma}>0$ because the restriction to $C_{\Gamma}$ of the fiber of a node has at most length $2$. Therefore $\deg (\Omega_{\PP^1}^{\vee} \otimes L_{\Gamma}(-D_{\Gamma}))<0$ and we are done. Equivalently, it is clear that if $h_{\Gamma}>1$, then the degree we want to compute is negative. If $h_{\Gamma}=1$, then $C_{\Gamma}$ has at least a point of intersection with $C-C_{\Gamma}$, therefore $n_{\Gamma}>0$. If $h_{\Gamma}=0$, then $C_{\Gamma}$ intersect $C-C_{\Gamma}$ in subscheme of length at least $3$, but because the restriction to $C_{\Gamma}$ of the fiber of a node of $Z$ has at most length $2$, the support of the intersection contains at least two points, which implies $n_{\Gamma}>1$. 
\end{proof}
To prove the theorem, we reduce to the case when the map $\pi$ is flat, or equivalently there are no separating node on $C$ and then we prove the statement in that case.

\begin{definition}\label{def:sep-dec}
	Let $C$ be an $A_r$-prestable curve and let $\{C_i\}_{i \in I}$ be a set subcurves of $C$ (see \Cref{def:subcurve} for the definition of subcurves). We  say that $\{C_i\}_{i \in I}$ is an $A_1$-separating decomposition of $C$ if the following three properties are satisfied:
\begin{enumerate}
	\item $C=\bigcup_{i \in I} C_i$,
	\item no $C_i$ has a separating node (for $C_i$ itself),
	\item $C_i\cap C_j$ is either empty or a separating node for every $i\neq j$. 
\end{enumerate} 
\end{definition}
 Such a decomposition exists and is unique. 
 Let $(C,\sigma)$ be a hyperelliptic $A_r$-stable curve and $\{C_i\}_{i \in I}$ be its $A_1$-separating decomposition. As usual, we denote by $\pi:C \rightarrow Z:=C/\sigma$ the quotient morphism. Fix an index $i \in I$. Let $Z_i$ be the schematic image of $C_i$ through $\pi$ and $\pi_i:C_i \rightarrow Z_i$ be the restriction of $\pi$ to $C_i$. Using again the local description in \Cref{prop:description-quotient}, we get that $Z_i$ is reduced (therefore a subcurve of $Z$) and $\pi_i$ is flat. Finally, let $L_i$ the quotient of the natural map 
$$\cO_{Z_i} \hookrightarrow \pi_{i,*}\cO_{C_i}$$
which is a line bundle because $\pi_i$ is flat finite of degree $2$ (and we are in characteristic different from $2$).

The following lemma generalizes the idea in \Cref{ex:not-involution}.

\begin{lemma}\label{lem:decomp-line-bundle}
In the situation above, we have an isomorphism of coherent sheaves over $Z$
$$ L \simeq \bigoplus_{i \in I} \iota_{Z_i,*}L_i$$

where $\iota_{Z_i}:Z_i \hookrightarrow Z$ is the closed immersion of the subcurve $Z_i$ in $Z$.

\end{lemma}

\begin{proof}
 
The commutative diagram
$$\begin{tikzcd}
\bigsqcup_{i \in I} C_i \arrow[r, "\bigsqcup \iota_{C_i}"] \arrow[d, "\bigsqcup \pi_i"'] & C \arrow[d, "\pi"] \\
\bigsqcup_{i\in I}Z_i \arrow[r, "\bigsqcup \iota_{Z_i}"]                    & Z                 
\end{tikzcd}
$$ 
 induces a commutative diagram at the level of structural sheaves
$$
\begin{tikzcd}
\cO_Z \arrow[r, hook] \arrow[d, hook] & \bigoplus_{i \in I} \iota_{Z_i,*}\cO_{Z_i} \arrow[d, hook] \\
\pi_*\cO_{C} \arrow[r, hook]          & \bigoplus_{i \in I}\pi_*\iota_{C_i,*}\cO_{C_i}.                 
\end{tikzcd}
$$ 
We want to prove that the map induced between the quotients of the two vertical maps is an isomorphism, in fact the quotient of the right-hand vertical map is trivially isomorphic to $\bigoplus_{i \in I} \iota_{Z_i,*}L_i$ by construction. By the Snake lemma, it is enough to prove that the induced morphism between the two horizontal quotients is an isomorphism. This follows from a local computation in the separating nodes.   
\end{proof}

\begin{remark}
	Because of the fundamental exact sequence for the differentials associated to the immersion $\iota_{Z_i}:Z_i\hookrightarrow Z$, we have that 
	$$ \hom_{\cO_{Z_i}}(\Omega_{Z_i},L_i)\simeq  \hom_{\cO_{Z_i}}(\iota_{Z_i}^*\Omega_{Z},L_i)$$
	and therefore the previous lemma implies that 
	$$ \hom_{\cO_Z}(\Omega_Z,L) \simeq \bigoplus_{i \in I} \hom_{\cO_{Z_i}}(\Omega_{Z_i},L_{i}).$$
	It is easy to see that the last isomorphism can be described using deformations just restricting a deformation of $\pi$ with both source and target fixed to a deformation of $\pi_i$ with both source and target fixed.
	
\end{remark}

\begin{proposition}
  The map $\bar{\gamma}$ factors through the inclusion of vector spaces
  
  $$ \bigoplus_{i\in I}\hom_{\cO_{Z_i}}(\Omega_{Z_i},L_i(-D_i)) \subset \hom_{\cO_Z}(\Omega_Z, L) $$
  
  where $D_i$ is the Cartier divisor on $Z_i$ defined as $\Sigma_{j \in I, j\neq i} (Z_i \cap Z_j)$.
\end{proposition}
\begin{proof}
	Let us start with a deformation of the hyperelliptic involution $\sigma_{\epsilon}:C[\epsilon]\rightarrow C[\epsilon]$. Let $p:\spec k \hookrightarrow C$ be a separating nodal point, then a local computation around $p$ in $C[\epsilon]$ shows that the following diagram is commutative
	$$\begin{tikzcd}
	{\spec k[\epsilon]} \arrow[r, "{p[\epsilon]}", hook] \arrow[d, "{p[\epsilon]}"', hook] & {C[\epsilon]} \\
	{C[\epsilon]} \arrow[ru, "\sigma_{\epsilon}"]                                          &              
	\end{tikzcd}$$
	where $p[\epsilon]$ is the trivial deformation of $p$. This means that every deformation of the hyperelliptic involution (where the curve $C$ is not deformed) cannot deform around the separating nodes. This also implies that the same is true for $\pi_{\epsilon}$ the quotient morphism (which can be associated to the element $\bar{\gamma}(\sigma_{\epsilon})$), namely $\pi_{\epsilon}\circ p[\epsilon]= q[\epsilon]$ where $q=\pi(p)$. Furthermore, the same is true for the restriction $\pi_{i,\epsilon}$ of $\pi_{\epsilon}$ to $C_i[\epsilon]$ for every $i \in I$. Therefore, the statement follows from the following fact: given the element $\delta_{i,\epsilon}$ representing $\pi_{i,\epsilon}$ in $\hom_{\cO_{Z_i}}(\Omega_{Z_i},L_i)$, the condition $ \pi_{i,\epsilon}\circ p[\epsilon]= q[\epsilon]$ translates into the condition $\delta_{i,\epsilon}(p)=0$, which implies $\delta_{i,\epsilon} \in \hom_{\cO_{Z_i}}(\Omega_{Z_i},L_i(-p))$. 
\end{proof}

\begin{remark}
		Notice that the intersection $Z_i \cap Z_j$ is a smooth point on $Z_i$, therefore a Cartier divisor.
\end{remark}

Finally, we have reduced ourself to study the case of $C$ having no separating nodes, or equivalently the quotient morphism $\pi:C \rightarrow Z$ being flat. We have to describe the whole $\hom_{\cO_Z}(\Omega_{Z},L)$ when $(C,\sigma)$ is a hyperelliptic $A_r$-prestable curve without separating node, because every deformation of the morphism $\pi$ gives rise to a deformation of $\sigma$. 

Let $\{ Z_i \}_{i \in I}$ be the $A_1$-separating decomposition of $Z$ (or equivalently the irreducible component decomposition as $Z$ has genus 0). We denote by $\pi_i:C_i \rightarrow Z_i$ the restriction of $\pi:C\rightarrow Z$ to $Z_i$, i.e. $C_i:=Z_i\times_Z C$. Again $L_i$ is the quotient line bundle of the natural morphism $\cO_{Z_i} \hookrightarrow \pi_{i,*}\cO_{C_i}$. 

\begin{remark}
 	Notice that in this situation we are just considering the pullback of $\pi$, while in the case of separating nodes we needed to work with the restriction to the subcurve $C_i$. The reason is that in the previous situation, $\pi^{-1}(Z_i)=\pi^{-1}\pi(C_i)$ was set-theoretically equal to $C_i$, but not schematically. In fact, $\pi^{-1}Z_i$ is not reduced, and it has embedded components supported on the restrictions of the separating nodes. 
\end{remark}
\begin{proposition}
	In the situation above,
	$$\hom_{\cO_Z}(\Omega_{Z},L) \simeq \bigoplus_{i \in I}\hom_{\cO_{Z_i}}(\Omega_{Z_i},L_i(-D_i))$$
	where $D_i$ is the Cartier divisor on $Z_i$ defined as $\sum_{j \in I. j \neq i}(Z_i \cap Z_j)$.
\end{proposition}

\begin{proof}
  Firstly, we consider the exact sequence
  $$ 0 \rightarrow \cO_Z \rightarrow \bigoplus_{i \in I}\iota_{Z_i,*} \cO_{Z_i} \rightarrow \bigoplus_{n \in N(Z)}k(n) \rightarrow 0 $$
  where $N(Z)$ is the set of nodal point of $Z$.
  Because $L$ is a line bundle ($\pi$ is flat), if we tensor the exact sequence with $L$ and then apply the functor $\hom_{\cO_Z}(\Omega_{Z},-)$, we end up with an exact sequence
  $$ 0 \rightarrow \hom_{\cO_Z}(\Omega_{Z},L) \rightarrow \bigoplus_{i \in I}\hom_{\cO_{Z_i}}(\iota_{Z_i}^*\Omega_Z,\iota_{Z_i}^*L) \rightarrow \bigoplus_{n \in N(Z)}\hom_{\cO_Z}(\Omega_Z,K(n)). $$
  The flatness of $\pi$ implies that $\iota_{Z_i}^*L \simeq L_{Z_i}$ and using the fundamental exact sequence of the differentials we get
  $$\hom_{\cO_{Z_i}}(\iota_{Z_i}^*\Omega_Z,L_i) \simeq \hom_{\cO_{Z_i}}(\Omega_{Z_i},L_i).$$
  Notice that $\hom_{\cO_Z}(\Omega_{Z},k(n))=k(n)dx\oplus k(n)dy$ where $dx,dy$ are the two generators of $\Omega_Z$ locally at the node $n \in Z$. Therefore, an element $f \in \hom_{\cO_Z}(\Omega_{Z},L)$ is the same as an element $\{f_i\} \in \bigoplus_{i \in I}\hom_{\cO_{Z_i}}(\Omega_{Z_i},L_i)$ such that $f_i(n)=0$ for every $n \in N(Z) \cap Z_i$.
\end{proof}

\subsection*{Universally closedness}

This section is dedicated to prove that $\eta$ is universally closed. The valuative criterion tells us that $\eta$ is universally closed if and only if for every diagram 
$$
\begin{tikzcd}
\spec Q \arrow[r] \arrow[d, hook] & \Htilde_g^r \arrow[d, "\eta"] \\
\spec R \arrow[r]                 & \Mtilde_g^r                  
\end{tikzcd}
$$
where $R$ is a noetherian complete DVR with algebraically closed residue field $k$ and field of fractions $K$, there exists a lifting 
$$
\begin{tikzcd}
\spec K \arrow[r] \arrow[d, hook]    & \Htilde_g^r \arrow[d, "\eta"] \\
\spec R \arrow[r] \arrow[ru, dashed] & \Mtilde_g^r                  
\end{tikzcd}
$$
which makes everything commutes. 

This amounts to extending the hyperelliptic involution from the general fiber of a family of $A_r$-stable curves over a DVR to the whole family. The precise statement is the following. 

\begin{theorem}\label{theo:univ-closed}
	Let $C_R\rightarrow \spec R$ be a family of $A_r$-stable curves over $R$, which is a noetherian complete DVR with fraction field $K$ and algebraically closed residue field $k$. Suppose there exists a hyperelliptic involution $\sigma_K$ of the general fiber $C_K \rightarrow \spec K$. Then there exists a unique hyperelliptic involution $\sigma_R$ of the whole family $C_R\rightarrow \spec R$ such that $\sigma_R$ restrict to $\sigma_K$ over the general fiber.
\end{theorem} 

\begin{remark}
	In the ($A_1$)-stable case, \Cref{theo:univ-closed} follows from the finiteness of the inertia, as $\overline{\mathcal{M}}_g$ is a Deligne-Mumford separated (in fact proper) stacks. However if $r\geq 2$, the inertia of $\Mtilde_g^r$ is not proper. As a matter of fact, $I_{\Mtilde_g^r}[2]$, i.e. the $2$-torsion of $I_{\Mtilde_g^r}$, is not finite over $\Mtilde_g^r$ either. It is clearly quasi-finite, but the properness fails as the following example shows. 
	
	Let $k=\CC$ and $\AA^1$ be the affine line over the complex number. We consider the following situation: consider the projective line $p_1:\PP^1\times \AA^1\rightarrow \AA^1$ over $\AA^1$ and the line bundle $p_2^*\cO_{\PP^1}(-4)$ (where $p_2:\PP^1\times \AA^1\rightarrow \PP^1$ is the natural projection). Every non-zero section $f\in p_{1,*}p_2^*\cO_{\PP^1}(8)\simeq \H^0(\PP^1,\cO(8)) \otimes_{\CC} \cO_{\AA^1}$ gives rise to a cyclic cover of 
	$$\begin{tikzcd}
	C \arrow[rd] \arrow[rr, "2:1"] &       & \PP^1\times\AA^1 \arrow[ld, "p_1"] \\
	& \AA^1 &                                   
	\end{tikzcd}$$
	 such that $C_f\rightarrow \AA^1$ is a family of (arithmetic) genus $3$ curves (see \cite{ArVis}). One can prove that they are all $A_r$-stable if $r\geq 8$. Furthermore, every automorphism of the data $(\PP^1\times \AA^1,\cO(-4),f)$ gives us an automorphism of $C$ which commutes with the cyclic cover map. We have already proven that the association is fully faithful and these are the only possible automorphisms. 
	 
	 Consider the section $f:=x_0^2x_1^2(x_0-x_1)^2(tx_0-x_1)^2 \in \H^0(\PP^1,\cO(8))\otimes \CC[t]$ where $[x_0:x_1]$ are the homogeneous coordinates of the projective line. Then it is easy to show that the family $C_f$ lives in $\Mtilde_3^3$. If $t\neq 0,1$, we have constructed a $A_1$-stable genus $3$ curve which can be obtained by gluing two projective line in four points with the same cross ratio, which is exactly $t$. Whereas if $t$ is either $0$ or $1$, we obtain a genus $3$ curve which is $A_3$-stable, as two of the four nodes in the generic case collapse into a tacnode. 
	 
	 Let $\phi_t$ be an element of $\PGL_2(k(t))$ defined by the matrix
	 $$
	 \begin{bmatrix}
	 0  &  1 \\
	 t  &  0
	 \end{bmatrix}, 
	 $$
	 hence an easy computation shows that $\phi_t$ is an involution of the data $$(\PP^1_{\CC(t)},\cO(-4),f)$$, and therefore of $C_f\otimes_{\AA^1}\spec \CC(t)$. First of all, this is not hyperelliptic: the quotient of $C_f\otimes_{\AA^1}\spec \CC(t)$ by this involution is a genus $1$ curve geometrically obtained by intersecting two projective lines in two points. Furthermore, it does not have a limit for $t=0$. We really need the hyperelliptic condition to have our result.
\end{remark}

First of all, we can reduce to considering morphisms $\spec K \rightarrow \Htilde_g^r$ which land in a dense open of $\Htilde_g^r$. Therefore \Cref{prop:smooth-hyp} implies that it is enough to prove the theorem when the family $C_R \rightarrow \spec R$ is generically smooth. In particular this implies $C_R$ is a $2$-dimensional normal scheme.

Using the normality of $C_R$ and the properness of the morphism $C_R \rightarrow \spec R$, one can prove that $\sigma_K$ can be uniquelly extended to an open $U$ of $C_R$ whose complement has $\codim 2$. Let us call $\sigma':C_R \dashrightarrow C_R$ the extension of $\sigma_K$ to the open $U$. 

\begin{lemma}\label{lem:contract}
	In the situation above, suppose that $\sigma'$ does not contract any one-dimensional subscheme of $C_R$ to a point. Then there exists an extension of $\sigma'$ to a regular morphism $\sigma_R:C_R\rightarrow C_R$ which is a hyperelliptic involution. 
\end{lemma}

\begin{proof}
	Let $U\subset C_R$ be the open subscheme where $\sigma'$ is defined and let $\Theta\subset C_R\times_{\spec R} C_R$ be the closure of the graph $U\hookrightarrow C_R\times_{\spec R} C_R$ associated to $\sigma'$. We denote by $p_i$ the restriction of the $i$-th projection to $\Theta$ for $i=1,2$. Suppose there exists a one-dimensional irreducible scheme $\Gamma\subset \Theta$ such that $p_2(\Gamma)$ is just a point, then $p_1(\Gamma)$ has to be one-dimensional but this would imply that $\sigma'$ contracts it. Therefore $p_2$ is quasi-finite and proper, thus finite. We have then a finite birational morphism to a normal variety, therefore it is an isomorphism. Hence we have that $\sigma'^{-1}$ can be defined over $C_R$, i.e. it is a regular morphism. We denote it by $\widetilde{\sigma}$. Because both $\widetilde{\sigma}$ and $\sigma'$ restricts to $\sigma_K$ at the generic fiber, we have that $\widetilde{\sigma}$ is in fact an extension of $\sigma'$ to a regular morphism and it is an involution. The hyperelliptic property follows from \Cref{prop:open-closed-imm}.
\end{proof}

The previous lemma implies that it is enough to prove that $\sigma'$ does not contract any one-dimensional subscheme of $C_R$. To do this, we use a classical but fundamental fact for smooth hyperelliptic curves. 

\begin{lemma}\label{lem:can-com}
	Let $C$ be a smooth hyperelliptic curve of genus $g$ over an algebraically closed field. Then the canonical morphism $$\phi_{|\omega_C|}:C\longrightarrow \PP^{g-1}$$
	factors through the hyperelliptic quotient.
\end{lemma}

Before continuing with the proof of \Cref{theo:univ-closed}, we recall why the genus $2$ case is special.

\begin{remark}	\label{rem:genus-2}
	If $g=2$, we need to prove that $\eta$ is an isomorphism.  In this situation, given a family $\pi:C\rightarrow S$ of $A_r$-stable curves, there is always a hyperelliptic involution. Indeed, one can consider the morphism $\phi_{|\omega_C^{\otimes 2}|}:C \rightarrow \PP(\pi_*\omega_C^{\otimes 2}) $. A straightforward computation proves that $\omega_C^{\otimes 2}$ is globally generated and the image $Z$ of the morphism associated to $\omega_C^{\otimes 2}$ is a family of genus $0$ curve. In fact we have a factorization of the morphism $\phi_{|\omega_C^{\otimes 2}|}$ 
	$$
	\begin{tikzcd}
	C \arrow[rd,"\pi"]  \arrow[r, "p"] & Z \arrow[d] \arrow[r, hook] & \PP(\pi_*\omega_C^{\otimes 2}) \arrow[ld]   \\
		& S              &                               
		\end{tikzcd}$$
	and one can construct an involution $\sigma$ of $C$ such that the quotient morphism is exactly $p$. 
	
	The same strategy cannot work for higher genus, as we know that $\eta$ is not surjective.  
\end{remark}

The idea is to use the canonical map to prove that the involution $\sigma'$ does not contract any one-dimensional subscheme of $C_R$, or equivalently any irreducible component of the special fiber $C_k:=C\otimes_R k$, where $k$ is the residue field of $R$. Indeed, $\sigma'$ commutes with the canonical map on a dense open, namely the generic fiber, because of \Cref{lem:can-com}. Reducedness and separatedness of $C_R$ imply that they commute wherever they are both defined. Let $\sigma'_k$ the restriction of $\sigma'$ to the special fiber and let $\Gamma$ be an irreducible component of $C_k$. Suppose we have the two following properties:
\begin{itemize}
	\item[1)] the open of definition of the canonical morphism of $C_k$ intersects $\Gamma$,
	\item[2)] the canonical morphism of $C_k$ does not contract $\Gamma$ to a point;
\end{itemize}
thus $\sigma'_k$ does not contract $\Gamma$ to a point, and neither does $\sigma'$. 

In the rest of the section, we describe the base point locus of $|\omega_C|$ and we prove that the canonical morphism contracts only a specific type of irreducible components. Then, we prove that $\sigma'$ does not contract these particular components using a variation of the canonical morphism. Thus we can apply \Cref{lem:contract} to get \Cref{theo:univ-closed}.

\begin{proposition}\label{prop:base-point-can}		
	Let $C/k$ be an $A_r$-stable genus $g$ curve over an algebraically closed field. The canonical map is defined on the complement of the set ${\rm SN}(C)$, which contains two type of closed points, namely:
	\begin{itemize}
		\item[$(1)$]   $p$ is a separating nodal point;
		\item[$(2)$]   $p$ belongs to an irreducible component of arithmetic genus $0$ which intersect the rest of curve in separating nodes.
	\end{itemize} 
\end{proposition} 
\begin{proof}
	Follows from \cite[Theorem D]{Cat}.
\end{proof}

\Cref{prop:base-point-can} implies that there may be irreducible components of $C_k$ where the canonical map is not defined. We prove in \Cref{lem:not-poss-comp} that in fact they cannot occur in our situation.

Now we prove two lemmas which partially imply the previous one but we need them in this particular form. The first one is a characterization of the points of type (2) as in \Cref{prop:base-point-can}.

\begin{lemma}\label{lem:char-type2}
	Let $C$ be an $A_r$-prestable curve of genus $g\geq 1$ over $k$ and $p\in C$ be a smooth point. Then $\H^0(C,\cO(p))=2$ if and only if $p$ is of type $(2)$ as in \Cref{prop:base-point-can}.
\end{lemma}

\begin{proof}
	The \emph{if} part follows from an easy computation. Let us prove the \emph{only if} implication. Let $\Gamma$ be the irreducible component that contains $p$. Then $h^0(\Gamma,\cO(p))\geq h^0(C,\cO(p))$, which implies $h^0(\Gamma,\cO(p))=2$ or equivalently $\Gamma$ is a projective line. Because $\cO(p)\vert_{\tilde{\Gamma}}\simeq \cO_{\tilde{\Gamma}}$ for every irreducible component $\tilde{\Gamma}$ different from $\Gamma$, we have that $h^0(C,\cO(p))=2$ implies that for every connected subcurve $C'\subset C$ not cointaining $\Gamma$, we have that $C'\cap \Gamma$ is a $0$-dimensional subscheme of length at most $1$. This is equivalent to the fact that $\Gamma$ intersects $C-\Gamma$ only in separating nodes. 
\end{proof}

The second lemma helps us describe the canonical map of $C$ when restricted to its $A_1$-separating decomposition.

\begin{lemma}\label{lem:sep-decom-hodge}
Let $C/k$ be an $A_r$-stable curve of genus $g$ and let $\{C_i\}_{i \in I}$ its $A_1$-separating decomposition (see \Cref{def:sep-dec}). Then we have that 
$$\H^0(C,\omega_C)=\bigoplus_{i \in I}\H^0(C_i,\omega_{C_i}).$$	
\end{lemma}

\begin{proof}
	Let $n \in C$ be a separating nodal point and let $C_1$ and $C_2$ the two subcurves of $C$ such that $C=C_1\cup C_2$ and $C_1\cap C_2=\spec k(n)$. Then we have the exact sequence of coherent sheaves
	$$ 0 \rightarrow \omega_C \rightarrow \iota_{1,*}\omega_{C_1}(n)\oplus \iota_{2,*}\omega_{C_2}(n) \rightarrow k(n) \rightarrow 0$$
	where $\iota_i:C_i \hookrightarrow C$ is the natural immersion for $i=1,2$. Taking the global sections, $h^1(\omega_C)=1$ implies that $\H^0(C,\omega_C)=\H^0(C_1,\omega_{C_1}(n))\oplus\H^0(C_2,\omega_{C_2}(n))$. Finally, we claim that $\H^0(C,\omega_C(p))=\H^0(C,\omega_C)$ for every smooth point $p$ on a connected reduced Gorenstein curve $C$. The global sections of the following exact sequence 
	$$0 \rightarrow \cO_C(-p) \rightarrow \cO_C \rightarrow k(p) \rightarrow 0$$
	implies that the claim is equivalent to the vanishing of $\H^0(C,\cO_C(-p))$. This is an easy exercise.
\end{proof}
Let $\{C_i\}_{i \in I}$ be the $A_1$-separating decomposition of $C$ and $g_i$ be the genus of $C_i$. From now on, we suppose that there are no points of type $(2)$ as in \Cref{prop:base-point-can}. In particular, $g_i>0$ for every $i \in I$.
The previous result implies that the composite 
$$
\begin{tikzcd}
C_i \arrow[r, "\iota_i", hook] & C \arrow[r, "\phi_{|\omega_C|}", dashed] & \PP^{g-1}
\end{tikzcd}
$$
factors through a map $f:C_i\dashrightarrow \PP^{g_i-1}$ induced by the vector space $\H^0(C_i,\omega_{C_i})$ inside the complete linear system of $\omega_{C_i}(\Sigma_i)$, where $\Sigma_i$ is the Cartier divisor on $C_i$ defined by the intersection of $C_i$ with the curve $C-C_i$, or equivalently the restriction of the separating nodal points on $C_i$. This follows because the restriction
$$ \H^0(C,\omega_C)  \longrightarrow \H^0(C_i,\omega_{C_i})$$
is surjective as long as there are no points of type $(2)$ as in \Cref{prop:base-point-can}  on $C_i$. Notice that the rational map $f$ is defined on the open $C_i\setminus \Sigma_i$, and we have that $f\vert_{C_i\setminus \Sigma_i}\equiv \phi_{|\omega_{C_i}|}\vert_{C_i\setminus \Sigma_i}$.

We have two cases to analyze: $C_i$ is either an $A_r$-stable genus $g_i$ curve or only $A_r$-prestable. In the first case, $C_i$ is an $A_r$-stable curve without separating nodes, and therefore the canonical morphism $\phi_{|\omega_{C_i}|}$ is globally defined and finite, because $\omega_{C_i}$ is ample. Therefore the canonical morphism of $C$ does not contract any irreducible component of the subcurve $C_i$. 

Suppose now that $C_i$ is not $A_r$-stable but only prestable and let $\Gamma$ be an irreducible component where $\omega_{C_i}\vert_{\Gamma}$ is not ample, in particular $g_{\Gamma}\leq 1$. In this case, the canonical morphism is not helpful.

\begin{remark}
	Let $C$ be a ($A_1$)-stable genus $2$ curve such that it is constructed as intersecting two smooth genus $1$ curves in a separating node. One can prove easily that the canonical morphism is trivial. Similarly, if we have a genus $g$ stable curve constructed by intersecting smooth genus $1$ curves in separating nodes, the same is still true.
\end{remark}
 
The idea is to construct a variation of the canonical morphism which does not contract the component $\Gamma$ and it still commutes with the involution. The following lemma gives us a possible candidate.

\begin{lemma}\label{lem:var-can}
	Let $C_R\rightarrow \spec R$ be a generically smooth family of $A_r$-stable genus $g$ curves over $R$, a discrete valuation ring, $p \in C_k$ be a separating node of the special fiber of the family and $C_1$ and $C_2$ be the two subcurves of $C_k$ such that $C_k=C_1\cup C_2$ and $C_1\cap C_2=\{p\}$.  Then there exists an integer $m$ such that the followings are true:
	\begin{itemize}
		\item[(i)] $mC_1$ and $mC_2$ are Cartier divisors of $C_R$,
		\item[(ii)] (if we denote by $\cO(mC_1)$ and $\cO(mC_2)$ the induced line bundles) we have $\cO(mC_2)\vert_{C_1}=\cO_{C_1}(p)$  and $\cO(mC_2) \vert_{C_2}=\cO_{C_2}(-p)$, \item[(iii)] the line bundle $\omega_{C_R}(mC_i)$ verifies base change on $\spec R$.
	\end{itemize}
\end{lemma} 

\begin{proof}
	The existence of an integer $m$ that verifies (i) follows from the theory of Du Val singularities, as a separating node in a normal surface \'etale locally is of the form $y^2+x^2+t^m$ where $t$ is a uniformizer of the DVR. An \'etale-local computation shows that $\cO(mC_1)\vert_{C_2}=\cO_{C_2}(p)$ and because as line bundles $\cO(-mC_1-mC_2)$ is the pullback from the DVR of the ideal $(t^m)$ we get $\cO(mC_2)\vert_{C_2}=\cO_{C_2}(-p)$. 
	
	Because $R$ is reduced, to prove $(iii)$ it is enough to prove that $$\dim_k\H^0(C_k,\omega_{C/R}(mC_i)\vert_{C_k})=g$$
	and then use Grauert's theorem. 
	Without loss of generality, we can suppose $i=1$.
	Let us denote by $\cL$ the line bundle $\omega_{C/R}(mC_1)$ restricted to $C_k$. Then we have the usual exact sequence
	$$0\rightarrow \H^0(C_k,\cL)\rightarrow \H^0(C_1,\cL\vert_{C_1})\oplus \H^0(C_2,\cL\vert_{C_2}) \rightarrow \cL(p)$$ 
	where we have that $\cL\vert_{C_1}\simeq \omega_{C_1}$ while $\cL\vert_{C_2}\simeq \omega_{C_2}(2p)$. Because $\omega_{C_2}(2p)$ is globally generated in $p$, we get that $$h^0(\cL)=h^0(\omega_{C_1})+h^0(\omega_{C_2}(2p))-1=h^0(\omega_{C_1})+h^0(\omega_{C_2})=g_1+g_2$$ where $g_i$ is the arithmetic genus of $C_i$ for $i=1,2$. Because $p$ is a separating node, thus $g_1+g_2=g$ and we are done.
	
\end{proof}

\begin{remark}
In the situation of \Cref{lem:var-can}, the involution $\sigma'$ commutes with the rational map associated with the complete linear system of $\omega_{C_R}(mC_i)$ for $i=1,2$. This follows from the fact that $\omega_{C_R}(mC_i)$ restrict to the canonical line bundle in the generic fiber. 
\end{remark}

  If $g_{\Gamma}=1$, then $C_i=\Gamma$ (because it is connected). If $g_{\Gamma}=0$, we have that the intersection $\Gamma \cap (C_i-\Gamma)$ has length $2$ with $C_i-\Gamma$ connected subcurve of $C_i$ (because $g_{C_i}>0$). Let $p$ be a separating node in $\Gamma \cap (C-\Gamma)$. We apply \Cref{lem:var-can} and let us denote by $C_1$ the subcurve containing $\Gamma$. For the sake of notation, we define $\cL_{\Gamma}:=\omega_{C_R}(mC_2)$. 

\begin{proposition}
	In the situation above, the morphism $\phi_{\cL_{\Gamma}}$ induced by the complete linear system of $\cL_{\Gamma}$ does not contract $\Gamma$. 
\end{proposition}

\begin{proof} 
	We need to prove that the morphism $\phi_{\cL_{\Gamma}}$ restricted to the special fiber does not contract $\Gamma$. Using the description of $\H^0(C_k,\cL_{\Gamma}\vert_{C_k})$ of \Cref{lem:var-can}, we can prove an explicit description of $\phi_{\cL_{\Gamma}}$ in both cases. 
	
	If $g_{\Gamma}=1$, then the restriction of $\phi_{\cL_{\Gamma}}$ to $\Gamma$ is the map induced by the vector space $\H^0(\Gamma,\omega_{\Gamma}(2p))$ inside the complete linear system of $\omega_{\Gamma}(2p + (\Sigma_{\Gamma}-p))$ where $\Sigma_{\Gamma}$ is the Cartier divisor on $\Gamma$ associated with the intersection $\Gamma \cap (C-\Gamma)$. We are using that there are no points of type $(2)$ as in \Cref{prop:base-point-can}, which implies that the morphism 
	$$ \H^0(C,\phi_{\cL}) \rightarrow \H^0(\Gamma,\phi_{\cL_{\Gamma}})$$
	is surjective. Similarly as we have done before, we have that $\phi_{\cL_{\Gamma}}\vert_{\Gamma}$ is defined everywhere except in support of the Cartier divisor $\Sigma_{\Gamma}-p$ and it coincides with $\omega_{\Gamma}(2p)$ where it is defined . This implies $\Gamma$ is not contracted to a point. 
	
	Finally, if $g_{\Gamma}=0$, let $D$ be the Cartier divisor on $\Gamma$ associated with the intersection $\Gamma\cap(C_i-\Gamma)$ of length $2$ and by ${\rm res}_D$ one of the following morphisms: 
	\begin{itemize}
		\item ${\rm res}_{q_1}+{\rm res}_{q_2}$ when $D=q_1+q_2$, 
		\item ${\rm res}_q$ when $D=2q$.
	\end{itemize}
	We have that $\phi_{\cL_{\Gamma}}$ restricted to $\Gamma$ is the morphism induced by kernel $V$ of 
	$$ {\rm res}_{D}(2p+(\Sigma-p)):\H^0(\Gamma,\omega_{\Gamma}(2p+(\Sigma_{\Gamma}-p)+D)) \longrightarrow k$$
    where the morphism of vector spaces ${\rm res}_{D}(2p+(\Sigma-p))$ is the tensor of the morphism ${\rm res}_D$ by the line bundle $\cO(2p+(\Sigma-p))$. This again is a conseguence of \Cref{lem:not-poss-comp}. As before, we can prove that $\phi_V$ is defined everywhere except in the support of $\Sigma-p$ and it coincides with the morphism associated with the kernel $W$ of the morphism 
    $$ {\rm res}_{D}(2p):\H^0(\Gamma,\omega_{\Gamma}(2p+D)) \longrightarrow k$$
	where the morphism of vector spaces ${\rm res}_{D}(2p)$ is the tensor of the morphism ${\rm res}_D$ by the line bundle $\cO(2p)$. A straightforward computation shows that $\phi_W$ does not contract $\Gamma$.
\end{proof}

	We have assumed that there are no points of type $(2)$ as in \Cref{prop:base-point-can}. The last lemma explain why this assumption can be made. We decided to leave it as the last lemma because it is the only statement where we used a result for the $A_1$-stable case which we cannot prove indipendently for the $A_r$-stable case. 
	
	\begin{lemma}\label{lem:not-poss-comp}
		Let $C\rightarrow \spec R$ be an $A_r$-stable curve of genus $g$ and suppose the generic fiber is smooth and hyperelliptic. Then if $C_k$ is the special fiber, there are no point $p\in C_k$ of type $(2)$ as in \Cref{prop:base-point-can}. 
	\end{lemma}
\begin{proof} 
	The proof relies on the fact that the stable reduction for the $A_n$-singularities is well-know, see \cite{Hass} and \cite{CasMarLaz}. Let $C_K$ be the generic fiber, then we know that there exists, a family $\widetilde{C}$ such that the special fiber $\widetilde{C}_k$ is stable. The explicit description of the stable reduction implies that if there is a component $\Gamma$ in $C_k$ of genus $0$ intersecting the rest of curve only in separating nodes, then the same component should appear in the stable reduction. But this is absurd, as the stable limit of a smooth hyperelliptic curve is a stable hyperelliptic curve, which does not have such irreducible component.
\end{proof}

\chapter{The Chow ring of $\Mtilde_3^7$}\label{chap:2}

In this chapter we explain the stratey used for the computation of the Chow ring of $\Mtilde_3^7$. The idea is to use a gluing lemma, whose proof is an exercise in homological algebra.

Let $i:\cZ\hookrightarrow\cX$ be a closed immersion of smooth global quotient stacks over $\kappa$ of codimension $d$ and let $\cU:=\cX\setminus \cZ$ be the open complement and $j:\cU \hookrightarrow \cX$ be the open immersion. It is straightforward to see that the pullback morphism $i^*:\ch(\cX)\rightarrow \ch(\cZ)$ induces a morphism $ \ch(\cU) \rightarrow \ch(\cZ)/(c_d(N_{\cZ|\cX}))$, where $N_{\cZ|\cX}$ is the normal bundle of the closed immersion. This morphism is denoted by $i^*$ by abuse of notation. 

Therefore, we have the following commutative diagram of rings:
$$
\begin{tikzcd}
\ch(\cX) \arrow[d, "j^*"] \arrow[rr, "i^*"] &  & \ch(\cZ) \arrow[d, "q"]     \\
\ch(\cU) \arrow[rr, "i^*"]                  &  & \frac{\ch(\cZ)}{(c_d(N_{\cZ|\cX}))}
\end{tikzcd}
$$
where $q$ is just the quotient morphism.

\begin{lemma}\label{lem:gluing}
  In the situation above, the induced map 
  $$\zeta: \ch(\cX)\longrightarrow \ch(\cZ)\times_\frac{\ch(\cZ)}{(c_d(N_{\cZ|\cX}))} \ch(\cU)$$
  is surjective and $\ker \zeta= i_* {\rm Ann}(c_d(N_{\cZ|\cX}))$. In particular, if $c_d(N_{\cZ|\cX})$ is a non-zero divisor in $\ch(\cZ)$, then $\zeta$ is an isomorphism. 
 \end{lemma}

From now on, we refer to the condition \emph{$c_d(N_{\cZ|\cX})$ is not a zero divisor} as the gluing condition.

\begin{remark}
	Notice that if $\cZ$ is a Deligne-Mumford separated stack, the gluing condition is never satisfied because the rational Chow ring of $\cZ$ is isomorphic to the one of its moduli space (see \cite{Vis1} and \cite{EdGra}).
	
	However, there is hope that the positive dimensional stabilizers  we have in $\Mtilde_g^r$ allow the gluing condition to occur. For instance, consider $\Mtilde_{1,1}^2$, or the moduli stack of genus $1$ marked curves with at most $A_2$-singularities. We know that 
	$$ \Mtilde_{1,1}^2 \simeq [\AA^2/\gm]$$ 
	therefore its Chow ring is a polynomial ring, which has plenty of non-zero divisors. See \cite{DiLorPerVis} for the description of $\Mtilde_{1,1}$ as a quotient stack.
\end{remark}

This is the reason why we introduced $\Mtilde_g^r$, which is a non-separated stacks. We are going to compute the Chow ring of $\Mtilde_3^r$ for $r=7$ using a stratification for which we can apply \Cref{lem:gluing} iteratively.

Now we introduce the stratification and we describe the strata as quotient stacks to compute their Chow rings. For the rest of the chapter, we denote by $\Mtilde_{g,n}$ the stack $\Mtilde_{g,n}^{2g+1}$ which is the largest moduli stack of $n$-pointed $A_r$-stable curves of genus $g$ we can consider (see \Cref{rem: max-sing}). We denote by $\Ctilde_{g,n}$ the universal curve of $\Mtilde_{g,n}$. 

Every Chow ring is considered with $\ZZ[1/6]$-coefficients unless otherwise stated. This assumption is not necessary for some of the statements, but it makes our computations easier. We do not know if the gluing condition still holds with integer coefficients. 

First of all, we recall the definitions of some substacks of $\Mbar_3$:
\begin{itemize}
	\item $\overline{\Delta}_1$ is the codimension 1 closed substack of $\Mbar_3$ classifying stable curves with a separating node, 
	\item $\overline{\Delta}_{1,1}$ is the codimension 2 closed substack of $\Mbar_3$ classifying stable curves with two separating nodes, 
	\item $\overline{\Delta}_{1,1,1}$ is the codimension 3 closed substack of $\Mbar_3$ classifying stable curves with three separating nodes. 
\end{itemize}

An easy combinatorial computation shows that we cannot have more than $3$ separating nodes for genus $3$ stable curves. The same substacks can be defined for $\Mtilde_3^r$ for any $r$, but we need to prove that they are still closed inside $\Mtilde_3^r$. This is a conseguence of \Cref{lem:sep-sing}.

We denote by $\Detilde_{1,1,1}\subset \Detilde_{1,1} \subset \Detilde_1$ the natural generalization of $\Debar_{1,1,1} \subset \Debar_{1,1}\subset \Debar_1$ in $\Mtilde_3$.

We also consider $\Htilde_3^7$ as a stratum of the stratification. The following diagram
$$
\begin{tikzcd}
	&                            & \Htilde_3^7 \arrow[rd] &             \\
	{\Detilde_{1,1,1}} \arrow[r] & {\Detilde_{1,1}} \arrow[r] & \Detilde_1 \arrow[r]   & \Mtilde_3^7
\end{tikzcd}
$$
represents the poset associated to the stratification. As before, we write $\Htilde_3$ instead of $\Htilde_3^7$.

Now, we describe the strategy used to compute the Chow ring of $\Mtilde_3$. Our approach focuses firstly on the computation of the Chow ring of $\Mtilde_3 \setminus \Detilde_1$. This is the most difficult part as far as computations is concerned. We first compute the Chow ring of $\Htilde_3 \setminus \Detilde_1$, which can be done without the gluing lemma. Then we apply the gluing lemma to $\Mtilde_3 \setminus (\Htilde_3 \cup \Detilde_1)$ and $\Htilde_3 \setminus \Detilde_1$ to get a description for the Chow ring of $\Mtilde_3 \setminus \Detilde_1$.

Notice that neither $\Detilde_1$ and $\Detilde_{1,1}$ are smooth stacks therefore we cannot use \Cref{lem:gluing}. Nevertheless, both  $\Detilde_1 \setminus \Detilde_{1,1}$ and $\Detilde_{1,1} \setminus \Detilde_{1,1,1}$ are smooth, therefore we apply \Cref{lem:gluing} to $\Detilde_1 \setminus \Detilde_{1,1}$ and $\Mtilde_3\setminus \Detilde_1$ to describe the Chow ring of $\Mtilde_3 \setminus \Detilde_{1,1}$, and then apply it again to $\Mtilde_3 \setminus \Detilde_{1,1}$ and $\Detilde_{1,1} \setminus\Detilde_{1,1,1}$. Finally, the same procedure allows us to glue also $\Detilde_{1,1,1}$ and get the description of the Chow ring of $\Mtilde_3$.
 
In this chapter, we describe the following strata and their Chow rings:
\begin{itemize}
	\item $\Htilde_3\setminus \Detilde_1$,
	\item $\Mtilde_3\setminus (\Htilde_3 \cup \Detilde_1)$,
	\item $\Detilde_1\setminus \Detilde_{1,1}$,
	\item $\Detilde_{1,1}\setminus \Detilde_{1,1,1}$,
	\item $\Detilde_{1,1,1}$,
\end{itemize}
and in the last section we give the complete description of the Chow ring of $\Mtilde_3$.
\section{Chow ring of $\Htilde_3 \setminus \Detilde_1$}\label{sec:H3tilde}
 
In this section we are going to describe $\Htilde_3 \setminus \Detilde_1$ and compute its Chow ring.

Recall that we have an isomorphism (see \Cref{prop:descr-hyper}) between $\Htilde_3^7$ and $\Ctilde_3^7$. We are going to describe $\Htilde_3\setminus \Detilde_1$ as a subcategory of $\Ctilde_3^7$ through the isomorphism cited above.

Consider the natural morphism (see the proof of \Cref{prop:smooth-hyp})
$$ \pi_3: \Htilde_3 \setminus \Detilde_1 \rightarrow \cP_3$$ 
where $\cP_3$ is the moduli stack parametrizing pairs $(Z/S,\cL)$ where $Z$ is a twisted curve of genus $0$ and $\cL$ is a line bundle on $Z$. The idea is to find the image of this morphism. First of all, we can restrict to the open of $\cP_3$ parametrizing pairs $(Z/S,\cL)$ such that $Z/S$ is an algebraic space, because we are removing $\Detilde_1$. In fact, if there are no separating points, $Z$ coincides with the geometric quotient of the involution (see the proof of \Cref{prop:cyclic-covers}). Moreover, we prove an upper bound to the number of irreducible components of $Z$.

\begin{lemma}\label{lem:max-comp}
	Let $(C,\sigma)$ be a hyperelliptic $A_r$-stable curve of genus $g$ over an algebraically closed field and let $Z:=C/\sigma$ be the geometric quotient. If we denote by $v$ the number of irreducible components of $Z$, then we have $v\leq 2g-2$.  Furthermore, if $C$ has no separating nodes, we have that $v\leq g-1$.
\end{lemma}

\begin{proof}
	Let $\Gamma$ be an irreducible component of $Z$. We denote by $g_{\Gamma}$ the genus of the preimage $C_{\Gamma}$ of $\Gamma$ through the quotient morphism, by $e_{\Gamma}$ the number of nodes lying on $\Gamma$ and by $s_{\Gamma}$ the number of nodes on $\Gamma$ such that the preimage is either two nodes or a tacnode (see \Cref{prop:description-quotient}). We claim that the stability condition on $C$ implies that
	$$ 2g_{\Gamma}-2+e_{\Gamma}+s_{\Gamma}>0$$
	for every $\Gamma$ irreducible component of $Z$. If $C_{\Gamma}$ is integral, then it is clear. Otherwise, $C_{\Gamma}$ is a union of two projective line meeting on a subscheme of length $n$. In this situation, $s_{\Gamma}=e_{\Gamma}:=m$ and the inequality is equivalent to $n+m\geq 3$, which is the stability condition for the two projective lines. Using the identity  $$g=\sum_{\Gamma}(g_{\Gamma}+\frac{s_{\Gamma}}{2}),$$
	 a simple computation using the claim gives us the thesis.
	 Suppose $(C,\sigma)$ is not in $\Detilde_1$. 
	 Notice that the involution $\sigma$ commutes with the canonical morphism, because it does over the open dense substack of $\Htilde_g$ parametrizing smooth curves. If we consider the factorization of the canonical morphism  
	 $$ C \rightarrow Z \rightarrow \PP(\H^0(C,\omega_C))=\PP^{g-1}$$ 
	 we know that $\cO_Z(1)$ has degree $g-1$ because the quotient morphism $C\rightarrow Z$ has degree $2$ and $Z\rightarrow \PP^{g-1}$ is finite. It follows that $v\leq g-1$. 
\end{proof}

\begin{remark}
	The first inequality is sharp even for ($A_1$-)stable hyperelliptic curves. Let $v:=2m$ be an even number and let $(E_1,e_1),(E_2,e_2)$ be two smooth elliptic curves and $P_1,\dots,P_{2m-2}$ be $2m-2$ projective lines. We glue $P_{2i-1}$ to $P_{2i}$ in $0$ and $\infty$ for every $i=1,\dots,m-1$ and we glue $P_{2i}$ to $P_{2i+1}$ in $1$ for every $i=1,\dots,m-2$. Finally we glue $(E_1,e_1)$ to $(P_1,1)$ and $(E_2,e_2)$ to $(P_{2m-2},1)$. It is clear that the curve is $A_1$-stable hyperelliptic of genus $m+1$ and its geometric quotient has $2m$ components. The odd case can be dealt similarly.
	\begin{figure}[H]
		\centering
		\includegraphics[width=1\textwidth]{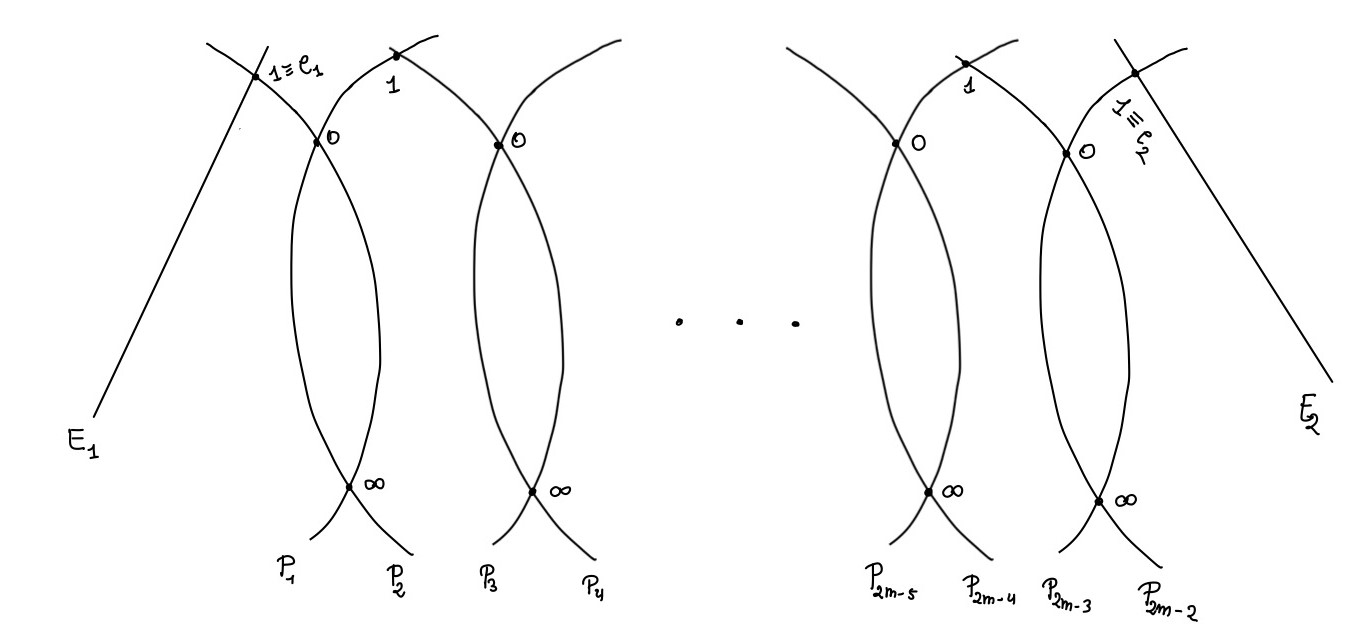}
		\label{fig:Contro}
	\end{figure}
	The same is true for the second inequality. Suppose $v:=2m$ is an even positive integer. Let $(E_1,e_1,f_1),(E_2,e_2,f_2)$ be two $2$-pointed smooth genus $1$ curves and $P_1,\dots,P_{2m-2}$ be $2m-2$ projective lines. Now we glue $P_{2i-1}$ to $P_{2i}$ in $0$ and $\infty$ for every $i=1,\dots,m-1$ and we glue $P_{2i}$ to $P_{2i+1}$ in $1$ and $-1$ for every $i=1,\dots,v-2$. Finally we glue $(E_1,e_1,f_1)$ to $(P_1,1,-1)$ and $(E_2,e_2,f_2)$ to $(P_{2m-2},1,-1)$. It is clear that the curve is $A_1$-stable hyperelliptic of genus $2m+1$ and its geometric quotient has $2m$ components. The odd case can be dealt similarly.
	\begin{figure}[H]
		\centering
		\includegraphics[width=1\textwidth]{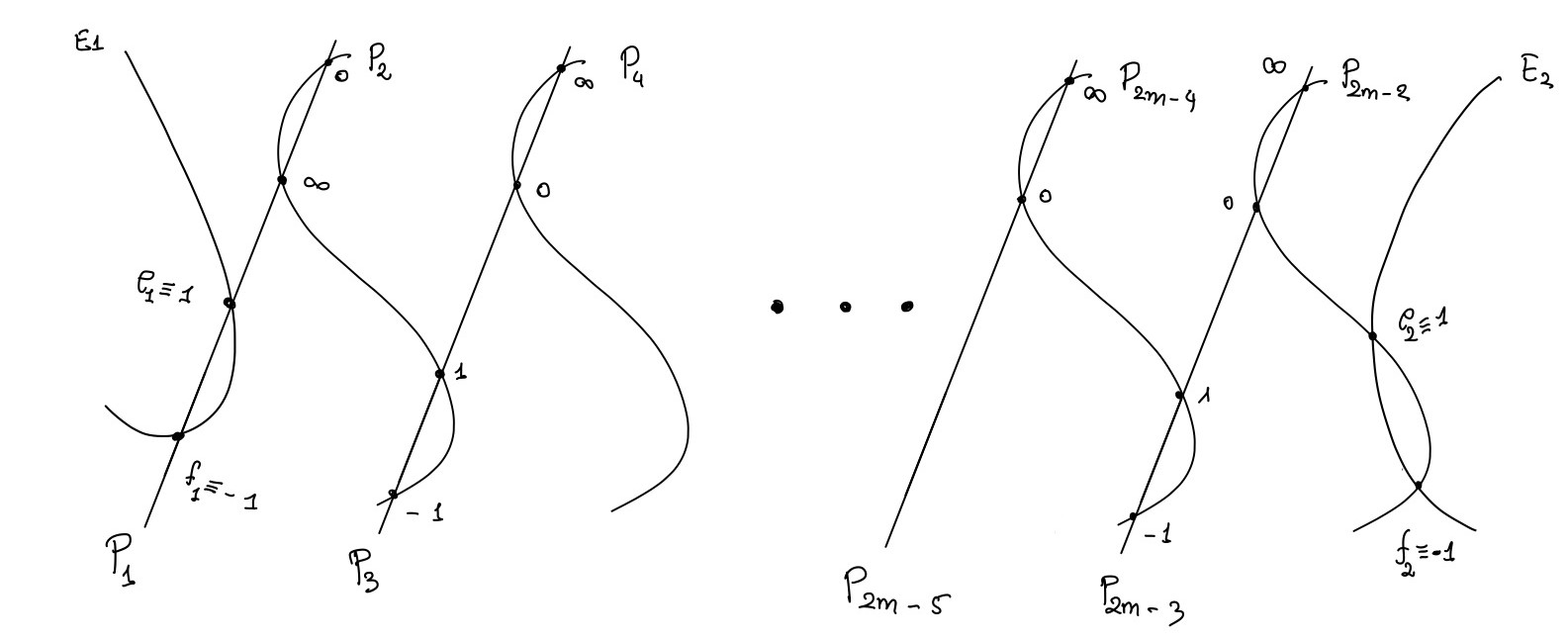}
		\label{fig:Contro1}
	\end{figure}
\end{remark}

\Cref{lem:max-comp} assures us that the geometric quotient $Z$ has at most two irreducible components if $C$ has genus $3$ and does not have separating nodes. Therefore, the datum $(Z/S,\cL,i)$ is in the image of $\pi_3$ only if the fiber $Z_s$ has at most $2$ irreducible components for every geometric point $s \in S$. Moreover, we know that not all the pairs $(\cL,i)$ give us a $A_r$-stable hyperelliptic curve of genus $3$. We need to translate the conditions in \Cref{def:hyp-A_r} in our setting. Because there are no stacky points in $Z$, the conditions in (a) are empty. Furthermore, if $Z$ is integral, it is immediate that there is nothing to prove as we are considering $r=7$. Therefore we can suppose $Z$ is the union of two irreducible components $\Gamma_1$ and $\Gamma_2$  intersecting in a single node. Let us call $n_1$ and $n_2$ the degrees of the restrictions of $\cL$ to $\Gamma_1$ and $\Gamma_2$ respectively. Thus (b1)+(c1) implies that $n_1\leq -2$ and $n_2\leq -2$. Finally the condition $\chi(\cL)=-3$ gives us $n_1=n_2=-2$. 

We denote by $\cM_0^{\leq 1}$ the moduli stack parametrizing  families $Z/S$ of genus $0$ nodal curves with at most $1$ node and $\cM_0^{1}$ the closed substack of $\cM_0^{\leq 1}$ parametrizing families $Z/S$ with exactly $1$ node. We follows the notation as in \cite{EdFul2}. 

Finally, let us denote by $\cP_3'$ the moduli stack parametrizing pairs $(Z/S,\cL)$ where $Z/S$ is a family of genus $0$ nodal curves with at most 1 node and $\cL$ is a line bundle that restricted to the geometric fibers over every point $s \in S$ has degree $-4$ if $Z_s$ is integral or bidegree $(-2,-2)$ if $Z_s$ is reducible.

Therefore $\Htilde_3\setminus\Detilde_1$ can be seen as an open substack of the vector bundle $\pi_*\cL^{\otimes -2}$ of $\cP_3'$, where $(\pi:\cZ \rightarrow \cP_3',\cL)$ is the universal object of $\cP_3'$. We can explicitly describe the complement of $\Htilde_3 \setminus \Detilde_1$ inside $\VV(\pi_{*}\cL^{\otimes-2})$. The only conditions the section $i$ has to verify in this particular case are (b1) and (b2) of \Cref{def:hyp-A_r}. The closed substack of $\VV(\pi_*\cL^{\otimes -2})$ that parametrizes sections which do not verify (b1) or (b2) can be described  as the union of two components. The first one is the zero section of the vector bundle itself. The second component $D$ parametrizes sections $h$ supported on the locus $\cM_0^{\geq 1}$ of reducible genus $0$ curves which vanish at the node $n$ and such that the vanishing locus $\VV(h)_n$ localized at the node is either positive dimensional or it has length different from $2$.

\begin{remark}
	The $\gm$-gerbe $\cP_3'\rightarrow \cM_0^{\leq 1}$ is trivial, as there exists a section defined by the association $Z/S \mapsto \omega_{Z/S}^{\otimes 2}$. Therefore, $\cP_3' \simeq \cM_0^{\leq 1}\times \cB\gm$.
\end{remark} 

We have just proved the following description.

\begin{proposition}
	The stack $\Htilde_3 \setminus \Detilde_1$ is isomorphic to the stack $\VV(\pi_*\omega_{\cZ/\cP_3'}^{\otimes -4})\setminus (D \cup 0)$, where $0$ is the zero section of the vector bundle and $D$ parametrizes sections $h$ supported on $\cM_0^1 \times \cB\gm$ such that the vanishing locus $\VV(h)$ localized at the node does not have length $2$.
\end{proposition}

Lastly, we state a fact that is well-known to experts, which is very helpful for our computations. It is one of the reasons why we choose to do the computations inverting $2$ and $3$ in the Chow rings. Consider a morphism $f:\cX \rightarrow \cY$ which is representable, finite and flat of degree $d$ between quotient stacks and consider the cartesian diagram of $f$ with itself: 
$$
 \begin{tikzcd}
\cX \times_{\cY} \cX \arrow[r, "p_1"] \arrow[d, "p_1"] & \cX \arrow[d, "f"] \\
\cX \arrow[r, "f"]                             & \cY.               
\end{tikzcd}
$$
 
\begin{lemma}\label{lem:chow-tor}
	In the situation above, the diagram in the category of groups
	$$ 
	\begin{tikzcd}
	\ch(\cY)[1/d] \arrow[r, "f^*"] & \ch(\cX)[1/d] \arrow[r, "p_1^*"', bend right] \arrow[r, "p_2^*", bend left] & \ch(\cX \times_{\cY} \cX)[1/d]
	\end{tikzcd}
	$$
	is an equalizer.
\end{lemma}

\begin{proof}
	This is just an application of the formula $g_*g^*(-)=d(-)$ for finite flat representable morphisms of degree $d$. 
\end{proof}

\begin{remark}
	
	We are interested in the following application: if $f:\cX \rightarrow \cY$ is $G$-torsor where $G$ is a costant group, we have an honest action of $G$ over $\ch(\cX)$ and the lemma tells us that the pullback $f^*$ is isomorphism between $\ch(\cY)$ and $\ch(\cX)^{G-{\rm inv}}$, if we invert the order of $G$ in the Chow groups.
\end{remark}

\subsection*{Relation coming from the zero-section}
By a standard argument in intersection theory, we know that 
$$ \ch(\VV\setminus 0)=\ch(\cX)/(c_{\rm top}(\VV))$$
for a vector bundle $\VV$ on a stack $\cX$, where $c_{\rm top}(-)$ is the top Chern class. 

To compute the Chern class, first we need to describe the Chow ring of $\cP_3'$. In \cite{EdFul2}, they compute the Chow ring of $\cM_{0}^{\leq 1}$ with integral coefficients, and their result implies that 
$$ \ch(\cP_3')=\ch(\cM_0^{\leq 1}\times \cB\gm)=\ZZ[1/6,c_1,c_2,s]$$
where $c_i:=c_i\Big(\pi_*\omega_{\cZ/\cM_0^{\leq 1}}^{\vee}\Big)$ for $i=1,2$ and $s:=c_1\Big(\pi_*(\cL\otimes\omega_{\cZ/\cM_0^{\leq 1}}^{\otimes -2})\Big)$.

\begin{remark}
	We are using the fact that $\ch(\cX\times \cB\gm)\simeq \ch(\cX)\otimes \ch(\cB\gm)$ for a smooth quotient stack $\cX$. This follows from Lemma 2.12 of \cite{Tot}.
\end{remark}

We need to compute $c_{9}(\pi_*\cL^{\otimes -4})$. If we denote by $\cS$ the line bundle on $\cP_3'$ such that $\pi^*(\cS):=\cL\otimes\omega_{\cZ/\cM_0^{\leq 1}}^{\otimes -2}$, we have that
$$ \pi_*\cL^{\otimes -4} = \pi_*\omega_{\cZ}^{\otimes -4} \otimes \cS^{\otimes -2}.$$ 

\begin{proposition}
	We have an exact sequence of vector bundles over $\cP_3'$
	$$
	0\rightarrow \det (\pi_*\omega_{\cZ}^{\vee})\otimes {\rm Sym}^2 \pi_*\omega_{\cZ}^{\vee} \rightarrow {\rm Sym}^4 \pi_*\omega_{\cZ}^{\vee} \rightarrow \pi_*\omega_{\cZ}^{\otimes -4} \rightarrow 0.$$ 
\end{proposition}

\begin{proof}
	Let $S$ be a $\kappa$-scheme and let $(Z/S,L)$ be an object of $\cP_3'(S)$. Consider the $S$-morphism 
	$$
	\begin{tikzcd}
	Z \arrow[rd, "\pi"] \arrow[r, "i", hook] & \PP(\cE) \arrow[d, "p"] \\
	& S                      
	\end{tikzcd}
	$$
	induced by the complete linear system of the line bundle $\omega_{Z/S}^{\vee}$, namely $\cE:=(\pi_*\omega_{Z/S}^{\vee})^{\vee}$. Then $i$ is a closed immersion and we have the following facts:
	\begin{itemize}
		\item $i^*\cO_{\PP(\cE)}(1)=\omega_{Z/S}^{\vee}$,
		\item $\cO_{\PP(\cE)}(-Z)\simeq\omega_{\PP(\cE)/S}(1)$; 
	\end{itemize} 
   see the proof of Proposition 6 of \cite{EdFul2} for a detailed discussion. Because 
   $$ \pi_*\omega_{Z/S}^{\otimes -4}=p_*i_*(\omega_{Z/S}^{\otimes -4})=p_*i^*\cO_{\PP(\cE)}(4)$$
   we can consider the exact sequence
   $$ 0 \rightarrow \cO_{\PP(\cE)}(4-Z) \rightarrow \cO_{\PP(\cE)}(4) \rightarrow i^{*}\cO_{\PP(\cE)}(4) \rightarrow 0. $$ 
   If we do the pushforward through $p$, the sequence remain exact for every geometric fiber over $S$, because $Z$ is embedded as a conic. Therefore we get 
   $$ 0\rightarrow p_*(\cO_{\PP(\cE)}(5)\otimes \omega_{\PP(\cE)/S}) \rightarrow p_*\cO_{\PP(\cE)}(4) \rightarrow \pi_*\omega_{Z}^{\otimes -4} \rightarrow 0$$ 
   and using the formula $\omega_{\PP(\cE)}=\cO_{\PP(\cE)}(-2) \otimes p^*\det \cE^{\vee}$, we get the thesis.
\end{proof}

We have found the first relation in our strata, which is
$$c_9:=\frac{c_{15}(\cS^{\otimes -2}\otimes {\rm Sym}^4 \pi_*\omega_{\cZ}^{\vee} )}{c_6(\cS^{\otimes -2}\otimes\det (\pi_*\omega_{\cZ}^{\vee})\otimes {\rm Sym}^2 \pi_*\omega_{\cZ}^{\vee})}$$
and can be described completely in terms of the Chern classes $c_1,c_2$ of $\pi_*\omega_{\cZ}^{\vee}$ and  $s=c_1(\cS)$.

\subsection*{Relations from the locus $D$}

We concentrate now on the locus $D$. First of all notice that $D$ is contained in the restriction of $\pi_*\cL^{\otimes -4}$ to $\cM_0^{1}\times \cB\gm$.

\begin{remark}
	In \cite{EdFul2}, the authors describe the stack $\cM_0^{\leq 1}$ as the quotient stack $[S/\GL_3]$ where $S$ is an open of the six-dimensional $\GL_3$-representation of homogeneous forms in three variables, namely $x,y,z$, of degree $2$. The action can be described as 
	$$ A.f(x):=\det(A)f(A^{-1}(x,y,z))$$ 
	for every $A \in \GL_3$ and $f \in S$, and the open subscheme $S$ is the complement of the closed invariant subscheme parametrizing non-reduced forms. 
	
	The proof consists in using the line bundle $\pi_*\omega_{\cZ}^{\vee}$, which is very ample, to describe the $\cM_0^{\leq 1}$ as the locus of reduced conics in $\PP^2$ with an action of $\GL_3$. For a more detailed discussion, see \cite{EdFul2}. 
	
	In this setting, $\cM_0^1$ correspond to the closed locus $S^1$ of $S$ parametrizing reducible reduced conics in $\PP^2$. It is easy to see that the action of $\GL_3$ over $S^1$ is transitive, therefore $\cM_0^1 \simeq \cB H$, with $H$ the subgroup of $\GL_3$ defined as the stabilizers of the element $xy \in S^1$. A straightforward computation shows that $H\simeq (\gm\ltimes \ga)^2 \ltimes C_2$, where $C_2$ is the costant group with two elements.  
\end{remark}

As we are inverting $2$ in the Chow rings, we can use \Cref{lem:chow-tor} to describe Chow ring of $\cM_0^1$ as the invariant subring of $\ch(\cB\gm^2)$ of a specific action of $C_2$. The $\ga$'s do not appear in the computation of the Chow ring thanks to Proposition 2.3 of \cite{MolVis}. We can see that the elements of the form $(t_1,t_2,1)$ of $\gm^2 \ltimes C_2$ correspond to the matrices in $\GL_3$ of the form 
$$
\begin{pmatrix*}
	t_1 & 0 & 0 \\
	0 & t_2 & 0 \\
	0 & 0   &  1  
\end{pmatrix*};
$$
 the elements of the form $(t_1,t_2,-1)$ correspond to the matrices 
$$ 
\begin{pmatrix*}
0 & t_1 & 0 \\
t_2 & 0 & 0 \\
0 & 0   &  1  
\end{pmatrix*}.
$$
It is immediate to see that the action of $C_2$ over $\gm^2$ can be described as $(-1).(t_1,t_2)=(t_2,t_1)$. Therefore if we denote by $t_1$ and $t_2$ the generator of the Chow ring of the two copies of $\gm$ respectively, we get  
$$ \ch(\cB(\gm^2\ltimes C_2)) \simeq \ZZ[1/6, t_1+t_2, t_1t_2].$$ 

A standard computation shows the following result.

\begin{lemma}
	If we denote by $i:\cM_0^1 \into \cM_0^{\leq 1}$ the (regular) closed immersion, we have that $i^{*}(c_1)=t_1+t_2$ and $i^{*}(c_2)=t_1t_2$ and therefore $i^{*}$ is surjective. Moreover, we have the equality $[\cM_0^1]=-c_1$ in $\cM_0^{\leq 1}$. 
\end{lemma}

\begin{proof}
	The description of $i^{*}(c_i)$ for $i=1,2$ follows from the explicit description of the inclusion 
	$$\cM_0^1 = [\{xy\}/H] \into [S/\GL_3]=\cM_0^{\leq 1}.$$
	
	Regarding the second part of the statement, it is enough to observe that $\cM_0^1=[S^1/\GL_3] \into [S/\GL_3]$ where $S^1$ is the hypersurface of $S$ described by the vanishing of the determinant of the general conic. A straightforward computation of the $\GL_3$-character associated to the determinant formula shows the result.
	\end{proof}
Finally, we focus on $D$. The vector bundle $\pi_*\cL^{\otimes -4}$ (or equivalently $\pi_*\omega_{\cZ}^{\otimes -4} \otimes \cS^{\otimes -2}$) can now be seen as a $9$-dimensional $H$-representation. Specifically, we are looking at sections of $\pi_*\omega_{\cZ}^{\otimes -4} \otimes \cS^{\otimes -2}$ on the curve $xy=0$, which are a $9$-dimensional vector space $\AA(4,4)$ parametrizing a pair of binary forms of degree $4$, which have to coincides in the point $x=y=0$. Let us denote by $\infty$ the point $x=y=0$, which is in common for the two components. With this notation, $D$ parametrizes pairs $(f,g)$ such that $f(\infty)=g(\infty)=0$ and either the coefficient of $xz^3$ or the one of $yz^3$ vanishes. 

\begin{remark}
	This follows from the local description of the $2:1$-cover. In fact, if $\infty$ is the intersection of the two components, we have that \'etale locally the double cover looks like
	$$ k[[x,y]]/(xy) \into k[[x,y,t]](t^2-h(x,y))$$ 
	where $h$ is exactly the section of $\pi_*\omega_{\cZ}^{\otimes -4} \otimes \cS^{\otimes -2}$. Because we can only allow nodes or tacnodes as fibers over a node in the quotient morphism by \Cref{prop:description-quotient}, we get that $h$ is either a unit (the quotient morphism is \'etale) or $h$ is of the form $xp(x)+yq(y)$ such that $p(0)\neq 0$ and $q(0)\neq 0$.  
\end{remark}

The action of $\gm^2\times \gm$ (where the second group of the product is the one whose generator is $s$) over the coefficient of $x^iz^{4-i}$ (respectively $y^iz^{4-i}$) can be described by the character $t_1^is^{-2}$ (respectively $t_2^is^{-2}$).

\begin{lemma}
	The ideal of relations coming from $D$ in $\VV(\pi_*\cL^{\otimes -2})\vert_{\cM_0^1 \times \cB\gm}$ is generated by the two classes $2s(4s-(t_1+t_2))$ and $2s(4s^2-2s(t_1+t_2)+t_1t_2))$. Therefore we have that the ideal of relations coming from $D$ in $\cP_3'$ is generated by the two relations $D_1:=2sc_1(c_1-4s)$ and $D_2:=2sc_1(4s^2-2sc_1+c_2)$. 
\end{lemma}

\begin{proof}
	Because of \Cref{lem:chow-tor}, we can start by computing the ideal of relations in the $\gm^2$-equivariant setting (i.e. forgetting the action of $C_2$) and then considering the invariant elements (by the action of $C_2$). It is clear the ideal of relation $I$ in the $\gm^2$-equivariant setting is of the form $(2s(2s-t_1),2s(2s-t_2))$. Thus the ideal $I^{\rm inv}$ is generated by the elements $2s(4s-(t_1+t_2))$ and $2s(2s-t_1)(2s-t_2)$.
\end{proof}

As a corollary, we get the Chow ring of $\Htilde_3\setminus \Detilde_1$. Before describing it, we want to change generators. We can express $c_1$, $c_2$ and $s$ using the classes $\lambda_1$, $\lambda_2$ and $\xi_1$ where $\lambda_i$ as usual is the $i$-th Chern class of the Hodge bundle $\HH$ and $\xi_1$ is the fundamental class of $\Xi_1$, which is defined as the pullback
$$
\begin{tikzcd}
\Xi_1 \arrow[d] \arrow[r, hook]      & \Htilde_3\setminus \Detilde_1 \arrow[d] \\
\cM_0^1\times \cB\gm \arrow[r, hook] & \cP_3'=\cM_0^{\leq 1}\times \cB\gm.     
\end{tikzcd}
$$

\begin{lemma}\label{lem:lambda-class-H}
	In the situation above, we have that $s=(-\xi_1 - \lambda_1)/3$, $c_1=-\xi_1$ and $c_2= \lambda_2 - (\lambda_1^2 - \xi_1^2)/3$. Furthermore, we have the following relation
	$$\lambda_3=(\xi_1+\lambda_1)(9\lambda_2+(\xi_1+\lambda_1)(\xi_1-2\lambda_1))/27.$$
\end{lemma}

\begin{proof}
	First of all, the relation $\xi_1=-c_1$ is clear from the construction of $\xi_1$, as we have already computed the fundamental class $\cM_0^1$ in $\ch(\cM_0^{\leq 1})$.
	
	Let $f:C\rightarrow Z$ be the quotient morphism of an object in $\Htilde_3\setminus \Detilde_1$ and let $\pi_C:C\rightarrow S$ and $\pi_Z:Z\rightarrow S$ be the two structural morphisms. Grothendieck duality implies that 
	$$ f_*\omega_{C/S} = \hom_Z(f_*\cO_C, \omega_{Z/S})$$ 
	but because $f$ is finite flat, we know that $f_*\cO_C=\cO_Z \oplus L$, i.e. $f_*\omega_{C}=\omega_{Z}\oplus (\omega_{Z}\otimes L^{\vee})$. Recall that $L\simeq \omega_{Z/S}^{\otimes 2} \otimes \pi_Z^*\cS$ for a line bundle $\cS$ on the base. Therefore if we consider the pushforward through $\pi_Z$, we get
$$ \pi_{C,*}\omega_{C/S}=\pi_{Z,*}(\omega_{Z/S}^{\vee})\otimes \cS^{\vee}$$
and the formulas in the statement follow from simple computations with Chern classes.
 \end{proof}
\begin{corollary}\label{cor:chow-hyper}
	We have the following isomorphism of rings:
	$$\ch(\Htilde_3)= \ZZ[1/6,\lambda_1,\lambda_2,\xi_1]/(c_9,D_1,D_2)$$
	where $D_1=2\xi_1(\lambda_1+\xi_1)(4\lambda_1+\xi_1)/9$, $D_2:=2\xi_1(\xi_1+\lambda_1)(9\lambda_2+(\xi_1+\lambda_1)^2)/27$ and $c_9$  is a polynomial in degree $9$.
\end{corollary}

\begin{remark}
    The polynomial $c_9$ has the following form:
    
    \begin{equation*}
    \begin{split}
    c_9 = & -\frac{16192}{19683}\lambda_1^9 - \frac{23200}{6561}\lambda_1^8\xi_1 - \frac{31040}{6561}\lambda_1^7\xi_1^2 +
    \frac{1376}{729}\lambda_1^7\lambda_2 - \frac{320}{6561}\lambda_1^6\xi_1^3 + \\& + \frac{4576}{243}\lambda_1^6\xi_1\lambda_2 +
    \frac{30784}{6561}\lambda_1^5\xi_1^4  + \frac{10144}{243}\lambda_1^5\xi_1^2\lambda_2 + \frac{3968}{81}\lambda_1^5\lambda_2^2 +
    \frac{16256}{6561}\lambda_1^4\xi_1^5 +  \\ & + \frac{15136}{729}\lambda_1^4\xi_1^3\lambda_2 + \frac{992}{27}\lambda_1^4\xi_1\lambda_2^2
    - \frac{320}{243}\lambda_1^3\xi_1^6 - \frac{5792}{243}\lambda_1^3\xi_1^4\lambda_2 -
    \frac{11072}{81}\lambda_1^3\xi_1^2\lambda_2^2 - \\ & - \frac{7264}{27}\lambda_1^3\lambda_2^3 - \frac{7360}{6561}\lambda_1^2\xi_1^7  -
    \frac{5216}{243}\lambda_1^2\xi_1^5\lambda_2 - \frac{11392}{81}\lambda_1^2\xi_1^3\lambda_2^2 -
    \frac{2848}{9}\lambda_1^2\xi_1\lambda_2^3 + \\ & + \frac{640}{6561}\lambda_1\xi_1^8 +  \frac{1952}{729}\lambda_1\xi_1^6\lambda_2 +
   \frac{832}{27}\lambda_1\xi_1^4\lambda_2^2 + \frac{1568}{9}\lambda_1\xi_1^2\lambda_2^3 +384\lambda_1\lambda_2^4 + \\ & +
    \frac{2912}{19683}\xi_1^9 + \frac{352}{81}\xi_1^7\lambda_2 + \frac{3808}{81}\xi_1^5\lambda_2^2 +
    \frac{5984}{27}\xi_1^3\lambda_2^3 + 384\xi_1\lambda_2^4.
    \end{split}
    \end{equation*}
    
\end{remark}

\subsection*{Normal bundle of $\Htilde_3\setminus \Detilde_1$ in $\Mtilde_3 \setminus \Detilde_1$}
We end up the section with the computation of the first Chern class of the normal bundle of the closed immersion $\Htilde_3\setminus \Detilde \into \Mtilde_3 \setminus \Detilde_1$. For the sake of notation, we denote the normal bundle by $N_{\cH|\cM}$.

\begin{proposition}
	The fundamental class of  $\Hbar_3$  in $\ch(\Mbar_3)$ is equal to $9\lambda_1-\delta_0-3\delta_1$.
\end{proposition}
\begin{proof}
	This is Theorem 1 of \cite{Est}. It is important to notice that in the computations the author just need to invert $2$ in the Picard group to get the result.
\end{proof}

\begin{remark}
	As in the ($A_1$-)stable case, we define by $\Delta_0$ the closure of the substack of $\Mtilde_3$ which parametrizes curves with a non-separating node. Alternately, we can consider the stack $\Delta$ in the universal curve $\Ctilde_3$ of $\Mtilde_3$, defined as the vanishing locus of the first Fitting ideal of $\Omega_{\Ctilde_3|\Mtilde_3}$. If we take a connected component $\Sigma \subset \Delta$, one can see that the induced morphism $\Sigma \rightarrow \Mtilde_3$ is a closed embedding. For a more detailed discussion, see Appendix A of \cite{DiLorVis}.
	
	We denote by $\Detilde$ the image with its natural stacky structure and by $\Detilde_0$ the complement of the inclusion $\Detilde_1 \subset \Detilde$. Thanks to \Cref{lem:sep-sing}, we know that $\Detilde_1\into \Detilde$ is also open, therefore we get that $\Detilde_0$ is a closed substack of $\Mtilde_3$. We denote by $\delta_0$ its fundamental class in the Chow ring of $\Mtilde_3$.
\end{remark}

Because $\Mtilde_3\setminus\Mbar_3$ has codimension $2$, we get that the same formula works in our context. Because $\delta_1$ is defined as the fundamental class of $\Detilde_1$, we just need to compute $\delta_0$ restricted to $\Htilde_3\setminus \Detilde_1$ to get the description we want. To do so, we compute the restriction of $\delta_0$ to $\Htilde_3 \setminus (\Detilde_1 \cup \Xi_1)$ and to $\Xi_1\setminus \Detilde_1$ and then glue the informations together.

First of all, notice that $\Htilde_3 \setminus (\Detilde_1 \cup \Xi_1)$ is an open inside a $9$-dimensional representation $V$ of $\PGL_2\times \gm$, as we have $2:1$-covers of $\PP^1$. Unwinding the definitions, we get the following result.

\begin{lemma}\label{lem:ar-vis}
	The representation $V$ of $\PGL_2 \times \gm$ above coincides with the one given by Arsie and Vistoli in Corollary 4.6 of \cite{ArVis}.  
\end{lemma} 
 
This implies that we can see it as an open inside $[\AA(8)/(\GL_2/\mu_4)]$, where $\AA(8)$ is the vector space of binary forms of degree $8$ and $\GL_2/\mu_4$ acts by the equation $A.f(x)=f(A^{-1}x)$. By the theory developed in \cite{ArVis} (see \Cref{prop:descr-hyper} in our situation), it is clear that the sections $f \in \AA(8)$  describe the branching locus of the quotient morphism. In particular, worse-than-nodal singularities on the $2:1$-cover of the projective line correspond to points on $\PP^1$ where the branching divisor is not \'etale, or equivalently points where $f$ has multiplicity more than $1$. Therefore $\delta_0$ is represented by the closed invariant subscheme of singular forms inside $\AA(8)$. This was already computed by Di Lorenzo (see the first relation in Theorem 6 in \cite{DiLor}), and we have that $ \delta_0=28\tau$ with $$\tau=c_1(\pi_*(\omega_C(-W)^{\otimes 2}))$$
where $W$ is the ramification divisor in $C$. Notice that if $f:C\rightarrow \PP^1$ is the cyclic cover of the projective line, we have that $W\simeq f^{*}\cL^{\otimes \vee}$. A computation using Grothendieck duality gives us that $\tau=-s$.

 We have that $\delta_0:=as+bc_1$ in $\ch(\Htilde_3\setminus \Detilde_1)$ for some elements in $\ZZ[1/6]$. 
 
 The computations above implies that if we restrict to the open complement of $\Xi_1 \setminus \Detilde_1$, we get $a=-28$.
 
\begin{remark}
	Notice that although $\GL_2/\mu_4$ is not special, if we invert $2$ in the Chow rings, we have that the pullback morphism 
	$$ \ch(\cB\GL_2/\mu_4) \longrightarrow \ch(\cB \GL_2)$$
	is an isomorphism. Therefore one can do the computations using the maximal torus in $\GL_2$ and apply the formula in \Cref{rem:gener} with $N=8$ and $k=2$ to get the same result.
\end{remark}

The restriction to $\Xi_1\setminus \Detilde_1$ is a bit more complicated, because we have that $\Xi_1 \subset \Detilde_0$. Recall the description
$$ \Xi_1\setminus \Detilde_1 = [\AA(4,4)\setminus D/H]$$
where $\AA(4,4)$ is the vector space of pairs of binary forms $(f(x,z),g(y,z))$ of degree $4$ such that $f(0,1)=g(0,1)$. We define an open $\Xi_1^0$ of $\Xi_1\setminus \Detilde_1$ which are the pairs $(f,g)$ such that $f(0,1)=g(0,1)\neq 0$. Clearly, $D$ does not intersect $\Xi_1^0$.

We do the computations on $\Xi_1^0$ and verify they are enough to determine the coefficient $b$ in the description $\delta_0=-28s+bc_1$. 

\begin{remark}
	The class $[\Xi_1\setminus \Xi_1^0]$ in $\ch(\Xi_1)$ is equal to $-2s$. In fact, it can be described as the vanishing locus of the coefficient of $z^4$ for the pair $(f,g) \in \AA(4,4)$. Therefore the Picard group (up to invert $2$) of $\Xi_1^0$ is free generated by $c_1$.
\end{remark}

Let us define a closed substack inside $\Xi_1^0$: we define $\Delta'$ as the locus parametrizing pairs $(f,g)$ such that either $f$ or $g$ are singular forms.

\begin{lemma}
	In the setting above, we have the equality 
	$$ \Delta'=12c_1$$ 
	in the Picard group of $\Xi_1^0$.
\end{lemma}

\begin{proof}
	As a conseguence of \Cref{lem:chow-tor}, we can do the $\gm^2$-equivariant computations of the equivariant class of $\Delta'$. We have that $\Delta'=\Delta_1' \cup \Delta_2'$ where $\Delta_1'$ (respectively $\Delta'_2$) is the substack parametrizing pairs $(f,g)$ such that $f$ (respectively $g$) is a singular form. 

We reduce ourself to compute the class the locus of singular forms inside $\AA(4)$. The result then follows from a straightforward computation. 
\end{proof}
Now we are ready to compute the restriction of $\delta_0$.
\begin{lemma}
	In the situation above, we have 
	$$ \delta_0\vert_{\Xi_1^0}= -2c_1 + [\Delta'] $$
	inside the Chow ring of $\Xi_1^0$.
\end{lemma}

\begin{proof}
	Because we are computing the Chern class of a line bundle, we can work up to codimension two, or equivalently we can restrict everything to $\Mbar_3\setminus \Debar_1$, which by abuse of notation is denoted by $\cM$. We denote by $\cC$ the universal curve over $\cM$. Consider the closed substack $\Delta$ in $\cC$ of singular points of the morphism $\pi:\cC\rightarrow \cM$ defined by the first Fitting ideal of $\Omega_{\cC/\cM}$. We get a morphism $\pi\vert_{\Delta}:\Delta \rightarrow \cM$ whose image is exactly $\Delta_0$, and the morphism is finite birational. Moreover, $\Delta$ is smooth and its connected components map isomorphically to the irreducible components of $\Delta_0$. Therefore if $\{\Gamma_i\}_{i \in I}$ is the set of irreducible components of $\Delta_0$, we have that 
	$$\delta_0\vert_{\Xi_1^0}=\sum_{i \in I} N_{\Gamma_i/\cM}\vert_{\Xi_1^0}$$ 
	because $\Gamma_i$ is a smooth Cartier divisor of $\cM$. For a more detailed discussion, see Appendix A of \cite{DiLorVis}.
	
	Let us look at the geometric point of $\Xi_1^0$. A curve in $\Xi_1^0$ is a $2:1$-cover of a reducible reduced conic, or equivalently it can described as two genus $1$ curves meeting at a pair of nodes. The two nodes are the fiber over the point $\infty$ in the intersection of the two components of the conic. Therefore, it is clear that  $\Xi_1^0$ is contained in two of the $\Gamma_i$'s, say  $\Gamma_1$ and $\Gamma_2$ and intersect only two of the others, say $\Gamma_3$ and $\Gamma_4$, transversally, namely when one of the two genus $1$ curves is singular.
	
	 The cicle $\Gamma_3+\Gamma_4$ restricted to $\Xi_1^0$ is exactly the fundamental class of $\Delta'$. It remains to compute $c_1(N_{\Gamma_i|\cM})\vert_{\Xi_1^0}$ for $i=1,2$. Consider the commutative $H$-equivariant diagram
	$$
	\begin{tikzcd}
	p \arrow[r, hook] \arrow[d, Rightarrow, no head] & C \arrow[d, "f"] \\
	q \arrow[r, hook]                                & Z               
	\end{tikzcd}
	$$
	where $q$ is the node of the conic and $p$ is one of the two nodes lying over $q$. We know that $f$ is \'etale over the node of $Z$ as we restricted to the open $\Xi_1^0$. Therefore the normal bundle $N_{p|C}$ is isomorphic equivariantly to $N_{q|Z}$ and it is enough to compute the Chern class of $N_{q|Z}$ as a character of $\gm^2 \ltimes C_2$. A straightforward computation shows that 
	$$ c_1(N_{q|Z})=c_1$$
	where $c_1$ is the $\gm^2$-character with weight $(1,1)$. 
\end{proof}

\begin{corollary}\label{cor:norm-hyper}
	The first Chern class of $N_{\cH|\cM}$ is equal to $(2\xi_1-\lambda_1)/3$.
\end{corollary}
 
\section{Description of $\Mtilde_3 \setminus (\Detilde_1\cup \Htilde_3)$}\label{sec:m3tilde-open}

We focus now on the open stratum. Recall that the canonical bundle of a smooth genus $g$ curve is either very ample or the curve is hyperelliptic and the quotient morphism factors through the canonical morphism. This cannot be true for $A_r$-stable curves as in $\Detilde_1$ we have that the dualizing line bundle is not globally generated, see \Cref{prop:base-point-can}. Nevertheless, if we remove $\Detilde_1$, we have the same result for genus $3$ curves. 

\begin{lemma}
	Suppose $C$ is an $A_r$-stable curve of genus $3$ over an algebraically closed field which does not have separating nodes and it is not hyperelliptic. Then the canonical morphism (induced by the complete linear system of the dualizing sheaf) is a closed immersion.
\end{lemma}

\begin{proof}
	This proof is done using the theory developed in \cite{Cat} to deal with most of the cases and analizying the rest of the them separately. 
	
	Firstly, we prove that if $C$ is a $2$-connected $A_r$-stable curve of genus $3$ (i.e. it is in $\Mtilde_3\setminus \Detilde_1$) then we have that $C$ is hyperelliptic if and only if there exist two smooth points $x,y$ such that $\dim \H^0(C,\cO(x+y))=2$ (notice that this is the definition of hyperelliptic as in \cite{Cat}).
	
	One implication is clear. Suppose there exists two smooth points $x,y$ such that $\dim \H^0(C,\cO(x+y))=2$. Proposition 3.14 of \cite{Cat} gives us that we can have two possibilities:
	\begin{itemize}
		\item[(a)] $x,y$ belongs  to $2$ different irreducible components $Y_1,Y_2$ of genus $0$ such that every connected component $Z$ of $C-Y_1-Y_2$ intersect $Y_1$ in a node and $Y_2$ in a (different) node,
		\item[(b)] $x,y$ belong to an irreducible hyperelliptic curve $Y$ such that for every connected component $Z$ of $C-Y$ intersect $Y$ in a Cartier divisor isomorphic to $\cO(x+y)$.
	\end{itemize} 
	Regarding (a), the stability condition implies that the only possibilities are either that there are no other connected components of $C-Y_1-Y_2$, which implies $C$ is hyperelliptic, or we have only one connected component $Z$ of $C-Y_1-Y_2$ which is of genus $1$. Because $Z$ is of genus $1$ and intersect $Y_1\cup Y_2$  in two points, we have that there exists a unique hyperelliptic involution of $Z$ that exchanges them, see \Cref{lem:genus1}. Therefore again $C$ is hyperelliptic. In case (b), the stability condition implies that the only possibilities are either that $C$ is irreducible, and therefore hyperelliptic, or $C$ is the union of two genus $1$ curves intersecting in a length $2$ divisor. Again it follows from \Cref{lem:genus1} that $C$ is hyperelliptic.
	
	Now, we focus on an other definition given in \cite{Cat}. The author define $C$ to be strongly connected if there are no pairs of nodes $x,y$ such that $C\setminus \{x,y\}$ is disconnected. Furthermore, the author define $C$ very strongly connected if it is strongly connected and there is not a point $p \in C$ such that $C\setminus \{p\}$ is disconnected.  
	
	In our situation, a curve $C$ is not very strongly connected if 
	\begin{itemize}
		\item[(1)] $C$ is the union of two genus $1$ curves meeting at a divisor of length $2$,
		\item[(2)] $C$ is the union of two genus $0$ curves meeting in a singularity of type $A_7$,
		\item[(3)] $C$ is the union of a genus $0$ and a genus $1$ curve meeting in a singularity of type $A_5$.
	\end{itemize}
	
	Case (1) is always hyperelliptic for \Cref{lem:genus1}. An easy computation shows that the case (3) is never hyperelliptic and the canonical morphism identify $C$ with the union of a cubic and a flex tangent in $\PP^2$. Finally, one can show that in case (2) the canonical morphism restricted to the two components is a closed embedding, therefore it is clear that it is either a finite flat morphism of degree $2$ over its image ($C$ hyperelliptic) or it is a closed immersion globally on $C$. 
	
	It remains to prove the statement in the case $C$ is very strongly connected. This is Theorem G of \cite{Cat}.
\end{proof}

\begin{remark}
	Notice that this lemma is really specific to genus $3$ curves and it is false in genus $4$. Consider a genus $2$ smooth curve $C$ meeting a genus $1$ smooth curve $E$ in two points, which are not a $g_1^2$ for $C$. Then the canonical morphism is $2:1$ restricted to $E$ but it is birational on $C$.
\end{remark}

The previous lemma implies that the description of $\cM_3 \setminus \cH_3$ proved by Di Lorenzo in Proposition 3.1.3 of \cite{DiLor2} can be generalized in our setting. Specifically, we have the following isomorphism: 
$$ \Mtilde_3 \setminus (\Detilde_1\cup \Htilde_3) \simeq [U/\GL_3]$$ 
where $U$ is an invariant open subscheme inside the space $\AA(3,4)$ of (homogeneous) forms in three coordinates of degree $4$ which is a representation of $\GL_3$ with the action described by the formula $A.f(x):=\det(A) f(A^{-1}x)$. The complement parametrizes forms $f$ such that the induced projective curve $\VV(f)$ in $\PP^2$ is not $A_r$-prestable. 

We use the description as a quotient stack to compute its Chow ring. The strategy is similar to the one adopted in \cite{DiLorFulVis} with a new idea to simplify computations. 

As usual, we pass to the projectivization of $\AA(3,4)$ which we denote by $\PP^{14}$. We induce an action of $\GL_3$ on $\PP^{14}$ setting $A.[f]=[f(A^{-1}x)]$, and if we denote by $\overline{U}$ the projectivization of $U$, we get 
$$\ch_{\GL_3}(U)= \ch_{\GL_3}(\overline{U})/(c_1-h)$$
where $c_i$ is the $i$-th Chern class of the standard representation of $\GL_3$ and $h=\cO_{\PP^{14}}(1)$ the hyperplane section of the projective space of ternary forms. The idea is to compute the relations that come from the closed complement of $\overline{U}$ and then set $h=c_1$ to get the Chow ring of $\Mtilde_3\setminus (\Htilde_3\cup \Detilde_1)$ as the quotient of $\ch(\cB\GL_3)$ by these relations.

\begin{remark}
	Notice that $\lambda_i:=c_i(\HH)$ where $\HH$ is the Hodge bundle can be identified with the Chern classes of the dual of the standard representation.
\end{remark}

We consider the quotient (stack) of $\PP^{14}\times \PP^{2}$ by the following $\GL_3$-action
$$A.([f],[p]):=([f(A^{-1}x),Ap]),$$
and we denote by $Q_4$ the universal quartic over $[\PP^{14}/\GL_3]$, or equivalently the substack of $[\PP^{14}\times \PP^2/\GL_3]$ parametrizing pairs $([f],[p])$ such that $f(p)=0$.

Now, we introduce a slightly more general definition of $A_n$-singularity.

\begin{definition}\label{def:A-sing}
	We say that a point $p$ of a curve $C$ is an $A_{\infty}$-singularity if we have an isomorphism 
$$ \widehat{\cO}_{C,p}\simeq k[[x,y]]/(y^2).$$
Furthermore, we say that $p$ is a $A$-singularity if it an $A_n$-singularity for $n$ either a positive integer or $\infty$.
\end{definition}

We describe when an $A_{\infty}$-singularity can occur for plane curves.

\begin{lemma}
	A point $p$ of a plane curve $f$ is an $A_{\infty}$-singularity if and only if $p$ lies on a unique irreducible component $g$ of $f$ where $g$ is the square of a smooth plane curve.
\end{lemma}

\begin{proof}
	Denote by $A$ the localization of $k[x,y]/(f)$ at the point $p$, which we can suppose to be the maximal ideal $(x,y)$. Because $A$ is an excellent ring and the completion is non-reduced, we get that $A$ is also non reduced. Let $h$ be a nilpotent element in $A$. Because the square of the nilpotent ideal in the completion is zero, we get that $h^2=0$. Since $k[x,y]$ is a UFD, we get the thesis.
\end{proof}

The reason why we introduced $A$-singularity is that they have an explicit description in terms of derivative of the defining equation for plane curves.

\begin{lemma}\label{lem:A-sing}
	A point $p$ on a plane curve defined by $f$ is not an $A$-singularity if and only if both the gradient and the Hessian of $f$ vanishes at the point $p$.
\end{lemma}

\begin{proof}
	If $f$ is an $A$-singularity, one can compute its Hessian and gradient (up to a change of coordinates) looking at the complete local ring, therefore it is a trivial computation.
	
	On the contrary, if the gradient does not vanish, it is clear that $p$ is a smooth point of $f=0$. Otherwise, if the gradient vanishes but there is a double derivative different from zero, we can use Weierstrass preparation theorem and the square completion procedure ($ {\rm char}(\kappa)\neq 2$) to get the result. 
\end{proof}

Now, we introduce a weaker version of Chow envelopes, which depends on what coefficients we consider for the Chow groups.

\begin{definition}
	Let $R$ be a ring. We say that a representable proper morphism $f:X\rightarrow Y$ is an algebraic Chow envelope for $Y$ with coefficients in $R$  if the morphism $f_*:\ch(X)\otimes_{\ZZ}R \rightarrow \ch(Y)\otimes_{\ZZ}R$ is surjective. 
\end{definition}

\begin{remark}
Recall the definition of Chow envelope between algebraic stacks as in Definition 3.4 of \cite{DiLorPerVis}. Because they are working with integer coefficients, Proposition 3.5 of \cite{DiLorPerVis} implies that an algebraic Chow envelope is an algebraic Chow envelope for every choice of coefficients.
\end{remark}

From now on, algebraic Chow envelopes with coefficients in $\ZZ[1/6]$ are simply called algebraic Chow envelopes.

Consider now the substack $X\subset Q_4$ parametrizing pairs $(f,p)$ such that $p$ is singular but not an $A$-singularity of $f$ . Thanks to the previous lemma, we can describe $X$ as the vanishing locus of the second derivates of $f$ in $p$. By construction, $X$ cannot be an algebraic Chow envelope for the whole complement of $\overline{U}$ because we have quartics $f$ which are squares of smooth conics and they do not appear in $X$. Therefore, we have to add the relations coming from this locus. 

We want to study the fibers of the proper morphism $X\rightarrow \PP^{14}$ to prove that it is an algebraic Chow envelope for its image. First, we need to understand how many and what kind of singular point appears in a reduced quartic in $\PP^2$.

\begin{lemma}\label{lem:sing-points}
	A reduced quartic in $\PP^2$ has at most $6$ singular points. If it has exactly $5$ singular points, then it is the union of a smooth conic and a reducible reduced one. If it has exactly $6$ singular points, then it is the union of two reducible reduced conics that do not share a component.
\end{lemma}

\begin{proof}
 	Suppose the quartic $F:=\VV(f)$ is irreducible. Then we can have at most $3$ singular points. In fact, suppose $p_1,\dots,p_4$ are four singular points. Then there exists a conic $Q$ passing through the four points and another smooth point of $f$. Thus $Q \cap F$ would have length at least $9$, which is impossible by Bezout's theorem. 
 	
 	The same reasoning cannot apply if $F$ is the union of two smooth conics meeting at four points, which is a possible situation. Nevertheless, if we suppose that $F$ has at least $5$ different singular points we would have that there exists a conic $Q$ inside $F$, therefore $F=Q\cup Q'$ with $Q'$ another conic because $F$ is a quartic. It is then clear that the singular points are at most $6$ and one can prove the rest of the statement.
\end{proof}

 We denote by $Z_{\{2\}}^{[2]}$ the substack of quartics which are squares of smooth conics and by $\overline{Z}_{\{2\}}^{[2]}$ its closure in $\PP^{14}$. We use the same notation as in \cite{DiLorFulVis}.

Let us denote by $z_2$ the fundamental class of $\overline{Z}_{\{2\}}^{[2]}$ in $\ch_{\GL_3}(\PP^{14})$. We also denote by $\rho$ the morphism $X\rightarrow \PP^{14}$and by $i_T:T\into \PP^{14}$ the closed complement of $\overline{U}$ in $\PP^{14}$. We are ready for the main proposition.

\begin{proposition}
	The ideal generated by $\im{i_{T,*}}$ is equal to the ideal generated by $\im{\rho_{*}}$ and by $z_2$.
\end{proposition}

\begin{proof}
 Consider $\PP^5$ the space of conics in $\PP^2$ with the action of $\GL_3$ defined by the formula $A.f(x):=f(A^{-1}x)$ and the equivariant morphism $\beta:\PP^5 \rightarrow \PP^{14}$ defined by the association $f \mapsto f^2$. We are going to prove that 
 $$ \rho \sqcup \beta : X \sqcup \PP^5 \longrightarrow \PP^{14}$$
 is an algebraic Chow envelope for $T$ and then prove that the only generator we need for the image of $\beta_*$ is the fundamental class $\beta_*(1)$, which coincides with $z_2$.
 
 Let $L/\kappa$ be a field extension. First of all, \Cref{lem:sing-points} tells us that a reduced quartic in $\PP^2$ has at most $6$ singular points. Therefore if $f$ is an $L$-point of $\PP^{14}$ which represents a reduced quartic, the fiber $\rho^{-1}(f)\rightarrow \spec L$ is a finite flat morphism of degree at most $6$. As we are inverting $2$ and $3$ in the Chow rings, the only case we need to worry is when the morphism has degree $5$, i.e. when $f$ is union of a singular conic and a smooth conic. However, in that situation we have a rational point, namely the section that goes to the intersection of the two lines that form the singular conic. This prove that $\rho$ is an algebraic Chow envelope for the open of reduced curves in $T$.
 
 Consider a $L$-point $f$ of $\PP^{14}$ which represents a non-reduced quartic. Then $f$ is one of the following:
 \begin{enumerate}
 	\item[(1)] $f$ is the product of a double line and a reduced conic that does not contain the line,
 	\item[(2)] $f$ is the product of a triple line and a different line,
 	\item[(3)] $f$ is the product of two different double lines,
 	\item[(4)] $f$ is the fourth power of a line,
  	\item[(5)] $f$ is a double smooth conic.
 \end{enumerate}
For a more detailed discussion, see Section 1 of \cite{DiLorFulVis}.

We are going to prove that in situations from (1) to (4), the fiber $\rho^{-1}(f)$ is an algebraic Chow envelope for $\spec L$. In cases (1) and (3), we have that the fiber is finite of degree respectively 2 and 1, therefore an algebraic Chow envelope. In cases (2) and (4) the fiber is a line, which is a projective bundle and therefore an algebraic Chow envelope (we don't have to worry about non-reduced structures).

Clearly the fiber in the case (5) is empty, as $f$ has only $A_{\infty}$-singularities as closed points. Therefore we really need the morphism $\beta$ which is an algebraic Chow envelope for case (5). 

It remains to prove the image of $\beta_*$ is generated by $\beta_*(1)$. This follows from the fact that $\beta^{*}(h_{14})=2h_{5}$, where $h_{14}$ (respectively $h_5$) is the hyperplane section of $\PP^{14}$ (respectively $\PP^5$).

\end{proof}

We conclude this section computing explicitly the relations we need and finally getting the Chow ring of $\Mtilde_3\setminus (\Detilde_1 \cup \Htilde_3)$.

The computation for the class $z_2$ can be done using an explicit localization formula. As a matter of fact, the exact computation was already done in Proposition 4.6 of \cite{DiLorFulVis}, although it was not shown as it was not relevant for their computations.

\begin{remark}
	After the identification $h=-\lambda_1$, the localization formula gives us
	\begin{equation*}
	\begin{split}
	z_2=-1152\lambda_1^3\lambda_3^2 + 256\lambda_1^2\lambda_2^2\lambda_3 + 5824\lambda_1\lambda_2\lambda_3^2 - 1152\lambda_2^3\lambda_3 - 10976\lambda_3^3.
	\end{split}
	\end{equation*}
\end{remark}

In order to compute the ideal generated by $\im{\rho_*}$, we introduce a simplified description of $Q_4$ and of $X$. This description is not completely necessary for this specific computation, but it turns out very useful for the computations of the $A_n$-strata in \Cref{chap:3}.

\begin{lemma}
	The universal quartic $Q_4$ is naturally isomorphic to the quotient stack $[\PP^{13}/H]$ where $H$ is a parabolic subgroup of $\GL_3$.
\end{lemma}

\begin{proof}
	The action of $\GL_3$ is transitive over $\PP^2$, therefore we can consider the subscheme $\PP^{14}\times \{[0:0:1]\}$ in $\PP^{14}\times \PP^2$. If we denote by $H$ the stabilizers of the point $[0:0:1]$ in $\PP^{2}$, we get that 
	$$[\PP^{14}\times \PP^2/\GL_3]\simeq [\PP^{14}/H].$$
	
	The statement follows from noticing that the equation $f([0:0:1])=0$ determines a hyperplane in $\PP^{14}$.
\end{proof}

\begin{remark}\label{rem:H}
	The group $H$ can be described as the subgroup of $\GL_3$ of matrices of the form
	$$
	\begin{pmatrix}
	a & b & 0 \\
	c & d & 0 \\
	f & g & h \\
	\end{pmatrix}.
	$$
	Notice that the submatrix
	$$
	\begin{pmatrix}
	a & b \\
	c & d
	\end{pmatrix}
	$$
	is in $\GL_2$, and in fact we can describe $H$ as $(\GL_2\times \gm)\ltimes \ga^2$.
\end{remark}

This description essentially centers our focus in the point $[0:0:1]$. This implies that the coordinates of the space $\PP^{14}$ can be identified to the coefficient of the Taylor expansion of $f$ in $[0:0:1]$ and this is very useful for computations. \Cref{lem:A-sing} implies that $X$ can be described inside $[\PP^{14}/H]$ as a projective subbundle of codimension $6$. Namely, $X$ is the substack described by the equations $$a_{00}=a_{10}=a_{01}=a_{20}=a_{11}=a_{02}=0$$
where $a_{ij}$ is the coefficient of the monomial $x^iy^jz^{4-i-j}$. One can verify easily that this set of equations is invariant for the action of $H$. The fact that $X$ is a projective subbundle implies that we can repeat the exact same strategy adopted in Section 5 of \cite{FulVis} to prove our version of Theorem 5.5 as in \cite{FulVis}.

For the sake of notation, we denote by $h_2$ (respectively $h_{14}$) the hyperplane section of $\PP^2$ (respectively $\PP^{14}$).

\begin{proposition}
	There exists a unique polynomial $p(h_{14},h_{2})$ with coefficients in $\ch(\cB \GL_3)$ such that 
	\begin{itemize}
		\item $p(h_{14},h_{2})$ represent the fundamental class of $X$ in the $\GL_3$-equivariant Chow ring of $\PP^{14}\times \PP^2$,
		\item the degree of $p$ with respect to the variable $h_{14}$ is strictly less than $15$,
		\item the degree of $p$ with respect to the variable $h_{2}$ is strictly less than $3$.
	\end{itemize} 
	Furthermore, if $p$ is of the form $$p_2(h_{14})h_2^2+p_1(h_{14})h_2+p_0(h_{14})$$
	 with $p_0,p_1,p_2 \in \ch_{\GL_3}(\PP^{14})$ and $\deg p_i \leq 14$, we have that $\im{\rho_*}$ is equal to the ideal generated by $p_0$, $p_1$ and $p_2$ in $\ch_{\GL_3}(\PP^{14})$. 
\end{proposition}

\begin{proof}
	For a detailed proof of the proposition see Section 5 of \cite{FulVis}.
\end{proof}

A computation shows the description of the Chow ring of $\Mtilde_3\setminus (\Detilde_1 \cup \Htilde_3)$.

\begin{corollary}\label{cor:chow-quart}
	We have an isomorphism of rings
	$$ \ch\Big(\Mtilde_3\setminus (\Detilde_1 \cup \Htilde_3)\Big) \simeq \ZZ[\lambda_1,\lambda_2,\lambda_3]/(z_2,p_0,p_1,p_2)$$
	where the generators of the ideal can be described as follows:
	\begin{itemize}
		\item $p_2=12\lambda_1^4 - 44\lambda_1^2\lambda_2 + 92\lambda_1\lambda_3$,
		\item $p_1=-14\lambda_1^3\lambda_2 + 2\lambda_1^2\lambda_3 + 48\lambda_1\lambda_2^2 - 96\lambda_2\lambda_3$,
		\item $p_0=15\lambda_1^3\lambda_3 - 52\lambda_1\lambda_2\lambda_3 + 112\lambda_3^2$,
		\item $z_2=-1152\lambda_1^3\lambda_3^2 + 256\lambda_1^2\lambda_2^2\lambda_3 + 5824\lambda_1\lambda_2\lambda_3^2 - 1152\lambda_2^3\lambda_3 - 10976\lambda_3^3.$
	\end{itemize}
\end{corollary}

\begin{remark}
	We remark that $z_2$ is not in the ideal generated by the other relations, meaning that it was really necessary to introduce the additional stratum of non-reduced curves for the computations.
\end{remark}

\section{Description of $\Detilde_1 \setminus \Detilde_{1,1}$}\label{sec:detilde-1}

To describe $\Detilde_1\setminus \Detilde_{1,1}$, we recall the gluing morphism in the case of stable curves:
$$ \Mbar_{1,1} \times \Mbar_{2,1} \longrightarrow \Debar_1,$$ 
which is an isomorphism outside the $\Debar_{1,1}$. The same is true for the $A_r$-case. 

Thus, we need to describe the preimage of $\Detilde_{1,1}$ through the morphism. Denote by $\ThTilde_1\subset \Mtilde_{2,1}$ the pullback of $\Detilde_1 \subset \Mtilde_2$ through the morphism $\Mtilde_{2,1} \rightarrow \Mtilde_2$ which forgets the section. Thus one can easily prove that  
$$ \Mtilde_{1,1}\times (\Mtilde_{2,1}\setminus \ThTilde_1) \simeq \Detilde_1 \setminus \Detilde_{1,1}.$$

It remains to describe the stack $\Mtilde_{2,1}\setminus \ThTilde_1$. We start by describing $\Ctilde_2 \setminus \ThTilde_1$, the universal curve over $\Mtilde_2\setminus \Detilde_1$. 

\begin{proposition}
	We have the following isomorphism
	$$  \Ctilde_2 \setminus \ThTilde_1 \simeq [\widetilde{\AA}(6)\setminus 0/B_2],$$ 
	where $B_2$ is the Borel subgroup of lower triangular matrices inside $\GL_2$ and $\widetilde{\AA}(6)$ is a $7$-dimensional $B_2$-representation.
\end{proposition} 

\begin{proof}
	This is a straightforward generalization of Proposition 3.1 of \cite{DiLorPerVis}.We remove the $0$-section because we cannot allow non-reduced curve to appear (condition (b1) in \Cref{def:hyp-A_r}). 
\end{proof}

 The representation $\widetilde{\AA}(6)$ can be described as follows: consider the $\GL_2$-representation $\AA(6)$ of binary forms with coordinates $(x_0,x_1)$ of degree $6$ with an action described by the formula
$$ A.f(x_0,x_1):=\det(A)^{2}f(A^{-1}x),$$  
and consider also the $1$-dimensional $B_2$-representation $\AA^1$ with an action described by the formula 
$$ A.s:=\det(A)a_{22}^{-3}s$$
where
$$
A:=\begin{pmatrix}
	a_{11} & a_{12} \\
	0 & a_{22}
\end{pmatrix}.
$$

We define $\widetilde{\AA}(6)$ to be the invariant subscheme  of $\AA(6)\times \AA^1$ defined by the pairs $(f,s)$ that satisfy the equation $f(0,1)=s^2$. If we use $s$ instead of the coefficient of $x_1^6$ in $f$ as a coordinate, it is clear that $\widetilde{\AA}(6)$ is isomorphic to a $7$-dimensional $B_2$-representation.

The previous proposition combined with \Cref{prop:contrac} gives us the following description of the Chow ring of $\Mtilde_{2,1}\setminus \ThTilde_1$.

\begin{corollary}\label{cor:mtilde_21}
	$\Mtilde_{2,1}\setminus \ThTilde_1$ is the quotient by $B_2$ of the complement of the subrepresentation of $\widetilde{\AA}(6)$ described by the equations $s=a_5=a_4=a_3=0$ where $a_i$ is the coefficient of $x_0^{6-i}x_1^i$.
\end{corollary}

\begin{proof}
	\Cref{prop:contrac} gives that $\Mtilde_{2,1}\setminus \ThTilde_1$ is isomorphic to the open of $\Ctilde_2 \setminus \ThTilde_1$ parametrizing pairs $(C/S,p)$ such that $C$ is a $A_5$-stable genus $2$ curve and $p$ is a section whose geometric fibers over $S$ are $A_n$-singularity for $n\leq 2$. In particular, we need to describe the closed invariant subscheme $D_3$ of $\widetilde{\AA}(6)$ that parametrizes pairs $(C,p)$ such that $p$ is a $A_n$-singularities with $n\geq 3$. 
	
	To do so, we need to explicit the isomorphism 
	$$ \Ctilde_2 \setminus \ThTilde_1 \simeq [\widetilde{\AA}(6)\setminus 0/B_2]$$
	as described in Proposition 3.1 of \cite{DiLorPerVis}. Given a pair $(f,s) \in \widetilde{\AA}(6)\setminus 0$, we construct a curve $C$ as $\spec_{\PP^1}(\cA)$ where $\cA$ is the finite locally free algebra of rank two over $\PP^1$ defined as $\cO_{\PP^1}\oplus \cO_{\PP^1}(-3)$. The multiplication is induced by the section $f:\cO_{\PP^1}(-6)\into \cO_{\PP^1}$. The section $p$ of $C$ can be defined seeing $s$ as a section of  $\H^0(\cO_{\PP^1}(-3)\vert_{\infty})$, where $\infty \in \PP^1$. This implies that the point $p$ is a $A_n$-singularity for $n\geq 3$ if and only if $\infty$ is a root of multiplicity at least $4$ for the section $f$.  The statement follows.  
\end{proof}
\Cref{cor:mtilde_21} gives us the description of the Chow ring of $\Detilde_{1}\setminus \Detilde_{1,1}$. We denote by $t$ the generator of the Chow ring of $\Mtilde_{1,1}$ defined as $t:=c_1(p^*\cO(p))$ for every object $(C,p)\in \Mtilde_{1,1}$. This is the $\psi$-class of $\Mtilde_{1,1}$.

\begin{proposition}\label{prop:descr-detilde-1}
	We have the following isomorphism
	$$ \ch(\Detilde_1 \setminus \Detilde_{1,1}) \simeq \ZZ[1/6,t_0,t_1,t]/(f)$$
	where $f \in \ZZ[1/6,t_0,t_1]$ is the polynomial:
	$$f=2t_1(t_0+t_1)(t_0-2t_1)(t_0-3t_1).$$
\end{proposition}

\begin{proof}
	Because $\Mtilde_{1,1}$ is a vector bundle over $\cB \gm$, the morphism 
	$$ \ch(\Mtilde_{2,1}\setminus \ThTilde_1)\otimes \ch(\cB\gm) \rightarrow \ch(\Detilde_1 \setminus \Detilde_{1,1})$$
	is an isomorphism. Therefore it is enough to describe the Chow ring of $\Mtilde_{2,1}\setminus \ThTilde_1$. The previous corollary gives us that if $T_2$ is the maximal torus of $B_2$ and $t_0,t_1$ are the two generators for the character group of $T_2$, we have that 
	$$\ch(\Mtilde_{2,1}\setminus \ThTilde_1)\simeq \ZZ[1/6,t_0,t_1]/(f)$$
	where $f$ is the fundamental class associated with the vanishing of the coordinates $s,a_5,a_4,a_3$ of $\widetilde{\AA}(6)$. Again, we are using Proposition 2.3 of \cite{MolVis}.  This computation in the $\gm^2$-equivariant setting is straightforward and it gives us the result.
\end{proof}

Once again, we use the results in Appendix A of \cite{DiLorVis} to describe the first Chern of the normal bundle of $\Detilde_1 \setminus \Detilde_{1,1}$ in $\Mtilde_3 \setminus \Detilde_{1,1}$.

\begin{proposition}\label{prop:relation-detilde-1}
	The closed immersion $\Detilde_1\setminus \Detilde_{1,1} \into \Mtilde_3 \setminus \Detilde_{1,1}$ is regular and the first Chern class of the normal bundle is of equal to $t+t_1$. Moreover, we have the following equalities in the Chow ring of $\Detilde_{1}\setminus \Detilde_{1,1}$:
	\begin{itemize}
		\item[(1)] $\lambda_1=-t-t_0-t_1$,
		\item[(2)] $\lambda_2=t_0t_1+t(t_0+t_1)$,
		\item[(3)] $\lambda_3=-t_0t_1t$,
		\item[(4)] $[H]\vert_{\Detilde_{1}\setminus \Detilde_{1,1}}=t_0-2t_1$.
	\end{itemize} 
\end{proposition}

\begin{proof}
The closed immersion is regular because the two stacks are smooth.  Because of Appendix A of \cite{DiLorVis}, we know that the normal bundle is the determinant of a vector bundle $N$ of rank $2$ tensored by a $2$-torsion line bundle whose first Chern class vanishes when we invert $2$ in the Chow ring. Moreover, we can describe $N$ in the following way. Suppose $(C/S,p)$ is an object of $(\Detilde_{1}\setminus \Detilde_{1,1})(S)$, where $p$ is the section whose geometric fibers over $S$ are separating nodes. Then $N$ is the normal bundle of $p$ inside $C$, which has rank $2$. If we decompose $C$ as a union of an elliptic curve $(E,e)\in \Mtilde_{1,1}$ and a $1$-pointed genus $2$ curve $(C_0,p_0) \in \Mtilde_{2,1}$, it is clear that $N=N_{e/E}\oplus N_{p_0/C_0}$. Therefore it is enough to compute the two line bundles separately.
	
  The first Chern class of the line bundle $N_{e/E}$ is exactly the $\psi$-class of $\Mtilde_{1,1}$, which is $t$ by definition. Similarly, the line bundle $N_{p_0/C_0}$ is the $\psi$-class of $\Mtilde_{2,1}$. Using the description of $\Ctilde_2$ as a quotient stack (see the proof of \Cref{cor:mtilde_21}), one can prove that 
	  \begin{itemize}
	  	\item $N_{p_0/C_0}=\cO_{\PP^1}(1)^{\vee}\vert_{p_{\infty}}=E_1$, where $E_1$ is the character of $B_2$ whose Chern class is $t_1$;
	  	\item $\pi_*\omega_{C_0}=E_0^{\vee}\oplus E_1^{\vee}$, where $E_0$ is the character of $B_2$ whose Chern class is $t_0$.
	\end{itemize}

We now prove the description of the fundamental class of the hyperelliptic locus. We have that $\Htilde_3$ intersects transversely $\Detilde_1\setminus \Detilde_{1,1}$ in the locus where the section of $\Mtilde_{2,1}$ is a Weierstrass point for the (unique) involution of a genus $2$ curve. The description of $\Ctilde_2 \setminus \ThTilde_1$ as a quotient stack (see the proof of \Cref{cor:mtilde_21}) implies that $[H]\vert_{\Detilde_1\setminus \Detilde_{1,1}}$ is equal to the class of the vanishing locus $s=0$ in $\widetilde{\AA}(6)$, which is easily computed as we know explicitly the action of $B_2$. The only thing that remains to do is the descriptions of the $\lambda$-classes. They follow from the following lemma.
\end{proof}

\begin{lemma}\label{lem:hodge-sep}
	Let $C/S$ be an $A_r$-stable curve over a base scheme $S$ and $p$ a separating node (i.e. a section whose geometric fibers over $S$ are separating nodes). Consider $b:\widetilde{C}\arr C$ the blowup of $C$ in $p$ and denote by $\widetilde{\pi}$ the composition $\pi \circ b$. Then we have $$\pi_*\omega_{C/S}\simeq \widetilde{\pi}_*\omega_{\widetilde{C}/S}.$$
\end{lemma}

\begin{proof}
	Denote by $D$ the dual of the conductor ideal, see \Cref{lem:blowup}. We know by the Noether formula that $b^*\omega_{C/S}=\omega_{\widetilde{C}/S}(D)$, therefore if we tensor the exact sequence 
	$$0 \rightarrow \cO_C \rightarrow f_*\cO_{\widetilde{C}} \rightarrow Q \rightarrow 0$$ 
	by $\omega_{C/S}$, we get the following injective morphism 
	$$ \omega_{C/S} \into f_*\omega_{\widetilde{C}/S}(D). $$
	As usual, the smoothness of $\Mtilde_g^r$ implies that we can apply Grauert's theorem to prove that the line bundles $\omega_{\widetilde{C}}$ and $\omega_{\widetilde{C}}(D)$ satisfy base change over $S$. Consider now the morphism on global sections
	$$ \pi_*\omega_{C/S} \into\widetilde{\pi}_*\omega_{\widetilde{C}/S}(D),$$
	because they both satisfy base change, we can prove the surjectivity restricting to the geometric points of $S$. The statement for algebraically closed fields has already been proved in \Cref{lem:sep-decom-hodge}. 
	
	Finally, given the morphism 
	$$ \omega_{\widetilde{C}} \into \omega_{\widetilde{C}}(D) $$ 
	we consider the global sections 
	$$ \widetilde{\pi}_*\omega_{\widetilde{C}/S} \into \widetilde{\pi}_*\omega_{\widetilde{C}/S}(D)$$ 
	and again the surjectivity follows restricting to the geometric fibers over $S$. 
\end{proof}

\section{Description of $\Detilde_{1,1} \setminus \Detilde_{1,1,1}$}\label{sec:detilde-1-1}

Let us consider the following morphism 
$$ \Mtilde_{1,2} \times \Mtilde_{1,1} \times \Mtilde_{1,1} \longrightarrow \Detilde_{1,1}$$ 
defined by the association $$(C,p_1,p_2),(E_1,e_1),(E_2,e_2) \mapsto E_1 \cup_{e_1\equiv p_1} C \cup_{p_2 \equiv e_2} E_2$$
where the cup symbol represent the gluing of the two curves using the two points specified. The operation is obviously non commutative.

The preimage of $\Detilde_{1,1,1}$ through the morphism is the product $\Mtilde_{1,1}\times \Mtilde_{1,1}\times \Mtilde_{1,1}$ where $\Mtilde_{1,1}\subset \Mtilde_{1,2}$ is the universal section of the natural functor $\Mtilde_{1,2}\rightarrow \Mtilde_{1,1}$.

\begin{lemma}\label{lem:detilde-1-1}
	The morphism 
	$$ \pi_2:(\Mtilde_{1,2}\setminus \Mtilde_{1,1}) \times \Mtilde_{1,1} \times \Mtilde_{1,1} \longrightarrow \Detilde_{1,1}\setminus \Detilde_{1,1,1}$$ 
	described above is a $C_2$-torsor.
\end{lemma}

\begin{proof}
	First of all, let us denote by $\cX\rightarrow \cY$ the morphism of algebraic stacks in the statement. It is easy to verify that it is representable. 
	
	Consider the fiber product $\cX \times_{\cY} \cX$. We can construct a morphism $\eta: \cX \times C_2 \rightarrow \cX \times_{\cY} \cX$ over $\cX$ using the following $C_2$-action on the objects of $\cX$:
	$$ \Big( (C,p_1,p_2), (E_1,e_1), (E_2,e_2) \Big) \mapsto \Big((C,p_2,p_1), (E_2,e_2), (E_1,e_1)\Big)$$
	whereas on morphism it is defined in the natural way. We use the definition of action over a stack as in \cite{Rom}.
	
	It is easy to see that $\eta$ is an isomorphism when restricted to the geometric fibers over $\cX$. This implies that $\cX \times_{\cY} \cX$ is quasi-finite and representable over $\cX$ and the length of the geometric fibers is constant. Because $\cX$ is smooth, we get that $\cX \times_{\cY} \cX$ is flat over $\cX$ and therefore that $\eta$ is an isomorphism.
\end{proof}

Denote by $\Ctilde_{1,1}$ the universal curve over $\Mtilde_{1,1}$. It naturally comes with a universal section $\Mtilde_{1,1}\into \Ctilde_{1,1}$.

\begin{lemma}\label{lem:mtilde-12}
	We have the isomorphism 
	$$ \Mtilde_{1,2}\setminus\Mtilde_{1,1} \simeq \Ctilde_{1,1}\setminus \Mtilde_{1,1}$$
	and therefore we have the following isomorphism of rings
	$$ \ch( \Mtilde_{1,2}\setminus\Mtilde_{1,1})\simeq \ZZ[1/6,t]. $$  
\end{lemma}

\begin{proof}
 The isomorphism is a corollary of \Cref{prop:contrac}. The computation of the Chow ring of $\Ctilde_{1,1}\setminus \Mtilde_{1,1}$  is Lemma 3.2 of \cite{DiLorPerVis}.
\end{proof}

\begin{remark}
	It is important to notice that $\Mtilde_{1,2}$ has at most $A_3$-singularities while $\Ctilde_{1,1}$ has at most cusps, because it is the universal curve of $\Mtilde_{1,1}$.
\end{remark}
\begin{corollary}
	The algebraic stack $\Detilde_{1,1}\setminus \Detilde_{1,1,1}$ is smooth.
\end{corollary}

 \Cref{lem:chow-tor} implies that 
$$\ch(\Detilde_{1,1}\setminus \Detilde_{1,1,1})\simeq \ZZ[1/6,t,t_1,t_2]^{\rm inv},$$
where the action of $C_2$ is defined on object by the following association 
$$ \Big( (C,p_1,p_2), (E_1,e_1), (E_2,e_2) \Big) \mapsto \Big((C,p_2,p_1), (E_2,e_2), (E_1,e_1)\Big).$$

A computation shows that the involution acts on the Chow ring of the product in the following way
$$ (t,t_1,t_2) \mapsto (t,t_2,t_1)$$
and therefore we have the description we need.

\begin{proposition}
	In the situation above, we have the following isomorphism:
	$$ \ch(\Detilde_{1,1}\setminus \Detilde_{1,1,1})\simeq \ZZ[1/6,t,c_1,c_2]$$ 
	where $c_1:=t_1+t_2$ and $c_2:=t_1t_2$.
\end{proposition}

It remains to describe the normal bundle of the closed immersion $\Detilde_{1,1}\setminus \Detilde_{1,1,1} \into \Mtilde_{3}\setminus\Detilde_{1,1,1}$ and the other classes. As usual, we denote by $\delta_1$ the fundamental class of $\Detilde_1$ in $\Mtilde_3$.

\begin{proposition}\label{prop:relat-detilde-1-1}
	We have the following equalities in the Chow ring of $\Detilde_{1,1}\setminus \Detilde_{1,1,1}$:
	\begin{itemize}
		\item[(1)] $\lambda_1=-t-c_1$,
		\item[(2)] $\lambda_2=c_2+tc_1$,
		\item[(3)] $\lambda_3=-tc_2$,
		\item[(4)] $[H]\vert_{\Detilde_{1,1}\setminus \Detilde_{1,1,1}} = -3t$,
		\item[(5)] $\delta_1\vert_{\Detilde_{1,1}\setminus \Detilde_{1,1,1}} = 2t+c_1$.
	\end{itemize}
	Furthermore, the second Chern class of the normal bundle of the closed immersion $\Detilde_{1,1}\setminus \Detilde_{1,1,1} \into \Mtilde_{3}\setminus\Detilde_{1,1,1}$ is equal to $c_2+tc_1+t^2$.
\end{proposition} 

\begin{proof}
	The restriction of the $\lambda$-classes can be computed using \Cref{lem:hodge-sep}. The proof of formula (5) is exactly the same of the computation of the normal bundle of $\Detilde_{1}\setminus \Detilde_{1,1}$ in $\Mtilde_3 \setminus \Detilde_{1,1}$. The only thing to remark is that the two $\psi$-classes of $\Mtilde_{1,2}\setminus \Mtilde_{1,1}$ coincide with the generator $t$ of $\ch(\Mtilde_{1,2}\setminus \Mtilde_{1,1})$.
	
	As far as the fundamental class of the hyperelliptic locus is concerned, it is clear that it coincides with the fundamental class of the locus in $\Mtilde_{1,2}\setminus \Mtilde_{1,1}$ parametrizing $2$-pointed stable curves of genus $1$ such that the two sections are both fixed points for an involution. This computation can be done using the description of $\Mtilde_{1,2}\setminus \Mtilde_{1,1}$ as $\Ctilde_{1,1}\setminus \Mtilde_{1,1}$, in particular as in Lemma 3.2 of \cite{DiLorPerVis}. In fact, they proved that $\Ctilde_{1,1}\setminus \Mtilde_{1,1}$ is an invariant subscheme $W$ of a $\gm$-representation $V$ of dimension $4$, where the action can be described as $$t.(x,y,\alpha,\beta)=(t^{-2}x,t^{-3}y,t^{-4}\alpha,t^{-6}\beta)$$
	for every $t \in \gm $ and $(x.y.\alpha,\beta) \in V$. Specifically, $W$ is the hypersurface in $V$ defined by the equation $y^2=x^3+\alpha x+\beta$, which is exactly the dehomogenization of the Weierstrass form of an elliptic curve with a flex at the infinity. A straightforward computation shows that the hyperelliptic locus is defined by the equation $y=0$. 
	
	Finally, the normal bundle of the closed immersion can be described using the theory developed in Appendix A of \cite{DiLorVis} as the sum 
	$$(N_{p_1|C}\otimes N_{e_1|E_1})\oplus (N_{p_2|C}\otimes N_{e_2|E_2})$$
	for every element $[(C,p_1,p_2), (E_1,e_1), (E_2,e_2)]$ in $\Detilde_{1,1}\setminus \Detilde_{1,1,1}$. 
\end{proof}

\section{Description of $\Detilde_{1,1,1}$}\label{sec:detilde-1-1-1}

Finally, to describe the last strata, we use the morphism introduced in the previous paragraph to define
$$c_6:\Mtilde_{1,1} \times \Mtilde_{1,1}\times \Mtilde_{1,1} \rightarrow \Detilde_{1,1,1}$$
which can be described as taking three elliptic (possibly singular) curves $(E_1,e_1)$, $(E_2,e_2)$, $(E_3,e_3)$ and attach them to a projective line with three distinct points $(\PP^1,0,1,\infty)$ using the order $e_1\equiv 0$, $e_2\equiv 1$, $e_3 \equiv \infty$. We denote by $S_3$ the group of permutation of a set of three elements.

\begin{lemma}\label{lem:descr-delta-1-1-1}
	The morphism $c_6$ is a $S_3$-torsor.
\end{lemma}

\begin{proof}
	The proof of \Cref{lem:detilde-1-1} can be adapted perfectly to this statement. 
\end{proof}

The previous lemma implies that we have an action of $S_3$ on $\ch(\Mtilde_{1,1}^{\times 3})$. Therefore it is clear that 
$$ \ch(\Detilde_{1,1,1})=\ZZ[1/6,c_1,c_2,c_3]$$
where $c_i$ is the $i$-th symmetric polynomial in the variables $t_1,t_2,t_3$, which are the generators of the Chow rings of the factors.. 

\begin{proposition}
	We have that the following equalities in the Chow ring of $\Detilde_{1,1,1}$:
	\begin{itemize}
		\item[(1)] $\lambda_1=-c_1$,
		\item[(2)] $\lambda_2=c_2$,
		\item[(3)] $\lambda_3=-c_3$,
		\item[(4)] $[H]\vert_{\Detilde_{1,1,1}}=0$,
		\item[(5)] $\delta_1\vert_{\Detilde_{1,1,1}}=c_1$,
		\item[(6)] $\delta_{1,1}\vert_{\Detilde_{1,1,1}}=c_2$.
	\end{itemize}
	Furthermore, the third Chern class of the normal bundle of the closed immersion $\Detilde_{1,1,1}\into \Mtilde_3$ is equal to $c_3$.
\end{proposition}

\begin{proof}
	We can use \Cref{lem:hodge-sep} to get the description of  the $\lambda$-classes. To compute $\delta_1$ and $\delta_{1,1}$ and the third Chern class of the normal bundle one can use again the results of Appendix A of \cite{DiLorVis} and adapt the proof of \Cref{prop:relat-detilde-1-1}. Notice that in this case $$N_{\Detilde_{1,1,1}|\Mtilde_3}=N_{e_1|E_1}\oplus N_{e_2|E_2}\oplus N_{e_3|E_3}.$$ Finally, we know that there are no hyperelliptic curves of genus $3$ with three separating nodes. In fact, any involution restricts to the identity over the projective line, because it fixes three points. But then its fixed locus is not finite, therefore it is not hyperelliptic. 
\end{proof}

\section{The gluing procedure and the Chow ring of $\Mtilde_3$}\label{sec:chow-m3tilde}

In the last section of this chapter, we explain how to calculate explicitly the Chow ring of $\Mtilde_3$. It is not clear a priori how to describe the fiber product that appears in \Cref{lem:gluing}.

Let $\cU$, $\cZ$ and $\cX$ as in \Cref{lem:gluing} and denote by $i:\cZ \into \cX$ the closed immersion and by $j:\cU \into \cZ$ the open immersion.  Let us set some notations.
\begin{itemize}
\item $ \ch(\cU)$ is generated by the elements $x'_1,\dots,x'_r$ and let $x_1,\dots,x_r$ be some liftings of $x_1',\dots,x_r'$ in $\ch(\cX)$; we denote by $\eta$ the morphism 
$$ \ZZ[1/6,X_1,\dots,X_r,Z] \longrightarrow \ch(\cX)$$
where $X_h$ maps to $x_h$ for $h=1,\dots,r$ and $Z$ maps to $[\cZ]$, the fundamental class of $\cZ$; we denote by $\eta_{\cU}$ the composite 
$$ \ZZ[1/6,X_1,\dots,X_r] \into \ZZ[1/6,X_1,\dots,X_r,Z] \longrightarrow \ch(\cX)\rightarrow \ch(\cU).$$ 
Furthermore, we denote by $p_1'(X), \dots, p_n'(X)$ a choice of generators for $\ker(\eta_{\cU})$, where $X:=(X_1,\dots,X_r)$.
\item  $ \ch(\cZ)$ is generated by elements $y_1,\dots,y_s \in \ch(\cZ)$; we denote by $a$ the morphism 
$$ \ZZ[1/6,Y_1,\dots,Y_s] \longrightarrow \ch(\cZ)$$ 
where $Y_h$ maps to $y_h$ for $h=1,\dots,s$. Furthermore, we denote by $q_1'(Y), \dots, q_m'(Y)$ a choice of generators for $\ker(a)$, where $Y:=(Y_1,\dots,Y_s)$.
\end{itemize}

Because $a$ is surjective, there exists a morphism
$$ \eta_{\cZ}: \ZZ[1/6,X_1,\dots,X_r,Z] \longrightarrow \ZZ[1/6,Y_1,\dots,Y_r]$$
which is a lifting of the morphism $i^*$, i.e. it makes the following diagram
$$
\begin{tikzcd}
	{\ZZ[1/6,X_1,\dots,X_r,Z]} \arrow[d, "\eta"] \arrow[r, "\eta_{\cZ}", dotted] & {\ZZ[1/6,Y_1,\dots,Y_s]} \arrow[d, "a"] \\
	\ch(\cX) \arrow[r, "i^*"]                                                    & \ch(\cZ)                               
\end{tikzcd}
$$
commute.

The cartesianity of the diagram in \Cref{lem:gluing} implies the following lemma.

\begin{lemma}\label{lem:gluing-sur}
	In the situation above, suppose that the morphism $\eta_{\cZ}$ is surjective. Then $\eta$ is surjective.
\end{lemma}

\begin{proof}
	This follows because $ j^*\circ \eta$ is always surjective and the hypothesis implies that also $i^*\circ \eta$ is surjective.
\end{proof}

In the hypothesis of the lemma, we can also describe explicitly the relations, i.e. the kernel of $\eta$. First of all, denote by $q_h(X,Z)$ some liftings of $q_h'(Y)$ through the morphism $\eta_{\cZ}$ for $h=1,\dots,m$. A straightforward computation shows that we have $Zq_h(X,Z) \in \ker \eta$ for $h=1,\dots,m$.
We have found our first relations. We refer to these type of relations as liftings of the closed stratum's relations. 

Another important set of relations comes from the kernel of $\eta_{\cZ}$. In fact, if $v \in \ker\eta_{\cZ}$, then a simple computation shows that $\eta(Zv)=0$. Therefore if $v_1,\dots,v_l$ are generators of $\ker\eta_{\cZ}$, we get that $Zv_h \in \ker\eta$ for $h=1,\dots,l$.

Finally, the last set of relations are the relations that come from $\cU$, the open stratum. The element $p_h(X)$ in general is not a relation as its restriction to $\ch(\cZ)$ can fail to vanish. We can however do the following procedure to find a modification of $p_h$ which still vanishes restricted to the open stratum and it is in the kernel of $\eta$. Recall that we have a well-defined morphism
$$ i^*:\ch(\cU) \longrightarrow \ch(\cZ)/(c_{\rm top}(N_{\cZ|\cX}))$$ 
which implies that $\eta_{\cZ}(p_h) \in (q'_1,\dots,q'_m,c_{\rm top}(N_{\cZ|\cX}))$. We choose an element $g'_h \in \ZZ[1/6,Y_1,\dots,Y_s]$ such that 
$$ \eta_{\cZ}(p_h) + g_h'c_{\rm top}(N_{\cZ|\cX}) \in (q'_1,\dots,q'_m)$$
and consider a lifting $g_h$ of $g'_h$ through the morphism $\eta_{\cZ}$. A straightforward computation shows that $p_h+Zg_h \in \ker \eta$ for every $h=1,\dots,n$.

\begin{proposition}\label{prop:desc-gluing}
	In the situation above, $\ker\eta$ is generated by the elements $Zq_1,\dots,Zq_m,Zv_1,\dots,Zv_l,p_1+Zg_1,\dots,p_n+Zg_n$.
\end{proposition}

\begin{proof}
	Consider the following commutative diagram
	$$
	\begin{tikzcd}
		{\ZZ[1/6,X_1,\dots,X_r,Z]} \arrow[r, "\eta"] \arrow[d, "\eta_{\cZ}"] & \ch(\cX) \arrow[d, "i^*"] \arrow[r, "j^*"] & \ch(\cU) \arrow[d, "i^*"]           \\
		{\ZZ[1/6,Y_1,\dots,Y_s]} \arrow[r, "a"]                              & \ch(\cZ) \arrow[r, "b"]                    & \ch(\cZ)/(c_{\rm top}(N_{\cZ|\cX}))
	\end{tikzcd}
	$$
	where $b$ is the quotient morphisms. Recall that the right square is cartesian. Notice that because $\eta_{\cZ}$ is surjective, all the other morphisms are surjectives.
	
	We denote by $c$ the top Chern class of the normal bundle $N_{\cZ|\cX}$. Let $p(X)+Zg(X,Z)$ be an element in $\ker \eta$. Because we have $$(j^*\circ \eta)(p(X)+Zg(X,Z))=0,$$ thus $p \in (p_1',\dots,p_n')$ which implies there exists $b_1,\dots,b_n \in \ZZ[1/6,X_1,\dots,X_r]$ such that 
	$$ p = \sum_{h=1}^n b_hp_h . $$
	Now, we pullback the element to $\cZ$, thus we get
	$$ (i^*\circ \eta)(p+Zg)=a\Big(\sum_{h=1}^n \eta_{\cZ}(b_h) \eta_{\cZ}(p_h) + c \eta_{\cZ}(g)\Big)$$ 
	or equivalently   
	$$\sum_{h=1}^n \eta_{\cZ}(b_h) \eta_{\cZ}(p_h) + c \eta_{\cZ}(g) \in (q_1',\dots, q_m').$$
	By construction $\eta_{\cZ}(p_h)=-g_h'c + (q_1',\dots,q_m')$ and $\eta_{\cZ}(g_h)=g_h'$, therefore we get 
	$$ c  \eta_{\cZ}\Big(g-\sum_{h=1}^nb_hg_h\Big) \in (q_1',\dots,q_m'). $$
	Because $c$ is a non-zero divisor in $\ch(\cZ)$ we have that
	$$ \eta_{\cZ}\Big(g-\sum_{h=1}^nb_hg_h\Big) \in (q_1',\dots,q_m')$$ 
	or equivalently 
	$$g= \sum_{h=1}^n b_hg_h + t$$
	with $t \in \ker(a \circ \eta_{\cZ})$. Therefore we have that
	$$ p+Zg=\sum_{h=1}^n b_h(p_h+Zg_h)+Zt$$
	with $t \in \ker(a \circ \eta_{\cZ})$. One can check easily that $\ker(a \circ \eta_{\cZ})$ is generated by $(v_1,\dots,v_l,q_1,\dots,q_m)$.
\end{proof}

Now, we sketch how to apply this procedure to the first two strata, namely $\Htilde_3\setminus \Detilde_{1}$ and $\Mtilde_3 \setminus (\Htilde_3 \cup \Detilde_1)$, to get the Chow ring of $\Mtilde_3 \setminus \Detilde_1$. This is the most complicated part as far as computations is concerned. The other gluing procedures are left to the curious reader as they follows the exact same ideas. 

In our situation, we have $\cU:=\Mtilde_{3}\setminus (\Htilde_3 \cup \Detilde_1)$ and $\cZ:= \Htilde_3 \setminus \Detilde_1$. We know the description of their Chow rings thanks to \Cref{cor:chow-hyper} and \Cref{cor:chow-quart}. Let us look at the generators we need. The Chow ring of $\cU$ is generated by $\lambda_1$, $\lambda_2$ and $\lambda_3$. The Chow ring of $\cZ$ is generated by $\lambda_1$, $\lambda_2$ and $\xi_1$. Therefore the morphism 
$$\eta_{\cZ} : \ZZ[1/6,\lambda_1,\lambda_2,\lambda_3, H] \longrightarrow \ZZ[1/6,\lambda_1,\lambda_2,\xi_1]$$ 
is surjective because $\eta_{\cZ}(H)=(2\xi_1-\lambda_1)/3$. We can also describe the $\ker \eta_{\cZ}$ which is generated by any lifting of the description of $\lambda_3$ in $\cZ$ (see \Cref{lem:lambda-class-H}). This gives us our first relation (after multiplying it by $H$, the fundamental class of the hyperelliptic locus). Furthermore, we can consider the ideal of relations in $\Htilde_{3}\setminus \Detilde_1$ which is generated by the relations $c_9$, $D_1$ and $D_2$ we described in \Cref{cor:chow-hyper}. Therefore we have other three relations. 

Lastly, we consider the four relations as in \Cref{cor:chow-quart} and compute their image through $\eta_{\cZ}$. The hardest part is to find a description of these elements in terms of the generators of the ideal of relations of $\Htilde_3 \setminus \Detilde_1$ and the top Chern class of the normal bundle of the closed immersion. We do not go into details about the computations, but the main idea is to notice that every monomial of the polynomials we need to describe can be written in terms of the relations and the top Chern class.

We state our theorem, which gives us the description of the Chow ring of $\Mtilde_3$. We write the explicit relations in \Cref{rem:relations-Mtilde}.

\begin{theorem}
	We have the following isomorphism 
	$$ \ch(\Mtilde_3)\simeq \ZZ[1/6,\lambda_1,\lambda_2,\lambda_3,H,\delta_1,\delta_{1,1},\delta_{1,1,1}]/I$$
	where $I$ is generated by the following relations:
	\begin{itemize}
		\item $k_h$, which comes from the generator of $\ker i_H^*$, where $i_H: \Htilde_3\setminus \Detilde_1 \into \Mtilde_3\setminus \Detilde_1$;
		\item $k_{1}(1)$ and $k_1(2)$, which come from the two generators of $\ker i_{1}^*$ where $i_{1}: \Detilde_1\setminus \Detilde_{1,1} \into \Mtilde_3\setminus \Detilde_{1,1}$;
		\item $k_{1,1}(1)$, $k_{1,1}(2)$ and $k_{1,1}(3)$, which come from the three generators of $\ker i_{1,1}^*$ where $i_{1,1}: \Detilde_{1,1}\setminus \Detilde_{1,1,1} \into \Mtilde_3\setminus \Detilde_{1,1,1}$;
		\item $k_{1,1,1}(1)$, $k_{1,1,1}(2)$, $k_{1,1,1}(3)$ and $k_{1,1,1}(4)$, which come from the four generators of $\ker i_{1,1,1}^*$ where $i_{1,1,1}: \Detilde_{1,1,1} \into \Mtilde_3$;
		\item $m(1)$, $m(2)$, $m(3)$ and $r$, which are the litings of the generators of the relations of the open stratum $\Mtilde_3\setminus (\Htilde_3 \cup \Detilde_1)$;
		\item $h(1)$, $h(2)$ and $h(3)$, which are the liftings of the generators of the relations of the stratum $\Htilde_3 \setminus \Detilde_1$; 
		\item $d_1(1)$, which is the lifting of the generator of the relations of the stratum $\Detilde_1\setminus \Detilde_{1,1}$.
	\end{itemize}
	Furthemore, $h(2)$, $h(3)$ and $d_1(1)$ are in the ideal generated by the other relations. 
\end{theorem}

\begin{remark}\label{rem:relations-Mtilde}
	We write explicitly the relations. 
	\begin{itemize}
		\item[]
		\begin{equation*}
		\begin{split}
		k_h& =\frac{1}{8}\lambda_1^3H + \frac{1}{8}\lambda_1^2H^2 + \frac{1}{4}\lambda_1^2H\delta_1 - \frac{1}{2}\lambda_1\lambda_2H - \frac{1}{8}\lambda_1H^3 +
		\frac{7}{8}\lambda_1H\delta_1^2 + \\ & + \frac{3}{2}\lambda_1\delta_1\delta_{1,1} - \frac{1}{2}\lambda_2H^2 + \lambda_3H - \frac{1}{8}H^4 - \frac{1}{4}H^3\delta_1 +
		\frac{1}{8}H^2\delta_1^2 + \frac{3}{4}H\delta_1^3 + \\ & + \frac{3}{2}\delta_1^2\delta_{1,1}
		\end{split}
		\end{equation*}
		
		\item[] 
		\begin{equation*}
		\begin{split}
		k_1(1)=\frac{1}{4}\lambda_1^2\delta_1 + \frac{1}{2}\lambda_1H\delta_1 + 2\lambda_1\delta_1^2 + \lambda_2\delta_1 + \frac{1}{4}H^2\delta_1 + H\delta_1^2 + \frac{7}{4}\delta_1^3 -\delta_1\delta_{1,1}
		\end{split}
		\end{equation*}
		
		\item[] 
		\begin{equation*}
		\begin{split}
		k_1(2)& =\frac{1}{4}\lambda_1^3\delta_1 + \frac{1}{2}\lambda_1^2H\delta_1 + \frac{5}{4}\lambda_1^2\delta_1^2 + \frac{1}{4}\lambda_1H^2\delta_1 + \frac{3}{2}\lambda_1H\delta_1^2 +
		\frac{7}{4}\lambda_1\delta_1^3 + \\ & + \lambda_1\delta_1\delta_{1,1} - \lambda_1\delta_{1,1,1} + \lambda_3\delta_1+  \frac{1}{4}H^2\delta_1^2 + H\delta_1^3 + \frac{3}{4}\delta_1^4 +
		\delta_1^2\delta_{1,1}
		\end{split}
		\end{equation*}
		
		\item[] 
		\begin{equation*}
		\begin{split}
		k_{1,1}(1)=-\lambda_1^2\delta_{1,1} - 2\lambda_1\delta_1\delta_{1,1} - \lambda_2\delta_{1,1} - \delta_1^2\delta_{1,1} + \delta_{1,1}^2
		\end{split}
		\end{equation*}
		
		\item[] 
		\begin{equation*}
		\begin{split}
		k_{1,1}(2)=3\lambda_1\delta_{1,1} + H\delta_{1,1} + 3\delta_1\delta_{1,1}
		\end{split}
		\end{equation*}
		
		\item[] 
		\begin{equation*}
		\begin{split}
		k_{1,1}(3)&=2\lambda_1^3\delta_{1,1} + 5\lambda_1^2\delta_1\delta_{1,1} + \lambda_1\lambda_2\delta_{1,1} + 4\lambda_1\delta_1^2\delta_{1,1} + \lambda_2\delta_1\delta_{1,1} + \\ & + \lambda_2\delta_{1,1,1} + \lambda_3\delta_{1,1} +
		\delta_1^3\delta_{1,1}
		\end{split}
		\end{equation*}
		
		\item[] 
		\begin{equation*}
		\begin{split}
		k_{1,1,1}(1)=\lambda_1\delta_{1,1,1} + \delta_1\delta_{1,1,1}
		\end{split}
		\end{equation*}
		
		\item[] 
		\begin{equation*}
		\begin{split}
		k_{1,1,1}(2)=\lambda_2\delta_{1,1,1} - \delta_{1,1}\delta_{1,1,1}
		\end{split}
		\end{equation*}
		
		\item[] 
		\begin{equation*}
		\begin{split}
		k_{1,1,1}(3)=\lambda_3\delta_{1,1,1} + \delta_{1,1,1}^2
		\end{split}
		\end{equation*}		
	
		\item[] 
		\begin{equation*}
		\begin{split}
		k_{1,1,1}(4)=H\delta_{1,1,1}
		\end{split}
		\end{equation*}
		
		\item[] 
		\begin{equation*}
		\begin{split}
		m(1)&=12\lambda_1^4 - \frac{7}{3}\lambda_1^3H + 27\lambda_1^3\delta_1 - 44\lambda_1^2\lambda_2 - \frac{706}{9}\lambda_1^2H^2 - \frac{65}{2}\lambda_1^2H\delta_1 + \\ & 
		+ 84\lambda_1^2\delta_1^2 - 32\lambda_1^2\delta_{1,1}  - 38\lambda_1\lambda_2H + 92\lambda_1\lambda_3 - \frac{715}{9}\lambda_1H^3 - \\ & -
		\frac{1340}{9}\lambda_1H^2\delta_1 - 25\lambda_1H\delta_1^2 + 69\lambda_1\delta_1^3  - 130\lambda_1\delta_1\delta_{1,1} + 92\lambda_1\delta_{1,1,1} + \\ & +
		6\lambda_2H^2 - \frac{46}{3}H^4 - \frac{1205}{18}H^3\delta_1 - \frac{562}{9}H^2\delta_1^2 - \frac{101}{6}H\delta_1^3 -
		54\delta_1^2\delta_{1,1}
		\end{split}
		\end{equation*}
		
		\item[] 
		\begin{equation*}
		\begin{split}
		m(2)&=-\frac{55}{18}\lambda_1^4H + \frac{9}{2}\lambda_1^4\delta_1 - 14\lambda_1^3\lambda_2 - \frac{31}{3}\lambda_1^3H^2 - \frac{58}{9}\lambda_1^3H\delta_1 +
		69\lambda_1^3\delta_1^2 - \\ & - \frac{272}{3}\lambda_1^3\delta_{1,1} -  \frac{173}{9}\lambda_1^2\lambda_2H + 2\lambda_1^2\lambda_3 - \frac{137}{18}\lambda_1^2H^3
		- \frac{167}{4}\lambda_1^2H^2\delta_1 + \\ & + \frac{1831}{36}\lambda_1^2H\delta_1^2 + \frac{459}{2}\lambda_1^2\delta_1^3 - 
		\frac{461}{3}\lambda_1^2\delta_1\delta_{1,1} + 2\lambda_1^2\delta_{1,1,1} + 48\lambda_1\lambda_2^2 + \\ & + \frac{1}{9}\lambda_1\lambda_2H^2 - \frac{1}{3}\lambda_1H^4 -
		\frac{605}{18}\lambda_1H^3\delta_1 - \frac{955}{18}\lambda_1H^2\delta_1^2 +  139\lambda_1H\delta_1^3 + \\ & + 291\lambda_1\delta_1^4 -
		49\lambda_1\delta_1^2\delta_{1,1} + 48\lambda_2^2H - 96\lambda_2\lambda_3 + \frac{16}{3}\lambda_2H^3 + 48\lambda_2\delta_1\delta_{1,1} - \\ & - 96\lambda_2\delta_{1,1,1}
		- \frac{241}{36}H^4\delta_1 - \frac{1111}{36}H^3\delta_1^2 - \frac{63}{4}H^2\delta_1^3 + \frac{367}{4}H\delta_1^4 + 126\delta_1^5
		\end{split}
		\end{equation*}
		
		\item[] 
		\begin{equation*}
		\begin{split}
		m(3)&=\frac{419}{648}\lambda_1^5H - 9\lambda_1^5\delta_1 + \frac{17903}{972}\lambda_1^4H^2 - \frac{19763}{648}\lambda_1^4H\delta_1 -
		\frac{285}{4}\lambda_1^4\delta_1^2 + \\ & + \frac{57344}{27}\lambda_1^4\delta_{1,1} - \frac{401}{162}\lambda_1^3\lambda_2H + 15\lambda_1^3\lambda_3 +
		\frac{100795}{1944}\lambda_1^3H^3 - \\ & - \frac{16057}{972}\lambda_1^3H^2\delta_1 - \frac{100555}{648}\lambda_1^3H\delta_1^2 -
		\frac{861}{4}\lambda_1^3\delta_1^3 + \frac{614635}{54}\lambda_1^3\delta_1\delta_{1,1}  + \\ & + 15\lambda_1^3\delta_{1,1,1} - \frac{6433}{81}\lambda_1^2\lambda_2H^2 -
		32\lambda_1^2\lambda_2\delta_{1,1} + \frac{11561}{216}\lambda_1^2H^4 + \\ & + \frac{12349}{324}\lambda_1^2H^3\delta_1 +
		\frac{559}{12}\lambda_1^2H^2\delta_1^2  - \frac{120883}{324}\lambda_1^2H\delta_1^3 - \frac{1263}{4}\lambda_1^2\delta_1^4 + \\ & +
		\frac{198799}{9}\lambda_1^2\delta_1^2\delta_{1,1} - 22\lambda_1\lambda_2^2H - 52\lambda_1\lambda_2\lambda_3 - 151\lambda_1\lambda_2H^3  + \\ & +
		54\lambda_1\lambda_2\delta_1\delta_{1,1} - 52\lambda_1\lambda_2\delta_{1,1,1} + \frac{2845}{162}\lambda_1H^5 + \frac{19415}{324}\lambda_1H^4\delta_1 + \\ & +
		\frac{59303}{324}\lambda_1H^3\delta_1^2 + \frac{10946}{243}\lambda_1H^2\delta_1^3 - \frac{66367}{162}\lambda_1H\delta_1^4 -
		\frac{903}{4}\lambda_1\delta_1^5 + \\ & + 18519\lambda_1\delta_1^3\delta_{1,1} - 22\lambda_2^2H^2 - \frac{1333}{18}\lambda_2H^4  +
		86\lambda_2\delta_1^2\delta_{1,1} + 112\lambda_3^2 + \\ & + 112\lambda_3\delta_{1,1,1} - \frac{407}{216}H^6 + \frac{14521}{648}H^5\delta_1 +
		\frac{40205}{648}H^4\delta_1^2 + \frac{147917}{972}H^3\delta_1^3 - \\ & - \frac{8563}{243}H^2\delta_1^4 -
		\frac{104075}{648}H\delta_1^5 - 63\delta_1^6 + \frac{11377}{2}\delta_1^4\delta_{1,1}
		\end{split}
		\end{equation*}
		
		\item[] 
		\begin{equation*}
			\begin{split}
				h(1)&=\frac{3}{4}\lambda_1^3H + \frac{13}{4}\lambda_1^2H^2 + \frac{9}{4}\lambda_1^2H\delta_1 + \frac{13}{4}\lambda_1H^3 +  \frac{13}{2}\lambda_1H^2\delta_1 +
				\frac{9}{4}\lambda_1H\delta_1^2 + \\ & + \frac{3}{4}H^4 + \frac{13}{4}H^3\delta_1+ \frac{13}{4}H^2\delta_1^2 + \frac{3}{4}H\delta_1^3
			\end{split}
		\end{equation*}
	
		\item[] 
		\begin{equation*}
		\begin{split}
		r&=-\frac{7}{81}\lambda_1^8H + \frac{247145}{2916}\lambda_1^7H^2 - \frac{39125}{81}\lambda_1^7H\delta_1 - 20800\lambda_1^7\delta_{1,1} - \\ & -
		\frac{1286}{243}\lambda_1^6\lambda_2H + \frac{1573727}{8748}\lambda_1^6H^3  -\frac{400579}{162}\lambda_1^6H^2\delta_1 -
		\frac{618187}{162}\lambda_1^6H\delta_1^2 + \\ & + \frac{31710800}{81}\lambda_1^6\delta_1\delta_{1,1} - \frac{736943}{729}\lambda_1^5\lambda_2H^2  -
		24288\lambda_1^5\lambda_2\delta_{1,1} - \frac{2193853}{8748}\lambda_1^5H^4 - \\ & - \frac{4516739}{729}\lambda_1^5H^3\delta_1 -
		\frac{35110427}{2916}\lambda_1^5H^2\delta_1^2  - \frac{3448919}{243}\lambda_1^5H\delta_1^3 + \\ & +
		\frac{253855904}{81}\lambda_1^5\delta_1^2\delta_{1,1} - \frac{136}{3}\lambda_1^4\lambda_2^2H - \frac{2054561}{729}\lambda_1^4\lambda_2H^3+ \\ & +
		133280\lambda_1^4\lambda_2\delta_1\delta_{1,1} - \frac{2986483}{2916}\lambda_1^4H^5 - \frac{40469125}{4374}\lambda_1^4H^4\delta_1 - \\ & -
		\frac{52534855}{2916}\lambda_1^4H^3\delta_1^2- \frac{33507299}{1458}\lambda_1^4H^2\delta_1^3 -
		\frac{17079953}{486}\lambda_1^4H\delta_1^4 + \\ & + \frac{246525952}{27}\lambda_1^4\delta_1^3\delta_{1,1} + \frac{26584}{9}\lambda_1^3\lambda_2^2H^2  
		- 9248\lambda_1^3\lambda_2^2\delta_{1,1} - \frac{364370}{243}\lambda_1^3\lambda_2H^4 + \\ & + 911024\lambda_1^3\lambda_2\delta_1^2\delta_{1,1} -
		1152\lambda_1^3\lambda_3^2 - 1152\lambda_1^3\lambda_3\delta_{1,1,1}  - \frac{315269}{324}\lambda_1^3H^6 - \\ & -
		\frac{6063527}{729}\lambda_1^3H^5\delta_1 - \frac{80352425}{4374}\lambda_1^3H^4\delta_1^2 -
		\frac{8710724}{2187}\lambda_1^3H^3\delta_1^3  - \\ & - \frac{26273782}{729}\lambda_1^3H^2\delta_1^4 -
		\frac{13515761}{243}\lambda_1^3H\delta_1^5 + \frac{41134658}{3}\lambda_1^3\delta_1^4\delta_{1,1} + 896\lambda_1^2\lambda_2^3H + \\ & +
		256\lambda_1^2\lambda_2^2\lambda_3  + \frac{225704}{27}\lambda_1^2\lambda_2^2H^3 - 11680\lambda_1^2\lambda_2^2\delta_1\delta_{1,1} +
		256\lambda_1^2\lambda_2^2\delta_{1,1,1} + \\ & + \frac{19484}{9}\lambda_1^2\lambda_2H^5 + 1716232\lambda_1^2\lambda_2\delta_1^3\delta_{1,1} -
		\frac{70385}{324}\lambda_1^2H^7 - \frac{716609}{162}\lambda_1^2H^6\delta_1 - \\ & - \frac{3175405}{243}\lambda_1^2H^5\delta_1^2 +
		\frac{14917697}{2187}\lambda_1^2H^4\delta_1^3 + \frac{27927485}{1458}\lambda_1^2H^3\delta_1^4 - \\ & -
		\frac{24157867}{486}\lambda_1^2H^2\delta_1^5 - \frac{12315655}{243}\lambda_1^2H\delta_1^6 +
		11373175\lambda_1^2\delta_1^5\delta_{1,1}  + \\ & + 1664\lambda_1\lambda_2^3H^2 - 1152\lambda_1\lambda_2^3\delta_{1,1} +
		192920/27\lambda_1\lambda_2^2H^4 + 7328\lambda_1\lambda_2^2\delta_1^2\delta_{1,1} + \\ & + 5824\lambda_1\lambda_2\lambda_3^2+
		5824\lambda_1\lambda_2\lambda_3\delta_{1,1,1} + \frac{66191}{27}\lambda_1\lambda_2H^6 + 1353816\lambda_1\lambda_2\delta_1^4\delta_{1,1} + \\ & +
		\frac{12985}{108}\lambda_1H^8 - \frac{104593}{81}\lambda_1H^7\delta_1  - \frac{1752295}{324}\lambda_1H^6\delta_1^2 +
		\frac{906349}{729}\lambda_1H^5\delta_1^3 + \\ & + \frac{57919459}{2187}\lambda_1H^4\delta_1^4 + \frac{6920350}{729}\lambda_1H^3\delta_1^5
		- \frac{53724649}{1458}\lambda_1H^2\delta_1^6 - \\ & - \frac{5743493}{243}\lambda_1H\delta_1^7 + 4958470\lambda_1\delta_1^6\delta_{1,1} -
		1152\lambda_2^3\lambda_3 + 768\lambda_2^3H^3 - \\ & - 1152\lambda_2^3\delta_1\delta_{1,1} - 1152\lambda_2^3\delta_{1,1,1} +
		\frac{16064}{9}\lambda_2^2H^5 + 9760\lambda_2^2\delta_1^3\delta_{1,1} + \frac{5399}{9}\lambda_2H^7 + \\ & + 391040\lambda_2\delta_1^5\delta_{1,1}  -
		10976\lambda_3^3 - 10976\lambda_3^2\delta_{1,1,1} + \frac{171}{4}H^9 - \frac{7903}{54}H^8\delta_1 - \\ & -
		\frac{304223}{324}H^7\delta_1^2 - \frac{365225}{486}H^6\delta_1^3  + \frac{4136302}{729}H^5\delta_1^4 +
		\frac{43734445}{4374}H^4\delta_1^5 - \\ & - \frac{3256102}{2187}H^3\delta_1^6 - \frac{14121601}{1458}H^2\delta_1^7 -
		\frac{1042615}{243}H\delta_1^8 + 887989\delta_1^7\delta_{1,1}
		\end{split}
		\end{equation*}

	\end{itemize}
\end{remark}
\chapter{The Chow ring of $\Mbar_3$}\label{chap:3}

This chapter is dedicated to the computation of the Chow ring of $\Mbar_3$ and the comparison with the result of Faber. 

The first part focuses on describing the strata of $A_r$-singularities that we eventually remove from $\Mtilde_3$ to get the Chow ring of $\Mbar_3$.

The second part focuses on the abstract computations, namely on finding the generators of the ideal of relations coming from the closed stratum of singularity of type $A_r$ with $r\geq 2$. 

In the third part, we describe how to compute these relations in the Chow ring of $\Mtilde_3$: the idea is to use the stratification introduced in \Cref{chap:2}. In fact, we compute every relation restricting it to every strata and then gluing the informations to get an element in $\Mtilde_3$ .

In the four part, we compare our description to the one in \cite{Fab}.

\section{The substack of $A_r$-singularity}\label{sec:3-1}
In this section we describe the closed substack of $\Mtilde_g^r$ which parametrizes $A_r$-stable curves with at least a singularity of type $A_h$ with $h\geq 2$. We do so by stratifying this closed substack considering singularities of type $A_h$ with a fixed $h$ greater than $2$.

Let $g\geq 2$ and $r\geq 1$ be two integers and $\kappa$ be the base field of characteristic greater than $2g+1$. We recall the sequence of open subset (see \Cref{rem: max-sing})
$$ \Mtilde_g^0 \subset \Mtilde_g^1 \subset \dots \subset \Mtilde_g^r$$ 
and we define $\widetilde{\cA}_{\geq n}:=\Mtilde_g^r\setminus \Mtilde_g^{n-1}$ for $n=0,\dots,r+1$ setting $\Mtilde_g^{-1}:=\emptyset$.

We now introduce an alternative to $\widetilde{\cA}_{\geq n}$ which is easier to describe. Suppose $n$ is a positive integer smaller or equal than $r$ and let $\cA_{\geq n}$ be the substack of the universal curve $\Ctilde_g^r$ of $\Mtilde_g^r$ parametrizing pairs $(C/S,p)$ where $p$ is a section whose geometric fibers over $S$ are $A_r$-singularities for $r\geq n$. We give to $\cA_{\geq n}$ the structure of closed substack of $\Ctilde_g^r$ inductively on $n$. Clearly if $n=0$ we have $\cA_{\geq 0}=\Ctilde_g^r$. To define $\cA_{\geq 1}$, we need to find the stack-theoretic structure of the singular locus of the natural morphism $\Ctilde_g^r \rightarrow \Mtilde_g^r$. This is standard and it can be done by taking the zero locus of the $1$-st Fitting ideal of $\Omega_{\Ctilde_g^r|\Mtilde_g^r}$. We have that $\cA_{\geq 1}\rightarrow  \Mtilde_g^r$ is now finite and it is unramified over the nodes, while it ramifies over the more complicated singularities. Therefore, we can denote by $\cA_{\geq 2}$ the substack of $\cA_{\geq 1}$ defined by the $0$-th Fitting ideal of $\Omega_{\cA_{\geq 1}|\Mtilde_g^r}$. A local computation shows us that $\cA_{\geq 2} \rightarrow \Mtilde_g^r$ is unramified over the locus of $A_2$-singularities and ramified elsewhere. Inductively, we can iterate this procedure considering the $0$-th Fitting ideal of $\Omega_{\cA_{\geq n-1}|\Mtilde_g^r}$ to define $\cA_{\geq n}$. 

A local computation shows that the geometric points of $\cA_{\geq n}$ are exactly the pairs $(C,p)$ such that $p$ is an $A_{n'}$-singularity for $n\leq n'\leq r$.

Let us define $\cA_n:=\cA_{\geq n}\setminus \cA_{\geq n+1}$ for  $n=0,\dots,r-1$. We have a stratification of $\cA_{\geq 2}$
$$ \cA_{r}=\cA_{\geq r} \subset \cA_{\geq r-1} \subset \dots \subset \cA_{\geq 2}$$ 
where the $\cA_n$'s are the associated locally closed strata for $n=2,\dots, r$. 

The first reason we choose to work with $\cA_{\geq n}$ instead of $\widetilde{\cA}_{\geq n}$ is the smoothness of the locally closed substack $\cA_n$ of $\Ctilde_g^r$.

\begin{proposition}
	The stack $\cA_n$ is smooth.
\end{proposition}

\begin{proof}
	We can adapt the proof of Proposition 1.6 of \cite{DiLorPerVis} perfectly. The only thing to point out is that the \'etale model induced by the deformation theory of the pair $(C,p)$  would be $y^2=x^n+a_{n-2}x^{n-2}+\dots+a_1x+a_0$, thus the restriction to $\cA_n$ is described by the equation $a_{n-2}=\dots=a_1=a_0$. The smoothness of $\Mtilde_g^r$ implies the statement.
\end{proof}

Before going into details for the odd and even case, we describe a way of desingularize a $A_n$-singularity.

\begin{lemma}\label{lem:blowup-an}
	Let $(C,p) \in \cA_n(S)$, then the ideal $I_p$ associated to the section $p$ verifies the hypothesis of \Cref{lem:blowup}. If we denote by $b:\widetilde{C}\rightarrow C$ the blowup morphism and by $D$ the preimage $b^{-1}(p)$, then $D$ is finite flat of degree $2$ over $S$.
	\begin{itemize}
		\item If $n=1$, $D$ is a Cartier divisor of $\widetilde{C}$ \'etale of degree $2$ over $S$.	
		\item If $n\geq 2$, the $0$-th Fitting ideal of $\Omega_{D|S}$ define a section $q$ of $D\subset \widetilde{C}\rightarrow S$ such that $\widetilde{C}$ is an $A_r$-prestable curve and $q$ is an $A_{n-2}$-singularity of $\widetilde{C}$. 
	\end{itemize} 
\end{lemma}

\begin{proof}
	Suppose that the statement is true when $S$ is reduced. Because $\cA_n$ is smooth, we know that up to an \'etale cover of $S$, our object $(C,p)$ is the pullback of an object over a smooth scheme, therefore reduced. All the properties in \Cref{lem:blowup-an} are stable by base change and satisfies \'etale descent, therefore we have the statement for any $S$. 
	
	Assume that $S$ is reduced. To prove the first part of the statement, it is enough to prove that the geometric fiber $\cO_{X_s}/I_s^m$ has constant length over every point $s \in S$. This follows from a computation with the complete ring in $p_s$. Regarding the rest of the statement, we can restrict to the geometric fibers over $S$ and reduce to the case $S=\spec k$ where $k$ is an algebraically closed field over $\kappa$. The statement follows from a standard blowup computation.
\end{proof}

Let us start with $(C,p)\in \cA_n(S)$. We can construct a (finite) series of blowups which desingularize the family $C$ in the section $p$. 

Suppose $n$ is even. If we apply \Cref{lem:blowup-an} iteratively, we get the successive sequence of blowups over $S$
$$ 
\begin{tikzcd}
\widetilde{C}_m \arrow[r, "b_m"] & \widetilde{C}_{m-1} \arrow[r, "b_{m-1}"] & \dots \arrow[r, "b_1"] & \widetilde{C}_0:=C
\end{tikzcd}
$$
with sections $q_h:S \rightarrow \widetilde{C}_h$ where $m:=n/2$, the morphism $b_h:\widetilde{C}_m:={\rm Bl}_{q_{h-1}}\widetilde{C}_{h-1}\rightarrow \widetilde{C}_{h-1}$ is the blowup of $\widetilde{C}_{h-1}$ with center $q_{h-1}$ ($q_0:=p$) and $q_h$ is the section of $\widetilde{C}_h$ over $q_{h-1}$ as in \Cref{lem:blowup-an}. We have that $\widetilde{C}_m$ is an $A_r$-prestable curve of genus $g-m$ and $q_m$ is a smooth section. 

On the contrary, if $n:=2m-1$ is odd, the same sequence of blowups gives us a curve $\widetilde{C}_m$  which has arithmetic genus either $g-m$ or $g-m+1$ depending on whether the geometric fibers of $\widetilde{C}_m$ are connected or not and an  \'etale Cartier divisor $D$ of degree $2$ over $S$.

\begin{definition}
	Let $(C,p)$ be an object of $\cA_n(S)$ with $n=2m$ or $n=2m-1$ for $S$ any scheme. The composition of blowups $b_m\circ b_{m-1} \circ \dots \circ b_{1}$ described above is called the relative $A_n$-desingularization and it is denoted as $b_{C,p}$. By abuse of notation, we refer to the source of the relative $A_n$-desingularization as relative $A_n$-desingularization. We also denote by $J_b$ the conductor ideal associated to it. 
	
	 We say that an object $(C/S,p)$ of $\cA_n$ is a separating $A_n$-singularity if the geometric fibers over $S$ of the relative $A_n$-desingularization are not connected. 
\end{definition}
\begin{remark}
	By construction, the relative $A_n$-desingularization is compatible with base change.
\end{remark}

\begin{lemma}\label{lem:conductor}
	Let $(C,p)$ be an object of $\cA_n(S)$ with $n=2m$ or $n=2m-1$ and let $b_{C,p}:\widetilde{C}\rightarrow C$ the relative $A_n$-desingularization. We have that $J_b$ is flat over $S$ and its formation is compatible with base change over $S$. Furthermore, we have that $J_b=I_{b^{-1}(p)}^m$ as an ideal of $\widetilde{C}$, where $I_{b^{-1}(p)}$ is the ideal associated to the preimage of $p$ through $b$.
\end{lemma}

\begin{proof}
	The flatness and compatibility with base change are standard. If $S$ is the spectrum of an algebraically closed field, we know that the equality holds. We can consider the diagram
	$$
	\begin{tikzcd}
	&                       & f_*I_{b^{-1}(p)}^m \arrow[d]                       &             &   \\
	0 \arrow[r] & \cO_C \arrow[r, hook] & f_*\cO_{\widetilde{C}} \arrow[r, two heads] & Q \arrow[r] & 0
	\end{tikzcd}
	$$ 
	and show that the composite morphism $f_*I_{b^{-1}(p)}^m\rightarrow Q$ is the zero map, restricting to the geometric fibers over $S$ (in fact they are both finite and flat over $S$). Therefore $I_{b^{-1}(p)}^m$ can be seen also as an ideal of $\cO_{C}$. Because the conductor ideal is the largest ideal of $\cO_{C}$ which is also an ideal of $\cO_{\widetilde{C}}$, we get an inclusion $I_{b^{-1}(p)}^m\subset J_f$ whose surjectivity can be checked on the geometric fibers over $S$.
\end{proof}

\begin{remark}\label{rem:stab}
	The stability condition for $\widetilde{C}$ can be described using the Noether formula (see Proposition 1.2 of \cite{Cat}). We have that $\omega_{C/S}$ is ample if and only if $\omega_{\widetilde{C}/S}(J_b^{\vee})$ is ample.
\end{remark}

 Lastly, we prove that the stack parametrizing separating $A_n$-singularities for a fixed positive integer $n$ is closed inside $\Ctilde_g^r$. 
  
\begin{lemma}\label{lem:sep-sing}
	Let $C\rightarrow \spec R$ a family of $A_r$-prestable curves over a DVR and denote by $K$ the function field of $R$. Suppose there exists a generic section $s_K$ of the morphism such that its image is a separating $A_{r_0}$-singularity (with $r_0\leq r$). Then the section $s_R$ (which is the closure of $s_K$) is still a separating $A_{r_0}$-singularity. 
\end{lemma}

\begin{proof}
	Because $s_K$ is a separating $A_{r_0}$-singularity, then $r_0$ is necessarily odd. Furthemore, we have that the special fiber $s_k:=s_R\otimes_R k$ is an $A_{r_1}$-singularity with $r_1\geq r_0$.
	
	Let us call $m_R$ the ideal associated with the section $s_R$. Because $C/\spec R$ is $A_r$-prestable, we can compute $\dim_L(\cO_C/I_R^h\otimes_R L)$ when $L$ is the algebraic closure of either the function field $K$ or the residue field $k$. We have that 
	$$ \dim_L L[[x,y]]/(y^2-x^n,m^h)=2h-1$$ 
	for every $h\geq 1$ and every $n\geq 2$, where $m=(x,y)$. 
	
	Therefore, $\dim_L(\cO_C/I_R^h\otimes_R L)$ is constant on the geometric fibers over $\spec R$ and we get that $\cO_C/I_R^h$ is $R$-flat  because $R$ is reduced.  Consider now $\widetilde{C}_1:={\rm Bl}_{I_R}C$, which is still $R$-flat, proper and finitely presented thanks to \Cref{lem:blowup} and commutes with base change. We denote by $b:\widetilde{C}_1 \rightarrow C$ the blowup morphism. We have that a local computations shows that if $r_0\geq 3$ we have $b^{-1}(s_R)_{\rm red}=\spec R$ and thus it defines a section $q_1$ of $\widetilde{C}_1$ which is a $A_{r_0-2}$-singularity at the generic fiber and a $A_{r_1-2}$-singularity at the special fiber. We can therefore iterate this procedure $r_0/2$ times until we get $\widetilde{C}_{r_0/2}\rightarrow \spec R$ which is a flat proper finitely presented morphism whose generic fiber is not geometrically connected. Therefore the special fiber is not geometrically connected as it is geometrically reduced. This clearly implies that $r_1=r_0$ and that $s_R$ is a separating section.
\end{proof}

\subsection*{Description of $\cA_n$ for the even case}

Let $n:=2m$ be an even number with $m\geq 1$. Firstly, we study the $A_{2m}$-singularity locally, and then we try to describe everything in families of projective curves.

Let $(C,p)$ be a $1$-dimensional reduced $1$-pointed scheme of finite type over an algebraically closed field where $p$ is an $A_{2m}$-singularity. Consider now the partial normalization in $p$ which gives us a finite birational morphism 
$$b:\widetilde{C} \longrightarrow C$$ 
which is infact an homeomorphism. This implies that the only think we need to understand is how the structural sheaf changes through $b$. We have the standard exact sequence
$$ 0 \rightarrow \cO_C \rightarrow \cO_{\widetilde{C}} \rightarrow Q \rightarrow 0$$
where $Q$ is a coherent sheaf on $C$ with support on the point $p$. Consider now the conductor ideal $J_b$ of the morphism $b$, which is both an ideal of $\cO_C$ and of $\cO_{\widetilde{C}}$.  Consider the morphism of exact sequences
$$
\begin{tikzcd}
0 \arrow[r] & \cO_C \arrow[r] \arrow[d, two heads] & \cO_{\widetilde{C}} \arrow[r] \arrow[d, two heads] & Q \arrow[r] \arrow[d, Rightarrow, no head] & 0 \\
0 \arrow[r] & \cO_C/J_b \arrow[r]                    & \cO_{\widetilde{C}}/J_b \arrow[r]                    & Q \arrow[r]                                & 0,
\end{tikzcd}
$$
it is easy to see that the vertical morphism on the right is an isomorphism. Therefore \Cref{lem:cond-diag} and \Cref{lem:conductor} imply that to construct an $A_{2m}$-singularity we need the partial normalization $\widetilde{C}$, the section $q$ and a subalgebra of $\cO_{\widetilde{C}}/m_q^{2m}$. Notice that not every subalgebra works. 

\begin{remark}
	First of all, a local computation shows that the extension $\cO_{C}/J_b \into \cO_{\widetilde{C}}/m_q^{2m}$ is finite flat of degree $2$. Luckily, the converse is also true: if $B=k[[t]]$ and $I=(t^{2m})$, then the subalgebras $C$ such that $C \into B/J_b$ is finite flat of degree $2$ are exactly the ones whose pullback through the projection $B \rightarrow B/J_b$ is an $A_{2m}$-singularity. This should serve as a motivation for the alternative description we are going to prove for $\cA_{2m}$.
\end{remark}  

The idea now is to prove that the same exact picture works for families of curves. The result holds in a greater generality, but we describe directly the case of families of $A_r$-prestable curves. 

We want to construct an algebraic stack whose objects are triplets $(\widetilde{C}/S,q,A)$ where $(\widetilde{C}/S,q)$ is a $A_r$-prestable $1$-pointed curve of genus $g-m$ with some stability condition (see \Cref{rem:stab}) and $A\subset \cO_{\widetilde{C}}/I_q^{2m}$ is a finite flat extension  of degree $2$ of flat $\cO_S$-algebras, where $I_q$ is the ideal sheaf associated to the section $q$. 

Firstly, we introduce the stack $\Mtilde_{h,[l]}^r$ parametrizing $A_r$-prestable 1-pointed curves $(\widetilde{C},q)$ such that $\omega_{\widetilde{C}}(lq)$ is ample. This is not difficult to describe, in fact we have a natural inclusion
$$ \Mtilde_{h,[l+1]}^r\subset \Mtilde_{h,[l]}^r$$ 
which is an equality for $l\geq 2$ if $h\geq 1$ and for $l\geq 3$ if $h=0$. We have that the only curves that live in $\Mtilde_{h,[l]}^r\setminus \Mtilde_{h,1}$ are curves that have one (irreducible) tail of genus $0$ and the section lands on the tail. We have the following result.

\begin{proposition}\label{prop:desc-str}
	In the situation above, if $h\geq 1$ we have an isomorphism 
	$$ \Mtilde_{h,[l]}^r \simeq \Mtilde_{h,1}^r\times [\AA^1/\gm]$$ 
	for $l\geq 2$. If $h=0$, we have an isomorphism 
	$$ \Mtilde_{0,[l]}\simeq \cB(\gm \rtimes \ga)$$ 
	for $l\geq 3$.
\end{proposition} 

\begin{proof}
	The case $h=0$ is straightforward. Clearly for $h\geq 1$ it is enough to construct the isomorphism for $l=2$.
	
	In the proof of Theorem 2.9 of \cite{DiLorPerVis}, they proved the description for $r=2$, but the proof can be generalized easily for any $r$. We sketch an alternative proof. First of all, it is easy to see that $\Mtilde_{h,[l]}^r$ is smooth using deformation theory.
	
	We follow the construction introduced in Section 2.3 of \cite{DiLorPerVis}: consider $\Ctilde_{h,1}^r$ the universal curve of $\Mtilde_{h,1}^r$ and define $\cD_{h,1}$ to be the blowup of $\Ctilde_{h,1}^r \times [\AA^1/\gm]$ in the center $\Mtilde_{h,1}^r\times \cB\gm$, where $\Mtilde_{h,1}^r\hookrightarrow \Ctilde_{h,1}^r$ is the universal section. If we denote by $q$ the proper transform of the closed substack
	$$\cM_{h,1}^r \times [\AA^1/\gm] \into \Ctilde_{h,1}^r \times [\AA^1/\gm]$$
	we get that $(\cD_{h,1}\rightarrow \Mtilde_{h,1}^r \times [\AA^1/\gm],q)$ define a morphism 
	$$ \varphi:\Mtilde_{h,1}^r \times [\AA^1/\gm] \longrightarrow \Mtilde_{h,[l]}^r$$
	which by construction is birational, as it is an isomorphism restricted to $\Mtilde_{h,1}^r$. Furthermore, we have that it is an isomorphism on geometric points, i.e. 
	$$ \varphi(k): (\Mtilde_{h,1}^r \times [\AA^1/\gm])(k) \longrightarrow (\Mtilde_{h,[l]}^r)(k) $$
	is an equivalence of groupoids for any $k$ algebraically closed field over $\kappa$. See Proposition 2.5 of \cite{DiLorPerVis}. This implies that $\varphi$ is representable by algebraic spaces and quasi-finite.
	
	Suppose that $\varphi$ is separated.  Because $\Mtilde_{h,[l]}^r$ is smooth, we can use the Zariski Main Theorem for algebraic spaces to prove that $\varphi$ is an isomorphism.
	To check separatedness, we can use the valuative criterion. Suppose we are given a commutative diagram 
	$$
	\begin{tikzcd}
		\spec K \arrow[d] \arrow[r]                                & \spec R \arrow[d, "y"] \arrow[ld, "x_1"', shift right] \arrow[ld, "x_2", shift left] \\
		{\Mtilde_{h,1}^r \times [\AA^1/\gm]} \arrow[r, "\varphi"'] & {\Mtilde_{h,[l]}^r}                                                            
	\end{tikzcd}
	$$
	where $R$ is a DVR and $K$ is its fraction field. We need to prove that $x_1\simeq x_2$ as objects of $\Mtilde_{h,1}^r\times [\AA^1/\gm]$. First of all, one can prove easily that $$\varphi\vert_{\Mtilde_{h,1}^r\times \cB\gm}: \Mtilde_{h,1}^r\times \cB\gm \longrightarrow \Mtilde_{h,[l]}^r$$ 
	is a closed immersion. Because $\varphi\vert_{\Mtilde_{h,1}^r}$ is an isomorphism, it is enough to prove the statement for the full subcategory $\cM_R$ of $\Mtilde_{h,1}^r\times [\AA^1/\gm](\spec R)$ of morphisms $x:\spec R \rightarrow \Mtilde_{h,1}^r\times [\AA^1/\gm]$ such that $\spec K$ factors through $\Mtilde_{h,1}^r\into \Mtilde_{h,1}^r\times [\AA^1/\gm]$ and $x$ lands in $\Mtilde_{h,1}\times \cB\gm$ when restricted to the residue field of $R$. Let $\cM_R'$ be the image $\varphi(R)(\cM_R)$ in $\Mtilde_{h,[l]}^r(R)$. We can define a functor
	$$ \phi_R :\cM_R \longrightarrow (\Mtilde_{h,1}^r \times [\AA^1/\gm])(\spec R)$$
	such that $\phi_R(\varphi(R)(x))\simeq x$ for every $x \in \cM_R$. We construct the morphism $\phi_R$ and leave it to the reader to prove that it is a left inverse of $\varphi(R)$. To do so, we first define the morphism 
	$$ \cM_R' \longrightarrow [\AA^1/\gm](\spec R)$$ 
	in the following way. An object of $\cM_R$ is a  $1$-pointed $A_r$-prestable curve $(\widetilde{C}_R,q)$ over $R$ such that the generic fiber is a $1$-pointed $A_r$-stable curve of genus $h$ while the special fiber has a rational tail $\Gamma$ which contains the section and intersects the rest of the curve in a separating node $n$. Because \'etale locally the node $n$ has equation $xy=s$ where $s \in R$, we have a morphism $s:\spec R \rightarrow \AA^1$ which is well-defined up to an invertible element of $R$. We have defined an object in $[\AA^1/\gm](R)$.
	
	Finally, we define a morphism $\Mtilde_{h,[l]}^r \rightarrow \Mtilde_{h,1}^r$, which in particular gives us a functor $\cM_R' \rightarrow \Mtilde_{h,1}^r(\spec R)$ and it is straightforward to prove that factors through $\cM_R$. Consider the universal curve $\Ctilde_{h,[l]}^r$ over $\Mtilde_{h,[l]}^r$ with the section $p$.  Then we can consider the morphism induced by the complete linear system of $\omega_{\Ctilde_{h,[l]}^r|\Mtilde_{h,[l]}^r}(p)^{\otimes 3}$ and we denote by $\Ctilde'$ its stacky image. This morphism contracts the tails of genus $0$. One can prove using the results in \cite{Knu} that $\Ctilde'\rightarrow \Mtilde_{h,[l]}^r$ defines a morphism of stacks
	$$ p: \Mtilde_{h,[l]}^r \rightarrow \Mtilde_{h,1}^r$$
	such that the composition 
	$$p \circ \varphi: \Mtilde_{h,1}^r\times [\AA^1/\gm] \longrightarrow \Mtilde_{h,1}^r$$ 
	is the natural projection. 
	
\end{proof}

\begin{remark}
	In the proof above, one can prove that the morphism $\phi_R$ is also a right-inverse of $\varphi(R)$ restricted to $\cM_R$. This implies that $\varphi$ is proper and thus there is no need to use the Zariski Main theorem.
\end{remark}

Finally we are ready to describe $\cA_n$. In Appendix A, we define the stack $\cF_n^c$ which parametrizes pointed finite flat curvilinear algebras of length $n$ and $\cE_{m,d}^c$ which parametrizes finite flat curvilinear extensions of degree $d$ of pointed finite flat curvilinear algebras of degree $m$. We prove that the natural morphism 
$$ \cE_{m,d}^c \longrightarrow \cF_{md}^c$$
defined by the association $(A\into B) \mapsto B$ is an affine bundle. See \Cref{lem:descr-affine-bundle}. 

Let 
$$ \cE_{m,2}^c\longrightarrow \cF_{2m}^c$$
as above and consider the morphism of stacks
$$ \Mtilde_{g-m,[2m]}^r\longrightarrow \cF_{2m}^c$$ 
defined by the association $(\widetilde{C},q) \mapsto \cO_{\widetilde{C}}/I_q^{2m}$ where $I_q$ is the ideal defined by the section $q$. 

We denote by $\cA'_{2m}$ the fiber product $\cE_{m,2}^c \times_{\cF_{2m}^c} \Mtilde_{g-m,[2m]}^r$. By definition, an object of $\cA_{2m}'$ over $S$ is of the form $(\widetilde{C},q, A \subset \cO_{\widetilde{C}}/I_q^{2m})$ where $A \subset \cO_{\widetilde{C}}/I_q^{2m}$ is a finite flat extension of algebras of degree $2$. Given two objects $(\pi_1:\widetilde{C}_1/S,q_1, A_1 \subset \cO_{\widetilde{C}_1}/I_{q_1}^2m)$ and $(\pi_2:\widetilde{C}_2/S,q_2, A_2 \subset \cO_{\widetilde{C}}/I_{q_2}^2m)$, a morphism over $S$ between them is a pair $(f,\alpha)$ where $f:(\widetilde{C}_1,q_1)\rightarrow (\widetilde{C}_2,q_2)$  is a morphism in $\Mtilde_{g-m,[2m]}^r(S)$ while $\alpha:A_2 \rightarrow A_1$ is an isomorphism of finite flat algebras over $S$ such that the diagram 
$$
\begin{tikzcd}
A_2 \arrow[r, "\alpha"] \arrow[d, hook]            & A_1 \arrow[d, hook]                        \\
\pi_{2,*}(\cO_{\widetilde{C}_2}/I_{q_2}^{2m}) \arrow[r] & \pi_{1,*}(\cO_{\widetilde{C}_1}/I_{q_1}^{2m})
\end{tikzcd}
$$
is commutative. 

We want to construct a morphism from $\cA'_{2m}$ to $\cA_{2m}$. Let $S$ be a scheme and $(\widetilde{C},q, A \subset \cO_{\widetilde{C}}/I_q^{2m})$ be an object of $\cA'_{2m}(S)$. We consider the diagram 
$$
\begin{tikzcd}
\spec_{\cO_S}(\cO_{\widetilde{C}}/I_q^{2m}) \arrow[d, "2:1"] \arrow[r, hook] & \widetilde{C} \arrow[dd, bend left] \\
\spec_{\cO_S}(A) \arrow[rd]                                                  &                                     \\
& S                                  
\end{tikzcd}
$$
and complete it with the pushout (see \Cref{lem:pushout}), which has a morphism over $S$:
$$
\begin{tikzcd}
\spec_{\cO_S}(\cO_{\widetilde{C}}/I_q^{2m}) \arrow[d, "2:1"] \arrow[r, hook] & \widetilde{C} \arrow[dd, bend left] \arrow[d, dotted] \\
\spec_{\cO_S}(A) \arrow[rd] \arrow[r, dotted]                                & C \arrow[d, dotted]                                   \\
& S.                                                    
\end{tikzcd}
$$
In this case, $\widetilde{C}$ and $C$ share the same topological space, whereas the structural sheaf $\cO_C$ of the pushout is the fiber product 
$$
\begin{tikzcd}
\cO_C:=\cO_{\widetilde{C}}\times_{\cO_{\widetilde{C}}/I_q^{2m}} A \arrow[d, two heads] \arrow[r, hook] & \cO_{\widetilde{C}} \arrow[d, two heads] \\
A \arrow[r, hook]                                                                                 & \cO_{\widetilde{C}}/I_q^{2m};         
\end{tikzcd}
$$
we define $I_p:=I_q\vert_{\cO_C}$ which induces a section $p:S \rightarrow C$ of $C\rightarrow S$.

\begin{lemma}
	In the situation above, $C/S$ is an $A_r$-stable curve of genus $g$ and $p_s$ is a $A_{2m}$-singularity for every geometric point $s \in S$. Furthermore, the formation of the pushout commutes with base change over $S$.
\end{lemma}

\begin{proof}
	The fact that $C/S$ is flat, proper and finitely presented is a consequence of \Cref{lem:pushout}. The same is true for the compatibility with base change over $S$. We only need to check that $C_s$ is an $A_r$-stable curve for every geometric point $s \in S$. Therefore we can assume $S=\spec k$ with $k$ an algebraically closed field over $\kappa$. Connectedness is trivial as the topological space is the same. We need to prove that $p$ is an $A_{2m}$-singularity. Consider the cartesian diagram of local rings
	$$
	\begin{tikzcd}
	{\cO_{C,p}:=\cO_{\widetilde{C},q}\times_{\cO_{\widetilde{C},p}/m_q^{2m}} A} \arrow[d, two heads] \arrow[r, hook] & {\cO_{\widetilde{C},q}} \arrow[d, two heads] \\
	A \arrow[r, hook]                                                                                     & {\cO_{\widetilde{C},q}/m_q^{2m}}            
	\end{tikzcd}
	$$
	and pass to the completion with respect to $m_p$, the maximal ideal of $\cO_{C,p}$. Because the extensions are finite, we get the following cartesian diagram of rings
	$$
	\begin{tikzcd}
	B \arrow[d, two heads] \arrow[r, hook]   & {k[[t]]} \arrow[d, two heads] \\
	{k[[t]]/(t^m)} \arrow[r, "\phi_2", hook] & {k[[t]]/(t^{2m})};            
	\end{tikzcd}
	$$
	using the description of $\phi_2$ as in \Cref{lem:triv-ext}, it is easy to see that it is defined by the association $t \mapsto t^2$ up to an isomorphism of $k[[t]]/(t^{2m})$. This concludes the proof.
\end{proof}

Therefore we have constructed a morphism of algebraic stacks
$$F:\cA'_{2m} \longrightarrow \cA_{2m}$$ 
defined on objects by the association
$$(\widetilde{C},q,A\subset \cO_{\widetilde{C}}) \mapsto (\widetilde{C}\bigsqcup_{\spec (\cO_{\widetilde{C}}/I_q^{2m})} \spec A, p) $$
and on morphisms in the natural way using the universal property of the pushout. 

To construct the inverse, we use the relative $A_n$-desingularization which is compatible with base change. \Cref{lem:conductor} implies that we can define a functor $G:\cA_{2m}\rightarrow \cA_{2m}'$ on objects 
$$ (C/S,p) \mapsto (\widetilde{C},q, \cO_C/J_b \subset \cO_{\widetilde{C}}/I_q^{2m}) $$
where $b:\widetilde{C}\rightarrow C$ is the relative $A_{2m}$-desingularization, $J_b$ is the conductor ideal relative to $b$ and $q$ is the smooth section of $\widetilde{C}\rightarrow S$ defined as the vanishing locus of the $0$-th Fitting ideal of $\Omega_{b^{-1}(p)|S}$. It is defined on morphisms in the obvious way. 

\begin{proposition}\label{prop:descr-an-pari}
	The two morphisms $F$ and $G$ are quasi-inverse of each other.
\end{proposition}

\begin{proof}
	The statement is a conseguence of \Cref{prop:pushout-blowup} and \Cref{prop:blowup-pushout}.
\end{proof}

\begin{corollary}\label{cor:descr-an-pari}
	$\cA_{2m}$ is an affine bundle of dimension $(m-1)$ over the stack $\Mtilde^r_{g-m,[2m]} $ for $m\geq 1$.
\end{corollary}

\begin{proof}
	Because $\cA_{2m}'$ is constructed as the fiber product $\cE_{m,2}^c \times_{\cF_{2m}^c} \Mtilde_{g-m,[2m]}^r$, the statement follows from \Cref{lem:descr-affine-bundle}.
\end{proof}

\subsection{Description of $A_n$ for the odd case}

Let $n:=2m-1$ and $(C,p)$ be a $1$-dimensional reduced $1$-pointed scheme of finite type over an algebraically closed field and $p$ is an $A_{2m-1}$-singularity. Consider now the partial normalization in $p$ which gives us a finite birational morphism 
$$b:\widetilde{C} \longrightarrow C$$ 
which is not an homeomorphism, as we know that $f^{-1}(p)$ is a reduced divisor of $\widetilde{C}$ of length $2$. We can use \Cref{lem:cond-diag} to prove that the extension $\cO_C \hookrightarrow f_*\cO_{\widetilde{C}}$ can be constructed pulling back a subalgebra of $f_*\cO_{\widetilde{C}}/J_b$ through the quotient $f_*\cO_{\widetilde{C}}\rightarrow f_*\cO_{\widetilde{C}}/J_b$, as in the even case. We can describe the subalgebra in the following way.

Consider the divisor $b^{-1}(p)$ which is the disjoint union of two closed points, namely $q_1$ and $q_2$. Then the composition
$$
\begin{tikzcd}
\cO_C/J_b \arrow[r] & f_*\cO_{\widetilde{C}}/J_b=\cO_{mq_1}\oplus \cO_{mq_2} \arrow[r] & \cO_{mq_i}
\end{tikzcd}
$$
is an isomorphism for $i=1,2$, where $\cO_{mq_i}$ is the structure sheaf of the support of the Cartier divisor $mq_i$ for $i=1,2$ and the right map is just the projection. Therefore the subalgebra $\cO_C/J_b$ of $f_*\cO_{\widetilde{C}}/J_b$ is determined by an isomorphism between $\cO_{mq_1}$ and $\cO_{mq_2}$. Recall that \Cref{rem:genus-count} implies that we can have two different situation: either $\widetilde{C}$ is connected and its genus is $g-m$  or it has two connected components of total genus $g-m+1$ and the two points lie in different components. 
\begin{definition}
	Let $0\leq i\leq (g-m+1)/2$. We define $\cA_{2m-1}^{i}$ to be the substack of $\cA_{2m-1}$ parametrizing $1$-pointed curves $(C/S,p)$ such that the geometric fibers over $S$ of the relative $A_{2m-1}$-desingularization are the disjoint union of two curves of genus $i$ and $g-m-i+1$.
	
	Furthermore, we define $\cA_{2m-1}^{\rm ns}$ the substack of $\cA_{2m-1}$ parametrizing curves $(C/S,p)$ such that the geometric fibers of the relative $A_{2m-1}$-desingularization are connected of genus $g-m$. 
\end{definition}

\begin{proposition}
	The algebraic stack $\cA_{2m-1}$ is the disjoint union of $\cA_{2m-1}^{\rm ns}$ and $\cA_{2m-1}^i$ for every $0 \leq i\leq  (g-m+1)/2$.
\end{proposition}

\begin{proof}
	Let $(C/S,p)$ be an object of $\cA_{2m-1}$ and consider the relative \\ $A_n$-desingularization $$b:(\widetilde{C},q)\rightarrow (C,p)$$
	over $S$. Because $\widetilde{C}\rightarrow S$ is flat, proper, finitely presented and the fibers over $S$ are geometrically reduced, we have that the number of connected components of the geometric fibers of $\widetilde{C}$ over $S$ is locally constant. See Proposition 15.5.7 of \cite{EGA}. Therefore we have that $\cA_{2m-1}^{\rm ns}$ is open and closed inside $\cA_{2m-1}$. Furthermore, we have that the objects of the complement $\cA_{2m-1}^{\rm s}$ of $\cA_{2m-1}^{\rm ns}$ are pairs $(C/S,p)$ such that the geometric fibers of $\widetilde{C}$ over $S$ have two connected components. Suppose that $S=\spec R$ with $R$ a strictly henselian ring over $\kappa$. We know that $b^{-1}(p)$ is \'etale of degree $2$ over $R$, thus it is the disjoint union of two copies of $R$. We denote by $q$ one of the two sections of $\widetilde{C}\rightarrow \spec R$ and we denote by $C_s^0$ the connected component of the fiber $\widetilde{C}_s$ which contains $q_s$ for every point $s \in \spec R$. Proposition 15.6.5 and Proposition 15.6.8 in \cite{EGA} imply that the set-theoretic union 
	$$ C_0:=\bigcup_{s \in S} C_s^0$$
	is a closed and open subscheme of $\widetilde{C}$, therefore the morphism $C_0\rightarrow S$ is still proper, flat and finitely presented. In particular, the arithmetic genus of the fibers is locally constant. It is easy to see that in particular $\cA_{2m-1}^{\rm s}$ is the disjoint union of $\cA_{2m-1}^{\rm i}$ for $0\leq i\leq (g-m+1)/2$.
\end{proof}

We start by describing $\cA_{2m-1}^{\rm ns}$. Let $\Mtilde_{h,2[l]}^r$ be a fibered category in groupoid over the category of schemes whose objects are of the form $(\widetilde{C}/S,q_1,q_2)$ where $\widetilde{C}\rightarrow S$ is a flat, proper, finitely presented morphism of scheme, $q_1$ and $q_2$ are smooth sections, $\widetilde{C}_s$ is an $A_r$-prestable curve of genus $h$ for every geometric point $s \in S$ and $\omega_{\widetilde{C}}(l(q_1+q_2))$ is relatively ample over $S$. The morphisms are defined in the obvious way. 

\begin{proposition}
	We have an isomorphism  of fibered category
	$$\Mtilde_{h,2[l]}^r\simeq  \Mtilde_{h,2}^r \times [\AA^1/\gm] \times [\AA^1/\gm]$$
	for $l\geq 2$ and $h\geq 1$. Furthermore, we have 
	$$ \Mtilde_{0,2[l]}\simeq \cB\gm \times [\AA^1/\gm]$$
	for $l\geq 3$.
\end{proposition}

\begin{proof}
	The proof is an adaptation of the one of \Cref{prop:desc-str}.
\end{proof}

We want to construct an algebraic stack with a morphism over $\Mtilde_{h,2[m]}^r$ whose fibers parametrize the isomorphisms between the two finite flat $S$-algebras $\cO_{mq_1}$ and $\cO_{mq_2}$.

\begin{remark}
	Recall that we have a smooth stack $\cE_{m,d}^c$ (see Appendix A for a detailed discussion) which parametrizes finite flat extensions $A\into B$ of degree $d$ with $B$ curvilinear of length $m$. If $d=1$, the stack $\cE_{m,1}^c$ parametrizes isomorphisms of finite flat algebras of length $m$. We also have a map 
	$$ \cE_{m,1}^c \longrightarrow \cF_m^c \times \cF_m^c$$ 
	defined on objects by the association $(A\into B) \mapsto (A,B)$, which is a $\gm \times \ga^{m-2}$-torsor. See Appendix A for a more detailed discussion.
\end{remark}

Consider the morphism
$$\Mtilde_{g-m,2[m]}^r \longrightarrow \cF_{m}^c \times \cF_m^c$$
defined by the association 
$$(\widetilde{C},q_1,q_2)\mapsto (\cO_{\widetilde{C}}/I_{q_1}^m,\cO_{\widetilde{C}}/I_{q_2}^m);$$
and let $I_{2m-1}^{\rm ns}$ be the fiber product $\Mtilde_{h,2[m]}^r \times_{(\cF_{m}^c\times\cF_{m}^c)} \cE_{m,1}^c$. It parametrizes objects $(C,q_1,q_2,\phi)$ such that $(C,q_1,q_2) \in \Mtilde_{h,2[m]}^r$ and an isomorphism $\phi$ between $\cO_{mq_1}$ and $\cO_{mq_2}$ as $\cO_S$-algebras which commutes with the sections.

We can construct a morphism 
$$ I_{2m-1}^{\rm ns} \longrightarrow \cA_{2m-1}^{\rm ns}$$ 
in the following way: let $(\widetilde{C},q_1,q_2,\phi) \in I_{2m-1}^{\rm ns}(S)$, then we have the diagram
$$
\begin{tikzcd}
\spec_S(\cO_{mq_1})\bigsqcup \spec_S(\cO_{mq_2}) \arrow[d, "{(\id,\phi)}"] \arrow[r, hook] & \widetilde{C} \\
\spec_S(\cO_{mq_1})                                                                     &              
\end{tikzcd}
$$
where the morphism $(\id,\phi)$ is \'etale of degree $2$. We denote by $C$ the pushout of the diagram and $p$ the image of $q$. We send $(\widetilde{C},q_1,q_2,\phi)$ to $(C,q)$. Notice that because both $q_1$ and $q_2$ are smooth sections, we have that $\cO_{mq_i}$ is the scheme-theoretic support of a Cartier divisor of $\widetilde{C}$, therefore it is flat for $i=1,2$. \Cref{lem:pushout} assures us that this construction is functorial and commutes with base change. 

\begin{proposition}\label{prop:descr-an-odd-ns}
	The pushout functor
	$$F^{\rm ns}: I_{2m-1}^{\rm ns} \longrightarrow \cA_{2m-1}^{\rm ns}$$ 
	is representable finite \'etale of degree $2$.
\end{proposition}

\begin{proof}
	
	It is a direct conseguence of both \Cref{prop:pushout-blowup} and \Cref{prop:blowup-pushout}. In fact  the two propositions assure us that the object $(C,p)$ is uniquely determined by the relative $A_{2m-1}$-desingularization $b:\widetilde{C}\rightarrow C$, the fiber $b^{-1}(p)$ and an automorphism of the $m$-thickening of $b^{-1}(p)$ which restricted to the geometric fibers over $S$ acts exchanging the two points of $b^{-1}(p)$. Because $b^{-1}(p)$ is finite \'etale of degree $2$, it is clear that $F^{\rm ns}$ is a finite \'etale morphism of degree $2$.
\end{proof}
Let $0\leq i\leq (g-m+1)/2$. In the same way, we can define morphisms 
$$  \Mtilde_{i,[m]}\times \Mtilde_{g-i-m+1,[m]} \longrightarrow \cF_{m}^c \times \cF_m^c$$ 
defined by the association $$\Big((\widetilde{C}_1,q_1),(\widetilde{C}_{2},q_{2})\Big) \mapsto \Big((\cO_{\widetilde{C_1}}/I_{q_1}^m,\cO_{\widetilde{C}_{2}}/I_{q_{2}}^m)\Big)$$
and we denote by $I_{2m-1}^{i}$ the fiber product $$(\Mtilde_{i,[m]}\times \Mtilde_{g-m-i+1,[m]})\times_{(\cF_{m}^c\times\cF_{m}^c)} \cE_{m,1}^c.$$

Similarly to the previous case, we can construct a functor 
$$ F^i: I_{2m-1}^i \longrightarrow \cA_{2m-i}^i$$ 
using the pushout construction. Again, we have the following result.

\begin{proposition}\label{prop:descr-an-odd-i}
	The functor $F^i$ is an isomorphism for $i \neq (g-m+1)/2$ whereas is finite \'etale of degree $2$ for $i = (g-m+1)/2$.
\end{proposition}

\begin{proof}
	The proof is exactly the same as \Cref{prop:descr-an-odd-ns}.
\end{proof}

\section{The relations from the $A_n$-strata}\label{sec:3-2}

In this section, we are going to describe the image of the pushforward of the closed immersion $\widetilde{\cA}_{\geq 2} \into \Mtilde_3$ in $\ch(\Mtilde_3)$. First of all, we explain why we can reduce to study the stacks $\cA_n$.  

We have the following result.

\begin{proposition}
	The functor forgetting the section gives us a natural morphism 
	$$ \cA_{\geq n} \rightarrow \widetilde{\cA}_{\geq n}$$
	which is finite birational and it is surjective at the level of Chow groups.
\end{proposition} 

\begin{proof}
	It follows from the fact that every $A_r$-stable genus $3$ curve has at most three singularity of type $A_n$ for $n\geq 2$. 
\end{proof}

Consider now the proper morphisms 
$$ \rho_{\geq n}:\cA_{\geq n} \longrightarrow \Mtilde_g^r$$ 
and their restrictions to $\cA_n$
$$ \rho_n :\cA_n \longrightarrow \Mtilde_g^r\setminus \widetilde{\cA}_{\geq n+1}$$
which is still proper; let $\{f_i\}_{i \in I_n}$ be a set of elements of $\ch(\cA_n)$ indexed by some set $I_n$ such that $\im{\rho_{n,*}}$ is generated by the set $\{\rho_{n,*}(f_i)\}_{i \in I_n}$. We choose a lifting $\tilde{f}_i$ of every $f_i$ to the Chow group of $\cA_{\geq n}$ for every $n=2,\dots,r$ and every $i \in I_n$. We have the following result.

\begin{lemma}\label{lem:strata}
	In the setting above, we have that $\im{\rho_{\geq 2,*}}$ is generated by $\{\rho_{\geq n,*}(\tilde{f}_i)\}_{\forall n, \forall i \in I_n}$
\end{lemma}

\begin{proof}
	This is a direct conseguence of Lemma 3.3 of \cite{DiLorFulVis}.
\end{proof}

The previous lemma implies that we need to focus on finding the generators of the relations coming from the strata $\cA_n$ for $n=2,\dots,r$. Therefore in the remaining part of the section we study the morphism $\rho_n$ and we prove that $\rho_n^*$ is surjective at the level of Chow rings for every $n\geq 3$. The same is not true for $n=2$ but we describe geometrically the generators of the image of $\rho_{2,*}$. 

\subsection*{Generators for the image of $\rho_{n,*}$ if $n$ is even}

Recall that we are interested in the case $r=7$ and $g=3$ and therefore the characteristic of $\kappa$ is greater than $7$. We are going to prove that the morphism 
$$ \rho_n^*: \ch(\Mtilde_3^7) \longrightarrow \ch(\cA_n)$$ 
is surjective for $n=4,6$.

\begin{proposition}\label{prop:rho-6-surj}
	The morphism $\rho_6^*$ is surjective.
\end{proposition}

\begin{proof}
	We start by considering the isomorphism proved in \Cref{cor:descr-an-pari}. Because in this situation $g=3$ and $m=3$ we have that 
	$$\cA_6 \simeq [V/\gm \ltimes \ga]$$
	where $V$ is the vector bundle considered in \Cref{lem:descr-affine-bundle}. As a matter of fact, we proved that the following commutative diagram of stacks
	$$
	\begin{tikzcd}
		\cA_6 \arrow[d] \arrow[r]                         & {\cE_{3,2}^c\simeq [V/G_{6}]} \arrow[d] \\
		{\Mtilde_{0,[3]}\simeq \cB(\gm \ltimes \ga)} \arrow[r, "\cB f"] & \cF_{6}^c:=\cB G_6                     
	\end{tikzcd}
	$$
	is cartesian. The morphism $\cB f$ can be described as the morphism  of classifying stacks induced by the morphism of groups schemes
	$$ f:\gm \ltimes \ga\simeq \aut(\PP^1,\infty) \longrightarrow G_6$$
	defined by the association $\phi \mapsto \phi\otimes_{\cO_{\PP^1}} \cO_{\PP^1}/m_{\infty}^6$, where $m_{\infty}$ is the maximal ideal of the point $\infty$. For the definition of $G_6$, see \Cref{cor:descr-finite-alg}. A simple computation, using the explicit formula of the action of $G_6$ on $V$, shows that 
	$$ \cA_6\simeq [V/\gm \ltimes \ga] \simeq [\AA^1/\gm]$$
	where the action of $\gm$ on $\AA^1$ has weigth $-3$. More explicitly, if we identify $\cO_{\PP^1}/m_{\infty}^6$ with the algebra $\kappa[t]/(t^6)$, an element $\lambda \in \AA^1$ is equivalent to the inclusion of algebras $$\kappa[t]/(t^3) \into \kappa[t]/(t^6)$$ defined by the association $t \mapsto t^2+\lambda t^5$.  
	
	We want to understand the pullback of the hyperelliptic locus, i.e. $\rho_6^*(H)$. It is clear that the locus $\rho_6^{-1}(\Htilde_3)$ is the locus in $[\AA^1/\gm]$ such that the involution $t \mapsto -t$ fixes the inclusion $t \mapsto t^2+\lambda t^5$. This implies $\lambda=0$ and therefore $\rho_6^*(\lambda)=-3s$ where $s$ is the generator of $\ch([\AA^1/\gm])$. 
\end{proof}

\begin{proposition}
	The morphism $\rho_4^*$ is surjective.
\end{proposition} 

\begin{proof}
	Again, \Cref{cor:descr-an-pari} shows that $\cA_{4}$ is an affine bundle of $\Mtilde_{1,1}\times [\AA^1/\gm]$ and therefore 
	$$ \ch(\cA_{4})\simeq \ZZ[1/6,t,s]$$ 
	where $s$ is the generator of the Picard group of $[\AA^1/\gm]$ and $t$ is the $\psi$-class of $\Mtilde_{1,1}$ (which is a generator of the Chow ring of $\Mtilde_{1,1}$). Exactly as it happens for the pinching morphism described in \cite{DiLorPerVis}, we have $\rho_4^*(\delta_1)=s$ (see Lemma 5.9 of \cite{DiLorPerVis}). We need to compute now $\rho_4^*(H)$. Consider now the open immersion 
	$$ \cA_4\vert_{\Mtilde_{1,1}} \into \cA_4$$
	induced by the open immersion $\Mtilde_{1,1} \into \Mtilde_{1,1} \times [\AA^1/\gm]$. We have that 
	$$ \ch(\cA_4\vert_{\Mtilde_{1,1}})=\ZZ[1/6,t,s]/(s)$$ 
	therefore it is enough to prove that $\rho_4^*(H)$ restricted to this open is of the form $-2t$ to have the surjectivity of $\rho_4^*$. 
	
	We know that $\Mtilde_{1,1}\simeq [\AA^2/\gm]$, therefore it is enough to restrict to $\cA_4\vert_{\cB\gm}$ because the pullback of the closed immersion $\cB\gm \into \Mtilde_{1,1}$ is an isomorphism of Chow rings. Similarly to the proof of \Cref{prop:rho-6-surj}, a simple computation shows that $\cA_4\vert_{\cB\gm}$ is isomorphic to $[\AA^1/\gm]$ where $\gm$ acts with weight $-2$. An element in $\lambda \in \AA^1$ is equivalent to the inclusion of algebras $\kappa[t]/(t^2) \into \kappa[t]/(t^4)$ defined by the association $t\mapsto t^2+\lambda t^3$.
	
	The locus $\Htilde_3$ coincides with the locus in $[\AA^1/\gm]$ described by the equation $\lambda=0$. Therefore $\rho_4^*(H)=-2s$ and we are done. 
\end{proof}

Before going to study the morphism $\rho_2$, we need to understand its source. We have that 
$$\cA_{2}\simeq \cA_{2}'\simeq \Mtilde_{2,1} \times [\AA^1/\gm].$$
Recall that $\Mtilde_{2,1}$ is an open substack of $\Ctilde_2$ thanks to \Cref{prop:contrac}. Therefore, the Chow ring of $\Mtilde_{2,1}$ is a quotient of the one of $\Ctilde_2$. 

\begin{lemma}\label{lem:chow-ring-C2}
	The Chow ring of $\Ctilde_2$ is a quotient of the polynomial ring generated by
	\begin{itemize}
		\item the $\lambda$-classes $\lambda_1$ and $\lambda_2$ of degree $1$ and $2$ respectively,
		\item the $\psi$-class $\psi_1$,
		\item two classes $\theta_1$ and $\theta_2$ of degree $1$ and $2$ respectively;
	\end{itemize}
	furthermore, the ideal of relations is generated by 
	\begin{itemize}
		\item  $\lambda_2-\theta_2-\psi_1(\lambda_1-\psi_1)$,
		\item  $\theta_1(\lambda_1+\theta_1)$,
		\item  $\theta_2\psi_1$,
		\item  $\theta_2(\lambda_1+\theta_1-\psi_1)$,
		\item  an homogeneous polynomial of degree $7$.
	\end{itemize}
\end{lemma}

\begin{proof}
	We do not describe all the computation in details. The idea is to use the stratification introduced in Section 4 of \cite{DiLorPerVis},i.e. 
	$$\ThTilde_2 \subset \ThTilde_1 \subset \Ctilde_2$$ 
	where $\ThTilde_1$ is the pullback of $\Detilde_1$ through to morphism $\Ctilde_2 \rightarrow \Mtilde_2$ and $\ThTilde_2$ is the closed substack of $\Ctilde_2$ parametrizing pairs $(C,p)$ such that $p$ is a separating node. We denote by $\theta_1$ and $\theta_2$ the fundamental classes of $\ThTilde_1$ and $\ThTilde_2$. Notice that the only difference with our situation is in the open stratum $\Ctilde_2 \setminus \ThTilde_1$. In fact, the authors of \cite{DiLorPerVis} proved in Proposition 4.1 that  
	$$\Ctilde_2^2\setminus \ThTilde_1 \simeq [U/B_2]$$
	where $U$ is an open inside a $B_2$-representation $\widetilde{\AA}(6)$.
	The same proof generalizes in the case $r=7$ (see \Cref{cor:mtilde_21}) and it gives us that 
	$$ \Ctilde_2^7\setminus \ThTilde_1 \simeq [\widetilde{\AA}(6)\setminus 0/B_2]$$
	and therefore the zero section in $\widetilde{\AA}(6)$ gives us a relation of degree $7$.  We also have the following isomorphisms:
	\begin{itemize}
		\item $\ThTilde_1 \setminus \ThTilde_2 \simeq (\Ctilde_{1,1}\setminus \Mtilde_{1,1})\times \Mtilde_{1,1}$,
		\item $\ThTilde_2 \simeq \Detilde_1$;
	\end{itemize}
 thus we have the following descriptions of the Chow rings of the strata:
	\begin{itemize}
		\item $\ch(\Ctilde_2\setminus \ThTilde_1) \simeq \ZZ[1/6,t_0,t_1]/(f_7)$,
		\item $\ch(\ThTilde_1 \setminus \ThTilde_2
		) \simeq \ZZ[1/6,t,s]$,
		\item $\ch(\ThTilde_2) \simeq \ZZ[1/6,\lambda_1,\lambda_2]$
	\end{itemize}
where $f_7$ is an homogeneous polynomial of degree $7$. Finally, one can prove the following identities:
\begin{itemize}
	\item $\lambda_1\vert_{\Ctilde_2\setminus \ThTilde_1} = -t_0-t_1$, $\lambda_1\vert_{\ThTilde_1\setminus \ThTilde_2}=-t-s$;
	\item $\lambda_2\vert_{\Ctilde_2\setminus \ThTilde_1}=t_0t_1$, $\lambda_2\vert_{\ThTilde_1\setminus \ThTilde_2}=st$;
	\item $\psi_1\vert_{\Ctilde_2\setminus \ThTilde_1}=t_1$, $\psi_1\vert_{\ThTilde_1\setminus \ThTilde_2}=t$, $\psi_1\vert_{\ThTilde_2}=0$;
	\item $\theta_1\vert_{\ThTilde_1\setminus \ThTilde_2}=t+s$, $\theta_1\vert_{\ThTilde_2}=-\lambda_1$;
	\item $\theta_2\vert_{\ThTilde_2}=\lambda_2$.
\end{itemize}
The result follows from applying the gluing lemma.
\end{proof}

\begin{remark}
	Clearly, $\rho_2^*$ cannot be surjective at the level of Chow rings, as it not true even at the level of Picard groups. In fact, the Picard group of $\Mtilde_3$ is an abelian free group of rank $3$ while the Picard group of $\Mtilde_{2,1}\times [\AA^1/\gm]$ is an abelian free group of rank $4$. 
\end{remark}

We are ready for the proposition.
\begin{proposition}\label{prop:gener-rho-2}
	The image of the pushforward of
	$$\rho_2: \Mtilde_{2,1} \times [\AA^1/\gm] \simeq \cA_2 \longrightarrow \Mtilde_3\setminus \widetilde{\cA}_{\geq 3}$$
	is generated by the elements $\rho_{2,*}(1)$, $\rho_{2,*}(s)$ and $\rho_{2,*}(s\theta_1)$, where $s$ is the generator of the Chow ring of $[\AA^1/\gm]$ and $\theta_1$ is the fundamental class of the locus parametrizing curves with a separating node.
\end{proposition} 

\begin{proof}
	For this proof, we denote by $\lambda_1$ and $\lambda_2$ the Chern classes of the Hodge bundle of $\Mtilde_{2,1}$, whereas the $i$-th Chern class of the Hodge bundle of $\Mtilde_3$ is denoted by $c_i(\HH)$ for $i=1,2,3$.
	
	We need to describe the pullback of the generators of the Chow ring of $\Mtilde_3$ through $\rho_2$. By construction, it is easy to see that $\rho_2^*(\delta_1)=s+\theta_1$, $\rho_2^*(\delta_{1,1})=\theta_2+s\theta_1$ and $\rho_2^*(\delta_{1,1,1})=s\theta_2$. 
	
	Notice that $\rho_2^{-1}(\Htilde_3)$ is the fundamental class of the closed substack $\Mtilde_{2,\omega} \times [\AA^1/\gm]$, where $\Mtilde_{2,\omega}$ is the closed substack of $\Mtilde_{2,1}$ which parametrizes pairs $(C,p)$ such that $p$ is fixed by the (unique) involution of $C$. To compute its class, we need to use the stratification used in the proof of \Cref{lem:chow-ring-C2}.
	In the open stratum $\Mtilde_{2,1}\setminus \ThTilde_1$, \Cref{prop:relation-detilde-1} implies that the restriction of $[\Mtilde_{2,\omega}]$ is equal to $-\lambda_1-3 \psi_1$.
	In the stratum $\ThTilde_1\setminus \ThTilde_2$, we have that the restriction is of the form $-3\psi_1$. This implies that $\rho_2^*(H)=-\lambda_1-3\psi_1-\theta_1$.
	
	Finally, to compute the restriction of $c_i(\HH)$ for $i=1,2,3$, we can restrict to the closed substack $\Mtilde_{2,1}\times \cB\gm \into \Mtilde_{2,1} \times [\AA^1/\gm]$ as the pullback of the closed immersion is clearly an isomorphism because it is the zero section of a vector bundle. The explicit description of the isomorphism $\Mtilde_{2,1}\times [\AA^1/\gm] \simeq \cA_2$ (which was constructed in Section 2 of \cite{DiLorPerVis}) implies that the morphism $\rho_2\vert_{\Mtilde_{2,1}\times \cB\gm}$ maps an object $(\widetilde{C}/S,q)$ to the object $(C/S,p)$ in the following way: consider the projective bundle $\PP(N_{q}\oplus N_0)$ over $S$, where $N_q$ is the normal bundle of the section $q$ and $N_0$ is the pullback to $S$ of the $1$-dimensional representation of $\gm$ of weight $1$; we have two natural sections defined by the two subbundles $N_q$ and $N_0$ of $N_q\oplus N_0$, namely $\infty$ and $0$; the object $(C/S,p)$ is defined by gluing $\infty$ with $q$, pinching in $0$ and then setting $p:=0$. A computation identical to the one of Proposition 5.9 of \cite{DiLorPerVis} implies the following formulas:
	\begin{itemize}
		\item $\rho_2^*(c_1(\HH))=\lambda_1+\psi_1-s$,
		\item $\rho_2^*(c_2(\HH))=\lambda_2+\lambda_1(\psi_1-s)$,
		\item $\rho_2^*(c_3(\HH))=\lambda_2(\psi_1-s)$.
	\end{itemize}
	The description of the restrictions of the generators of $\ch(\Mtilde_3)$ gives us that the image of $\rho_{2,*}$ is the ideal generated by $\rho_{2,*}(s^i)$ for every $i$ non-negative integer. Moreover, we have that 
	$$\rho_2^*(\delta_{1,1,1})=s\theta_2=s(\rho_2^*(\delta_{1,1})-s(\rho_2^*(\delta_1)-s))$$
	which implies that $\rho_{2,*}(s^i)$ is in the ideal generated by $\rho_{2,*}(1)$, $\rho_{2,*}(s)$ and $\rho_2^*(s^2)$ for every $i\geq 3$. Finally, we have that 
	$$s\theta_1=s(\rho_2^*(\delta_1)-s)$$
	therefore we can use $s\theta_1$ as a generator with $\rho_{2,*}(1)$ and $\rho_{2,*}(s)$ instead of $\rho_{2,*}(s^2)$.
\end{proof}

\begin{remark}
	Notice that $\rho_{2,*}(s)$ is equal to the fundamental class of the image of the morphism $\Mtilde_{2,1} \times \cB\gm \rightarrow \Detilde_1 \into \Mtilde_{3}$.  We denote this closed substack $\Detilde_1^c$; it parametrizes curves $C$ obtained by gluing a genus $2$ curve with a genus $1$ cuspidal curve in a separating node.
	
	In the same way, $\rho_{2,*}(s\theta_1)$ is equal to the fundamental class of the image of the morphism $\ThTilde_1 \times \cB\gm \into \Detilde_{1,1} \into \Mtilde_3$. We denote this closed substack $\Detilde_{1,1}^c$; it parametrizes curves $C$ in $\Detilde_{1,1}$ such that one of the two elliptic tails is cuspidal.
\end{remark}

\subsection*{Generators for the image of $\rho_{n,*}$ if $n$ is odd}

Now we deal with the odd case. This is a bit more convoluted as we have several strata to deal with for every $n$. Let us recall the results we have proven.

First of all, $\cA_{2m-1}$ is the disjoint union of $\cA_{2m-1}^{\rm ns}$ and $\cA_{2m-1}^i$ for $0\leq i\leq (g-m+1)/2$. Because $g=3$, we have the following possibilities:
\begin{itemize}
	\item if $m=4$, we have only one component, namely $\cA_7^0$;
	\item if $m=3$, we have two components, namely $\cA_5^0$ and $\cA_5^{\rm ns}$;
	\item if $m=2$, we have three components, namely $\cA_5^0$, $\cA_5^1$ and $\cA_5^{\rm ns}$.
\end{itemize}
First of all, notice that $\cA_5^0$ is empty, due to the stability condition. Therefore we need to deal with $5$ components. 

We start with the case $m=4$.

\begin{proposition}\label{prop:rho-7-surj}
	The pullback of the morphism 
	$$\rho_7:\cA_7 \longrightarrow \Mtilde_3$$
	is surjective.
\end{proposition}

\begin{proof}
	The proof is similar to one of \Cref{prop:rho-6-surj}. First of all, we describe the Chow ring of $\cA_7^0$. We apply \Cref{prop:descr-an-odd-i} and \Cref{lem:chow-tor} and we get that 
	$$\ch(\cA_7^0) \simeq \ch(I_{7}^0)^{\rm inv}$$
 	where the invariants are taking with respect of the action of $C_2$ induced by the involution defined by the association $$\Big((C_1,p_1),(C_2,p_2),\phi\Big) \mapsto \Big((C_2,p_2),(C_1,p_1),\phi^{-1}\Big).$$
 	
 	By construction, we have that $I_7^0$ is the fiber product of the diagram
 	$$
 	\begin{tikzcd}
 		& {[E_{4,1}/G_4^{\times 2}]} \arrow[d] \\
 		\cB(\gm \ltimes \ga)^{\times 2} \arrow[r, "{\cB(f^{\times 2})}"] & \cB (G_4^{\times 2})                
 	\end{tikzcd}
 	$$
 	where the morphism $f$ is described in the proof of \Cref{prop:rho-6-surj}. A simple computation shows that 
 	$$ I_7^0 \simeq [\AA^1/\gm]$$
 	where $\AA^1$ is the $\gm$-representation with weight $2$. Furthermore, one can prove that $C_2$ acts trivially on $\gm$ and acts on $\AA^1$ by the rule $\lambda \mapsto -\lambda$. Therefore it is clear that
 	$$\ch(I_7^0) \simeq \ZZ[1/6,s]$$
 	where $s$ is the generator of the Chow ring of $\cB \gm$ and a simple computation shows the $\rho_7^{*}(H)=2s$.
\end{proof}

We now deal with the case $m=3$. We have to split it in two subcases, namely $\cA_5^0$ and $\cA_5^{\rm ns}$. We denote by $\rho_5^{0}$ and $\rho_5^{\rm ns}$ the restriction of $\rho_5$ to the two connected components $\cA_5^0$ and $\cA_5^{\rm ns}$ respectively.

\begin{proposition}\label{prop:rho-5-0-surj}
	The pullback of the morphism 
	$$\rho_5^{0}:\cA_5^0 \longrightarrow \Mtilde_3$$
	is surjective.
\end{proposition}

\begin{proof}
	In this case, \Cref{prop:descr-an-odd-i} tells us that $\cA_5^0$ is isomorphic to $I_5^0$. We have a commutative diagram
	$$
	\begin{tikzcd}
		& {[E_{3,1}/G_3\times G_3]} \arrow[d] \\
		{\cB(\gm \ltimes \ga)\times \Mtilde_{1,1}\times [\AA^1/\gm]} \arrow[r, "{(\cB f,g)}"] & \cB (G_3\times G_3)                
	\end{tikzcd}
	$$
	where the morphism $\cB f:\cB(\gm \ltimes \ga) \rightarrow \cB G_3$ is the same as in \Cref{prop:rho-7-surj}, whereas we recall that 
	$$ g:\Mtilde_{1,[3]} \longrightarrow G_3$$
	is defined by the association $(E,e)\mapsto \cO_{E}/m_e^3$. Because the morphism $f$ is injective (and therefore $\cB f$ is representable) and $E_{3,1}$ is a $\gm \ltimes \ga$-torsor, we have that $\cA_{5}^0 \simeq \Mtilde_{1,1}\times [\AA^1/\gm]$. Therefore we have an isomorphism 
	$$\ch(\cA_5^0) \simeq \ZZ[1/6,t,s]$$
	where $t$ (respectively $s$) is the generator of the Chow ring of $\Mtilde_{1,1}$ (respectively $[\AA^1/\gm]$). We can describe $\rho_5^0$ in the following way: if we have a geometric point $(E,e)$ in $\Mtilde_{1,[3]}$, the image is the genus $3$ curve $(C,p)$ obtained by gluing the projective line $\PP^1$ to $E$ identifying $\infty$ with $e$ in a $A_5$-singularity (using the pushout construction as we have defined earlier) and setting $p=e$. Notice that there is a unique way of creating the $A_5$-singularity up to a unique isomorphism of $(\PP^1,\infty)$.
	
	Clearly, $\rho_5^{0,*}(\delta_1)=s$. However, $\rho_5^{0,*}(H)=0$, therefore we need to understand the pullback of the Chern classes of the Hodge bundle. It is enough to prove that $\rho_5^{0,*}(\lambda_1)=-2t+ns$ for some integer $n$. Therefore we can restrict the computation to the open substack $\Mtilde_{1,1}\subset \Mtilde_{1,1}\times [\AA^1/\gm]$. Moreover, as the closed embedding $\cB\gm \into \Mtilde_{1,1}$ is a zero section of a vector bundle over $\cB\gm$, it is enough to do the computation restricting everything to $\cB\gm \into \Mtilde_{1,1}\times [\AA^1/\gm]$.  Therefore, suppose we have an elliptic cuspidal curve $(E,e)$ with $e$ a smooth point, the image through $\rho_5^0$ is a $1$-pointed genus $3$ curve $(C,p)$ constructed gluing the projective line $(\PP^1,\infty)$ and $(E,e)$ (identifying $e$ and $\infty$) in an $A_5$-singularity and setting $p:=e$. We need to understand the $\gm$-action over the vector space $\H^0(C,\omega_C)$. 
	Consider the exact sequence
	$$ \begin{tikzcd}
		0 \arrow[r] & \cO_C \arrow[r] & \cO_{E}\oplus \cO_{\PP^1} \arrow[r] & Q \arrow[r] & 0
	\end{tikzcd}$$
	and the induced long exact sequence on the global sections
	$$
	\begin{tikzcd}
		0 \arrow[r] & \kappa \arrow[r] & \kappa^{\oplus 2} \arrow[r] & Q \arrow[r] & {\H^1(C,\cO_C)} \arrow[r] & {\H^1(E,\cO_E)} \arrow[r] & 0.
	\end{tikzcd}
    $$
    We know that $\lambda_1=-c_1^{\gm}(\H^1(C,\cO_C))$ and $c_1^{\gm}(\H^1(E,\cO_E))=t$. It is enough to describe $c_1^{\gm}(Q)$. Recall that $Q$ fits in a exact sequence of $\gm$-representations
    $$
    \begin{tikzcd}
    	0 \arrow[r] & \cO_{3\infty}=\kappa[t]/(t^3) \arrow[r] & \cO_{3e}\oplus \cO_{3\infty}=\kappa[t]/(t^3)^{\oplus 2}\arrow[r] & Q \arrow[r] & 0
    \end{tikzcd}
	$$ 
	where as usual $\cO_{3e}$ (respectively $\cO_{3\infty}$) is the quotient $\cO_E/m_e^3$ (respectively $\cO_{\PP^1}/m_{\infty}^3)$. Therefore an easy computation shows that $c_1(Q)=-2t$ and thus the restriction of $\lambda_1$ to $\cB\gm$ is equal to $-2t$.
\end{proof}

Finally, a proof similar to the one of \Cref{prop:rho-7-surj} gives us the following result.

\begin{proposition}
	The pullback of the morphism 
	$$\rho_5^{\rm ns}:\cA_5^{\rm ns} \longrightarrow \Mtilde_3$$
	is surjective.
\end{proposition}

\begin{proof}
	 We leave to the reader to check the details. See also \Cref{prop:rho-3-surj} for a similar result in the non-separating case.
\end{proof}

It remains to prove the case for $m=2$, or the strata classifying tacnodes.

\begin{proposition}\label{prop: rho-3-1-surj}
	The pullback of the morphism 
	$$ \rho_3^{1}: \cA_3^1 \longrightarrow  \Mtilde_3$$
	is surjective.
\end{proposition}

\begin{proof}
	Thanks to \Cref{prop:descr-an-odd-i}, we can describe $\cA_3^1$ using a $C_2$-action on $I_3^1$, which is a $\gm$-torsor over the stack $\Mtilde_{1,[2]}\times \Mtilde_{1,[2]}$ where $\Mtilde_{1,[2]}\simeq \Mtilde_{1,1}\times [\AA^1/\gm]$. In fact, \Cref{lem:chow-tor} implies that 
	$$ \ch(\cA_3^1)\simeq \ch(I_3^1)^{\rm inv}.$$
	Because the pullback of the closed immersion 
	$$ \Mtilde_{1,1} \times \cB\gm \into \Mtilde_{1,1} \times [\AA^1/\gm]$$ 
	is an isomorphism at the level of Chow rings, it is enough to understand the $\gm$-torsor when restricted to $(\Mtilde_{1,1}\times \cB\gm)^{\times 2}$. Let us denote by $t_i$ (respectively $s_i$) the generator of the Chow ring of $\Mtilde_{1,1}$ (respectively $\cB\gm$) seen as the $i$-th factor of the product for $i=1,2$. Exactly as we have done in \Cref{prop:gener-rho-2}, we can describe the objects in this product as  $$\Big((E_1,e_1), (\PP(N_{e_1}\oplus N_{s_1}),0,\infty), (E_2,e_2), (\PP(N_{e_2}\oplus N_{s_2}),0,\infty)\Big).$$ We recall that $N_{e_i}$ is the normal bundle of the section $e_i$ and $N_{s_i}$ is the representation of $\gm$ (whose generator is $s_i$) with weight $1$ (for $i=1,2$). By construction, the first Chern class of $N_{e}$ is the $\psi$-class associated to an object $(E,e)$ in $\Mtilde_{1,1}$. 
	
	The $\gm$-torsor comes from identifying the two tangent spaces at $\infty$ of the two projective bundles. A computation shows that its class (namely the first Chern class of the line bundle associated to it) in the Chow ring of the product is of the form $t_1-s_1-t_2+s_2$. Therefore, if we denote by $b_i:=t_i-s_i$ for $i=1,2$, we have the following  description of the Chow ring of $I_3^1$:
	$$\ch(I_3^1)\simeq \ZZ[1/6,t_1,t_2,b_1,b_2]/(b_1-b_2).$$
	Furthermore, the $C_2$-action over $I_3^1$ translates into an action on the Chow ring  defined by the association 
	$$ (t_1,t_2,b_1,b_2) \mapsto (t_2,t_1,b_2,b_1)$$
	therefore it is easy to compute the ring of invariants. We have the following result:
	$$ \ch(\cA_3^1)\simeq \ZZ[1/6,d_1,d_2,b]$$
	where $d_1:=t_1+t_2$ and $d_2:=t_1t_2$. An easy computation shows that $\rho_{3}^{1,*}(\delta_1)=d_1-2b$ and $\rho_3^{1,*}(\delta_{1,1})=d_2-bd_1+b^2$. Finally, a computation identical to the one in the proof of \Cref{prop:rho-5-0-surj} for the $\lambda$-classes, shows us that $\rho_{3}^{1,*}(\lambda_1)=b-d_1$. The statement follows.
\end{proof}

Finally, we arrived at the end of this sequence of abstract computations.

\begin{proposition}\label{prop:rho-3-surj}
	The pullback of the morphism 
	$$\rho_3^{\rm ns}: \cA_3^{\rm ns} \longrightarrow \Mtilde_3$$
	is surjective.
\end{proposition}

\begin{proof}
	For this proof, we denote by $\lambda_i$ the Chern classes of the Hodge bundle of $\Mtilde_{1,2}$, while the Chern classes of the Hodge bundle of $\Mtilde_3$ is denoted by $c_i(\HH)$.
	
	Thanks to \Cref{prop:descr-an-odd-ns}, we have that the Chow ring of $\cA_3^{\rm ns}$ can be recovered by the Chow ring of the $\gm$-torsor $I_3^{\rm ns}$ over $\Mtilde_{1,2}\times [\AA^1/\gm]\times [\AA^1/\gm]$.
	
	First of all, we know that $\Mtilde_{1,2}\simeq \Ctilde_{1,1}$ thanks to \Cref{prop:contrac} and we know the Chow ring of $\Ctilde_{1,1}$ is isomorphic to 
	$$\ZZ[1/6,\lambda_1,\mu_1]/(\mu_1(\lambda_1+\mu_1));$$
	see Proposition 3.3 of \cite{DiLorPerVis}. It is important to remark that $\mu_1$ is the fundamental class of the locus in $\Ctilde_{1,1}$ parametrizing $(E,e_1,e_2)$ such that the two sections coincide.
	
	 We need to understand the class of the $\gm$-torsor $I_3^{\rm ns}$ over $\Mtilde_{1,2}\times [\AA^1/\gm] \times [\AA^1/\gm]$. As usual, we reduce to the closed substack $\Mtilde_{1,2} \times \cB\gm \times \cB\gm$.  If we denote by $s_1$ (respectively $s_2$) the generator of the Chow ring of the first $\cB\gm$ (respectively the second $\cB\gm$) in the product, we have that the same description we used in \Cref{prop: rho-3-1-surj} for the $\gm$-torsor applies here and therefore we only need to understand the description of the two $\psi$-classes in $\ch(\Mtilde_{1,2})$, namely $\psi_1$ and $\psi_2$. We claim that $\psi_1=\psi_2$. Consider the autoequivalence 
	$$ \Mtilde_{1,2} \longrightarrow \Mtilde_{1,2},$$ 
	which is defined by the association $(E,e_1,e_2) \mapsto (E,e_2,e_1)$ (and therefore acts on the Chow rings sending $\psi_1$ in $\psi_2$ and viceversa). It is easy to see that is isomorphic to the identity functor because of the unicity of the involution, see for instance \Cref{lem:genus1}.
	
	This implies that the class associated to the torsor $I_3^{\rm ns}$ is of the form $s_1-s_2$. Because the action of $C_2$ on $I_3^1$ translates into the involution 
	$$ (\lambda_1,\mu_1,s_1,s_2) \mapsto (\lambda_1,\mu_1,s_2,s_1)$$
	of the Chow ring, we finally have
	$$ \ch(\cA_3^{\rm ns})\simeq \ZZ[1/6, \lambda_1,\mu_1, s]/(\mu_1(\lambda_1+\mu_1))$$
	where $s:=s_1=s_2$. It is easy to see that $\rho_3^{\rm ns,*}(\delta_1)=\mu_1$. Moreover, the same ideas for the computations of the $\lambda$-classes used in \Cref{prop:rho-5-0-surj} gives us that $\rho_3^{\rm ns,*}(c_1(\HH))=-s$. Finally, it is enough to prove that $\rho_3^{\rm ns,*}(H)=-12\lambda_1$ modulo the ideal $(\mu_1,s)$, therefore we can restrict our computation to $\Mtilde_{1,2}\setminus \Mtilde_{1,1}\subset \Mtilde_{1,2}\times [\AA^1/\gm] \times [\AA^1/\gm]$. Notice that in this situation the $\gm$-torsor is trivial. Recall that we have the formula $H=9c_1(\HH)-\delta_0-3\delta_1$ by \cite{Est}. Therefore it follows that $\rho_3^{\rm ns,*}(H)=-\delta_0$. To compute $\delta_0$ in $\Mtilde_{1,2}\setminus \Mtilde_{1,1}$, we can consider the natural morphism 
	$$ \Ctilde_{1,1} \longrightarrow \Mtilde_{1,1}$$
    and use the fact that $\delta_0 = 12\lambda_1$ in $\ch(\Mtilde_{1,1})$.
\end{proof}

\Cref{lem:strata} implies that the image of $\rho_{\geq 2,*}$ in $\ch(\Mtilde_3)$ is generated by the following cycles:
\begin{itemize}
	\item the fundamental classes of $\Detilde_1^c$ and $\Detilde_{1,1}^c$;
	\item the fundamental classes of the images of $\rho_7$, $\rho_5^0$ and $\rho_3^1$, which are closed inside $\Mtilde_3$ because of \Cref{lem:sep-sing}; by abuse of notation, we denote these closed substack by $\cA_7$, $\cA_5^0$ and $\cA_3^1$ respectively;
	\item the fundamental classes of the closure of the images of $\rho_6$, $\rho_5^{\rm ns}$, $\rho_4$, $\rho_3^{\rm ns}$ and $\rho_2$; by abuse of notation, we denote these closed substack as $\cA_6$, $\cA_5$, $\cA_4$, $\cA_3$ and $\cA_2$ respectively.
\end{itemize}
 
\begin{remark}
	Notice that $\cA_6$, $\cA_4$ and $\cA_2$ are the stacks we previously denoted by $\widetilde{\cA}_{\geq 6}$, $\widetilde{\cA}_{\geq 4}$ and $\widetilde{\cA}_{\geq 2}$ respectively. Moreover, the stacks $\cA_5^{\rm ns}$ and $\cA_3^{\rm ns}$ are substacks of $\widetilde{\cA}_{\geq 5}$ and $\widetilde{\cA}_{\geq 3}$ respectively.
\end{remark}

\begin{corollary}\label{cor:relations}
	The Chow ring of $\Mbar_3$ is the quotient of the Chow ring of $\Mtilde_3$ by the fundamental classes of $\cA_7$, $\cA_6$, $\cA_5^0$, $\cA_5^{\rm ns}$, $\cA_4$, $\cA_3^1$, $\cA_3^{\rm ns}$, $\cA_2$, $\Detilde_1^c$ and $\Detilde_{1,1}^c$.
\end{corollary}

\section{Explicit description of the relations}\label{sec:strategy}

We illustrate the strategy to compute the explicit description of the relations listed in \Cref{cor:relations}. Suppose we want to compute the fundamental class of a closed substack $X$ of $\Mtilde_3$. First of all, we need to compute the classes of the restrictions of $X$ on every stratum, namely 
$$ X\vert_{\Mtilde_3\setminus (\Htilde_3 \cup \Detilde_1)}, X\vert_{\Htilde_3\setminus \Detilde_1}, X\vert_{\Detilde_1\setminus \Detilde_{1,1}}, X\vert_{\Detilde_{1,1}\setminus \Detilde_{1,1,1}}, X\vert_{\Detilde_{1,1,1}}.$$
Once we have the explicit descriptions, we need to patch them together. Let us show how to do it for the first two strata, i.e. $\Mtilde_3\setminus(\Htilde_3 \cup \Detilde_1)$ and $\Htilde_3 \setminus \Detilde_1$. Suppose we have the description of $X^q:=X\vert_{\Mtilde_3 \setminus (\Detilde_{1}\cup \Htilde_3)}$ and of $X^h:=X\vert_{\Htilde_3 \setminus \Detilde_1}$ in their respective Chow rings. Then we can compute $X\vert_{\Mtilde_{3}\setminus \Detilde_1}$ using \Cref{lem:gluing}. Suppose we are given
$$ X\vert_{\Mtilde_3\setminus \Detilde_1}=p + Hq  \in \ch(\Mtilde_3\setminus \Detilde_1)$$
an expression of $X$, we need to compute the two polynomials $p$ and $q$. If we restrict $X$ to $\Mtilde_3\setminus (\Htilde_3\cup \Detilde_1)$, we get that the polynomial $p$ can be chosen to be just any lifting of $X^q$. Now if we restrict to $\Htilde_3 \setminus \Detilde_1$ we get that
$$ i_h^*p + c_{\rm top}(N_{\cH|\cM})q = X^h$$
where as usual $i_h:\Htilde_3 \setminus \Detilde_1 \into \Mtilde_3 \setminus \Detilde_1$ is the closed immersion of the hyperelliptic stratum and $N_{\cH|\cM}$ is the normal bundle of this immersion. Because of the commutativity of the diagram in \Cref{lem:gluing}, we have that $X^h-i_h^*p$ is in the ideal in $\ch(\Htilde_3 \setminus \Detilde_1)$ generated by $c_{\rm top}(N_{\cH|\cM})$. However the top Chern class is a non-zero divisor, thus we have that we can choose $q$ to be just a lifting of $\widetilde{q}$, where $\widetilde{q}$ is an element in $\ch(\Htilde_3 \setminus \Detilde_1)$ such that $X^h-i_h^*p=c_{\rm top}(N_{\cH|\cM})\widetilde{q}$. Although we have done a lot of choices, it is easy to see that the presentation of $X\vert_{\Mtilde_3 \setminus \Detilde_1}$ is unique in the Chow ring of $\Mtilde_3\setminus \Detilde_1$, i.e. two different presentations differ by a relation in the Chow ring.

We show how to apply this strategy firstly for the computations of the fundamental classes $\delta_{1}^c$ and $\delta_{1,1}^c$ of $\Detilde_1^c$ and $\Detilde_{1,1}^c$ respectively.
\begin{proposition}
	We have the following description:
	$$ \delta_1^c = 6\big(\delta_1(H+\lambda_1+3\delta_1)^2+4\delta_{1,1}(\lambda_1-H-2\delta_1)+12\delta_{1,1,1}\big)$$
	and 
	$$ \delta_{1,1}^c = 24(\delta_{1,1}(\delta_1+\lambda_1)^2+\delta_{1,1,1}\delta_1)$$
	in the Chow ring of $\Mtilde_3$. 
\end{proposition}
\begin{proof}
First of all, we have that $\Detilde_1^c\subset \Detilde_1$, therefore the generic expression of the class $\delta_1^c$ is of the form 
$$ \delta_1 p_2 + \delta_{1,1} p_1 + \delta_{1,1,1} p_0$$ 
where $p_0,p_1,p_2$ are homogeneous polynomial in $\ZZ[1/6,\lambda_1,\lambda_2,\lambda_3, \delta_1, \delta_{1,1},\delta_{1,1,1}, H]$ of degree respectively $0,1,2$. 
We start by restricting $\delta_1^c$ to $\Mtilde_3\setminus \Detilde_{1,1}$. Here we have the sequence of embeddings 
$$ \Detilde_1^c\cap (\Detilde_1 \setminus \Detilde_{1,1}) \into (\Detilde_1 \setminus \Detilde_{1,1}) \into \Mtilde_3 \setminus \Detilde_{1,1}$$ 
which implies that $\delta_1^c$ restricted to $\ch(\Detilde_{1}\setminus \Detilde_{1,1})$ is equal $24t^2(t+t_1)$, where $24t^2$ is the fundamental class of the closed embedding $\Detilde_1^c \into \Detilde_{1}$ (restricted to the open $\Detilde_1 \setminus \Detilde_{1,1}$) while $(t+t_1)$ is the normal bundle of the closed embedding $i_{1}:\Detilde_1 \setminus \Detilde_{1,1} \into \Mtilde_3 \setminus \Detilde_{1,1}$. Because $t+t_1$ is not a zero divisor in the Chow ring, we have that $i_1^*(p_2)=24t^2$ which implies $p_2=6(H+\lambda_1+3\delta_1)^2$.

Now we have to compute the restriction of $\delta_1^c$ to $\Detilde_{1,1}\setminus \Detilde_{1,1,1}$. This is not trivial, because $\Detilde_1^c$ is contained in $\Detilde_1$ but the closed immersion $\Detilde_{1,1} \setminus \Detilde_{1,1,1} \into \Detilde_{1}\setminus \Detilde_{1,1,1}$ is not regular. As a matter of fact, one can prove that doing the naive computation does not work, i.e. the difference $\delta_1^c-\delta_1p_2$ restricted to $\Detilde_{1,1}\setminus \Detilde_{1,1,1}$ is not divisible by the top Chern class of the normal bundle of $i_{1,1}:\Detilde_{1,1}\setminus \Detilde_{1,1,1}\into \Mtilde_3 \setminus \Detilde_{1,1,1}$.  To do it properly, we can consider the following cartesian diagram
$$
\begin{tikzcd}
	{(\Mtilde_{2,1}\setminus \ThTilde_2)\times \cB\gm} \arrow[r, "\rho_2^c"]                                                          & {\Mtilde_3 \setminus \Detilde_{1,1,1}}                                  \\
	{(\Mtilde_{1,2}\setminus \Mtilde_{1,1})\times \Mtilde_{1,1}\times \cB\gm} \arrow[u, hook, "\mu_1\times \id"] \arrow[r, "\rho_{1,1}"] & {\Detilde_{1,1}\setminus \Detilde_{1,1,1}} \arrow[u, "{i_{1,1}}", hook]
\end{tikzcd}
$$
where 
$$\mu_1:(\Mtilde_{1,2}\setminus \Mtilde_{1,1})\times \Mtilde_{1,1}\simeq (\ThTilde_1\setminus \ThTilde_2) \into \Ctilde_2\setminus \ThTilde_2 \simeq \Mtilde_{2,1} \setminus \ThTilde_2$$
is described in the proof of \Cref{lem:chow-ring-C2} and $\rho_2^c$ is the restriction of $\rho_2$ to the closed substack $\Detilde_1^c$. Notice that $\mu_1\times \id$ is a regular embedding of codimension $1$ whereas $i_{1,1}$ is regular of codimension $2$. Excess intersection theory implies that 
$$ \delta_1^c=\rho_{2,*}^c(1)=\rho_{(1,1),*}(c_1(\rho_{1,1}^*N_{i_{1,1}}/N_{\mu_1\times \id}));$$
the normal bundle $N_{\mu_1}$ was described in Proposition 4.9 of \cite{DiLorPerVis} while $N_{i_{1,1}}$ was described in \Cref{prop:relat-detilde-1-1} (see also the proof of \Cref{lem:chow-ring-C2}). A computation shows that 
$$ c_1(\rho_{1,1}^*N_{i_{1,1}}/N_{\mu\times \id})=(t+t_2)$$
where $t$ is the generator of $\ch(\Mtilde_{1,2}\setminus \Mtilde_{1,1})$ while $t_2$ is the generator of $\ch(\cB\gm)$. Finally, we need to compute $\rho_{(1,1),*}(t+t_2)$. This can be done by noticing that $\rho_{1,1}$ factors through the morphism described in \Cref{lem:detilde-1-1}, i.e. the diagram 
$$ 
\begin{tikzcd}
	{(\Mtilde_{1,2}\setminus \Mtilde_{1,1})\times \Mtilde_{1,1}\times \cB\gm} \arrow[r, "\tilde{\rho}_{1,1}"] \arrow[rd, "\rho_{1,1}"] & {(\Mtilde_{1,2}\setminus \Mtilde_{1,1})\times \Mtilde_{1,1} \times \Mtilde_{1,1}} \arrow[d, "\pi_2"] \\
	& {\Detilde_{1,1}\setminus \Detilde_{1,1,1}}                                                 
\end{tikzcd}
$$
is commutative, where $\tilde{\rho}_{1,1}$ is induced by the zero section $\cB\gm \into [\AA^2/\gm]\simeq \Mtilde_{1,1}$. A simple computation using the fact that $\pi_2$ is a $C_2$-torsor gives us $$\delta_1^c\vert_{\Detilde_{1,1}\setminus \Detilde_{1,1,1}}=24(c_1t^2-2c_2t+c_1^3-3c_1c_2).$$
We can now divide $(\delta_1^c-\delta_1p_2)\vert_{\Detilde_{1,1}\setminus \Detilde_{1,1,1}}$ by $c_2+tc_1+t^2$, which is the $2$nd Chern class of the normal bundle of $i_{1,1}$, and get $i_{1,1}^*p_1=24(-2t-3c_1)$. Therefore we have $p_1=24(\lambda_1-H-2\delta_1)$.

Finally, we have to restrict $\delta_1^c$ to $\Detilde_{1,1,1}$. Again we can consider the following diagram
$$
\begin{tikzcd}
	{\Mtilde_{1,1}\times \Mtilde_{1,1}\times \cB\gm} \arrow[r, "c_2\times \id"] \arrow[d, "\tilde{\rho}_2", hook] & \ThTilde_2 \times \cB\gm \arrow[r, "\mu_2\times \id", hook] \arrow[d, "\alpha"] & {\Mtilde_{2,1} \times \cB\gm} \arrow[d] \\
	{\Mtilde_{1,1}\times \Mtilde_{1,1}\times \Mtilde_{1,1}} \arrow[r, "c_6"]                            & {\Detilde_{1,1,1}} \arrow[r, "{i_{1,1,1}}", hook]                     & \Mtilde_3                              
\end{tikzcd}
$$
where the right square is cartesian, the morphism $\mu_2:\ThTilde_2 \into \Ctilde_2\simeq \Mtilde_{2,1}$ is the closed immersion described in Section 4 of \cite{DiLorPerVis} (see also the proof of \Cref{lem:chow-ring-C2}), the morphism $c_2:\Mtilde_{1,1}\times \Mtilde_{1,1}\rightarrow \Detilde_1\simeq \ThTilde_2$ is the gluing morphism (see proof of Lemma 4.8 of \cite{DiLorPerVis}) and $c_6$ is the morphism described in \Cref{lem:descr-delta-1-1-1}. 
Excess intersection theory implies that 
$$\delta_1^c\vert_{\Detilde_{1,1,1}}= \alpha_*(N_{\mu_2});$$ 
moreover we have that
$$N_{\mu_2}= \frac{(c_2\times \id)_*\tilde{\rho}_2^*(t_1+t_2)}{2}$$ 
where $t_i$ is the generator of the $i$-th factor of the product $\Mtilde_{1,1}^{\times 3}$ for $i=1,2,3$. Therefore
$$\delta_{1}^c\vert_{\Detilde_{1,1,1}}=c_{6,*}(12t_3^2(t_1+t_2))=24(c_1c_2-3c_3)$$
in the Chow ring of $\Detilde_{1,1,1}$.
This leads finally to the description. The same procedure can be used to compute the class of $\delta_{1,1}^c$.
	
\end{proof}

\begin{remark}
The relation $\delta_1^c$ gives us that we do not need the generator $\delta_{1,1,1}$  in $\ch(\Mbar_3)$.
\end{remark}

\subsection*{The fundamental class of $\cA_5^0$ and $\cA_3^1$}

Now we concentrate on describing two of the strata of separating singularities, namely $\cA_{3}^1$ and $\cA_5^1$. Let us start with $\cA_3^1$.

\begin{proposition}
	We have the following description 
	$$ [\cA_3^1]=\frac{1}{2}(H+\lambda_1+\delta_1)(3H+\lambda_1+\delta_1) $$
	in the Chow ring of $\Mtilde_3$.
\end{proposition}

\begin{proof}
	 A generic object in $\cA_3^1$ can be described as two genus $1$ curves intersecting in a separating tacnode. We claim that the morphism $\cA_3^1 \rightarrow \Mtilde_3$ (which is proper thanks to \Cref{lem:sep-sing}) factors thorough $\Xi_1\subset\Htilde_3 \into \Mtilde_3$. In fact, given an element $((E_1,e_1),(E_2,e_2),\phi)$ in $\cA_3^1$ with $E_1$ and $E_2$ smooth genus $1$ curves, we can consider the hyperelliptic involution of $E_1$ (respectively $E_2$) induced by the complete linear sistem of $\cO(2e_1)$ (respectively $\cO(2e_2)$). It is easy to see that their differentials commute with the isomorphism $\phi$ and therefore we get an involution of the genus $3$ curve whose quotient is a genus $0$ curve with one node. Because we have considered a generic element and the hyperelliptic locus is closed, we get the claim. In particular, $\cA_3^1\vert_{\Mtilde_3 \setminus (\Htilde_3 \cup \Detilde_1)}$ is zero. 
	
	Consider now the description of $\Htilde_3\setminus \Detilde_1$. Because $\cA_3^1\setminus \Detilde_1 \into \Xi_1\setminus \Detilde_1 \into \Htilde_3 \setminus \Detilde_1$, we have that excess intersection theory gives us the equalities
	$$ [\cA_3^1]\vert_{\Htilde_3 \setminus \Detilde_1} = c_1(N_{\cH|\cM})[a_3]=\frac{2\xi_1-\lambda_1}{3}[a_3]$$ 
	where $a_3$ is the class of $\cA_3^1$ as a codimension $2$ closed substack in $\Htilde_3\setminus \Detilde_1$. Suppose we have an element $(Z/S,L,f)$ in $\Htilde_3 \setminus \Detilde_1$; first of all we have that $\cA_3^1\subset \Xi_1$, therefore $Z/S \in \cM_0^{1}$. Furthermore, because we have a separating tacnode between two genus $1$ curve, it is clear that the nodal section $n$ in $Z$ has to be in the branching locus, or equivalently $f(n)=0$. The converse is also true thanks to \Cref{prop:description-quotient}. Due to the description of $\Xi_1$, we have that 
	$$ a_3 = (-2s)[\Xi_1] = (-2s)(-c_1) $$
	and therefore using $\lambda$-classes as generators
	$$ [\cA_3^1]\vert_{\Htilde_3 \setminus \Detilde_1}=\frac{1}{3}a_3 (2\xi_1-\lambda_1) = \frac{2}{9}(\xi_1+\lambda_1)(2\xi_1-\lambda_1)\xi_1.$$ 
	
	Let us focus on the restriction to $\Detilde_1 \setminus \Detilde_{1,1} \simeq (\Mtilde_{2,1}\setminus \ThTilde_1) \times \Mtilde_{1,1}$. It is easy to see that the only geometric objects that are in $\cA_3^1$ are of the form $((C,p),(E,e))$ where $p$ lies in an almost-bridge of the genus $2$ curve $p$, i.e. $p$ is a smooth point in a projective line that intersects the rest of the curve in a tacnode. Recall that we have an open immersion
	$$\Mtilde_{2,1}\setminus \ThTilde_1 \simeq \Ctilde_{2}\setminus \ThTilde_1$$
	described in \Cref{prop:contrac}. Through this identification, $(C,p)$ corresponds to a pair $(C',p')$ where $C'$ is an ($A_r$-)stable genus $2$ curve and $p'$ is a cuspidal point. Therefore it is easy to see that in the notation of \Cref{cor:mtilde_21}, we have that the fundamental class of the separating tacnodes is described by the equations $s=a_5=a_4=0$.  We get the following expression
	$$ [\cA_3^1]\vert_{\Detilde_{1}\setminus\Detilde_{1,1}} = -2t_1(t_0-2t_1)(t_0-3t_1) \in \ch(\Detilde_1 \setminus \Detilde_{1,1}).$$ 
	
	The same idea works for $\Detilde_{1,1}\setminus \Detilde_{1,1,1}$ (using \Cref{lem:mtilde-12}) and gives us the following description 
	$$ [\cA_3^1]\vert_{\Detilde_{1,1}\setminus \Detilde_{1,1,1}}= -24t^3 \in \ch(\Detilde_{1,1} \setminus \Detilde_{1,1,1}).$$
	
	Finally, it is clear that $\cA_3^1 \cap \Detilde_{1,1,1} = \emptyset$. 
\end{proof}

Now we focus on the fundamental class of $\cA_5^0$. A generic object in $\cA_5^0$ can be described as a genus $1$ curve and a genus $0$ curve intersecting in a $A_5$-singularity. Notice that this implies that $\cA_5^0 \cap \Htilde_3=\emptyset$ because the only possible involution of an $A_5$-singularity has to exchange the two irreducible components (see \Cref{prop:description-quotient}). In the same way, it is easy to prove that $\cA_5^0 \cap \Detilde_{1,1} = \emptyset$. 

The intersection of $\cA_5^0$ with $\Detilde_{1}$ is clearly transversal, therefore
$$[\cA_5^0]\vert_{\Detilde_{1}\setminus \Detilde_{1,1}}=[\cA_5^0 \cap \Detilde_1].$$

\begin{lemma}\label{lem:a-5-0}
	We have the following equality
	$$ [\cA_5^0]\vert_{\Detilde_{1}\setminus \Detilde_{1,1}} = 72(t_0+t_1)^3t_0t_1 - 384(t_0+t_1)(t_0t_1)^2$$
	in the Chow ring of $\Detilde_1 \setminus \Detilde_{1,1}$.
\end{lemma}

\begin{proof}
	Because every $A_5$-singularity for a curve of genus $2$ is separating, it is enough to compute the fundamental class of the locus of $\cA_5$-singularities in $\Ctilde_2 \setminus \ThTilde_1$. We know that 
	$$\Ctilde_2 \setminus \ThTilde_1 \simeq [\widetilde{\AA}(6)\setminus 0/B_2]$$ 
	see \Cref{cor:mtilde_21} for a more detailed discussion. We have that an element $(f,s) \in \widetilde{\AA}(6)$ defines a genus $2$ curve with an $A_5$-singularity if and only if $f \in \AA(6)$ has a root of multiplicity $6$. Because $B_2$ is a special group, it is enough to compute the $T$-equivariant fundamental class of the locus parametrizing sections of $\AA(6)$ which have a root of multiplicity $6$, where $T$ is the maximal torus inside $B_2$. Therefore one can use the formula in \Cref{rem:gener}.
\end{proof}

It remains to describe the restriction of $\cA_5^0$ in $\Mtilde_3\setminus (\Htilde_3 \cap \Detilde_1)$. First of all, we pass to the projective setting. Recall that 
$$ \Mtilde_3\setminus (\Htilde_3 \cap \Detilde_1) \simeq [U/\GL_3]$$  
where $U$ is an invariant open inside a $\GL_3$-representation of the space of forms in three coordinates of degree $4$. Because $U$ does not cointain the zero section, we can consider the projectivization $\overline{U}$ in $\PP^{14}$ and we have the isomorphism 
$$ \ch_{\GL_3}(U)\simeq \ch_{\GL_3}(\overline{U})/(c_1-h)$$ 
where $c_1$ is the first Chern class of the standard representation of $\GL_3$.  We want to describe the locus $X_5^0$ which parametrizes pairs $(f,p)$ such that $p$ is a separating $A_5$-singularity of $f$, in fact, because a quartic in $\PP^2$ has at most one $A_5$-singularity, we can compute the pushforward through $\pi$ of the fundamental class of $X_5^0$ and then set $h=c_1$ to the get the fundamental class of $\cA_5^0$. Notice that pair $(f,p) \in X_5^0$ can be described as a cubic $g$ and a line $l$ such that $f=gl$ and $p$ is the only intersection between $g$ and $l$, or equivalently $p$ is a flex of $g$ and $l$ is the flex tangent. 

To describe $X_5^0$, we first introduce the locally closed substack $X_2$ of $[\PP^{14}\times \PP^2/\GL_3]$ which parametrizes pairs $(f,p)$ such that $p$ is a $A_r$-singularity of $f$ with $r\geq 2$ (eventually $r=\infty$).

Recall that we have the isomorphism  
$$ [\PP^{14}\times \PP^2/\GL_3] \simeq [\PP^{14}/H]$$
where $H$ is the stabilizer subgroup of $\GL_3$ of the point $[0:0:1]$ in $\PP^2$. The isomorphism is a consequence of the transitivity of the action of $\GL_3$ on $\PP^2$. Moreover, $H\simeq (\GL_2\times \gm)\ltimes \ga^2$; see \Cref{rem:H}. We denote by $G:=(B_2\times \gm)\ltimes \ga^2$ where $B_2 \into \GL_2$ is the Borel subgroup of upper triangular matrices inside $\GL_2$.

Finally, let us denote by $L$ the standard representation of $\gm$, by $E$ the standard representation of $\GL_2$ and $E_0 \subset E$ the flag induced by the action of $B_2$ on $E$; $E_1$ is the quotient $E/E_0$. We denote by $V_3$ the $H$-representation
$$(L^{\vee} \otimes {\rm Sym}^3E^{\vee}) \oplus {\rm Sym}^4E^{\vee}.$$

\begin{lemma}\label{lem:X-2}
	In the setting above, we have the following isomorphism
	$$ X_2 \simeq [V_3 \otimes E_0^{\otimes 2} \otimes L^{\otimes 2}/G]$$
	of algebraic stacks.
\end{lemma}

\begin{proof}
Thanks to the isomorphism  
$$ [\PP^{14}\times \PP^2/\GL_3] \simeq [\PP^{14}/H]$$
we can describe $X_2$ as a substack of the right-hand side. The coordinates of $[\PP^{14}/H]$ are the coefficients of the generic polynomial $p=a_{00}+a_{10}x+ a_{01}y+ \dots +a_{04}y^4$ which is the dehomogenization of the generic quartic $f$ in the point $[0:0:1]$.
First of all, notice that $X_2$ is contained in the complement of the closed substack $X$ parametrizing polynomials $f$ such that its first and second derivatives in $x$ and $y$ vanish in $(0,0):=[0:0:1]$, thanks to \Cref{lem:A-sing}. This is an $H$-equivariant subbundle of codimension $6$ in $[\PP^{14}/H]$ defined by the equations $$a_{00}=a_{10}=a_{01}=a_{20}=a_{11}=a_{02}=0.$$
We denote it simply by $[\PP^8/H]$. Moreover, $X_2$ is cointained in the locus parametrizing quartic $f$ such that $f$ is singular in $(0,0)$. This is an $H$-equivariant subbundle defined by the equations $a_{00}=a_{10}=a_{01}=0$. We denote it simply $[\PP^{11}/H]$, therefore $X_2 \into [\PP^{11}\setminus \PP^8/H]$. By construction, $\PP^{11}$ can be described as the projectivization of the $H$-representation $V$ defined as 
$$ (L^{\otimes -2}\otimes{\rm Sym}^2E^{\vee}) \oplus V_3$$
whereas $\PP^8$ is the projectivization of $V_3$.

Consider an element $p$ in $\PP^{11}$, described as the polynomial $a_{20}x^2+a_{11}xy+a_{02}y^2 + p_3(x,y) +p_4(x,y)$ where $p_3$ and $p_4$ are homogeneous polynomials in $x,y$ of degree respectively $3$ and $4$. Notice that $(0,0)$ is an $A_1$-singularity, i.e. an ordinary node, if and only if $a_{11}^2-4a_{02}a_{20}\neq 0$.  In fact, an $A$-singularity (eventually $A_{\infty}$) is a node if and only if it has two different tangent lines. Therefore $X_2$ is equal to the locus where the equality holds and therefore it is enough to describe $\VV(a_{11}^2-4a_{02}a_{20})$ in $[\PP^{11}\setminus \PP^8/H]$.

We define $W$ as the $H$-representation 
$$W:=(L^{\vee}\otimes E^{\vee}) \oplus (L^{\vee} \otimes {\rm Sym}^3E^{\vee}) \oplus {\rm Sym}^4E^{\vee}$$
and we consider the $H$-equivariant closed embedding of $W\into V$ induced by the morphism of $H$-schemes
$$ L^{\vee}\otimes E^{\vee} \into {\rm Sym}^2(L^{\vee} \otimes E^{\vee})\simeq L^{\otimes -2}\otimes {\rm Sym}^2E^{\vee}$$ 
which is defined by the association $(f_1,f_2) \mapsto (f_1^2,2f_1f_2,f_2^2)$. Now, consider the $\gm$-action on $W$ defined using weight $1$ on the $2$-dimensional vector space $(L^{\vee}\otimes E^{\vee})$ and weight $2$ on $V_3$. We denote by $\PP_{1,2}(W)$ the quotient stack and by $\PP_2(V_3)\into \PP_{1,2}(W)$ the closed substack induced by the embedding $V_3\into W$. One can prove that the morphism 
$$ [\PP_ {1,2}(W)\setminus \PP_2(V_3)/H] \into [\PP(V)\setminus \PP(V_3)/H]$$
induced by the closed immersion $W \into V$ is a closed immersion too and its scheme-theoretic image is exactly the locus $\VV(a_{11}^2-4a_{02}a_{20})$. We are considering the action of $H$ on a stack as defined in \cite{Rom}. Finally, because the action of $\GL_2$ over $E^{\vee}$ is transitive, we have the isomorphism 
$$ [\PP_{1,2}(W) \setminus \PP_2(V_3)/H] \simeq [\PP_{1,2}(W_0) \setminus \PP_2(V_3)/G]$$ 
where the $G$-representation $W_0\into W$ is defined as $(L^{\vee}\otimes E_0^{\vee})\oplus V_3$. We want to stress that this is true only if we remove the locus $\PP_2(V_3)$, in fact the subgroup of stabilizers of $W_0$ in $W$ is equal to $H$ when restricted to $V_3$. Finally, we notice that we have an isomorphism $$\PP_{1,2}(W_0)\setminus \PP_2(V_3)\simeq V_3 \otimes (E_0^{\otimes 2}\otimes L^{\otimes 2})$$ of $H$-stacks.
	
\end{proof}

\begin{remark}
	By construction, an element $f \in V_3 \otimes E_0^{\otimes 2} \otimes L^{\otimes 2}$ is associated to the curve $y^2=f(x,y)$. Notice that $E^{\vee}$ is the vector space $E_0^{\vee}\oplus E_1^{\vee}$ where $x$ is a generator for $E_1^{\vee}$ and $y$ is a generator for $E_0^{\vee}$. Moreover,  $L^{\vee}$ is generated by $z$.
\end{remark}

\begin{corollary}
	We have an isomorphism of rings
	$$ \ch(X_2)\simeq \ZZ[1/6,t_1,t_2,t_3]$$ 
	where $t_1,t_2,t_3$ are the first Chern classes of the standard representations of the three copies of $\gm$ in $G$. Specifically, $t_1,t_2,t_3$ are the first Chern classes of $E_1,E_0,L$ respectively.
\end{corollary}

Recall that $X$ is the closed substack of $\PP^{14}\times \PP^2$ which parametrizes pairs $(f,p)$ such that $p$ is a singular point of $f$ but not an $A$-singularity, see \Cref{def:A-sing}.
Thus, we have a closed immersion 
$$ i_2:=X_2 \into [(\PP^{14}\times \PP^2)\setminus X/\GL_3]$$ 
and we can describe its pullback at the level of Chow ring. We have an isomorphism
$$\ch_{\GL_3}(\PP^{14}\times \PP^2) \simeq \frac{\ZZ[1/6,c_1,c_2,c_3,h_{14},h_{2}]}{(p_2(h_2),p_{14}(h_{14}))}$$ 
where $c_i$ is the $i$-th Chern class of the standard representation of $\GL_3$, $h_{14}$ (respectively $h_{2}$) is the hyperplane section of $\PP^{14}$ (respectively of $\PP^{2}$) and $p_{14}$ (respectively $p_2$) is a polynomial of degree $15$ (respectively $3$) with coefficients in $\ch(\cB\GL_3)$. 
\begin{proposition}\label{prop:i-2}
	The closed immersion $i_2$ is the complete intersection in $[(\PP^{14}\times \PP^2)\setminus X/\GL_3]$ defined by equations 
	$$ a_{00}=a_{10}=a_{01}=a_{11}^2-4a_{20}a_{02}=0$$
	whose fundamental class is equal to
	$$2(h+k-c_1)(h+4k)((h+3k)^2-(c_1+k)(h+2k)+c_2).$$
	Moreover, the morphism $i_2^*$ is defined by the following associations:
\begin{itemize}
		\item $i_2^*(k)=-t_3$,
		\item $i_2^*(c_1)=t_1+t_2+t_3$, 
		\item $i_2^*(c_2)=t_1t_2+t_1t_3+t_2t_3$,
		\item $i_2^*(c_3)=t_1t_2t_3$,
		\item $i_2^*(h)=2(t_2+t_3)$,
\end{itemize}   
\end{proposition}

\begin{proof}
	It follows from the proof of \Cref{lem:X-2}.
\end{proof}

Because $i_2^*$ is clearly surjective, it is enough to compute the fundamental class of $X_5^0$ as a closed subscheme of $X_2$, choose a lifting through $i_2^*$ and then multiply it by the fundamental class of $X_2$.

We are finally ready to do the computation. 
\begin{corollary}
	The closed substack $X_5^0$ of $X_2$ is the complete intersection defined by the vanishing of the coefficients of $x^3, x^2y, x^4$ as coordinates of $V_3 \otimes E_0^{\otimes 2} \otimes L^{\otimes 2}$.
\end{corollary}

\begin{proof}
	The element of the representation $V_3 \otimes E_0^{\otimes 2} \otimes L^{\otimes 2}$ are the coefficients of polynomials of the form $p_3(x,y)+p_4(x,y)$ where $p_3$ (respectively $p_4$) is the homogeneous component of degree $3$ (respectively $4$). If we look at them in $X_2$, they define a polynomial $y^2+p_3(x,y)+p_4(x,y)$. It is clear now that a polynomial of this form is the product of a line and a cubic if and only if the coefficient of $x^3$ and $x^4$ are zero. Moreover, the condition that $(0,0)$ is the only intersection between the line and the cubic is equivalent to asking that the coefficient of $x^2y$ is zero.
\end{proof}

\begin{remark}
	 A straightforward computation gives use the fundamental class of $X_5^0$ and the strategy described at the beginning of \Cref{sec:strategy} gives us the description of the fundamental class of $\cA_5^0$ in the Chow ring of $\Mtilde_3$.
	We do not write down the explicit description because it is contained inside the ideal generated by the other relations. 
\end{remark}

\subsection*{Fundamental class of $A_n$-singularity}

We finally deal with the computation of the remaining fundamental classes. As usual, our strategy assures us that it is enough to compute the restriction of every fundamental class to every stratum. We do not give detail for every fundamental class. We instead describe the strategy to compute all of them in every stratum and leave the remaining computations to the reader.  

We start with the open stratum $\Mtilde_3 \setminus (\Htilde_3 \cup \Detilde_1)$, which is also the most diffucult one. Luckily, we have already done all the work we need for the previous case. We adopt the same exact idea we used for the computation of $\cA_5^0$. 

\begin{remark}
	First of all, we can reduce the computation to the fundamental class of the locus $X_n$ in $(\PP^{14}\times \PP^{2})\setminus X$ parametrizing $(f,p)$ such that $p$ is an $A_h$-singularity for $h\geq n$. As above, $X$ is the closed locus paramatrizing $(f,p)$ such that $p$ is a singular point of $f$ but not an $A$-singularity.  Consider the morphism 
	$$\pi: \PP^{14}\times \PP^{2} \rightarrow \PP^{14}$$
	and consider the restriction 
	$$\pi\vert_{X_n}: X_n \longrightarrow\cA_n\setminus (\Htilde_3 \cup \Detilde_1);$$
	this is finite birational because generically a curve in $\cA_n$ has only one singular point. Therefore it is enough to compute the $H$-equivariant class of $X_n$ in $(\PP^{14}\times \PP^{2})\setminus X$ and then compute the pushforward $\pi_*(X_n)$. This is an exercise with Segre classes.
	We give the description of the relevant strata in the Chow ring of $\Mtilde_3$ in \Cref{rem:relations-Mbar}.
\end{remark}

\begin{proposition}
	In the situation above, we have that 
	$$[X_n]=C_ni_{2,*}(1) \in \ch_{\GL_3}((\PP^{14}\times \PP^2)\setminus X)$$
	where
	$$i_{2,*}(1)=2(h+k-c_1)(h+4k)((h+3k)^2-(c_1+k)(h+2k)+c_2)$$ 
	while $C_2=1$ and $C_n=c_3c_4\dots c_{n}$ for $n\geq 3$ where 
	$$ c_m:= -mc_1+\frac{2m-1}{2}h+(4-m)k  $$ for every $3\leq m \leq 7$.
\end{proposition}

\begin{proof}
  \Cref{prop:i-2} and \Cref{lem:X-2} imply that it is enough to compute the fundamental class of $X_n$ in $X_2$. It is important to remind that the coordinates of $X_2$ are the coefficients of the polynomial $p_3(x,y)+p_4(x,y)$ where $p_3$ and $p_4$ are homogeneous polynomials in $x,y$ of degree respectively $3$ and $4$. Moreover, if we see it as an element of $[\PP^{14}\times \PP^{2}/\GL_3]$, it is represented by the pair $(y^2z^2+p_3(x,y)z+p_4(x,y), [0:0:1])$.  Therefore, we need to find a relation between the coefficients of $p_3$ and $p_4$ such that the point $(0,0):=[0:0:1]$ is an $A_h$-singularity for $h \geq n$.
  
  To do so, we apply Weierstrass preparation theorem. Specifically, we use Algorithm 5.2 in \cite{Ell}, which allows us to write the polynomial $y^2+p_3(x,y) + p_4(x,y)$ in the form $y^2+p(x)y+q(x)$ up to an invertible element in $k[[x,y]]$. The square completion procedure implies that, up to an isomorphism of $k[[x,y]]$, we can reduce to the form $y^2+[q(x)-p(x)^2/4]$. Although $q(x)$ and $p(x)$ are power series, we just need to understand the coefficients of $h(x):=q(x)-p(x)^2/4$ up to degree $8$. Clearly the coefficient of $1$, $x$ and $x^2$ are already zero by construction. In general for $n\geq 3$, if $c_n$ is the coefficient of $x^{n}$ inside $h(x)$, we have that $X_n$ is the complete intersection inside $X_2$ of the hypersurfaces $c_i=0$ for $3\leq i \leq n$. We can use now the description of $X_2$ as a quotient stack (see \Cref{lem:X-2}) to compute the fundamental classes. 
\end{proof}

\begin{remark}
	Notice that for $\cA_5$, we also have the contribution of the closed substack $\cA_5^0$ that we need to remove to get the fundamental class of the non-separating locus. To same is not true for $\cA_3$, because $\cA_3^1 \subset \Xi_1$.
\end{remark}

It remains to compute the fundamental class of $\cA_n$ restricted to the other strata. The easiest case is $\Detilde_{1,1,1}$, because clearly there are no $\cA_n$-singularities for $n\geq 3$. Therefore $\cA_n\vert_{\Detilde_{1,1,1}}=0$ for every $n\geq 3$. Regarding $\cA_2$, it enough to compute its pullback through the $6:1$-cover described in \Cref{lem:descr-delta-1-1-1}. We get the following result.

\begin{proposition}
	The restriction of $\cA_n$ to $\Detilde_{1,1,1}$ is of the form
	$$ 24(c_1^2-2c_2) \in \ZZ[1/6,c_1,c_2,c_3]\simeq  \ch(\Detilde_{1,1,1})$$
	for $n=2$ while it is trivial for $n\geq 3$.  
\end{proposition}

As far as $\Detilde_{1,1} \setminus \Detilde_{1,1,1}$ is concerned, we have that $\cA_n \cap \Detilde_{1,1} = \emptyset$ for $n\geq 4$. Moreover, $\cA_3 \cap \Detilde_{1,1}=\emptyset$ because every tacnode is separating in the stratum $\Detilde_{1,1}\setminus \Detilde_{1,1,1}$. Therefore again, we only need to do the computation for $\cA_2$, which is straightforward.

\begin{proposition}
	The restriction of $\cA_n$ to $\Detilde_{1,1}\setminus \Detilde_{1,1,1}$ is of the form
	$$ 24(t^2+c_1^2-2c_2) \in \ZZ[1/6,c_1,c_2,t]\simeq  \ch(\Detilde_{1,1}\setminus \Detilde_{1,1,1})$$
	for $n=2$ while it is trivial for $n\geq 3$.  
\end{proposition}

We now concentrate on the stratum $\Detilde_{1}\setminus \Detilde_{1,1}$. Recall that we have the isomorphism 
$$\Detilde_{1} \setminus \Detilde_{1,1}\simeq (\Mtilde_{2,1}\setminus \ThTilde_1) \times \Mtilde_{1,1}$$ 
as in \Cref{prop:descr-detilde-1} and we denote by $t_0,t_1$ the generators of the Chow ring of $\Mtilde_{2,1}$ and by $t$ the generator of the Chow ring of $\Mtilde_{1,1}$.

\begin{proposition}
	We have the following description of the $A_n$-strata in the Chow ring of $\Detilde_{1}\setminus \Detilde_{1,1}$:
	\begin{itemize}
		\item the fundamental classes of $\cA_5$, $\cA_6$ and $\cA_7$ are trivial,
		\item the fundamental class of $\cA_4$ is equal to $40(t_0+t_1)^2t_0t_1$,
		\item the fundamental class of $\cA_3$ is equal to 
		$-24(t_0+t_1)^3 + 48(t_0+t_1)t_0t_1$,
		\item the fundamental class of $\cA_2$ is equal to 
		$24(t_0+t_1)^2 - 48t_0t_1+24t^2$.
	\end{itemize} 
\end{proposition}

\begin{proof}
	If $n=6,7$ the intersection of $\cA_n$ with $\Detilde_{1}\setminus \Detilde_{1,1}$ is empty. Notice that if $n=5$, it is also trivial as we are interested in non-separating singularities. It remains to compute the case for $n=2,3,4$. It is clear that if $n=3,4$, the factor $\Mtilde_{1,1}$ of the product does not give a contribution, therefore it is enough to describe the fundamental class of the locus of $A_n$-singularities in $\Mtilde_{2,1}\setminus \ThTilde_1$ for $n=3,4$. We do it exactly as in the proof of \Cref{lem:a-5-0}. The computation for $n=2$ again is straightforward. 
\end{proof}

Last part of the computations is the restriction to the hyperelliptic locus.

\begin{remark}\label{rem:sep-tac}
	We recall the stratification
	$$ \cA_3^1\setminus \Detilde_1  \subset \Xi_1\setminus \Detilde_1 \subset \Htilde_3\setminus \Detilde_1$$ 
	defined in the following way: $\Xi_1$ parametrizes triplets  $(Z,L,f)$ in $\Htilde_3\setminus \Detilde_1$ such that $Z$ is genus $0$ curve with one node whereas $\cA_3^1$ parametrizes triplets $(Z,L,f)$ in $\Xi_1$ such that $f$ vanishes at the node. Using the results in \Cref{sec:H3tilde}, we get that 
	\begin{itemize}
		\item $\ch(\Htilde_3\setminus (\Detilde_1 \cup \Xi_1)) \simeq \ZZ[1/6,s,c_2]/(f_9)$,
		\item $\ch(\Xi_1\setminus (\Detilde_1 \cup \cA_3^1))\simeq \ZZ[1/6,c_1,c_2]$,
		\item $\ch(\cA_3^1 \setminus \Detilde_1)\simeq \ZZ[1/6,s]$
	\end{itemize}
	where $f_9$ is the restriction of the relation $c_9$ to the open stratum. Furthermore, the normal bundle of the closed immersion
	$$ \Xi_1\setminus (\Detilde_1\cup \cA_3^1) \into \Htilde_3 \setminus (\Detilde_1 \cup \cA_3^1)$$ 
	is equal to $-c_1$ whereas the normal bundle of the closed immersion 
	$$\cA_3^1 \setminus \Detilde_1 \into \Xi_1\setminus \Detilde_1$$
	is equal to $-2s$. 
\end{remark}

\Cref{lem:gluing} implies that we can compute the restriction of $\cA_n$ to the hyperelliptic locus using the stratification $$ \cA_3^1\setminus \Detilde_1  \subset \Xi_1\setminus \Detilde_1 \subset \Htilde_3\setminus \Detilde_1,$$
i.e. it is enough to compute the restriction of $\cA_n$ to $\cA_3^1\setminus \Detilde_1$, $\Xi_1\setminus (\Detilde_1\cup \cA_3^1)$ and $\Htilde_3\setminus (\Detilde_1 \cup \Xi_1)$.

\begin{lemma}
	The restriction of $\cA_n$ to $\cA_3^1\setminus \Detilde_1$ is empty if $n\geq 3$ whereas we have the equality $$[\cA_2]\vert_{ \cA_3^1\setminus \Detilde_1}=72s^2$$
	in the Chow ring of $\cA_3^1\setminus \Detilde_1$.

	Moreover, the restriction of $\cA_n$ to $\Xi_1\setminus (\Detilde_1\cup\cA_3^1)$ is empty if $n\geq 4$ whereas we have the two equalities
	\begin{itemize}
	\item$ [\cA_2]\vert_{\Xi_1\setminus (\Detilde_1 \cup\cA_3^1)} =24c_1^2-48c_2$ 
	\item$ [\cA_3]\vert_{\Xi_1\setminus (\Detilde_1\cup \cA_3^1)} = 24c_1^3-72c_1c_2$
	\end{itemize}
	in the Chow ring of $\Xi_1\setminus (\Detilde_1\cup \cA_3^1)$.
\end{lemma}

\begin{proof}
	This is an easy exercise and we leave it to the reader.
\end{proof}

\subsubsection*{Restriction to $\Htilde_3 \setminus (\Xi_2 \cup \Detilde_1)$}
It remains to compute the restriction of $\cA_n$ to the open stratum $\Htilde_3 \setminus (\Xi_2 \cup \Detilde_1)$.

\begin{remark}\label{rem:ar-vis-2}
	 We know that $\Htilde_3 \setminus (\Detilde_1\cup \Xi_1)$ is isomorphic to the quotient stack $[\AA(8)\setminus 0/(\GL_2/\mu_4)]$, where $\AA(8)$ is the space of homogeneous forms in two variables of degree $8$. Moreover, we have the isomorphism $\GL_2/\mu_4 \simeq \PGL_2\times \gm$. See \Cref{lem:ar-vis} and \cite{ArVis} for a more detailed discussion.
	 
	 Therefore we have that 
	 $$\ch(\Htilde_3\setminus (\Detilde_1\cup \Xi_1))\simeq \ZZ[1/6,s,c_2]$$ where $s$ is the first Chern class of the standard representation of $\gm$ while $c_2$ is the generator of the Chow ring of $\cB\PGL_2$. Notice that we have a morphism 
	 $$ \GL_2 \longrightarrow \PGL_2 \times \gm$$
	 defined by the association $A\mapsto ([A],\det{A}^2)$ which coincides with the natural quotient morphism 
	 $$ q:\GL_2 \longrightarrow \GL_2/\mu_4.$$
	 We have that $q^*(s)=2d_1$ and $q^*(c_2)=d_2$, where $d_1$ and $d_2$ are the first and second Chern class of the standard representation of $\GL_2$. 	
\end{remark}
 
Using the description of $\Htilde_3\setminus (\Detilde_1\cup \Xi_1)$ as a quotient stack highlighted in the previous remark, the restriction of $\cA_n$ to $\Htilde_3\setminus  (\Detilde_1\cup \Xi_1)$ is the locus $H_n$ which parametrizes forms $f \in \AA(8)$ such that $f$ has a root of multiplicity at least $n+1$. Thanks to \Cref{rem:ar-vis-2}, it is enough to compute the fundamental class of $H_n$ after pulling it back through $q$, i.e. compute the $\GL_2$-equivariant fundamental class of $H_n$. Because $\GL_2$ is a special group, we can reduce to do the computation of the $T$-equivariant fundamental class of $H_n$, where $T$ is the torus of diagonal matrices in $\GL_2$. Therefore, we can use the formula in \Cref{rem:gener} to get the explicit description of the $T$-equivariant class of $H_n$.

\section{The Chow ring of $\Mbar_3$ and the comparison with Faber's result}\label{sec:3-4}

We are finally ready to present our description of the Chow ring of $\Mbar_3$. 

\begin{theorem}\label{theo:chow-ring-m3bar}
	 Let $\kappa$ be the base field of characteristic different from $2,3,5,7$. The Chow ring of $\Mbar_3$ with $\ZZ[1/6]$-coefficients is the quotient of the graded polynomial algebra 
	 $$\ZZ[1/6,\lambda_1,\lambda_2,\lambda_3,\delta_{1},\delta_{1,1},\delta_{1,1,1},H]$$
	 where 
	 \begin{itemize}
	 	\item[] $\lambda_1,\delta_1,H$ have degree $1$, \item[]$\lambda_2,\delta_{1,1}$ have degree $2$, \item[]$\lambda_3,\delta_{1,1,1}$ have degree $3$.
 	\end{itemize}
 	The quotient ideal is generated by 15 homogeneous relations: 
	 \begin{itemize}
	 	\item[] $[\cA_2]$, which is in codimension $2$;
	 	\item[] $[\cA_3], [\cA_3^{1}], \delta_1^c, k_1(1),k_{1,1}(2)$, which are in codimension $3$,
	 	\item[] $[\cA_4], \delta_{1,1}^c, k_{1,1}(1), k_{1,1,1}(1), k_{1,1,1}(4), m(1),k_h,k_1(2)$, which are in codimension $4$,
	 	\item[] $ k_{1,1}(3)$, which is in codimension $5$.
	 \end{itemize}
\end{theorem}

\begin{remark}\label{rem:relations-Mbar}
	We write the relations explicitly. 
	\begin{itemize}

		\item[]
		\begin{equation*}
			\begin{split}
			[\cA_2]=24(\lambda_1^2-2\lambda_2)
			\end{split}
		\end{equation*}
		\item[]
		\begin{equation*}
			\begin{split}
				[\cA_3]&= 36\lambda_1^3 + 10\lambda_1^2H + 21\lambda_1^2\delta_1 - 92\lambda_1\lambda_2 - 4\lambda_1H^2 + 18\lambda_1H\delta_1 + \\ & +72\lambda_1\delta_1^2 +
				+ 88\lambda_1\delta_{1,1} - 20\lambda_2H + 56\lambda_3 - 2H^3 + 9H^2\delta_1 + 54H\delta_1^2 + \\ & + 87\delta_1^3 -
				4\delta_1\delta_{1,1}+ 56\delta_{1,1,1}
			\end{split}
		\end{equation*}

		\item[]
		\begin{equation*}
			\begin{split}
				[\cA_3^1]=\frac{H}{2}(\lambda_1+3H+\delta_1)(\lambda_1+H+\delta_1)
			\end{split}
		\end{equation*}
		 
		\item[]
		\begin{equation*}
			\begin{split}
				\delta_1^c&=6(\lambda_1^2\delta_1 + 2\lambda_1H\delta_1 + 6\lambda_1\delta_1^2 + 4\lambda_1\delta_{1,1} + H^2\delta_1 + 6H\delta_1^2 -\\ & - 4H\delta_{1,1}
				+ 9\delta_1^3 - 8\delta_1\delta_{1,1} + 12\delta_{1,1,1})
			\end{split}
		\end{equation*}
		\item[]
		\begin{equation*}
			\begin{split}
			k_1(1)=\delta_1\Big(\frac{1}{4}\lambda_1^2 + \frac{1}{2}\lambda_1H + 2\lambda_1\delta_1 + \lambda_2 + \frac{1}{2}H^2 + H\delta_1^2 + \frac{7}{4}\delta_1^2 -\delta_{1,1}\Big)
			\end{split}
		\end{equation*} 
		\item[]
		\begin{equation*}
			\begin{split}
				k_{1,1}(2)=\delta_{1,1}(3\lambda_1 + H + 3\delta_1)
			\end{split}
		\end{equation*} 
		\item[] 
		\begin{equation*}
			\begin{split}
				[\cA_4]&=36\lambda_1^4 + \frac{1048}{27}\lambda_1^3H + 66\lambda_1^3\delta_1 - 92\lambda_1^2\lambda_2 - \frac{146}{81}\lambda_1^2H^2 +
				\frac{517}{9}\lambda_1^2H\delta_1 + \\ & + 207\lambda_1^2\delta_1^2 - 176\lambda_1^2\delta_{1,1} - 84\lambda_1\lambda_2H + 56\lambda_1\lambda_3 +
				\frac{16}{81}\lambda_1H^3 +\\ & + \frac{3272}{81}\lambda_1H^2\delta_1 +  \frac{1282}{9}\lambda_1H\delta_1^2 + 222\lambda_1\delta_1^3 -
				340\lambda_1\delta_1\delta_{1,1} + \\ & + 56\lambda_1\delta_{1,1,1} + 8\lambda_2H^2 + \frac{130}{27}H^4  + \frac{2041}{81}H^3\delta_1 +
				\frac{4957}{81}H^2\delta_1^2 + \\ & + \frac{2101}{27}H\delta_1^3 + 45\delta_1^4 - 72\delta_1^2\delta_{1,1}
			\end{split}
		\end{equation*}
		\item[]
		\begin{equation*}
			\begin{split}
				\delta_{1,1}^c=24\big(\delta_{1,1}(\lambda_1+\delta_1)^2+\delta_{1,1,1}\big)
			\end{split}
		\end{equation*} 
		\item[]
		\begin{equation*}
			\begin{split}
				k_{1,1}(1)=\delta_{1,1}(\delta_{1,1}-\lambda_2-(\lambda_1+\delta_1)^2)
			\end{split}
		\end{equation*} 
		\item[] 
		\begin{equation*}
			\begin{split}
				k_{1,1,1}(1)=\delta_{1,1,1}(\lambda_1 + \delta_1)
			\end{split}
		\end{equation*}
		\item[]
		\begin{equation*}
			\begin{split}
				k_{1,1,1}(4)=H\delta_{1,1,1}
			\end{split}
		\end{equation*}
		\item[]
		\begin{equation*}
			\begin{split}
				m(1)&=12\lambda_1^4 - \frac{7}{3}\lambda_1^3H + 27\lambda_1^3\delta_1 - 44\lambda_1^2\lambda_2 - \frac{706}{9}\lambda_1^2H^2 - \frac{65}{2}\lambda_1^2H\delta_1 + \\ &
				+ 84\lambda_1^2\delta_1^2 - 32\lambda_1^2\delta_{1,1} - 38\lambda_1\lambda_2H + 92\lambda_1\lambda_3 - \frac{715}{9}\lambda_1H^3 - \\ & -
				\frac{1340}{9}\lambda_1H^2\delta_1 - 25\lambda_1H\delta_1^2  + 69\lambda_1\delta_1^3 - 130\lambda_1\delta_1\delta_{1,1} + 92\lambda_1\delta_{1,1,1} + \\ & +
				6\lambda_2H^2 - \frac{46}{3}H^4 - \frac{1205}{18}H^3\delta_1  - \frac{562}{9}H^2\delta_1^2 - \frac{101}{6}H\delta_1^3 -
				54\delta_1^2\delta_{1,1}
			\end{split}
		\end{equation*} 
		\item[] 
		\begin{equation*}
			\begin{split}
				k_h&= \frac{1}{8}\lambda_1^3H + \frac{1}{8}\lambda_1^2H^2 + \frac{1}{4}\lambda_1^2H\delta_1 - \frac{1}{2}\lambda_1\lambda_2H - \frac{1}{8}\lambda_1H^3 +
				\frac{7}{8}\lambda_1H\delta_1^2 + \\ & + \frac{3}{2}\lambda_1\delta_1\delta_2   \frac{1}{2}\lambda_2H^2 + \lambda_3H  - \frac{1}{8}H^4 - \frac{1}{4}H^3\delta_1 +
				\frac{1}{8}H^2\delta_1^2 + \frac{3}{4}H\delta_1^3 + \frac{3}{2}\delta_1^2\delta_2
			\end{split}
		\end{equation*}
			\item[]
		\begin{equation*}
			\begin{split}
				k_1(2)&=\frac{1}{4}\lambda_1^3\delta_1 + \frac{1}{2}\lambda_1^2H\delta_1 + \frac{5}{4}\lambda_1^2\delta_1^2 + \frac{1}{4}\lambda_1H^2\delta_1 + \frac{3}{2}\lambda_1H\delta_1^2 +
				\frac{7}{4}\lambda_1\delta_1^3 +\\ & + \lambda_1\delta_1\delta_2 - \lambda_1\delta_3 + \lambda_3\delta_1+ \frac{1}{4}H^2\delta_1^2 + H\delta_1^3 + \frac{3}{4}\delta_1^4 +
				\delta_1^2\delta_2
			\end{split}
		\end{equation*}
		\item[]
		\begin{equation*}
			\begin{split}
				k_{1,1}(3)&=2\lambda_1^3\delta_2 + 5\lambda_1^2\delta_1\delta_2 + \lambda_1\lambda_2\delta_2 + 4\lambda_1\delta_1^2\delta_2 + \lambda_2\delta_1\delta_2 + \lambda_2\delta_3 + \lambda_3\delta_2 +\\ & +
				\delta_1^3\delta_2
			\end{split}
		\end{equation*}
	\end{itemize}
\end{remark}

\begin{remark}
	Notice that the relation $[\cA_2]$ and $\delta_1^c$ gives us that $\lambda_2$ and $\delta_{1,1,1}$ can be obtained using the other generators. 
\end{remark}

Lastly, we compare our result with the one of Faber, namely Theorem 5.2 of \cite{Fab}. Recall that he described the Chow ring of $\Mbar_3$ with rational coefficients as the graded $\QQ$-algebra defined as a quotient of the graded polynomial algebra generated by $\lambda_1$, $\delta_1$, $\delta_0$ and $\kappa_2$. We refer to \cite{Mum} for a geometric description of these cycles. The quotient ideal is generated by $3$ relations in codimension $3$ and six relations in codimension $4$. 

First of all, if we invert $7$, we have that the relation $[\cA_3]$ implies that also $\lambda_3$ is not necessary as a generator. Therefore if we tensor with $\QQ$, our description can be simplified and we end up having exactly $4$ generators, namely $\lambda_1$, $\delta_1$, $\delta_{1,1}$ and $H$ and $9$ relations. Notice that the identity 
$$ [H]=9\lambda_1 - 3\delta_1 - \delta_0$$
allows us to easily pass from the generator $H$ to the generator $\delta_0$ that was used in \cite{Fab}. Finally, Table 2 in \cite{Fab} gives us the identity
$$ \delta_{1,1}=-5\lambda_1^2+\frac{\lambda_1\delta_0}{2} +\lambda_1\delta_1+\frac{\delta_1^2}{2}+\frac{\kappa_2}{2}$$ 
which explains how to pass from the generator $\delta_{1,1}$ to the generator $\kappa_2$ used in \cite{Fab}.

Thus we can construct two morphism of $\QQ$-algebras:
$$ \phi: \QQ[\lambda_1,H,\delta_1,\delta_{1,1}]\longrightarrow \QQ[\lambda_1, \delta_0,\delta_1,\kappa_2] $$
and 
$$ \varphi: \QQ[\lambda_1,\delta_0,\delta_1,\kappa_2] \longrightarrow \QQ[\lambda_1,H,\delta_1,\delta_{1,1}]$$
which are  one the inverse of the other. A computation shows that $\phi$ sends ideal of relations to the one in \cite{Fab} and $\varphi$ sends the ideal of relations in \cite{Fab} to the one we constructed.  
\appendix
\chapter{Moduli stack of finite flat extensions}

In this appendix, we describe the moduli stack of finite flat extensions of curvilinear algebras. This is used in Chapter 3 to describe the stack of $A_n$-singularities. 

Let $m$ be a positive integer and $\kappa$ be a base field. We denote by $F\cH_m$ the finite Hilbert stack of a point, i.e. the stack whose objects are pairs $(A,S)$ where $S$ is a scheme and $A$ is a locally free finite $\cO_S$-algebra of degree $m$. We know that $F\cH_m$ is an algebraic stack, which is in fact a quotient stack of an affine scheme of finite type by a smooth algebraic group. For a more detailed treatment see Section 96.13 of \cite{StProj}. 

Given another positive integer $d$, we can consider a generalization of the stack $F\cH_m$. Given a morphism $\cX \rightarrow \cY$, one can consider the finite Hilbert stack $F\cH_d(\cX/\cY)$ which parametrizes commutative diagram of the form
$$
\begin{tikzcd}
Z \arrow[r] \arrow[d, "f"] & \cX \arrow[d] \\
S \arrow[r]                & \cY          
\end{tikzcd}
$$
where $f$ is finite locally free of degree d. This again can be prove to be algebraic. If $\cY=\spec k$, we denote it simply by $F\cH_d(\cX)$. For a more detailed treatment see Section 96.12 of \cite{StProj}.

Finally, we define $E\cH_{m,d}$ ti be the fibered category in groupoids whose objects over a scheme $S$ are finite locally free extensions of $\cO_S$-algebras $A \into B$  of degree $d$ such that $A$ is a finite locally free $\cO_S$-algebra of degree $m$. Morphisms are defined in the obvious way. Clearly the algebra $B$ is finite locally free of degree $dm$. 
 
\begin{proposition}
	The stack $F\cH_{d}(F\cH_m)$ is naturally isomorphic to $E\cH_{m,d}$, therefore $E\cH_{m,d}$ is an algebraic stack.
\end{proposition}

\begin{proof}
	The proof follows from unwinding the definitions.
\end{proof}

We want to add the datum of a section of the structural morphism $\cO_S \into B$. This can be done passing to the universal object of $F\cH_{dm}$.

Let $n$ be a positive integer and let $B_{\rm univ}$ the universal object of $F\cH_{n}$; consider $\cF_{n}:=\spec_{F\cH_{n}}(B_{\rm univ})$ the generalized spectrum of the universal algebra over $F\cH_{n}$. It parametrizes pairs $(B,q)$ over $S$ where $B \in F\cH_{n}(S)$ and $q:B \rightarrow \cO_S$ is a section of the structural morphism $\cO_S \into B$. 
\begin{definition}
	We say that a pointed algebra $(A,p) \in \cF_n$ over an algebraically closed field $k$ is linear if $\dim_k m_p/m_p^2 \leq 1$, where $m_p$ is the maximal ideal associated to the section $p$. We say that $(A,p)$ is curvilinear if $(A,p)$ is linear and $\spec A=\{p\}$.
\end{definition}

 We consider the closed substack defined by the $1$-st Fitting ideal of $\Omega_{A|\cO_S}$ in $\cF_{n}$. This locus parametrizes non-linear pointed algebras . Therefore we can just consider the open complement, which we denote by $\cF_{n}^{\rm lin}$. We can inductively define closed substacks of $\cF_{n}^{\rm lin}$ in the following way: suppose $\cS_h$ is defined, then we can consider $\cS_{h+1}$ to be the closed substack of $\cS_h$ defined by the $0$-th fitting ideal of $\Omega_{\cS_h|F\cH_{n}}$. We set $\cS_1=\cF_{n}^{\rm lin}$. It is easy to prove that the geometric points of $\cS_h$ are pairs $(A,p)$ such that $A$ localized at $p$ has length greater or equal than $h$. Notice that this construction stabilizes at $h=n$ and the geometric points of $\cS_h$ are exactly the curvilinear pointed algebras. Finally, we denote by $\cF_{n}^{c}:=\cS_{n}$ the last stratum. As it is a locally closed substack of an algebraic stacks of finite type, it is algebraic of finite type too. 
 
 \begin{lemma}\label{lem:loc-triv-alg}
 	If $(B,q) \in \cF_{n}^c(S)$ for some scheme $S$, then it exists an \'etale cover $S'\rightarrow S$ such that $B\otimes_S S' \simeq \cO_{S'}[t]/(t^{n})$ and $q\otimes_S S'=q_0\otimes S'$ where
 	$$q_0:\kappa[t]/(t^{n}) \longrightarrow \kappa$$
 	is defined by the assocation $t\mapsto 0$.
 \end{lemma}
  
\begin{proof}
 We are going to prove that $\cF_{n}^c$ has only one geometric point (up to isomorphism) and its tangent space is trivial. Thus the thesis follows, by a standard argument in deformation theory.  
 
 Suppose then $S=\spec k$ is the spectrum of an algebraically closed field and $B$ is a finite $k$-algebra (of degree $n$). Because $(B,q)$ is linear, we have that $\dim_k(m_q/m_q^2)\geq 1$. We can construct then a surjective morphism 
 $$ k[[t]] \longrightarrow B_{m_q}$$
 whose kernel is generated by the monomial $t^n$. We have then $(B,q)\simeq (k[t]/t^{n'},q_0)$ for some positive integer $n'$. Because $B$ is local of length $n$, we get that $n'=n$.

Let us study the tangent space of $\cF_n^c$ at the pointed algebra $(k[t]/(t^{n}),q_0)$ where $k/\kappa$ is a field. We have that any deformation $(B_{\epsilon},q_{\epsilon})$ of the pair $(k[t]/t^n,q_0)$ is of the form
$$ k[t,\epsilon]/(p(t,\epsilon))$$
where $p(t,0)=t^{n}$. Because the section is defined by the association $t\mapsto b\epsilon$, we have also that $p(b\epsilon,\epsilon)=p(0,\epsilon)=0$. It is easy to see that $(B_{\epsilon},q_{\epsilon}) \in \cF^c_{n}$ only if $$p(b\epsilon,\epsilon)=p'(b\epsilon,\epsilon)=\dots=p^{(n-1)}(b\epsilon,\epsilon)=0$$
where the derivatives are done over $k[\epsilon]/(\epsilon^2)$, thus $p(t,\epsilon)=(t-b\epsilon)^{n}$. The algebra obtained is clearly isomorphic to trivial one.
\end{proof} 

Let $G_{n}:=\underaut\Big(\kappa[t]/(t^{n}),q_0\Big)$ be the automorphism group of the trivial algebra. One can describe $G$ as the semi-direct product of $\gm$ and a group $U$ which is isomorphic to an affine space of dimension $n-2$. 

\begin{corollary}\label{cor:descr-finite-alg}
	We have that $\cF_{n}^c$ is isomorphic to $\cB G_{n}$, the classifying stack of the group $G_{n}$.
\end{corollary}
  
 We denote by $\cE_{m,d}^c$ the fiber product $E\cH_{m,d}\times_{F\cH_{dm}}\cF_{dm}^c$. We get the morphism of algebraic stacks
$$ \cE_{m,d}^c \longrightarrow \cF_{dm}^c$$ 
defined by the association $(S,A\into B \rightarrow \cO_S) \mapsto (S,B\rightarrow \cO_S)$. Notice that the morphism is faithful, therefore representable by algebraic spaces.  Clearly the composite $A\rightarrow \cO_S$ is still a section. 

We now study the stack $\cE_{m,d}^c$.

\begin{lemma}\label{lem:triv-ext}
	If $(A\into B,q) \in \cE_{m,d}^c(S)$ for some scheme $S$, then it exists an \'etale cover $\pi:S'\rightarrow S$ such that 
	$$\pi^*\Big(A\into B,q\Big) \simeq \Big(\phi_d:\cO_{S'}[t]/(t^{m})\into \cO_{S'}[t]/(t^{dm}),q_0\otimes_{S}S'\Big)$$
	where $\phi_d(t)=t^dp(t)$ with $p(0) \in \cO_{S'}^{\times}$.
\end{lemma}

\begin{proof}
	First of all, an easy computation shows that any finite flat extension of pointed algebras of degree $d$ $$\cO_{S}[t]/(t^{m})\into \cO_{S}[t]/(t^{dm})$$ 
	is of the form $\phi_d$ for any scheme $S$. Therefore, it is enough to prove that if $A\into B$ is finite flat of degree $d$, then $B$ curvilinear implies $A$ curvilinear. In analogy with the proof of \Cref{lem:loc-triv-alg}, we prove the statement for $S=\spec k$ and then for $S=\spec k[\epsilon]/(\epsilon^2)$ with $k$ algebraically closed field.
	
	Firstly, suppose $S=\spec k$ with $k$ algebraically closed. By \Cref{lem:loc-triv-alg}, we know that $(B,q)\simeq (k[t]/(t^{dm}),q_0)$. We need to prove now that also $A$ is curvilinear. Clearly $A$ is local because of the going up property of flatness and we denote by $m_A$ its maximal ideal. 
	
	If we tensor $A\into B$ by $A/m_A$, we get $k\into k[t]/m_Ak[t]$. However flatness implies that the extension $k\into k[t]/(m_Ak[t])$ has degree $d$, therefore because $k[t]$ is a PID, it is clear that $m_Ak[t]\subset m_q^{d}$ and the morphism $m_A/m_A^2\rightarrow m_q^d/m_q^{d+1}$ is surjective. Let $a \in m_A$ be an element whose image is $t^d$ modulo $t^{d+1}$. If we now consider $A/a\into B/aB\simeq k[t]/(t^d)$, by flatness it still has length $d$, but this implies $A/a\simeq k$ or equivalently $m_A=(a)$. Therefore $A$ is curvilinear too. 
	
	Suppose now $S=\spec k[\epsilon]$.
    We know that, given a morphism of schemes $X\rightarrow Y$, we have the exact sequence
	$$ {\rm Def}_{X\rightarrow Y} \longrightarrow {\rm Def}_X \oplus {\rm Def}_Y \longrightarrow \ext^1_{\cO_X}(Lf^*{\rm NL}_Y,\cO_X)$$
	where ${\rm NL}_Y$ is the naive cotangent complex of $Y$. We want to prove that if $X\rightarrow Y$ is the spectrum of the extension $A\into B$,  then the morphism 
	$$ {\rm Def}_X \longrightarrow \ext^1_{\cO_X}(Lf^*{\rm NL}_Y,\cO_X)$$ 
    is injective. This implies the thesis because ${\rm Def}_Y=0$.

    We can describe the morphism
    $$
    {\rm Def}_X \longrightarrow \ext^1_{\cO_X}(Lf^*{\rm NL}_Y,\cO_X)
    $$ explicitly using the Schlessinger's functors $T^i$. More precisely, it can be described as a morphism
    $$T^1(A/k,A)\rightarrow T^1(A/k,B);$$
	an easy computation shows that $T^1(A/k,A)\simeq k[t]/(t^{m-1})$ and $T^1(A/k,B)\simeq k[t]/(t^{d(m-1)})$. Through these identifications, the morphism  $$T^1(A/k,A)\rightarrow T^1(A/k,B)$$ is exactly the morphism $\phi_d$ defined earlier, namely it is defined by the association $t\mapsto t^dp(t)$ with $p(0)\neq 0$. The injectivity is then straightforward.
\end{proof}

Now we are ready to describe the morphism $\cE_{m,d} \rightarrow \cF_{md}$. Let us define $A_0:=\kappa[t]/(t^m)$ and $B_0:=\kappa[t]/(t^{md})$. 
Let $E_{m,d}$ be the category fibered in groupoids whose objects are $$\Big(S,(A\into B,q) ,\phi_A:(A,q)\simeq (A_0\otimes S, q_0 \otimes S), \phi_B: (B,q )\simeq (B_0 \otimes S,q_0\otimes S)\Big)$$ where $(A\into B,q) \in \cE_{m,d}(S)$. The morphism are defined in the obvious way. It is easy to see that $E_{m,d}$ is in fact fibered in sets. As before, we set $G_m:=\underaut(A_0,q_0)$ and $G_{dm}:=\underaut(B_0,q_0)$. Clearly we have a right action of $G_{dm}$ and a left action of $G_m$ on $E_{m,d}$.

\begin{proposition}
	We have the follow isomorphism of fibered categories
	$$ \cE_{m,d} \simeq [E_{m,d}/G_m\times G_{md}]$$ 
	and through this identification the morphism $\cE_{m,d} \rightarrow \cF_{dm}$  is just the projection to the classifying space $\cB G_{md}$.
\end{proposition}
\begin{proof}
	It follows from \Cref{lem:triv-ext}.
\end{proof}

\Cref{lem:triv-ext} also tells us how to describe $E_{m,d}$: the map $\phi_d$  is completely determined by $p(t) \in \cO_S[t]/(t^{d(m-1)})$ with $p(0)\in \cO_S^{\times}$. Therefore, we have a morphism 
$$(\AA^1\setminus 0)\times \AA^{d(m-1)-1} \longrightarrow E_{m,d}$$ 
which is easy to see that it is an isomorphism. Consider now the subscheme of $E_{m,d}$ defined as $$V:=\{ f \in (\AA^1\setminus 0)\times \AA^{d(m-1)-1}|\, a_0=1, \, a_{kd}(f)=0\quad {\rm for }\quad k=1,\dots,m-2\}$$
where $a_{l}(f)$ is the coefficient of $t^{l}$ of the polynomial $f$. Clearly $V$ is an affine space of dimension $(m-1)(d-1)$.

\begin{lemma}\label{lem:descr-affine-bundle}
	In the situation above, we have the isomorphism

	$$ V \simeq [E_{m,d}/G_m],$$
	
	in particular the morphism $\cE_{m,d} \rightarrow \cF_{md}$ is an affine bundle of dimension $(m-1)(d-1)$.
\end{lemma}

\begin{proof}
	Consider the morphism 
	$$ G_m \times V \longrightarrow E_{m,d}$$
	which is just the restriction of the action of $G_m$ to $V$. A straightforward computation shows that it is an isomorphism. The statement follows.
\end{proof}

\chapter{Pushout and blowups in families}

In this appendix, we discussion two well-known constructions: pushout and blowup. Specifically, we study these two constructions in families and give conditions to get flatness and compatibility with base change. Moreover, we study when the two constructions are one the inverse of the other.

\begin{lemma}\label{lem:pushout}
	Let $S$ be a scheme. Consider three schemes $X$,$Y$ and $Y'$ which are flat over $S$. Suppose we are given $Y\hookrightarrow X$ a closed immersion and $Y\rightarrow Y'$ a finite flat morphism. Then the pushout $Y'\bigsqcup_Y X$ exists in the category of schemes, it is flat over $S$ and it commutes with base change. 
	
	Furthermore, if $X$ and $Y'$ are proper and finitely presented scheme over $S$, the same is true for $Y' \bigsqcup_Y X$.
\end{lemma}

\begin{proof}
	The existence of the pushout follows from Proposition 37.65.3 of \cite{StProj}. In fact, the proposition tells us that the morphism $Y'\rightarrow Y'\bigsqcup_Y X$ is a closed immersion and $X \rightarrow Y'\bigsqcup_Y X$ is integral, in particular affine. It is easy to prove that it is in fact finite and surjective because $Y\rightarrow Y'$ is finite and surjective.  Let us prove that $Y'\bigsqcup_Y X \rightarrow S$ is flat. Because flatness is local condition and all morphisms are affine, we can reduce to a statement of commutative algebra.
	
	Namely, suppose we are given $R$ a commutative ring and $A$,$B$ two flat algebras. Let $I$ be an ideal of $A$ such that $A/I$ is $R$-flat and $B\hookrightarrow A/I$ be a finite flat extension. Then $B\times_{A/I}A$ is $R$-flat. To prove this, we complete the fiber square with the quotients 
	$$
	\begin{tikzcd}
	0 \arrow[r] & B\times_{A/I}A \arrow[d, two heads] \arrow[r, hook] & A \arrow[d, two heads] \arrow[r] & Q' \arrow[d] \arrow[r] & 0 \\
	0 \arrow[r] & B \arrow[r, hook]                                   & A/I \arrow[r]                    & Q \arrow[r]                                 & 0
	\end{tikzcd}
	$$
	and we notice that $Q'\rightarrow Q$ is an isomorphism. Because the extension $B\hookrightarrow A/I$ is flat, then $Q$ (thus $Q'$) is $R$-flat and the $R$-flatness of $A$ and $Q'$ implies the flatness of $B\times_{A/I}A$.  
	
	Suppose now we have a morphism $T\rightarrow S$. This induces a natural morphism 
	$$\phi_T:Y'_T \bigsqcup_{Y_T} X_T \rightarrow (Y'\bigsqcup_Y X)_T$$ where by $(-)_T$ we denote the base change $(-)\times_S T$. Because being an isomorphism is a local condition, we can reduce to the following commutative algebra's statement. Suppose we have a morphism $R\rightarrow \widetilde{R}$; we can consider the same morphism of exact sequence of $R$-modules as above and tensor it with $\widetilde{R}$. We denote by $\widetilde{(-)}$ the tensor product $(-)\otimes_R \widetilde{R}$. The flatness of $Q$ implies that we get that the commutative diagram
	$$
	\begin{tikzcd}
	0 \arrow[r] & \widetilde{B\times_{A/I}A} \arrow[d, two heads] \arrow[r, hook] & \widetilde{A} \arrow[d, two heads] \arrow[r] & \widetilde{Q'}\arrow[d] \arrow[r] & 0 \\
	0 \arrow[r] & \widetilde{B} \arrow[r, hook]                                     & \widetilde{A/I} \arrow[r]                    & \widetilde{Q} \arrow[r]            & 0
	\end{tikzcd}
	$$
	is still a morphism of exact sequence. Because also $\widetilde{B} \times_{\widetilde{A/I}}  \widetilde{A}$ is the kernel of $\widetilde{A} \rightarrow \widetilde{Q'}$, we get that $\phi_T$ is an isomorphism. 
	
	Finally, using the fact that the pushout is compatible with base change, we get that we can apply the same strategy used in \Cref{lem:quotient} to reduce to the case $S=\spec R_0$ with $R_0$ of finite type over $k$. Thus we can use Proposition 37.65.5 of \cite{StProj} to prove that the pushout (in the situation when $Y'\rightarrow Y$ is flat) preserves the property of being proper and finitely presented.
\end{proof}

The second construction we want to analyze is the blow-up.

\begin{lemma}\label{lem:blowup}
	Let $X/S$ be a flat, proper and finitely presented morphism of schemes and let $I$ be an ideal of $X$ such that $\cO_{X}/I^{m}$ is flat over $S$ for every $m\geq 0$. Denote by ${\rm Bl}_IX\rightarrow X$ the blowup morphism, then ${\rm Bl}_IX \rightarrow S$ is flat, proper and  finitely presented and its formations commutes with base change.
\end{lemma}

\begin{proof}
	The flatness of $\cO_X/I^m$ implies the flatness of $I^m$, therefore $\oplus_{m\geq 0}I^m$ is clearly flat over $S$. The universal property of the blowup implies that it is enough to check that the formation of the blowup commutes with base change when we restrict to the fiber of a closed point $s \in S$. Therefore it is enough to prove that the inclusion $m_sI^m\subset m_s\cap I^m$ is an equality for every $m\geq 0$ with $m_s$ the ideal associated to the closed point $s$. The lack of surjectivity of that inclusion is encoded in $\tor_S^1(\cO_X/I^m,\cO_S/m_S)$ which is trivial due to the $S$-flatness of $\cO_X/I^m$. The rest follows from classical results.
\end{proof}
\begin{remark}
	Notice that the flatness of the blowup follows from the flatness of $I^m$ for every $m$ whereas we need the flatness of $\cO_X/I^m$ to have the compatibility with base change.
\end{remark}
\begin{lemma}\label{lem:cond-diag}
	Let $R$ be a ring and $A\into B$ be an extension of $R$-algebras. Suppose $I\subset A$ is an ideal of $A$ such that $I=IB$. Then the following commutative diagram
	$$
	\begin{tikzcd}
	A \arrow[r, hook] \arrow[d, two heads] & B \arrow[d, two heads] \\
	A/I \arrow[r, hook]                    & B/I                   
	\end{tikzcd}
	$$
	is a cartesian diagram of $R$-algebras. Furthermore, suppose we have a cartesian diagram of $R$-algebras
	$$
	\begin{tikzcd}
	A:=\widetilde{A}\times_{B/I}B \arrow[d] \arrow[r, hook] & B \arrow[d] \\
	\widetilde{A} \arrow[r, hook]                       & B/I                      
	\end{tikzcd}
	$$
	then the morphism $A\rightarrow \widetilde{A}$ is surjective and its kernel coincide (as an $R$-module) with the ideal $I$.
\end{lemma}

\begin{proof}
	It follows from a straightforward computation in commutative algebra.
\end{proof}	
 
 Finally, we prove that the two constructions are one the inverse of the other.
 
\begin{proposition}\label{prop:pushout-blowup}
	Suppose we are given a diagram
	$$
	\begin{tikzcd}
	\widetilde{D} \arrow[d] \arrow[r, hook] & \widetilde{X} \\
	D                                       &              
	\end{tikzcd}
	$$ of proper, flat, finitely presented schemes over $S$ such that $\widetilde{D} \rightarrow D$ is a finite flat morphism and $\widetilde{D}\into \widetilde{X}$ is a closed immersion of an effective Cartier divisor. Consider the pushout $X:=\widetilde{X}\amalg_{\widetilde{D}}D$ as in \Cref{lem:pushout} and denote by $I_D$ the ideal associated with the closed immersion $D\into X$. Then the pair $(X,I_D)$ over $S$ verifies the hypothesis of \Cref{lem:blowup}. Furthermore, if we denote by $(\overline{X},\overline{D})$ the blowup of the pair $(X,D)$, there exists a unique isomorphism $(\widetilde{X},\widetilde{D}) \simeq (\overline{X},\overline{D})$ of pairs over $(X,D)$.
\end{proposition}

\begin{proof}
	Consider the pushout diagram over $S$
	$$
	\begin{tikzcd}
	\widetilde{D} \arrow[d] \arrow[r, hook] & \widetilde{X} \arrow[d] \\
	D \arrow[r, hook]                       & X;                     
	\end{tikzcd}
	$$
	because every morphism is finite and flatness is a local condition, we can restrict ourself to the affine case and \Cref{lem:cond-diag} assures us that $I_D=I_{\widetilde{D}}$, and in particular $I_D^n=I_{\widetilde{D}}^n$ for every $n\geq 1$. Because $I_{\widetilde{D}}$ is a flat Cartier divisor, the same is true for its powers.
	
	Regarding the second part of the statement, we know that the unicity and existence of the morphism $$(\overline{X},\overline{D}) \longrightarrow (\widetilde{X},\widetilde{D})$$ 
	are assured by the universal property of the blowup. As being an isomorphism is a local property, we can reduce again to the affine case (all the morphisms involved are finite). We have an extension of algebras $A\into B$ with an ideal $I$ of $A$ such that $I=IB$ and $I$ is free of rank $1$ as a $B$-module. Therefore we can describe the Rees algebra as follows
	$$R_A(I):=\bigoplus_{n\geq 0}I^n=A\oplus tB[t]$$
	because $I$ is free of rank $1$ over $B$. It is immediate to see that the morphism $\spec B \rightarrow \proj_A(R_A(I))$ is an isomorphism over $\spec A$.
\end{proof}

\begin{proposition}\label{prop:blowup-pushout}
Let $D\into X$ be a closed immersion of proper, flat, finitely presented schemes over $S$ such that the ideal $I_D^n$ is $S$-flat for every $n\geq 1$, consider the blowup $b:\widetilde{X}:={\rm Bl}_DX\rightarrow X$ and denote by $\widetilde{D}$ the proper transform of $D$. Suppose that $\widetilde{D}\rightarrow D$ is finite flat  (in particular the morphism $b$ is finite birational). Moreover, suppose that the ideal $I_D$ is cointained in the conductor ideal $J_b$ of the morphism $b$. Then it exists a unique isomorphism  $\widetilde{X}\amalg_{\widetilde{D}} D\rightarrow X$, which makes everything commutes.  
\end{proposition}

\begin{proof}
	As in the previous proposition, the existence and unicity of the morphism are conseguences of the universal property of the pushout. Therefore as all morphisms are finite we can restrict to the affine case. The fact that the ideal $I_D$ is contained in the conductor ideal implies that we can use \Cref{lem:cond-diag} and conclude. We leave the details to the reader.
\end{proof}

\chapter{Discriminant relations}	

In this appendix, we generalize Proposition 4.2 of \cite{EdFul}. We do not need this result in its full generality in our work, only the formulas in \Cref{rem:gener}.

First of all, we set some notations. Everything is considered over a base field $\kappa$. Let $T$ be the $2$-dimensional split torus $\gm^2$ which embeds in $\GL_2$ as the diagonal matrices and let $E$ be the standard representation of $\GL_2$. Let $n$ be  positive integer. We denote by $\AA(n)$ the $n$-th symmetric power of the dual representation of $E$ and by $\PP^n$  the projective bundle $\PP(\AA(n))$. We denote by $\xi_n$ the hyperplane sections of $\PP^n$. Moreover, we denote by $h_i$ the element of $\ch_T(\PP^n)$ associated to the hyperplane defined by the equation $a_{i,n-i}=0$ for every $i=0,\dots,n$ where $a_{i,n-i}$ is the coordinate of $\PP^n$ associated to the coefficient of $x_0^ix_1^{n-i}$ and $x_0,x_1$ is a $T$-base for $E^{\vee}$. We have the identity
$$ h_i=\xi_n -(n-i)t_0- it_1$$ 
where $t_0,t_1$ are the generators of $\ch(\cB T)$ (acting respectively on $x_0$ and $x_1$). Let $\tau \in \ch(\cB T)$ be the element $t_0-t_1$, then the previous identity can be written as 
$$ h_i=h_0+i\tau.$$
Notice that we can reduce to the $T$-equivariant setting exactly as the authors do in \cite{EdFul}, because $\GL_2$ is a special group and therefore the morphism 
$$ \ch(\cB \GL_2) \longrightarrow \ch(\cB T)$$ 
is injective
 
Let $N$ and $k$ be two positive integers such that $k\leq N$. Inside $\PP^N$, we can define a closed subscheme $\Delta_k$ parametrizing (classes of) homogeneous forms in two variables $x_0,x_1$ which have a root of multiplicity at least $k$. 

We want to study the image of the pushforward of the closed immersion 
$$ \Delta_k \into  \PP^N;$$
we have the description of the Chow ring of $\PP^N$ as the quotient 
$$ \ch_{\GL_2}(\PP^N) \simeq  \ZZ[c_1,c_2,\xi_N]/(p_N(\xi_N))$$
where $p_N(\xi_N)$ is a monic polynomial in $h$ of degree $N+1$ with coefficients in $\ch(\cB\GL_2)\simeq \ZZ[c_1,c_2]$. The coefficient of $\xi_N^i$ is the $(N-i)$-th Chern class of the $\GL_2$-representation $\AA(N)$ for $i=0,\dots,N$. 

Exactly as it was done in \cite{Vis3} and generalized in \cite{EdFul2}, we introduce the multiplication morphism for every positive integer $r$ such that $r\leq N/m$
$$ \pi_r: \PP^r \times \PP^{N-kr} \longrightarrow \PP^N$$ 
defined by the association $(f,g)\mapsto f^kg$.
The $\GL_2$-action on the left hand side is again induced by the symmetric powers of the dual of $E$. Notice that we are not assuming that $N$ is a multiple of $k$. We have an analogue of Proposition 3.3 of \cite{Vis3} or Proposition 4.1 of \cite{EdFul2}.

\begin{proposition}
	Suppose that the characteristic of $\kappa$ is greater than $N$, then the disjoint union of the morphisms $\pi_r$ for $1\leq r\leq N/k$ is a Chow envelope for $\Delta_k\into \PP^N$. 
\end{proposition}

Therefore it is enough to study the image of the pushforward of $\pi_r$ for $r\leq N/k$.  We have that $\pi^*(\xi_N)=k\xi_r+\xi_{N-kr}$, therefore for a fixed $r$ we have that the image of $\pi_{r,*}$ is generated as an ideal by $\pi_{r,*}(\xi_r^m)$ for $0\leq m\leq r$.

\begin{remark}
	Fix $r\leq N/k$. We have that 
	$$ \pi_{r,*}(\xi_r^m) \in \big(\pi_{r,*}(1), \pi_{r,*}(h_0), \dots, \pi_{r,*}(h_0\dots h_{m-1})\big)$$ 
	in $\ch_T(\PP^r)$ for $m\leq r$. In fact, we have 
	$$ h_0\dots h_{m-1}= \prod_{i=0}^{m-1} (\xi_r-(r-i)t_0-it_1) = \xi_r^m + \sum_{i=0}^{m-1}\alpha_i\xi_r^i $$
	with $\alpha_i \in \ch(\cB T)$. Therefore we can prove it by induction on $m$. 
\end{remark}

Therefore, it is enough to describe the ideal generated by $\pi_{r,*}(h_0\dots h_m)$ for $1\leq r\leq N/k$ and $-1\leq m \leq r-1$. We define the element associated to $m=-1$ as $\pi_{r,*}(1)$. 

Our goal is to prove that the ideal is in fact generated by $\pi_{1,*}(1)$ and $\pi_{1}(h_0)$. To do so, we have to introduce some morphisms first. 

Let $n$ be an integer and $\rho_n:(\PP^1)^{\times n} \rightarrow \PP^n$ by the $n$-fold product morphism, which is an $S_n$-torsor, where $S_n$ is the $n$-th symmetric group. Furthermore, we denote by $\Delta^n:\PP^1 \into (\PP^1)^{\times n}$ the small diagonal in the $n$-fold product, i.e. the morphism defined by the association $f \mapsto (f,f,\dots,f)$. We denote by $h_i$  the fundamental class of $[\infty]:=[0:1]$ in the Chow ring of the $i$-th element of the product $(\PP^1)^n$ (and by pullback in the Chow ring of the product).
\begin{remark}
	Notice that we are using the same notation for two different elements: if we are in the projective space $\PP^n$, $h_i$ is the hyperplane defined by the vanishing of the $i+1$-th coordinate of $\PP^n$. On the contrary, if we are in the Chow ring of the product $(\PP^1)^n$, it represents the subvariety defined as the pullback through the $i$-th projection of the closed immersion $\infty \into \PP^1$. Notice that $\rho_{n,*}(h_0\dots h_s)$ is equal  to $s! (n-s)! h_0\dots h_s$ for every $s\leq n$.
\end{remark}
We have a commutative diagram of finite morphisms 
$$
\begin{tikzcd}
	(\PP^1)^r \times (\PP^1)^{N-kr} \arrow[r, "\alpha_r^k"] \arrow[d, "\rho_r \times \rho_{N-kr}"'] & (\PP^1)^N \arrow[d, "\rho_N"] \\
	\PP^r \times \PP^{N-kr} \arrow[r, "\pi_r"]                                                    & \PP^N                        
\end{tikzcd}
$$  
where $\alpha_r^k= (\Delta^k)^{\times r} \times \id_{(\PP^1)^{N-kr}}$.  We can use this diagram to have a concrete description of $\pi_{r,*}(h_0\dots h_m)$. In order to do so, we first need the following lemma to describe the fundamental  class of the image of $\alpha_r^k$. 

\begin{lemma}\label{lem:k-diag}
We have the following identity 
	$$ [\Delta^k]= \sum_{j=0}^{k-1} \tau^{k-1-j}\sigma_j^k(h_1,\dots,h_k)$$ 
	in the Chow ring of $(\PP^1)^{\times k}$ for every $k \geq 2$, where $\sigma_j^k(-)$ is the elementary symmetric function with $k$ variables of degree $j$. 
\end{lemma}

\begin{proof}
The diagonal $\Delta^k$ is equal to the complete intersection of the hypersurfaces of $(\PP^1)^k$
of equations $x_{0,i}x_{1,i+1}-x_{0,i+1}x_{1,i}$ for $1\leq i \leq k-1$ (we are denoting  by $x_{0,i},x_{i,1}$ the two coordinates of the $i$-th factor of the product). Therefore we have 
$$ \Delta^k=\prod_{i=1}^{k-1}(h_i+h_{i+1}+\tau). $$
Notice that in the Chow ring of $(\PP^1)^k$ we have $k$ relations of degree $2$ which can be written as $h_i^2+\tau h_i=0$ for every $i=1,\dots,k$. 

The case $k=2$ was already proven in Lemma 3.8 of \cite{Vis3}. We proceed by induction on $k$. We have
$$\Delta^{k+1}=\prod_{i=1}^{k}(h_i+h_{i+1}+\tau)=(h_{k+1}+h_k+\tau) \Delta^k$$ 
and thus by induction 
$$\Delta^{k+1}=\sum_{i=0}^{k-1}h_{k+1} \tau^{k-1-i}\sigma_i^k + \sum_{i=0}^{k-1}  \tau^{k-i}\sigma_i^k  +\sum_{i=0}^{k-1} h_k \tau^{k-1-i}\sigma_i^k.$$ 
Recall that we have the relations $$\sigma_j^k(x_1,\dots,x_k)=x_k\sigma_{j-1}^{k-1}(x_1,\dots,x_{k-1}) + \sigma_j^{k-1}(x_1,\dots,x_{k-1})$$ between elementary symmetric functions (with $\sigma_j^k=0$ for $j>k$), therefore we have
$$ \sum_{i=0}^{k-1} \tau^{k-1-i}  h_k\sigma_i^k = \sum_{i=0}^{k-1} \tau^{k-1-i}(h_k^2\sigma_{i-1}^{k-1}+h_k\sigma_{i}^{k-1})=\sum_{i=0}^{k-1} \tau^{k-1-i}h_k(-\tau \sigma_{i-1}^{k-1} +\sigma_i^{k-1})$$ 
where we used the relation $h_k^2+\tau h_k=0$ in the last equalities. Therefore we get
\begin{equation*}
	\begin{split}
		\Delta^{k+1}&=\sum_{i=0}^{k-1} \tau^{k-1-i}(h_{k+1}\sigma_i^k+h_k\sigma_i^{k-1}) + \sum_{i=0}^{k-1} \tau^{k-i}(\sigma_i^k-h_k\sigma_{i-1}^{k-1})= \\ & =\sum_{i=0}^{k-1}\tau^{k-1-i}(h_{k+1}\sigma_i^k+h_k\sigma_i^{k-1}) + \sum_{i=0}^{k-1} \tau^{k-i}\sigma_i^{k-1}
	\end{split} 
\end{equation*}
Shifting the index of the last sum, it is easy to get the following identity
$$ \Delta^{k+1}= h_{k+1}\sigma_{k-1}^k+h_k\sigma_{k-1}^{k-1} + \tau^k +\sum_{i=0}^{k-2}\tau^{k-1-i}(h_{k+1}\sigma_i^k+h_k\sigma_i^{k-1}+\sigma_{i+1}^{k-1})$$
and the statement follows from shifting the last sum again and from using the relations between the symmetric functions (notice that $h_k\sigma_{k-1}^{k-1}=\sigma_k^k$).
\end{proof}

\begin{remark}
	We define by $\theta_{m,r}$ the $T$-equivariant closed subvariety of $(\PP^1)^r \times (\PP^1)^{N-kr}$ of the form 
	$$\theta_{m,r}:=\infty^{m+1} \times (\PP^1)^{r-(m+1)} \times (\PP^1)^{N-kr}$$
	induced by the $T$-equivariant closed immersion $\infty \into \PP^1$.
	We  have that $$(\rho_r \times \rho_{N-kr})_*(\theta_{m.r})=(r-(m+1))! (N-kr)! h_0\dots h_m$$ for every $m\leq r-1$. 
\end{remark}

From now on, we set $d:=r-(m+1)\geq 0$. Thanks to the remark and the commutativity of the diagram we constructed, we can computate the pushforward $\rho_{N,*}\alpha_{r,*}^k(\theta_{m,r})$ and then divide it by $d!(N-kr)!$ to get $\pi_{r,*}(h_0\dots h_m)$.

We denote by $\alpha_l^{(k,d)}$ the integer
$$ \alpha_{(k,d)}:=\sum_{j_1+\dots+j_d=l}^{0\leq j_s \leq k-1}\binom{k}{j1}\dots \binom{k}{j_d}$$
and by $\beta_l^{(k,m,r)}$ the integer (because $l\leq d(k-1)$)
$$ \beta_l^{(k,m,r)}:=\frac{(N-(m+1)k-l)!}{(N-kr)! d!}.$$

\begin{lemma}\label{lem:pi-r-1}
	We get the following equality 
	$$ \pi_{r,*}(h_0\dots h_m) = \sum_{l=0}^{d(k-1)} \alpha_l^{(k,d)}\beta_{l}^{(k,m,r)} \tau^{d(k-1)-l} h_0\dots h_{(m+1)k+l-1}$$
	in the $T$-equivariant Chow ring of $\PP^N$.
\end{lemma} 

\begin{proof}
	Thanks to \Cref{lem:k-diag}, we have that 
	\begin{equation*}
		\begin{split}
			&\alpha_{r,*}^k(\theta_{m,r})=[\infty^{k(m+1)} \times (\Delta^k)^d \times (\PP^1)^{N-kr}]= \\ & =h_1\dots h_{(m+1)k} \prod_{i=1}^d \Big(\sum_{j=0}^{k-1} \tau^{k-1-j}\sigma_j^{k}(h_{(m+j)k+1},h_{(m+j)k+2},\dots,h_{(m+j+1)k})\Big);
		\end{split}
	\end{equation*}
	we need to take the image through $\rho_{N,*}$ of this element. However, we have that $\rho_{N,*}(h_{i_1} \dots h_{i_s})=\rho_{N,*}(h_1\dots h_s)$ for every $s$-uple of $(i_1,\dots,i_s)$ of distinct indexes because $\rho_N$ is a $S_N$-torsor. Therefore a simple computation shows that $\rho_{N,*}\alpha_{r,*}^k(\theta_{m,r})$ has the following form:
	$$\sum_{l=0}^{d(k-1)}\Big(\sum_{j_1+\dots+j_d=l}^{0\leq j_s \leq k-1}\binom{k}{j_1}\dots \binom{k}{j_d} \Big)(N-(m+1)k-l)!  \tau^{d(k-1)-l} h_0 \dots   h_{(m+1)k+l-1}.$$
	The statement follows.
\end{proof}

\begin{remark}\label{rem:gener}
	Notice that the expression makes sense also per $m=-1$ and in fact we get a description of the $\pi_{r,*}(1)$. 
	
	Let us describe the case $r=1$. Clearly we only have $d=1$ (or $m=-1$) and $d=0$ (or $m=0$). If $d=0$, the formula gives us 
	$$\pi_{1,*}(h_0)=h_0\dots h_{k-1} \in \ch_T(\PP^N);$$
	if $d=1$ a simple computation shows 
	$$ \pi_{1,*}(1)=\sum_{l=0}^{k-1}(k-l)!\binom{k}{l}\binom{N-k}{N-l}\tau^{k-1-l}h_0\dots h_{l-1} \in \ch_T(\PP^N).$$
	These two formula gives us the $T$-equivariant class of these two elements. As matter of fact, $\pi_{1,*}(1)$ is also a $\GL_2$-equivariant class by definition. As far as $\pi_{1,*}(h_0)$ is concerned, this is clearly not $\GL_2$-equivariant. Nevertheless we can consider
	 $$\pi_ {1,*}(\xi_1)=\pi_{1,*}(h_0)+t_1\pi_{1,*}(1)$$
	which is a $\GL_2$ -equivariant class.
	
	We can describe $\pi_{1,*}(1)$ geometrically. In fact, it is the fundamental class of the locus describing forms $f$ such that $f$ has a root with multiplicity at least $k$. This locus is strictly related to locus of $A_k$-singularities in the moduli stack of cyclic covers of the projective line of degree $2$.
\end{remark}

We denote by $I$ the ideal generated by the two elements described in the previous remark. First of all, we prove that almost all the pushforwards we need to compute are in this ideal. 

\begin{proposition}
	We have that
	$$ \pi_{r,*}(h_0\dots h_m) \in I $$
	for every $1\leq r \leq N/k$ and $0\leq m \leq r-1$. 
\end{proposition}
\begin{proof}
	\Cref{lem:pi-r-1} implies that it is enough to prove that $(m+1)k+l-1\geq k-1$ for every $l=0,\dots, d(k-1)$ where $d=r-(m+1)$, because it implies that every factor of $\pi_{r,*}(h_0\dots h_m)$ is divisible by $h_0\dots h_{k-1}$. This follows from $m \geq 0$. 
\end{proof}

Therefore it only remains to prove  that $\pi_{r,*}(1)$ is in the ideal $I$ for $r\geq 2$. To do so, we need to prove some preliminary results.

\begin{proposition}\label{prop:square-power}
	We have the following equality 
	$$ h_0^2 \dots h_{n-1}^2 h_n \dots h_{m-1} = \sum_{s=0}^n (-1)^s s! \binom{n}{s}\binom{m}{s}\tau^s h_0\dots h_{m+n-s-1}$$ 
	in the $T$-equivariant Chow ring of $\PP^N$ for every $n\leq m$.
\end{proposition}

\begin{proof}
	Denote by $a_{n,m}$ the left term of the equality. Because we have the identity $h_i=h_j+(j-i)\tau$ in the $T$-equivariant Chow ring of $\PP^N$, we have the following formula
	$$ a_{n,m}=a_{n-1,m+1}-(m-n+1)\tau a_{n-1,m}$$ 
	which gives us that $a_{n,m}$ is uniquely determined from the elements $a_{0,j}$ for $j\leq N$. This implies that it is enough to prove that the formula in the statement verifies the recursive formula above. This follows from straightforward computation.
\end{proof}

Before going forward with our computation, we recall the following combinatorial fact.

\begin{lemma}\label{lem:comb}
	For every pair of non-negative integers $k,m\leq N$ we have that
	$$\sum_{l=0}^{k-1} (-1)^l \binom{m}{l}\binom{N-l}{k-1-l}=\binom{N-m}{k-1}.$$ 
\end{lemma}

We are going to use it to prove the following result.

\begin{proposition}\label{prop:square-h}
	For every non-negative integer $ t\leq k-1$, we have the following equality
	$$ h_0\dots h_{t-1} \pi_{1,*}(1) = \sum_{f=0}^{k-1} \frac{(N-f-t)!}{(N-k-t)!} \binom{k}{f}\tau^{k-1-f}h_0\dots h_{t+f-1} + I$$ 
	in the $T$-equivariant Chow ring of $\PP^N$. Again, for $t=0$, we end up with the formula for $\pi_{1,*}(1)$.
\end{proposition}
\begin{proof}
	The left hand side of the equation in the statement can be written as
	\begin{equation*}
		\begin{split}
			\sum_{l=0}^{t}(k-l)!\binom{k}{l}\binom{N-k}{N-l}\tau^{k-1-l}h_0^2\dots h_{l-1}^2 h_l\dots h_{t-1} + \\ + \sum_{l=t+1}^{k-1}(k-l)!\binom{k}{l}\binom{N-k}{N-l}\tau^{k-1-l}h_0^2\dots h_{t-1}^2 h_t \dots h_{l-1}.;
		\end{split}
	\end{equation*}
	see \Cref{rem:gener}.
	If we apply \Cref{prop:square-power} to the two sums, we get 
	$$ \sum_{l=0}^t \sum_{s=0}^l (-1)^s\frac{k! (N-l)! t!}{(k-l)! (N-k)! s! (l-s)! (t-s)!} \tau^{k-1-l+s}h_0\dots h_{l+t-s-1}$$ 
	and 
	$$ \sum_{l=t+1}^t \sum_{s=0}^t (-1)^s\frac{k! (N-l)! t!}{(k-l)! (N-k)! s! (l-s)! (t-s)!} \tau^{k-1-l+s}h_0\dots h_{l+t-s-1}; $$
	if we exchange the sums in each factor and put everything together we end up with
	$$ \sum_{s=0}^t \sum_{l=s}^{k-1} (-1)^s\frac{k! (N-l)! t!}{(k-l)! (N-k)! s! (l-s)! (t-s)!} \tau^{k-1-l+s}h_0\dots h_{l+t-s-1}. $$
	Shifting the inner sum and setting $f:=l-s$, we get
	$$ \sum_{s=0}^t \sum_{f=0}^{k-1-s} (-1)^s\frac{k! (N-s-f)! t!}{(k-s-f)! (N-k)! s! f! (t-s)!} \tau^{k-1-f}h_0\dots h_{l+f-1}.$$ 
	Notice that we can extend the inner sum up to $k-1$ as all the elements we are adding are in the ideal $I$. Therefore we exchange the sums again and get 
	$$ \sum_{f=0}^{k-1} (-1)^s\frac{k!}{f!} \tau^{k-1-f}h_0\dots h_{l+f-1}\Big( \sum_{s=0}^t (-1)^s \binom{N-s-f}{k-s-f}\binom{t}{s}\Big) + I $$
	and we can conclude using \Cref{lem:comb}.
\end{proof}

We state now the last technical lemma.

\begin{lemma}\label{lem:tec}
	If we define $\Gamma_t$ to be the element 
	$$ \frac{(N-t)!}{(N-2k+1)!} \tau^{2(k-1)-t}h_0\dots h_{t-1}$$
	in the $T$-equivariant Chow ring of $\PP^N$, we have that $\Gamma_t \in I$ for every $t\leq k-1$. 
\end{lemma}

\begin{proof}
	We proceed by induction on $m=k-1-t$. The case $m=0$ follows from the previous proposition. 
	
	Suppose that $\Gamma_s \in I$ for every $s\geq k-t$. If we consider the element in $I$ 
	$$ \frac{(N-k-t)!}{(N-2k+1)!}\tau^{k-1-t}h_0 \dots h_{t-1} \pi_{1,*}(1)$$ 
	we can apply the previous proposition again and get 
	$$ \sum_{f=0}^{k-1} \binom{k}{f}\frac{(N-f-t)!}{(N-2k+1)!} \tau^{2(k-1)-(f+t)}h_0\dots h_{t+f-1} \in I$$ 
	which is the same as 
	$$ \sum_{f=0}^{k-1} \binom{k}{f} \Gamma_{f+t} \in I$$. The statement follows by induction.
\end{proof}
\begin{remark}\label{rem:nec}
	It is important to notice that more is true, the same exact proof shows us that 
	$$ \Gamma_t \in \binom{2(k-1)-t}{k-1-t} \cdot I$$ 
	for every $t \leq k-1$. This will not be needed, except for the case $t=0$, where this implies that in particular $\Gamma_0 \in 2 \cdot I$.
\end{remark}
Before going to prove the final proposition, we recall the following combinatorial fact.

\begin{lemma}\label{lem:comb2}
	We have the following numerical equality 
	$$ \sum_{j_1+\dots +j_r=l}^{0\leq j_s \leq k-1} \binom{k}{j_1}\dots \binom{k}{j_r} = \binom{rk}{l}$$
	for every $l\leq k-1$. In particular in our situation we have $$\alpha_{l}^{k,r}=\binom{rk}{l}$$
	for $l\leq k-1$. 
\end{lemma}

Finally, we are ready to prove the last statement.

\begin{proposition}
	We have $\pi_{r,*}(1)$ is contained in the ideal $I$ for $r \geq 2$. 
\end{proposition}

\begin{proof}
	Notice that 
	\begin{equation*}
		\begin{split}
			\pi_{r,*}(1)&=\sum_{l=0}^{r(k-1)} \alpha_l^{(k,r)}\beta_l^{(k,-1,r)}\tau^{r(k-1)-l}h_0\dots h_{l-1}=\\ & =\sum_{l=0}^{k-1} \alpha_l^{(k,r)}\beta_l^{(k,-1,r)}\tau^{r(k-1)-l}h_0\dots h_{l-1} + I
		\end{split}
	\end{equation*}
	therefore we have to study the first $k-1$ elements of the sum. Using \Cref{lem:comb2}, we get the following chain of equalities modulo the ideal I;
	\begin{equation*}
		\begin{split}
			\pi_{r,*}(1)&= \sum_{l=0}^{k-1} \binom{rk}{l}\frac{(N-l)!}{(N-rk)!r!}\tau^{r(k-1)-l}h_0\dots h_l= \\ & =
			\sum_{l=0}^{k-1}\binom{rk}{l}\frac{(N-2k+1)!}{(N-rk)!r!}\tau^{(r-2)(k-1)}\Gamma_l;
		\end{split}		
	\end{equation*}
	therefore it remains to prove that the coefficient
	$$
	\binom{rk}{l}\frac{(N-2k+1)!}{(N-rk)!r!}
	$$
	is an integer for every $r \geq 2$ and any $l\leq k-1$. First of all, we notice that this is the same as
	$$\binom{rk}{l}\binom{N-rk+r}{r}\frac{(N-2k+1)!}{(N-rk+r)!}$$
	which implies that for $r\geq 3$ and $l\leq k-1$ this is an integer. It remains to prove the statement for $r=2$, i.e. to prove that the number 
	$$ \binom{2k}{l}\frac{N-2k+1}{2}$$
	is an integer. Notice that it is clearly true for $l\geq 1$ but not for $l=0$. However, we have that $\Gamma_0 \in 2\cdot I$ by \Cref{rem:nec}, therefore we are done.
\end{proof}

\begin{corollary}
	The ideal generated by the relations induced by $\Delta^k$ in $\PP^N$ is generated by the two elements $\pi_{1,*}(1)$ and $\pi_{1,*}(h_0)$. See \Cref{rem:gener} for the explicit description.
\end{corollary}

\addcontentsline{toc}{chapter}{Bibliography}
\bibliographystyle{plain}
\bibliography{Bibliografia1}

\end{document}